\documentclass[reqno]{amsart}

\usepackage{amsmath}
\usepackage{amsthm}
\usepackage{amsfonts}
\usepackage{amssymb}
\usepackage{enumerate}
\usepackage{enumitem}
\usepackage[pdftex]{graphicx}
\usepackage{epstopdf}
\usepackage{todonotes}
\usepackage{etoolbox}
\usepackage{verbatim}
\usepackage{mathtools}

\newcommand{\indicator}[1]{\ensuremath{\mathbf{1}_{\{#1\}}}}
\newcommand{\oindicator}[1]{\ensuremath{\mathbf{1}_{{#1}}}}

\newbool{HideDetails}
\setbool{HideDetails}{true}

\newif\ifdetail

\newenvironment{Details}{}{}

\ifbool{HideDetails}{\AtBeginEnvironment{Details}{\comment}%
\AtEndEnvironment{Details}{\endcomment}}{}

\DeclareMathOperator{\var}{Var}
\DeclareMathOperator{\tr}{tr}

\DeclareMathOperator{\rank}{rank}

\newcommand{\Prob}{\mathbb{P}}
\newcommand{\E}{\mathbb{E}}
\newcommand{\C}{\mathbb{C}}
\renewcommand{\P}{\mathbb{P}}
\renewcommand\Re{\operatorname{Re}}
\renewcommand\Im{\operatorname{Im}}
\newcommand{\eps}{\varepsilon}

\def\R{\mathbb{R}}

\theoremstyle{plain}
  \newtheorem{theorem}{Theorem}[section]
  
  \newtheorem{lemma}[theorem]{Lemma}
  \newtheorem{corollary}[theorem]{Corollary}
  
  \newtheorem{proposition}[theorem]{Proposition}

\theoremstyle{definition}
  \newtheorem{definition}[theorem]{Definition}
  
  \newtheorem{assumption}[theorem]{Assumption}

\theoremstyle{remark}
  \newtheorem{remark}[theorem]{Remark}

\newcommand{\Var}{\text{Var}}
\newcommand{\lnorm}{\left\lVert}
\newcommand{\rnorm}{\right\rVert}

\newcommand{\blmat}[2]{\mathcal{#1}_{#2}}

\begin{document}
\title[Gaussian fluctuations for linear eigenvalue statistics]{Gaussian fluctuations for linear eigenvalue statistics of products of independent iid random matrices} 

\author[N. Coston]{Natalie Coston}
\address{Department of Mathematics, University of Colorado at Boulder, Boulder, CO 80309 }
\email{natalie.coston@colorado.edu}

\author[S. O'Rourke]{Sean O'Rourke}
\address{Department of Mathematics, University of Colorado at Boulder, Boulder, CO 80309 }
\email{sean.d.orourke@colorado.edu}

\begin{abstract}
Consider the product $X = X_{1}\cdots X_{m}$ of $m$ independent $n\times n$ iid random matrices.  When $m$ is fixed and the dimension $n$ tends to infinity, we prove Gaussian limits for the centered linear spectral statistics of $X$ for analytic test functions.  We show that the limiting variance is universal in the sense that it does not depend on $m$ (the number of factor matrices) or on the distribution of the entries of the matrices.  The main result generalizes and improves upon previous limit statements for the linear spectral statistics of a single iid matrix by Rider and Silverstein as well as Renfrew and the second author.  
\end{abstract}

\maketitle

\setcounter{tocdepth}{2}
\tableofcontents

\newpage

\section{Introduction and Background Material}
\label{Sec:Intro}

This paper is concerned with fluctuations of linear eigenvalue statistics for products of random matrices with independent and identically distributed (iid) entries.  

\begin{definition}[iid random matrix]
Let $\xi$ be a complex-valued random variable.  We say $X_n$ is an $n \times n$ \emph{iid random matrix} with atom variable $\xi$ if $X_n$ is an $n \times n$ matrix whose entries are iid copies of $\xi$.  
\end{definition}

Recall that the \textit{eigenvalues} of an $n \times n$ matrix $M_n$ are the roots in $\mathbb{C}$ of the characteristic polynomial $\det (zI - M_n)$, where $I$ is the identity matrix.  We let $\lambda_1(M_n), \ldots, \lambda_n(M_n)$ denote the eigenvalues of $M_n$ counted with (algebraic) multiplicity. The \textit{empirical spectral measure} $\mu_{M_n}$ of $M_n$ is given by
\begin{equation*} 
\mu_{M_n} := \frac{1}{n} \sum_{j=1}^n \delta_{\lambda_j(M_n)}. 
\end{equation*}
If $M_n$ is a random $n \times n$ matrix, then $\mu_{M_n}$ is also random.  In this case, we say $\mu_{M_n}$ converges \textit{weakly in probability} (resp. \textit{weakly almost surely}) to another Borel probability measure $\mu$ on the complex plane $\mathbb{C}$ if, for every bounded and continuous function $f:\mathbb{C} \to \mathbb{C}$, 
$$ \int_{\mathbb{C}} f d \mu_{M_n} \longrightarrow \int_{\mathbb{C}} f d \mu $$
in probability (resp. almost surely) as $n \to \infty$.  

For iid random matrices whose atom variable has finite variance, the limiting behavior of the empirical spectral measure is described by the circular law.    
Recall that the Hilbert-Schmidt norm $\|M\|_2$ of a matrix $M$ is defined by the formula
\begin{equation} \label{eq:def:hs}
	\|M\|_2 := \sqrt{ \tr (M M^\ast) } = \sqrt{ \tr (M^\ast M)}. 
\end{equation}

\begin{theorem}[Circular law; Corollary 1.12 from \cite{TVesd}] \label{thm:circ}
Let $\xi$ be a complex-valued random variable with mean zero and unit variance.  For each $n \geq 1$, let $X_n$ be an $n \times n$ iid random matrix with atom variable $\xi$, and let $A_n$ be a deterministic $n \times n$ matrix.  If $\rank(A_n) = o(n)$ and $\sup_{n \geq 1} \frac{1}{n} \|A_n\|^2_2 < \infty$, then the empirical measure $\mu_{\frac{1}{\sqrt{n}} X_n + A_n}$ of $\frac{1}{\sqrt{n}} X_n + A_n$ converges weakly almost surely to the uniform probability measure on the unit disk centered at the origin in the complex plane as $n \to \infty$.  
\end{theorem}

This result appears as \cite[Corollary 1.12]{TVesd}, but is the culmination of work by many authors including \cite{Bcirc, BSbook, Ed-cir, Gi, G1, G2, GTcirc, M, M:B, PZ, TVcirc, TVbull, TVesd}.  We refer the interested reader to the survey \cite{BC} for further details.

%
%


\subsection{Products of Independent iid Matrices}

The result presented in this paper focuses not on a single iid random matrix, but instead on the product of several independent iid matrices.  The analogue of the circular law (Theorem \ref{thm:circ}) in this case has been derived by several authors \cite{GTprod, ORSV, OS} under various assumptions on the factor matrices; the version presented below is from \cite{ORSV}. Similar results are stated in \cite{G4}.

\begin{theorem}[Theorem 2.4 from \cite{ORSV}] \label{thm:ORSV}
Let $m \geq 1$ be an integer and $\tau > 0$.  Let $\xi_1, \ldots, \xi_m$ be real-valued random variables with mean zero, and assume, for each $1 \leq k \leq m$, $\xi_k$ has nonzero variance $\sigma^2_k$ and satisfies $\E|\xi_k|^{2 + \tau} < \infty$.  For each $n \geq 1$ and $1 \leq k \leq m$, let $X_{n,k}$ be an $n \times n$ iid random matrix with atom variable $\xi_k$, and let $A_{n,k}$ be a deterministic $n \times n$ matrix.  Assume $X_{n,1}, \ldots, X_{n,m}$ are independent.  If
$$ \max_{1 \leq k \leq m} \rank(A_{n,k}) = O(n^{1 - \eps}) \quad \text{and} \quad \sup_{n \geq 1} \max_{1 \leq k \leq m} \frac{1}{n} \|A_{n,k}\|_2^2 < \infty $$
for some $\eps > 0$, then the empirical spectral measure $\mu_{P_n}$ of the product
$$ P_n :=  \left( \frac{1}{\sqrt{n}} X_{n,1} + A_{n,1} \right)\left( \frac{1}{\sqrt{n}} X_{n,2} + A_{n,2} \right)\cdots \left( \frac{1}{\sqrt{n}} X_{n,m} + A_{n,m} \right) $$
converges weakly almost surely to a (non-random) probability measure $\mu_m$ as $n \to \infty$.  Here, the probability measure $\mu_m$ is absolutely continuous with respect to Lebesgue measure on $\mathbb{C}$ with density
\begin{equation} \label{eq:density}
    \varphi_m(z) := \left\{
     \begin{array}{rr}
       \frac{1}{m \pi} \sigma^{-2/m} |z|^{\frac{2}{m} - 2}, & \text{if } |z| \leq \sigma, \\
       0, & \text{if } |z| > \sigma,
     \end{array}
   \right. 
\end{equation}
where $\sigma := \sigma_1 \cdots \sigma_m$.  
\end{theorem}

\begin{figure}[t]
	\includegraphics[scale=.35]{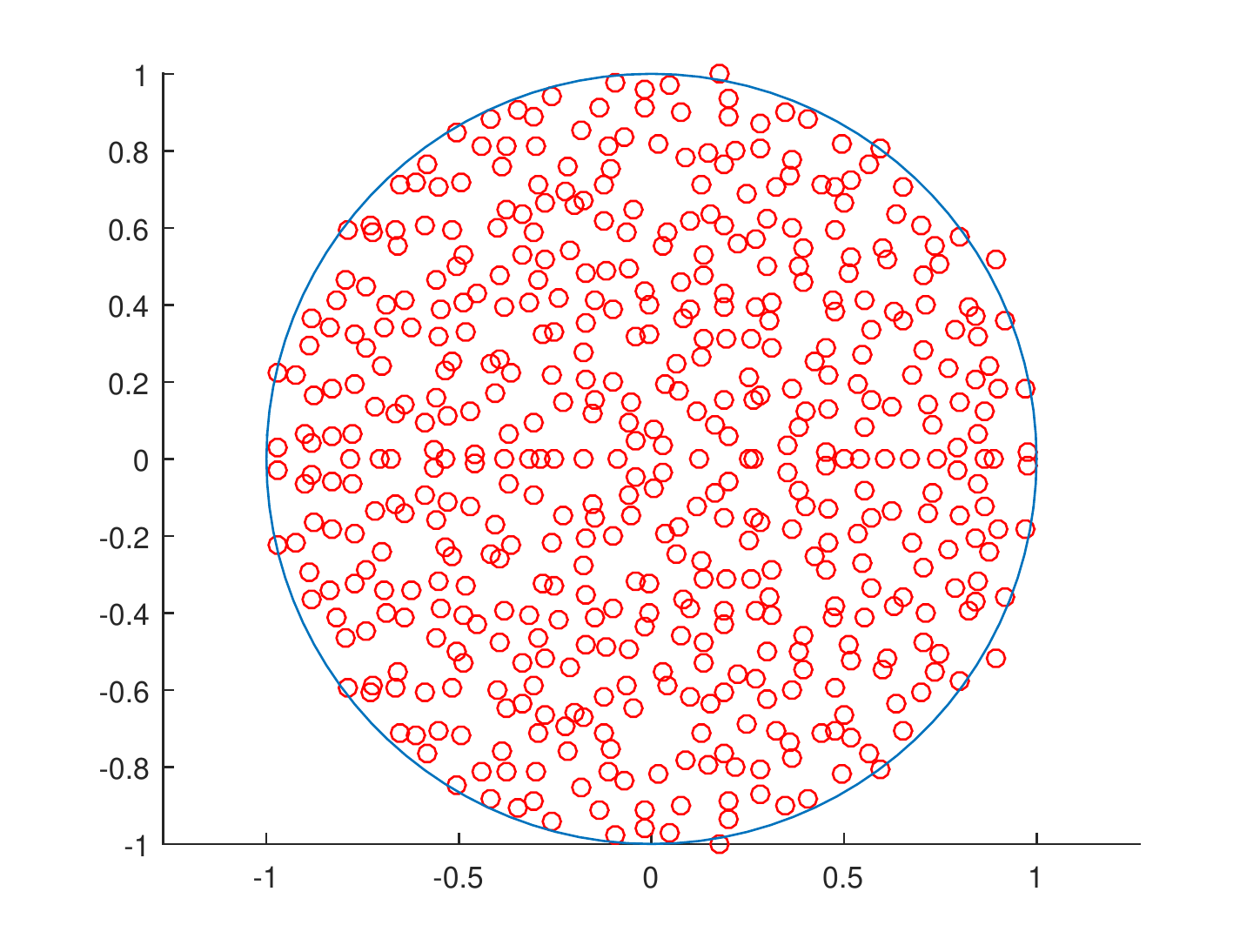}
	\includegraphics[scale=.35]{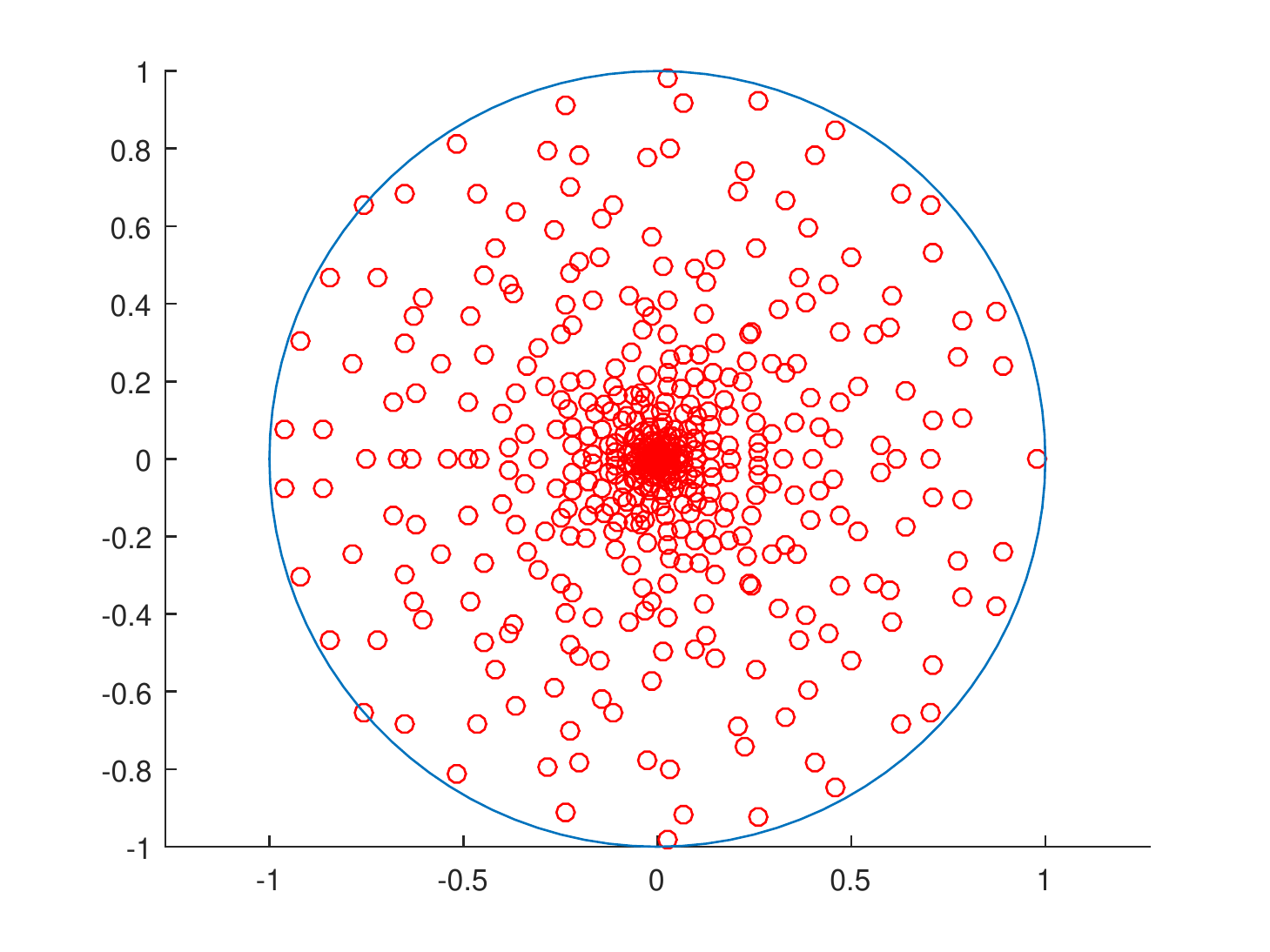}
	\caption{The leftmost figure shows the eigenvalues, denoted by the small circles, of a single $500 \times 500$ iid random matrix $\frac{1}{\sqrt{500}}X_{500}$ with Gaussian entries. The rightmost figure shows the eigenvalues, denoted by small circles, of a product of four independent $500\times 500$ random matrices, each scaled by $\frac{1}{\sqrt{500}}$, where the entries in each random matrix are independent iid Gaussian random variables. }
	\label{Fig:NoOutliers}
\end{figure}

\begin{remark}
When $\sigma = 1$, the density in \eqref{eq:density} is easily related to the circular law (Theorem \ref{thm:circ}).  Indeed, in this case, $\varphi_m$ is the density of $\psi^m$, where $\psi$ is a complex-valued random variable uniformly distributed on the unit disk centered at the origin in the complex plane.  
\end{remark}

Theorem \ref{thm:ORSV} can be viewed as a generalization of the circular law (Theorem \ref{thm:circ}).  Indeed, $\mu_1$ is simply the uniform measure on a disk of radius $\sigma$ centered at the origin.  We emphasis here that the limiting empirical spectral measure $\mu_m$ depends on $m$ and is different for each integer $m$.  Figure \ref{Fig:NoOutliers} provides a numerical illustration of Theorem \ref{thm:ORSV}.

The random matrix theory literature contains many papers concerning products of independent matrices with Gaussian entries; we refer the reader to \cite{ARRS, AB,ABK, AIK, AIK2, AKW, AS, BJW, BNS, Bsurv, BJLNS, F, F2, I, IK, KZ, S} and references therein.  Some other models of products and sums of random matrices have also been considered in \cite{B}.  

Recently in \cite{N2}, Nemish proved a local law version of Theorem \ref{thm:ORSV} up to the optimal scale, under the assumption that the entries of each iid matrix have subexponential decay.  Nemish's local law has also been extended by G\"{o}tze, Naumov, and Tikhomirov \cite{GNT} to include the case where the entries do not have subexponential decay but instead have finite $4 + \tau$ moment for some fixed $\tau > 0$.  The universality of the local correlation functions for the eigenvalues of such product matrices was recently established in \cite{KOV}.  

\begin{Details}
	Akemann and Strahov examine product matrix processes as multi-level point process formed by the singular values of the product of random matrices in \cite{AS2}. In this result, the matrices are complex and not necessarily independent of one another. The authors prove that despite the dependence between matrices, the processes still converges as multi-level determinantal processes and they provide the formula for correlation functions for these processes.
\end{Details}

\begin{Details}
In this note, we will consider the case when $A_{n,k}$ is the zero matrix, so we have a product of unperturbed matrices. Theorem \ref{thm:ORSV} is a special case of \cite[Theorem 2.4]{ORSV}. Indeed, \cite[Theorem 2.4]{ORSV} applies to so-called elliptic random matrices, which generalize iid matrices.  Theorem \ref{thm:ORSV} and the results in \cite{ORSV} are stated only for real random variables, but the proofs can be extended to the complex setting.  Similar results have also been obtained in \cite{B, GTprod, OS}.  The Gaussian case was originally considered by Burda, Janik, and Waclaw \cite{BJW}; see also \cite{Bsurv}.  We refer the reader to \cite{AB,ABK, AIK, AIK2, AKW, AS, BJLNS, F, F2, KZ, S} and references therein for many other interesting results concerning products random matrices with Gaussian entries.
\end{Details}

\subsection{Fluctuations of Linear Eigenvalue Statistics}
The linear eigenvalue statistics of a random matrix describe the fluctuations of the spectrum about its limiting distribution.  The uncentered linear eigenvalue statistics for an $n\times n$ matrix $M$ and sufficiently smooth test function $f$ (whose smoothness depends on the matrix ensemble under consideration) is defined by 
\begin{equation}
\tr f(M):=\sum_{i=1}^{n}f(\lambda_{i}(M))
\label{def:trfM}
\end{equation}
where $\lambda_1(M), \ldots, \lambda_n(M)$ denote the eigenvalues of $M$. 

In the classical central limit theorem, sums of $n$ iid random variables have variance on the order of $\sqrt{n}$.  In contrast, the variance of linear spectral statistics for many ensembles of random matrices is often on the order of a constant.  There are many results regarding the fluctuations of linear eigenvalue statistics for various ensembles of random matrices (and under various assumptions on the test functions $f$).  Because the subject is so well studied, we do not give a full treatment here.  We refer the reader to  \cite{AZ,BS:CLT,DS,DE,J,Kphil,KOV,LP,NP,OR:CLT,RiS,Sh,SS,So,SW} and the references therein for further details.  In the discussion below, we will only focus on linear statistics for iid random matrices and their products.

Rider and Silverstein, in the seminal paper \cite{RiS}, established Gaussian fluctuations for the linear eigenvalue statistics of iid random matrices with analytic test functions.  

\begin{theorem}[Theorem 1.1 from \cite{RiS}] \label{thm:RiS}
Let $\xi$ be a complex-valued random variable which satisfies the following conditions. 
\begin{enumerate}[label=\emph{(\roman*)}]
	\item $\E[\xi]=0,\;\;\text{and}\;\;\E[|\xi|^{2}]=1$,
	\item \label{item:cond:2mom} $\E[\xi^{2}]=0$,
	\item \label{item:cond:mom} $\E[|\xi|^{k}]\leq k^{\alpha k}$ for every $k>2$ and some $\alpha>0$,
	\item \label{item:cond:reim} $\Re(\xi)$ and $\Im(\xi)$ possesses a bounded joint density. 
\end{enumerate}
	For each $n\geq 1$, let $X_{n}$ be an $n\times n$ iid random matrix with atom variable $\xi$. Consider test functions $f_{1}, f_{2},\dots ,f_{k}$ analytic in a neighborhood of the disk $\{ z \in \C : |z|\leq 4\}$ and otherwise bounded. Then as $n\rightarrow\infty$, the vector
	\[\left(\tr f_{j}\left(\frac{1}{\sqrt{n}}X_{n}\right)-nf_{j}(0)\right)_{j=1}^{k}\]
	converges in distribution to a mean-zero multivariate Gaussian vector $(F(f_{1}),F(f_{2}),\dots,F(f_{k}))$ with covariances 
	\[\E\left[F(f_{l})\overline{F(f_{m})}\right]=\frac{1}{\pi}\int_{\mathbb{U}}\frac{d}{dz}f_{l}(z)\overline{\frac{d}{dz}f_{m}(z)}d^{2}z,\]
	in which $\mathbb{U}$ is the unit disk centered at the origin and $d^{2}z=d\Re(z)d\Im(z)$.  
\end{theorem}

Theorem \ref{thm:RiS} was later generalized and extended by Renfrew and the second author in \cite{OR:CLT}. The results in \cite{OR:CLT} remove several technical assumptions present in Theorem \ref{thm:RiS}.  Specifically, for the results in \cite{OR:CLT} to hold, conditions \ref{item:cond:2mom} and \ref{item:cond:reim} from Theorem \ref{thm:RiS} are no longer required, and condition \ref{item:cond:mom} is replaced by the finiteness of $\E|\xi|^{6+\tau}$.   In addition, the functions $f_1, \ldots, f_k$ are only required to be analytic in a neighborhood of the disk $\{z \in \C : |z| \leq 1\}$.  More generally, the results in \cite{OR:CLT} also hold for an ensemble of elliptic random matrices which include iid random matrices as a special case.

For products of independent iid random matrices, much less in known.  To the best of the authors' knowledge, the only result for fluctuations of linear eigenvalue statistics for products of iid random matrices is \cite[Theorem 3]{KOV}, which requires the factor matrices to match moments with the complex Ginibre ensemble; we state this result from \cite{KOV} below.    

\begin{theorem}[Theorem 3 from \cite{KOV}] \label{thm:KOV} 
	 Let $f:\C\rightarrow\R$ be a test function with at least two continuous derivatives, supported in the region $\{z\in\C\;:\;\tau_{0}<|z|<1-\tau_{0}\}$ for some fixed $\tau_{0}>0$. Let $m\geq 1$ be an integer and let 
	 \[P_{n}:=n^{-m/2}X_{n,1}\cdots X_{n,m}\]
	 be a matrix product such that each $X_{n,i}$ is an $n\times n$ iid random matrix (which are all jointly independent) with an atom variable $\xi_{i}$ which satisfies the following:
	 \begin{itemize}
	 	\item $\xi_{i}$ has mean zero and unit variance,
	 	\item $\xi_{i}$ has independent real and imaginary parts,
	 	\item $\xi_{i}$ satisfies the subgaussian decay condition that there exist constants $C,c>0$ (independent of $n$) such that for each $t>0$, $\P(|\xi_{i}|>t)\leq Ce^{-ct^{2}}$, and 
	 	\item $\xi_{i}$ matches moments with a standard complex Gaussian random variable to four moments: for all $a,b\geq 0$ such that $a+b\leq 4$, $\E\left[\Re(\xi_{i})^{a}\Im(\xi_{i})^{b}\right]=\E\left[\Re(\zeta)^{a}\Im(\zeta)^{b}\right]$ where $\zeta$ is a standard complex Gaussian random variable.  
 	\end{itemize}
 	Then the centered linear statistic
\[\tr f(P_{n}) -\E\left[\tr f(P_{n})\right]\] 
converges in distribution as $n\rightarrow\infty$ to the mean-zero Gaussian distribution with limiting variance
\[\frac{1}{4\pi}\int_{\mathbb{U}}|\bigtriangledown f(z)|^{2}d^{2}z \]
where $\mathbb{U}$ is the unit disk centered at the origin. 
\end{theorem}

%

\subsection*{Acknowledgments}
The paper is based on a chapter from N. Coston's doctoral thesis, and she would like to thank her thesis committee for their feedback and support.  The authors would also like to thank Philip Wood for providing useful feedback on an earlier draft of the manuscript. S. O'Rourke has been supported in part by NSF grants ECCS-1610003 and DMS-1810500.

\section{Main Results}
\label{Sec:Main}

\subsection{Fluctuations of Linear Eigenvalue Statistics for Product Matrices} 
\label{Subsec:Main:Products}
The main result of this paper is the analogue of Theorem \ref{thm:RiS} (and its generalization in \cite{OR:CLT}) for the product of independent iid random matrices.    

For $1\leq k\leq m$, let $\xi_{k}$ be a random variable which satisfies the following conditions. 
\begin{assumption} \label{assump:4PlusTau}
The real-valued random variables $\xi_1, \ldots, \xi_m$ (not necessarily defined on the same probability space) are said to satisfy Assumption \ref{assump:4PlusTau} if, for each $1 \leq k \leq m$, 
\begin{itemize}
\item $\xi_k$ has mean zero,
\item $\xi_k$ has nonzero variance $\sigma_k^2$, and
\item there exists $\tau>0$ such that $\E|\xi_{k}|^{4+\tau}<\infty$.
\end{itemize}
\end{assumption}

The following theorem is the main result of the paper.

\begin{theorem}[Fluctuations of linear statistics for products of iid random matrices] \label{thm:mainCLT}
Let $m \geq 1$ be a fixed integer, and assume $\xi_1, \ldots, \xi_m$ are real-valued random variables which satisfy Assumption \ref{assump:4PlusTau}.  For each $n \geq 1$, let $X_{n,1}, \ldots, X_{n,m}$ be independent $n \times n$ iid random matrices with atom variables $\xi_1, \ldots, \xi_m$, respectively.  Define the products
\begin{equation} 
P_n := n^{-m/2} X_{n,1} \cdots X_{n,m}
\label{Def:P_n} 
\end{equation}
and 
$$\sigma:=\sigma_{1}\cdots\sigma_{m}.$$
Let $\delta>0$, $s\geq 1$ be a fixed integer, and $f_{1},f_{2},\dots, f_{s}$ be test functions analytic in some neighborhood containing the disk $D_{\delta}:=\{z\in\C\;:\;|z|\leq 1+\delta\}$ and bounded otherwise.  Then there exist deterministic sequences $\mathcal{A}_{n}(f_{1}),\dots ,\mathcal{A}_{n}(f_{k})$ (with $\mathcal{A}_n(f_i)$ depending only  $n,f_{i},$ and the distribution of $\xi_{1},\dots\xi_{m}$) such that the random vector
\begin{equation} 
\left(\tr f_{i}(P_{n}/\sigma)-\mathcal{A}_{n}(f_{i})\right)_{i=1}^{s}
\label{Equ:OrigVector}
\end{equation}
converges in distribution to a mean-zero multivariate Gaussian random vector 
\[(F(f_{1}),\dots, F(f_{s}))\] 
with variance and covariance terms defined by
\begin{equation}\E\left[F(f_{i})F(f_{j})\right]=-\frac{1}{4\pi^{2}}\oint_{\mathcal{C}}\oint_{\mathcal{C}}f_{i}(z)f_{j}(w)(zw-1)^{-2}dzdw \label{Equ:Thm:CLT:Var}
\end{equation}
and \begin{equation}\E\left[F(f_{i})\overline{F(f_{j})}\right]=\frac{1}{4\pi^{2}}\oint_{\mathcal{C}}\oint_{\mathcal{C}}f_{i}(z)\overline{f_{j}(w)}(z\bar{w}-1)^{-2}dzd\bar{w}\label{Equ:Thm:CLT:Covar}
\end{equation}
where $\mathcal{C}$ is the contour around the boundary of the disk $D_{\delta}$.
\end{theorem}

A few remarks concerning Theorem \ref{thm:mainCLT} are in order. Heuristically, we may think of $\mathcal{A}_{n}(f_{i})$ as the centering term, and subtracting this quantity is similar to subtracting the expectation (as was done in Theorem \ref{thm:KOV}) or $n f_i(0)$ (as was done in Theorem \ref{thm:RiS}). For technical reasons, defining this term requires some notation and concepts which will be introduced in the forthcoming sections; see \eqref{Def:A_n} for details. 

While the limiting empirical spectral measure for the product of $m$ iid matrices does depend on $m$ (Theorem \ref{thm:ORSV}),  the variance and covariance terms, \eqref{Equ:Thm:CLT:Var} and \eqref{Equ:Thm:CLT:Covar}, for the fluctuations of the linear eigenvalue statistics do not depend on $m$.  In other words, the variance and covariance terms are the same as in the case of a single iid matrix ($m=1$); indeed,  \eqref{Equ:Thm:CLT:Var} and \eqref{Equ:Thm:CLT:Covar} match the analogous terms appearing in \cite{OR:CLT} for a single iid matrix.  In this sense, the fluctuations of the linear statistics appear to be more universal than the global distribution of the eigenvalues.  In certain cases, these covariance terms can be rewritten in terms of an iterated integral over the real and imaginary parts of $z$ as was done in \cite[Theorem 1.1]{RiS}.

We also remark that Theorem \ref{thm:mainCLT} can be extended to the case where each atom variable is complex-valued under the assumption that the real and imaginary parts of each atom variable are independent.  In this case, the covariance terms in Theorem \ref{thm:mainCLT} would change slightly.  The changes required for the complex case can easily be found by inspecting the proof; we refer the reader to Remarks \ref{Remark:ComplexCase4}, \ref{Remark:ComplexCase1}, \ref{Remark:ComplexCase2}, and \ref{Remark:ComplexCase3} for the details.

While Theorem \ref{thm:KOV} holds for a more general class of test functions than Theorem \ref{thm:mainCLT}, it also requires subgaussian decay and a moment matching condition on the entries.  In particular, Theorem \ref{thm:KOV} does not apply to iid matrices with real-valued entries.  Thus, Theorem \ref{thm:mainCLT}, while restricted to analytic functions, does apply to a much larger class of iid random matrices (such as the real Ginibre ensemble and Bernoulli matrices, whose entries take the values $\pm 1$ with equal probability).

Even for the case of a single iid matrix ($m=1$), Theorem \ref{thm:mainCLT} improves upon the existing results in the literature.  Compared to the main results of \cite{OR:CLT} (which were already an improvement over Theorem \ref{thm:RiS}), Theorem \ref{thm:mainCLT} applies to a more general class of iid matrices while still applying to the same class of test functions.


Figure \ref{Fig:CLT} provides a numerical illustration of Theorem \ref{thm:mainCLT} for various test functions and values of $m$.

\begin{figure}[t]
	\begin{center}
		\includegraphics[scale = 0.24]{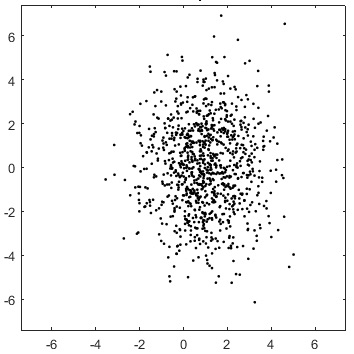}
		\includegraphics[scale = 0.24]{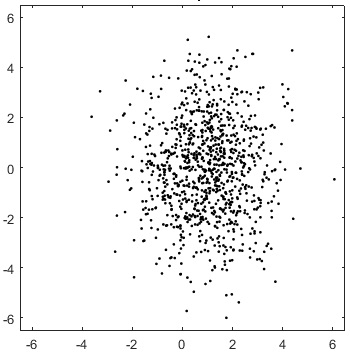}
		\includegraphics[scale = 0.24]{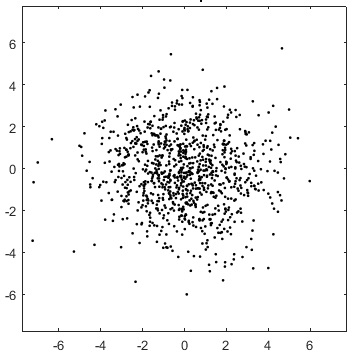}
		\includegraphics[scale = 0.24]{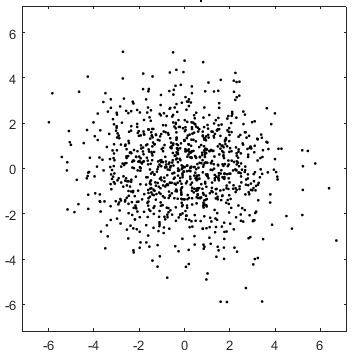}
	\end{center} 
	\caption[Simulations of linear statistics]{This figure provides an illustration of Theorem \ref{thm:mainCLT}. All plots show 1000 observations of the linear statistic in \eqref{def:trfM}. The leftmost plot shows linear statistics computed with a product of three Bernoulli$(-1,1)$ $300\times 300$ matrices scaled by $300^{-3/2}$, and with $f(z)=z^{2}+2\sqrt{-1}z$. The second plot from the left shows linear statistics for a product of ten Bernoulli$(-1,1)$ $300\times 300$ matrices scaled by $300^{-10/2}$, and with $f(z)=z^{2}+2\sqrt{-1}z$. The second plot from the right shows linear statistics for a product of three mean-zero Gaussian $300\times 300$ matrices scaled by $300^{-3/2}$, and $g(z)=\sqrt{-1}z^{3}+z^{2}$. Finally, the rightmost plot shows linear statistics for a product of ten mean-zero Gaussian $300\times 300$ matrices scaled by $300^{-10/2}$, and with $g(z)=\sqrt{-1}z^{3}+z^{2}$.}
	\label{Fig:CLT}
\end{figure}


Since the covariance formulas, \eqref{Equ:Thm:CLT:Var} and \eqref{Equ:Thm:CLT:Covar}, do not depend on $m$, it is an interesting open question to consider the case when the number of product matrices is allowed to depend on $n$, e.g., when $m$ grows (slowly) with $n$.  The proof given below requires $m$ to be fixed, and there are several key bounds which depend on $m$.  If this dependence could be tracked carefully, it may be possible that parts of the proof could be adapted to the case where $m$ grows with $n$.

\subsection{Outline and Overview}

The remainder of the paper is devoted to the proof of Theorem \ref{thm:mainCLT}.  Roughly speaking, the proof follows the main ideas from \cite{RiS,OR:CLT} for studying the linear eigenvalue statistics of a single iid matrix.  However, the product structure of the matrix $P_n$ introduces substantial new difficulties.  For instance, when $m > 1$, the entries of $P_n$ are no longer independent.  To get around this issue, we use a linearization technique.  That is, instead of studying the product matrix $P_n$, we introduce the linearized matrix $\mathcal{Y}$, which is a $mn \times mn$ block matrix defined as follows:
\[ \mathcal{Y} := \frac{1}{\sqrt{n}} \begin{bmatrix} 
				0 &  {X}_{n,1} &  0   &  &\cdots          & 0       \\
                         		0 & 0    & {X}_{n,2} &    0  &  \cdots    & 0        \\      
                           	\vdots &      & \ddots & \ddots &    &      \vdots   \\
                         		0 &    0  &  \cdots   &        &  0 & {X}_{n,m-1} \\
                       		{X}_{n,m} & 0     &  \cdots   &    & 0       &  0      
	\end{bmatrix}, \]
where any block not specified above is assumed to be zero.  
As has been observed previously \cite{BJW, COW, ORSV, OS}, the eigenvalues of $\mathcal{Y}^m$ are the same as the eigenvalues of $P_n$, up to some multiplicity factor.  Thus, the problem can be reduced to studying the linear eigenvalue statistics of $\mathcal{Y}$.  

The main challenge in studying the linear statistics of non-Hermitian random matrices is often computing the limiting variance.  For example, in \cite{RiS, OR:CLT}, for a single iid matrix the variance is computed by deriving a recursive equation, which then must be solved in the limit as $n$ tends to infinity to obtain the limiting variance.  In the case of analyzing the limiting variance for the linear statistics of $\mathcal{Y}$, the block structure of $\mathcal{Y}$ itself introduces new difficulties.  In this case, it does not seem possible to derive a single recursive equation as was done in \cite{RiS, OR:CLT} due to the prescience of so many deterministic zero blocks in $\mathcal{Y}$.  To get around this issue, we instead derive a system of $m$ recursive equations and then solve the system of equations simultaneously.  The derivation and solution of this system of recursive equations for the variance is the main technical advance of the present article and occupies the bulk of the proof.  Interestingly, the limiting variance we derive for the linear statistics of $\mathcal{Y}$ \emph{does} depend on $m$ and appears to have a form not encountered before in random matrix theory.  When this limiting variance is translated back to the variance for the product matrix $P_n$, the dependence on $m$ vanishes.

The paper is organized as follows.  In Section \ref{Sec:PrelimToolsAndNotation}, we present some preliminary results, tools, and notation that will be used throughout the paper. In Section \ref{Sec:Reduction}, we make some preliminary reductions and reduce the problem to the study of the linear statistics of the linearized matrix $\mathcal{Y}$.  Section \ref{Sec:ProofOfReduction} begins the proof of the main result, which by Cauchy's integral formula, involves studying a sequence of stochastic processes involving the trace of the resolvent matrix.  Section \ref{Sec:FiniteDimDist} proves the finite dimensional convergence of this sequence of stochastic processes, and Section \ref{Sec:Tightness} shows that this sequence of stochastic processes is tight, concluding the proof.  These last two sections are based on \cite{RiS, OR:CLT}.  We warn the reader that the last two sections appear to inherit many of the technical challenges present in \cite{RiS, OR:CLT} along with several additional challenges (based on the block structure of the linearized matrix $\mathcal{Y}$, as discussed above). As such, the material presented in these two sections is rather technical and some of the calculations are tedious.  Some appendices follow with auxiliary results.

\section{Preliminary Tools and Notation}
\label{Sec:PrelimToolsAndNotation}

This section is devoted to introducing some additional concepts and notation required for the proofs of our main results.

\subsection{Notation} 
\label{Subsec:PrelimTools:Notation}

We use asymptotic notation (such as $O,o, \Omega$) under the assumption that $n \to \infty$.  In particular, $X= O(Y)$, $Y = \Omega(X)$, $X \ll Y$, and $Y \gg X$ denote the estimate $|X| \leq C Y$, for some constant $C > 0$ independent of $n$ and for all $n \geq C$.  If we need the constant $C$ to depend on a parameter $k$, e.g. $C = C_k$, we indicate this with subscripts, e.g. $X = O_{k}(Y)$, $Y = \Omega_k(X)$, $X \ll_k Y$, and $Y\gg_k X$.  We write $X = o(Y)$ if $|X| \leq c(n) Y$ for some sequence $c(n)$ that goes to zero as $n \to \infty$.  Specifically, $o(1)$ denotes a term which tends to zero as $n \to \infty$.  If we need the sequence $c(n)$ to depend on a parameter $k$, e.g. $c(n) = c_k(n)$, we indicate this with subscripts, e.g. $X = o_k(Y)$.  

Throughout the paper, we view $m$ as a fixed integer.  Thus, when using asymptotic notation, we will allow the implicit constants (and implicit rates of convergence) to depend on $m$ without including $m$ as a subscript (i.e., we will not write $O_m$ or $o_m$).  

An event $E$, which depends on $n$, is said to hold with \emph{overwhelming probability} if $\Prob(E) \geq 1 - O_C(n^{-C})$ for every constant $C > 0$.  We let $\oindicator{E}$ denote the indicator function of the event $E$.  $E^{c}$ denotes the complement of the event $E$.  For $\delta > 0$, $D_\delta$ denotes the disk $\{z \in \mathbb{C} : |z| \leq 1 + \delta\}$.  

For a matrix $M$, we let $\|M\|$ denote the spectral norm of $M$.  $\|M\|_2$ denotes the Hilbert-Schmidt norm of $M$ (defined in \eqref{eq:def:hs}).  We let $I_n$ denote the $n \times n$ identity matrix.  Often we will just write $I$ for the identity matrix when the size can be deduced from context. 

The singular values of a matrix $M$ are the square roots of the eigenvalues of the matrix $M^{*}M$. For an $n\times n$ matrix, we will denote these $s_{1}(M_{n}),\dots ,s_{n}(M_{n})$. Note that all singular values real and non-negative, so we let $s_{1}(M_{n})\geq \cdots \geq s_{n}(M_{n})$ by convention. 

We write a.s., a.a., and a.e. for almost surely, Lebesgue almost all, and Lebesgue almost everywhere respectively.  We use $\sqrt{-1}$ to denote the imaginary unit and reserve $i$ as an index.

We let $C$ and $K$ denote constants that are non-random and may take on different values from one appearance to the next.  The notation $K_p$ means that the constant $K$ depends on another parameter $p$.  We allow these constants to depend on the fixed integer $m$ without explicitly denoting or mentioning this dependence.

\subsection{Linearization}
Let $M_1, \ldots, M_m$ be $n \times n$ matrices, and suppose we wish to study the product $M_1 \cdots M_m$.  A useful trick is to linearize this product and instead consider the $mn \times mn$ block matrix 
\begin{equation} \label{def:M}
	\mathcal{M} := \left[\begin{array}{ccccc}
	0 & M_{1} & 0 & \cdots & 0\\
	0 & 0 & M_{2} & \dots & 0\\
	\vdots & \vdots & \vdots & \ddots & \vdots\\
	0 & 0 & 0 & \cdots & M_{m-1}\\
	M_{m} & 0 & 0 & \dots & 0
	\end{array}\right].
\end{equation}
The following proposition relates the eigenvalues of $\mathcal{M}$ to the eigenvalues of the product $M_1 \cdots M_m$.  We note that similar linearization tricks have been used previously; see, for example, \cite{A, BJW, COW, ORSV, OS} and references therein.  

\begin{proposition}[Proposition 4.1 from \cite{COW}]\label{prop:linear}
Let $M_1, \ldots, M_m$ be $n \times n$ matrices.  Let $P := M_1 \cdots M_m$, and assume $\mathcal{M}$ is the $mn \times mn$ block matrix defined in \eqref{def:M}.  Then
$$ \det(\mathcal{M}^m - z I) = [\det( P - z I)]^m $$
for every $z \in \mathbb{C}$.  In other words, the eigenvalues of $\mathcal{M}^m$ are the eigenvalues of $P$, each with multiplicity $m$.  
\end{proposition}

\begin{Details}
\begin{proof}
A simple computation reveals that $\mathcal{M}^m$ is a block diagonal matrix of the form
$$ \mathcal{M}^m = \begin{bmatrix} 
				Z_1 &  & 0 \\
				 & \ddots &  \\
				0 &  & Z_m 
			\end{bmatrix}, $$
where $Z_1 := P$ and
$$ Z_k := M_k \cdots M_m M_1 \cdots M_{k-1} $$
for $1 < k \leq m$.  Since each product $Z_2, \ldots, Z_m$ has the same characteristic polynomial 
as $P$, it follows that
$$ \det( \mathcal{M}^m - zI) = \prod_{k=1}^m \det(Z_k - z I) = [\det (P - zI)]^m $$
for all $z \in \mathbb{C}$.  
\end{proof}
We will exploit Proposition \ref{prop:linear} many times in the coming proofs. 
	For instance, in order to study the product $X_{n,1} \cdots X_{n,m}$, we will consider a truncated version of the  $mn \times mn$ block matrix
	\begin{equation*}
	\frac{1}{\sqrt{n}}\begin{bmatrix} 
	0 &  {X}_{n,1} &     &            & 0       \\
	0 & 0    & {X}_{n,2} &            & 0        \\      
	&      & \ddots & \ddots     &         \\
	0 &      &     &          0 & {X}_{n,m-1} \\
	{X}_{n,m} &      &     &            &  0      
	\end{bmatrix}
	\end{equation*}
	and its resolvent
	\begin{equation*}
	\left( \frac{1}{\sqrt{n}}\begin{bmatrix} 
	0 &  {X}_{n,1} &     &            & 0       \\
	0 & 0    & {X}_{n,2} &            & 0        \\      
	&      & \ddots & \ddots     &         \\
	0 &      &     &          0 & {X}_{n,m-1} \\
	{X}_{n,m} &      &     &            &  0      
	\end{bmatrix} - z I \right)^{-1},
	\end{equation*}
	defined for $z \in \mathbb{C}$ provided $z$ is not an eigenvalues of the above block matrix. 
\end{Details}
\subsection{Matrix Notation}
Here and in the sequel, we will deal with matrices of various sizes.  The most common dimensions are $n \times n$ and $N \times N$, where we take $N := mn$.  Unless otherwise noted, we denote $n \times n$ matrices by capital letters (such as $M, X$) and larger $N \times N$ matrices using calligraphic symbols (such as $\mathcal{M}$, $\mathcal{Y}$).  

If $M$ is an $n \times n$ matrix and $1 \leq i,j \leq n$, we let $M_{ij}$ and $M_{(i,j)}$ denote the $(i,j)$-entry of $M$.  Similarly, if $\mathcal{M}$ is an $N \times N$ matrix, we let $\mathcal{M}_{ij}$ and $\mathcal{M}_{(i,j)}$ denote the $(i,j)$-entry of $\mathcal{M}$ for $1 \leq i,j \leq N$. Sometimes we will deal with $n \times n$ matrices notated with a subscript such as $M_n$.  In this case, for $1 \leq i,j \leq n$, we write $(M_n)_{ij}$ or $M_{n,(i,j)}$ to denote the $(i,j)$-entry of $M_n$. 



\subsection{Singular Value Inequalities and Useful Identities}
Let $M$ denote an $n\times n$ matrix. We often want to know about the smallest and largest singular values of a matrix. Recall $s_{1}(M)\geq\dots\geq s_{n}(M)$ denote the singular values of the matrix $M$. 
\begin{proposition}
	Let $M$ be an $n\times n$ matrix, and assume that $E \subset \C$ and $c > 0$.  If $$\inf_{z\in E} s_{n}(M-zI)\geq c,$$ then no eigenvalue of $M$ is contained in $E$ and 
	$$\sup_{z\in E}\lnorm G(z)\rnorm \leq \frac{1}{c}$$
	where $G(z)=(M-zI)^{-1}$ is the resolvent of $M$.
	\label{Prop:LargeAndSmallSingVals} 
\end{proposition}
The proof of Proposition \ref{Prop:LargeAndSmallSingVals} follows easily by observing that the operator norm of the resolvent can be bounded above by $1/s_{n}(M-zI)$; similar bounds were used in \cite{OR:CLT}.   

We will make use of the Sherman--Morrison rank one perturbation formula (see \cite[Section 0.7.4]{HJ}). Suppose $A$ is an invertible square matrix, and let $u$, $v$ be vectors. If $1+v^{*}A^{-1}u\neq 0$, then
\begin{equation}
(A+uv^{*})^{-1}=A^{-1}-\frac{A^{-1}uv^{*}A^{-1}}{1+v^{*}A^{-1}u}
\label{equ:ShermanMorrison1}
\end{equation}
and 
\begin{equation}
(A+uv^{*})^{-1}u=\frac{A^{-1}u}{1+v^{*}A^{-1}u}.
\label{equ:ShermanMorrison2}
\end{equation}
Also recall the Sherman--Morrison--Woodbury formula (for example, \cite[Theorem 1.1]{D}), which states that for an invertible $N\times N$ matrix $A$ and $a\times N$ matrices $V,U$ for some fixed $a<N$, 
\begin{equation}
(A+UV^{T})^{-1}U=A^{-1}U(I_{a}+V^{T}A^{-1}U)^{-1}
\label{Equ:ShermanMorrisonWoodbury}
\end{equation}
provided $I_{a}+V^{T}A^{-1}U$ is invertible.

Another identity we will make use of is the Resolvent Identity, which states that \begin{equation}
A^{-1}-B^{-1}=A^{-1}(B-A)B^{-1}
\label{Equ:ResolventIndentity}
\end{equation}
whenever $A$ and $B$ are invertible.

We also use Weyl's inequality for the singular values (see, for example, \cite[Problem III.6.5]{Bhatia}), which states that for $n\times n$ matrices $A$ and $B$, 
\begin{equation} \label{Equ:weyl}
\max_{1 \leq i \leq n}\left|s_{i}(A)-s_{i}(B)\right|\leq \lnorm A-B\rnorm.
\end{equation}

\begin{Details}
\subsection{Overview of Proof}
\label{Subsec:PrelimTools:Overview}
In order to prove that 
\[\left(\tr f_{i}(P_{n}/\sigma)\oindicator{E_{n}}-\E[\tr f_{i}(P_{n}/\sigma)\oindicator{E_{n}}]\right)_{i=1}^{s}\]
converges to the appropriate mean-zero multivariate Gaussian process under the assumptions of Theorem \ref{thm:mainCLT}, we first make some preliminary reductions. In particular, we first use the Cramer--Wold theorem to consider a single test function, then we move to a slightly more general result, Theorem \ref{thm:firstReduction}. We proceed by showing that it is sufficient to consider a rescaled linearized version of truncated matrices, which we will call $\blmat{Y}{n}$ and we define formally in \eqref{Def:Y_n}. We introduce a truncated event $\hat{E}_{n}$ and prove that it is sufficient to show that
$$\tr f(\blmat{Y}{n})\oindicator{\hat{E}_{n}}-\E[\tr f(\blmat{Y}{n})\oindicator{\hat{E}_{n}}]$$
converges in distribution to an appropriate mean zero Gaussian process. We next rewrite this trace using Cauchy's integral formula, 
\begin{align}
&\tr f(\blmat{Y}{n})\oindicator{\hat{E}_{n}}-\E[\ f(\blmat{Y}{n})\oindicator{\hat{E}_{n}}]\notag\\
&=-\frac{1}{2\pi i}\oint_{C}f(z)(\tr (\blmat{Y}{n}-zI)^{-1}\oindicator{\hat{E}{n}}-\E[\tr(\blmat{Y}{n}-zI)^{-1}\oindicator{\hat{E}{n}}])dz
\label{equ:CauchyIntFormula}
\end{align}
where $\mathcal{C}$ is a contour on the boundary of the disk $D_{\delta}$. By viewing
\begin{equation*}
\tr (\blmat{Y}{n}-zI)^{-1}\oindicator{\hat{E}{n}}-\E[\tr(\blmat{Y}{n}-zI)^{-1}\oindicator{\hat{E}{n}}]
\end{equation*}
as a sequence of stochastic process, the continuous mapping theorem implies that it is sufficient to prove that
this sequence of stochastic processes converges in distribution to a mean zero Gaussian process with appropriate covariance structure. In order to prove that this sequence of stochastic processes converges, we prove that finite dimensional distributions of this process converge and that this sequence is  tight in the space of continuous functions on the contour $\mathcal{C}$. 

\end{Details}

\section{Reductions} \label{Sec:Reduction}
In order to prove Theorem \ref{thm:mainCLT}, we will make a series of reductions by truncating the entries in each factor matrix and applying the linearization techniques discussed above.   
\begin{Details}
If it does, then \[\left(\tr f_{i}(P_{n}/\sigma)\oindicator{E_{n}}-\E\left[\tr f_{i}(P_{n}/\sigma)\oindicator{E_{n}}\right]\right)_{i=1}^{s}\] 
must converge to a mean-zero multivariate Gaussian vector as well. It only remains to calculate the covariance structure, which will follow from the variance of $\tr f(P_{n}/\sigma)\oindicator{E_{n}}-\E\left[\tr f(P_{n}/\sigma)\oindicator{E_{n}}\right]$. 
\end{Details} 
\begin{Details}
	Observe that
	\begin{align*}
	\tr f(P_{n}/\sigma)&=\tr\left(\gamma_{1}f_{1}(P_{n}/\sigma)+\gamma_{2}f_{2}(P_{n}/\sigma)+\cdots +\gamma_{s}f_{s}(P_{n}/\sigma)\right)\\
	&=\gamma_{1}\tr\left(f_{1}(P_{n}/\sigma)\right)+\gamma_{2}\tr\left(f_{2}(P_{n}/\sigma)\right)+\cdots +\gamma_{s}\tr\left(f_{s}(P_{n}/\sigma)\right)\\
	&=\gamma_{1}\left(-\frac{1}{2\pi i}\oint_{\mathcal{C}}f_{1}(z)\tr(\mathcal{G}_{n}(z))dz\right)+\gamma_{2}\left(-\frac{1}{2\pi i}\oint_{\mathcal{C}}f_{2}(z)\tr(\mathcal{G}_{n}(z))dz\right)\\
	&\quad+\dots+ \gamma_{s}\left(-\frac{1}{2\pi i}\oint_{\mathcal{C}}f_{s}(z)\tr(\mathcal{G}_{n}(z))dz\right)\\
	&=-\frac{1}{2\pi i}\oint_{\mathcal{C}}\left(\gamma_{1}f_{1}(z)+\gamma_{2}f_{2}(z)+\cdots +\alpha_{s}f_{s}(z)\right)\tr(\mathcal{G}_{n}(z))dz\\
	&=-\frac{1}{2\pi i}\oint_{\mathcal{C}}f(z)\tr(\mathcal{G}_{n}(z))dz.
	\end{align*}
	Thus the covariance structure induced by the function $f$ will determine the variance and covariance of the multi-dimensional Gaussian vector.
\end{Details}
Theorem \ref{thm:mainCLT} will follow from the following result. 

\begin{theorem}
	Let $m\geq 1$ be a fixed integer, and assume $\xi_1, \ldots, \xi_m$ are real-valued random variables which satisfy Assumption \ref{assump:4PlusTau}.  For each $n \geq 1$, let $X_{n,1}, \ldots, X_{n,m}$ be independent $n \times n$ iid random matrices with atom variables $\xi_1, \ldots, \xi_m$, respectively.  Define the products
	$$ P_n := n^{-m/2} X_{n,1} \cdots X_{n,m}\;\text{ and }\;\sigma=\sigma_{1}\cdots\sigma_{m}.$$
	Let $\delta>0$, and let $f$ be analytic in some neighborhood containing the disk $D_{\delta}:=\{z\in\C\;:\;|z|\leq 1+\delta\}$ and bounded otherwise. Then, there exists a constant $c>0$ and a deterministic sequence $\mathcal{A}_n(f)$ (depending on $n$, $f$, and the distribution of $\xi_1, \ldots, \xi_m$) such that the event
	\begin{equation}
	E_{n}:=\left\{\inf_{|z|> 1+\delta/2}s_{n}\left(P_{n}/\sigma-zI\right)\geq c\right\}
	\label{Def:E_n}
	\end{equation}
	holds with probability $1-o(1)$ and as $n\rightarrow\infty$,
	\begin{equation}
	\tr f(P_{n}/\sigma)\oindicator{E_{n}}-\mathcal{A}_{n}(f)
	\label{Equ:SingleFConv}
	\end{equation}
	converges in distribution to a mean zero Gaussian random variable $F(f)$ with covariance structure 
	\begin{equation} 
	\E\left[\left(F(f)\right)^{2}\right]=-\frac{1}{4\pi^{2}}\oint_{\mathcal{C}}\oint_{\mathcal{C}}f(z)f(w)(zw-1)^{-2}dzdw
	\label{Equ:FirstReductionVar1}
	\end{equation}
	and  
	\begin{equation}
	\E\left[F(f)\overline{F(f)}\right] =\frac{1}{4\pi^{2}}\oint_{\mathcal{C}}\oint_{\mathcal{C}}f(z)\overline{f(w)}(z\bar{w}-1)^{-2}dzd\bar{w}
	\label{Equ:FirstReductionVar2}
	\end{equation}
	where $\mathcal{C}$ is the contour around the boundary of the disk $D_{\delta}$.  In addition, the function $f \mapsto \mathcal{A}_n(f)$ is continuous and linear with the property that if $f(z) \in \mathbb{R}$ for all $z \in \mathbb{R} \cap D_{\delta}$, then $\mathcal{A}_n(f)$ is real-valued.  
	\label{thm:firstReduction}
\end{theorem}

We will now prove Theorem \ref{thm:mainCLT} assuming Theorem \ref{thm:firstReduction}. 

\begin{proof}[Proof of Theorem \ref{thm:mainCLT}]
	Assume Theorem \ref{thm:firstReduction}.  It follows from standard least singular value bounds that  there exists a constant $c>0$ such that $E_{n}$ holds with probability $1-o(1)$; the details are presented as Lemma \ref{Lem:E_nOverwhelming} from Appendix \ref{Sec:AppendexEvents}. Thus, the prescience (or lack thereof) of the indicator function in \eqref{Equ:SingleFConv} does not affect the limiting distribution. 
	
	To prove Theorem \ref{thm:mainCLT}, we will invoke the Cramer--Wold device, but we will need to be careful as we are dealing with complex-valued random variables.  If $f$ is analytic in a neighborhood containing the disk $D_\delta$, we can express $f$ as the power series
	\[ f(z) = \sum_{i=0}^\infty a_i z^i \]
	in the same disk.  We then define
	\begin{align*}
		\mathfrak{R}f(z) = \sum_{i=0}^\infty \Re(a_i) z^i \quad \text{and} \quad
		\mathfrak{I}f(z) = \sum_{i=0}^\infty \Im(a_i) z^i,
	\end{align*}
	which are both analytic in $D_\delta$.  (Notice that these are not the real and imaginary parts of $f$, which would not be analytic in $D_\delta$.)  
	
	In order to invoke the Cramer--Wold device and prove Theorem \ref{thm:mainCLT}, we consider
	\[ f(z) = \sum_{l=0}^s \left( \alpha_l \mathfrak{R}f_l(z) + \beta_l \mathfrak{I}f_l(z)  \right) \]
	for some real-valued constants $\alpha_1, \beta_1, \ldots, \alpha_s, \beta_s$.  
	
	As $f(z) \in \mathbb{R}$ for all $z \in \mathbb{R} \cap D_\delta$ and since the non-real eigenvalues of $P_n/\sigma$ come in complex conjugate pairs, it follows that $\tr f(P_n/\sigma) - \mathcal{A}_n(f)$ is real-valued.  Applying the Cramer--Wold device and the convergence of \eqref{Equ:SingleFConv}, we conclude that 
	\[ \left( \tr f_i(P_n/\sigma) - \mathcal{A}_n(f_i) \right)_{i=1}^s \]
	converges to a multivariate Gaussian vector.  The limiting covariances can now be extrapolated from \eqref{Equ:FirstReductionVar1} and \eqref{Equ:FirstReductionVar2}. 
	\begin{Details}
		Indeed, it is assumed that $f_{i}(z)$ is analytic in some neighborhood containing the disk $D_{\delta}$, there exists an expansion $f_{i}(z)=\sum_{j=0}^{\infty}(a_{i})_{j}z^{j}$ for coefficients $(a_{i})_{j}\in \C$. Then define the function $g_{i}(z)=\sum_{j=0}^{\infty}\overline{(a_{i})_{j}}z^{j}$ for each $1\leq i\leq s$ and note that $g_{i}(z)$ is also analytic in the same neighborhood as $f_{i}(z)$. Thus $f_{i}(z)+g_{i}(z)$ and $f_{i}(z)-g_{i}(z)$ are analytic in the same neighborhood as well. Now, since $P_{n}/\sigma$ has real entries, the eigenvalues come in complex conjugate pairs. Since the event $E_{n}$ holds with probability $1-o(1)$ we may say that, for all $1\leq i\leq s$, $\alpha_{i}\tr(f_{i}(P_{n}/\sigma)+g_{i}(P_{n}/\sigma))$ and $\beta_{i}\sqrt{-1}\tr(f_{i}(P_{n}/\sigma)-g_{i}(P_{n}/\sigma))$ are both real for any arbitrary $\alpha_{i},\beta_{i}\in\R$ with probability $1-o(1)$ as well. 
	\end{Details}
	\begin{Details}
	\begin{align*} 
	\E\left[\left(F(f)\right)^{2}\right]&=-\frac{1}{4\pi^{2}}\oint_{\mathcal{C}}\oint_{\mathcal{C}}f(z)f(w)(zw-1)^{2}dzdw\\
	&=-\frac{1}{4\pi^{2}}\oint_{\mathcal{C}}\oint_{\mathcal{C}}(\gamma_{1}f_{1}(z)+\dots+\gamma_{s}f_{s}(z))\\
	&\quad\quad\quad\quad\quad\quad\quad\times(\gamma_{1}f_{1}(w)+\dots+\gamma_{s}f_{s}(w))(zw-1)^{2}dzdw\\
	\end{align*}
	where $\mathcal{C}$ is the contour around the boundary of the disk $D_{\delta}$. A similar expression holds for $\E\left[F(f)\overline{F(f)}\right]$.
	and  
	\begin{align*}
	\E\left[F(f)\overline{F(f)}\right]&=\frac{1}{4\pi^{2}}\oint_{\mathcal{C}}\oint_{\mathcal{C}}f(z)\overline{f(w)}(z\bar{w}-1)^{-2}dzd\bar{w}\\
	&=\frac{1}{4\pi^{2}}\oint_{\mathcal{C}}\oint_{\mathcal{C}}(\gamma_{1}f_{1}(z)+\dots+\gamma_{s}f_{s}(z))\\
	&\quad\quad\quad\quad\quad\quad\quad\times(\overline{\gamma_{1}f_{1}(w)}+\dots+\overline{\gamma_{s}f_{s}(w)})(z\bar{w}-1)^{-2}dzd\bar{w}\\
	\end{align*}
	Note that by selecting $\gamma_{i}=1$ for some fixed $1\leq i\leq L$, and $\alpha_{j}=0$ for all $j\neq i$, this implies 
	\[\E\left[\left(F(f_{i})\right)^{2}\right]=-\frac{1}{4\pi^{2}}\oint_{\mathcal{C}}\oint_{\mathcal{C}}f_{i}(z)f_{i}(w)(zw-1)^{2}dzdw\]
	and  
	\[\E\left[F(f_{i})\overline{F(f_{i})}\right]=\frac{1}{4\pi^{2}}\oint_{\mathcal{C}}\oint_{\mathcal{C}}(f_{i}(z)\overline{f_{i}(w)})(z\bar{w}-1)^{-2}dzd\bar{w}.\]
	Next, if we consider $\gamma_{i}=\gamma_{j}=1$ and all other $\gamma_{k}=0$ for $k\neq i,j$, then we have the limiting variance of
	\[\tr\left(f_{i}(P_{n}/\sigma)+f_{j}(P_{n}/\sigma)\right)\]
	is given by 
	\begin{align*} 
	&-\frac{1}{4\pi^{2}}\oint_{\mathcal{C}}\oint_{\mathcal{C}}\left(f_{i}(z)+f_{j}(z)\right)\left(f_{i}(w)+f_{j}(w)\right)(zw-1)^{2}dzdw\\
	&\quad\quad =-\frac{1}{4\pi^{2}}\oint_{\mathcal{C}}\oint_{\mathcal{C}}\left(f_{i}(z)f_{i}(w)+f_{i}(z)f_{j}(w)\right.\\
	&\quad\quad\quad\quad\quad\quad\quad\quad\quad\quad\left.+f_{j}(z)f_{i}(w)+f_{j}(z)f_{j}(w)\right)(zw-1)^{2}dzdw\\
	&\quad\quad =\E[F(f_{i})F(f_{i})]-\frac{1}{4\pi^{2}}\oint_{\mathcal{C}}\oint_{\mathcal{C}}f_{i}(z)f_{j}(w)(zw-1)^{2}dzdw\\
	&\quad\quad\quad\quad\quad\quad\quad\quad\quad\quad-\frac{1}{4\pi^{2}}\oint_{\mathcal{C}}\oint_{\mathcal{C}}f_{j}(z)f_{i}(w)(zw-1)^{2}+\E[F(f_{j})F(f_{j})]\\
	&\quad\quad =\E[F(f_{i})F(f_{i})]+\E[F(f_{j})F(f_{j})]-\frac{1}{2\pi^{2}}\oint_{\mathcal{C}}\oint_{\mathcal{C}}f_{j}(z)f_{i}(w)(zw-1)^{2}\\
	\end{align*}
	and therefore 
	\[\E\left[F(f_{i})F(f_{j})\right]=-\frac{1}{4\pi^{2}}\oint_{\mathcal{C}}\oint_{\mathcal{C}}f_{j}(z)f_{i}(w)(zw-1)^{2}\]
	Additionally, if we select $\gamma_{i}=\gamma_{j}=\frac{1}{\sqrt{2}}+\frac{\sqrt{-1}}{\sqrt{2}}$ for some fixed $1\leq i,j\leq L$ and all other $\gamma_{k}=0$ for $k\neq i,j$, then we show that the limiting distribution of
	\[\tr\left( \left(\frac{1}{\sqrt{2}}+\frac{\sqrt{-1}}{\sqrt{2}}\right)f_{i}(P_{n}/\sigma)+\left(\frac{1}{\sqrt{2}}+\frac{\sqrt{-1}}{\sqrt{2}}\right)f_{j}(P_{n}/\sigma)\right)\]
	has variance
	\begin{align*} 
	&-\frac{1}{4\pi^{2}}\oint_{\mathcal{C}}\oint_{\mathcal{C}}\left(\left(\frac{1}{\sqrt{2}}+\frac{\sqrt{-1}}{\sqrt{2}}\right)f_{i}(z)+\left(\frac{1}{\sqrt{2}}+\frac{\sqrt{-1}}{\sqrt{2}}\right)f_{j}(z)\right)\\
	&\quad\quad\quad\quad\quad\quad\quad\times\left(\left(\frac{1}{\sqrt{2}}+\frac{\sqrt{-1}}{\sqrt{2}}\right)f_{i}(w)+\left(\frac{1}{\sqrt{2}}+\frac{\sqrt{-1}}{\sqrt{2}}\right)f_{j}(w)\right)(zw-1)^{2}dzdw\\
	&=-\frac{1}{4\pi^{2}}\oint_{\mathcal{C}}\oint_{\mathcal{C}}\left(\sqrt{-1}f_{i}(z)f_{i}(w)+\sqrt{-1}f_{i}(z)f_{j}(w)\right.\\
	&\quad\quad\quad\quad\quad\quad\quad\quad \left.+\sqrt{-1}f_{j}(z)f_{i}(w)+\sqrt{-1}f_{j}(z)f_{j}(w)\right)(zw-1)^{2}dzdw\\
	\end{align*}
	and 
	\begin{align*} 
	&\frac{1}{4\pi^{2}}\oint_{\mathcal{C}}\oint_{\mathcal{C}}\left(\left(\frac{1}{\sqrt{2}}+\frac{\sqrt{-1}}{\sqrt{2}}\right)f_{i}(z)+\left(\frac{1}{\sqrt{2}}+\frac{\sqrt{-1}}{\sqrt{2}}\right)f_{j}(z)\right)\\
	&\quad\quad\quad\quad\quad\quad\quad\times\overline{\left(\left(\frac{1}{\sqrt{2}}+\frac{\sqrt{-1}}{\sqrt{2}}\right)f_{i}(w)+\left(\frac{1}{\sqrt{2}}+\frac{\sqrt{-1}}{\sqrt{2}}\right)f_{j}(w)\right)}(z\bar{w}-1)^{2}dzd\bar{w}\\
	&=\frac{1}{4\pi^{2}}\oint_{\mathcal{C}}\oint_{\mathcal{C}}\left(f_{i}(z)\overline{f_{i}(w)}+f_{i}(z)\overline{f_{j}(w)}+f_{j}(z)\overline{f_{i}(w)}+f_{j}(z)\overline{f_{j}(w)}\right)(z\bar{w}-1)^{2}dzd\bar{w}\\
	&=\E\left[F(f_{i})\overline{F(f_{i})}\right]+\E\left[F(f_{j})\overline{F(f_{j})}\right]\\
	&\quad\quad\quad\quad+\frac{1}{4\pi^{2}}\oint_{\mathcal{C}}\oint_{\mathcal{C}}f_{i}(z)\overline{f_{j}(w)}(z\bar{w}-1)^{2}dzd\bar{w}+\frac{1}{4\pi^{2}}\oint_{\mathcal{C}}\oint_{\mathcal{C}}f_{j}(z)\overline{f_{i}(w)}(z\bar{w}-1)^{2}dzd\bar{w}\\
	\end{align*}
	and therefore 
	\[\E\left[F(f_{i})\overline{F(f_{j})}\right]=\frac{1}{4\pi^{2}}\oint_{\mathcal{C}}\oint_{\mathcal{C}}f_{i}(z)\overline{f_{j}(w)}(z\bar{w}-1)^{2}dzd\bar{w}.\]
	By choosing values of the coefficients $\gamma_{1},\dots,\gamma_{s}$ properly, this expression recovers the variance and covariance terms \eqref{Equ:Thm:CLT:Var} and \eqref{Equ:Thm:CLT:Covar}.
\end{Details} 
\end{proof}

\begin{Details}
From here, we may proceed by proving Theorem \ref{thm:firstReduction}. Thus, the goal is to prove that
\begin{equation*}
\tr f (P_{n}/\sigma)\oindicator{E_{n}}-\E[\tr f (P_{n}/\sigma)\oindicator{E_{n}}]
\end{equation*}
converges to a mean zero Gaussian random variable $F(f)$ with covariance structure
\[\E\left[\left(F(f)\right)^{2}\right]=-\frac{1}{4\pi^{2}}\oint_{\mathcal{C}}\oint_{\mathcal{C}}f(z)f(w)(zw-1)^{2}dzdw\]
and  
\[\E\left[F(f)\overline{F(f)}\right] =\frac{1}{4\pi^{2}}\oint_{\mathcal{C}}\oint_{\mathcal{C}}f(z)\overline{f(w)}(z\bar{w}-1)^{-2}dzd\bar{w}.\]
where $f$ is analytic on the disk $D_{\delta}=\{z\;:\;|z|\leq 1+\delta\}$ and bounded otherwise and $E_{n}=\left\{\inf_{|z|>1+\delta/2}s_{n}(P_{n}/\sigma-zI)\geq \frac{c}{2}\right\}$ as defined in (\ref{Def:E_n}).
\end{Details} 

\begin{remark}
	The proof above exploits the fact that the non-real eigenvalues come in complex conjugate pairs since the entries of the product matrix are real.  In the case where the entries in each matrix are allowed to be complex-valued, this approach would no longer hold.  In this case, one may apply a complex-valued version of the Cramer--Wold Theorem.
	\label{Remark:ComplexCase4}
\end{remark}
It remains to prove Theorem \ref{thm:firstReduction}.  By a simple rescaling, it is sufficient to prove Theorem \ref{thm:firstReduction} when $\sigma_{i}=1$ for $1\leq i\leq k$.  For this reason, for the remainder of the paper we assume all atom variables have unit variance unless stated otherwise.  

\begin{Details} 
\subsection{Unit Disk Rescaling}
Here, we rescale the product $P_{n}$ defined in \eqref{Def:P_n}. Observe the following lemma.
\begin{lemma}
	Let $m\geq 1$ be a fixed integer, and assume $\xi_1, \ldots, \xi_m$ are complex-valued random variables which satisfy Assumption \ref{assump:4PlusTau} with $\Var(\xi_{k})=1$ for $1\leq k\leq m$.  For each $n \geq 1$, let $X_{n,1}, \ldots, X_{n,m}$ be independent $n \times n$ iid random matrices with atom variables $\xi_1, \ldots, \xi_m$, respectively.  Define the product
	$$ P_n := n^{-m/2} X_{n,1} \cdots X_{n,m}.$$
	Let $\delta>0$ and let $f$ be analytic in some neighborhood containing the disk $D_{\delta}:=\{z\in\C\;:\;|z|\leq 1+\delta\}$ and bounded otherwise. Then, there exists a constant $c>0$ such that the event
	\begin{equation}
	E_{n}=\left\{\inf_{|z|> 1+\delta/2}s_{n}\left(P_{n}-zI\right)\geq c\right\}
	\end{equation}
	holds with probability $1-o(n^{-1})$ and as $n\rightarrow\infty$,
	$$\tr f(P_{n})\oindicator{E_{n}}-\E\left[\tr f(P_{n})\oindicator{E_{n}}\right]$$
	converges in distribution to a mean zero Gaussian random variable $F(f)$ with covariance structure
	\[\E\left[\left(F(f)\right)^{2}\right]=-\frac{1}{4\pi^{2}}\oint_{\mathcal{C}}\oint_{\mathcal{C}}f(z)f(w)(zw-1)^{2}dzdw\]
	and  
	\[\E\left[F(f)\overline{F(f)}\right] =\frac{1}{4\pi^{2}}\oint_{\mathcal{C}}\oint_{\mathcal{C}}f(z)\overline{f(w)}(z\bar{w}-1)^{-2}dzd\bar{w}.\]
	\label{lem:varianceOneReduction}
\end{lemma}

Note that the above lemma is a special case of Theorem \ref{thm:firstReduction} when all variances are equal to 1. It turns out that this is sufficient by a simple rescaling by $\sigma$. We now prove Theorem \ref{thm:firstReduction} assuming Lemma \ref{lem:varianceOneReduction}. 
\begin{proof}[Proof of Theorem \ref{thm:firstReduction}]
	Assume Lemma \ref{lem:varianceOneReduction} holds and let $f(z)$ be any function which is analytic on the disk $D_{\delta}$ and bounded otherwise. Recall that under the general assumptions, $X_{n,k}$ has atom variable $\xi_{k}$ and $\var(\xi_{k})=\sigma_{k}^{2}$ so consider matrices $\frac{1}{\sigma_{k}}X_{n,k}$ for $1\leq k\leq m$. Note that each matrix has been rescaled to have atom variable with variance 1 and note that
	\begin{equation*}
	n^{-m/2}\left(\frac{1}{\sigma_{1}}X_{n,1}\right)\left(\frac{1}{\sigma_{2}}X_{n,2}\right)\cdots \left(\frac{1}{\sigma_{m}}X_{n,m}\right) = \frac{1}{\sigma}P_{n}.
	\end{equation*}
	By assumption, there exists an event 
	\begin{equation}
	E_{n}=\left\{\inf_{|z|> 1+\delta/2}s_{n}\left(\frac{1}{\sigma}P_{n}-zI\right)\geq c\right\}
	\end{equation}
	which holds with probability $1-o(n^{-1})$ as $n\rightarrow\infty$ and 
	\begin{equation*}
	\tr f\left(\frac{1}{\sigma}P_{n}\right)\oindicator{E_{n}}-\E\left[\tr f\left(\frac{1}{\sigma}P_{n}\right)\oindicator{E_{n}}\right]
	\end{equation*}
	converges to a mean zero Gaussian with covariance structure as claimed.
\end{proof}

For the remaining sections, assume that $\sigma_{k} = 1$ for $1\leq k\leq m$.  
\end{Details}

\subsection{Truncation of iid Matrices}
Since $\xi_{k}$ is assumed to have finite $4+\tau$ finite moments for $1\leq k\leq m$, there exists $\varepsilon>0$ such that for all $1\leq k\leq m$,
\begin{equation}
\lim_{n\rightarrow\infty}n^{4\varepsilon}\E\left[\left|\xi_{k}\right|^{4}\indicator{|\xi_{k}|>n^{1/2-\varepsilon}}\right]=0.
\label{Equ:4thMomentToZero}
\end{equation}
\begin{Details} 
Indeed, by the dominated convergence theorem, we can write 
\begin{align*}
n^{4\varepsilon}\E\left[\left|\xi_{k}\right|^{4}\indicator{|\xi_{k}|>n^{1/2-\varepsilon}}\right]&\leq n^{4\varepsilon}\E\left[\frac{\left|\xi_{k}\right|^{4+\tau}}{n^{(1/2-\varepsilon)\tau}}\indicator{|\xi_{x}|>n^{1/2-\varepsilon}}\right]\\
&= n^{\varepsilon(4+\tau)-\tau/2}\E\left[\left|\xi_{k}\right|^{4+\tau}\indicator{|\xi_{k}|>n^{1/2-\varepsilon}}\right]\\
&=o(1)
\end{align*}
for $\varepsilon$ sufficiently small.
\end{Details}
\begin{Details} 
In the above, the $(4+\tau)$th moment is finite, so the indicator makes it $o(1)$ provided that the power of $n$ can be controlled. This will happen whenever $\varepsilon\leq \frac{\tau}{8+2\tau}$. 
\end{Details}

Next, for a real-valued random variable $\xi$ with mean zero, variance one, and finite $(4+\tau)$th moment, define
\begin{equation}
\tilde{\xi}:=\xi\indicator{|\xi|\leq n^{1/2-\varepsilon}}-\E\left[\xi\indicator{|\xi|\leq n^{1/2-\varepsilon}}\right]\;\;\;\;\text{   and   }\;\;\;\;\hat{\xi} :=\frac{\tilde{\xi}}{\sqrt{\Var(\tilde{\xi})}}.
\label{Equ:TruncateLeveln}
\end{equation}
Note that $\tilde{\xi}$ and $\hat{\xi}$ depend on $n$, but this dependence is not expressed in the notation.
\begin{lemma}Let $\xi$ be a real-valued random variable which satisfies Assumption \ref{assump:4PlusTau} with unit variance and define $\tilde{\xi}$ and $\hat{\xi}$ as in (\ref{Equ:TruncateLeveln}). Then the following statements hold:
	\begin{enumerate}[label=\emph{(\roman*)}]
		\item $|1-\Var(\tilde{\xi})|=o(n^{-1-2\varepsilon})$\label{item:Lem:TruncateLeveln:i}
		\item There exists an $N_{0}>0$ such that for any $n>N_{0}$, $\hat{\xi}$ has zero mean and unit variance, and almost surely
		$$|\hat{\xi}|\leq 4n^{1/2-\varepsilon}.$$\label{item:Lem:TruncateLeveln:ii}
		\item There exists $N_{0}>0$ such that for any $n>N_{0}$, 
		$$\E|\hat{\xi}|^{4}\leq 2^{8}\E|\xi|^{4}.$$\label{item:Lem:TruncateLeveln:iii}
	\end{enumerate}
\label{Lem:TruncateLeveln}
\end{lemma}	
The proof of this lemma is a standard truncation argument and can be found in Appendix \ref{Sec:AppendexTruncation}. 
\begin{remark}
	\label{Remark:ComplexCase1}
	In the case where $\xi$ is complex-valued, this truncation will need to be modified in order to preserve independence between the real and imaginary parts of $\xi$ (see, for example, \cite[Lemma 7.1]{COW}).
\end{remark}
Next, we define various truncated matrices and prove a number of lemmas involving operator norms and Hilbert-Schmidt norms of these truncated matrices. The following lemmas will be used later in the proof.

Let $X$ be an $n\times n$ random matrix filled with iid copies of a random variable $\xi$ which satisfies Assumption \ref{assump:4PlusTau}. Define the $n\times n$ matrices $\dot{X}$, $\tilde{X}$, and $\hat{X}$ to be the matrices with entries given by
\begin{align}
& \dot{X}_{(i,j)}:=X_{(i,j)}\indicator{|X_{(i,j)}|\leq n^{1/2-\varepsilon}}, \label{def:Xdot}\\
& \tilde{X}_{(i,j)}:=X_{(i,j)}\indicator{|X_{(i,j)}|\leq n^{1/2-\varepsilon}}-\E\left[X_{(i,j)}\indicator{|X_{(i,j)}|\leq n^{1/2-\varepsilon}}\right], \label{def:Xtilde}\\
&\hat{X}_{(i,j)}:=\frac{\tilde{X}_{(i,j)}}{\sqrt{\Var(\tilde{X}_{(i,j)})}}\label{def:Xhat}
\end{align}
for $1\leq i,j\leq n$.

\begin{lemma}
	Let $X_{n}$ be an $n\times n$ iid random matrix with atom variable $\xi$ which satisfies Assumption \ref{assump:4PlusTau} with unit variance and let $\hat{X}_{n}$ be the truncated matrix as defined in \eqref{def:Xhat}. Then
	\begin{equation*}
	\E\lnorm\frac{1}{\sqrt{n}}X_{n}-\frac{1}{\sqrt{n}}\hat{X}_{n}\rnorm_{2}^{2}=o\left(n^{-2\varepsilon}\right)\;\;\text{ and }\;\;	\P\left(\lnorm\frac{1}{\sqrt{n}}X_{n}-\frac{1}{\sqrt{n}}\hat{X}_{n}\rnorm>n^{-\varepsilon}\right)=o(1).
	\end{equation*}
	\label{Lem:XcloseXhat}
\end{lemma}

\begin{proof}
	By Markov's inequality,
	\begin{align*}
	\P\left(\lnorm\frac{1}{\sqrt{n}}X_{n}-\frac{1}{\sqrt{n}}\hat{X}_{n}\rnorm>n^{-\varepsilon}\right) & \leq n^{2\varepsilon}\E\lnorm\frac{1}{\sqrt{n}}X_{n}-\frac{1}{\sqrt{n}}\hat{X}_{n}\rnorm^{2}\\
	&\leq n^{2\varepsilon}\E\lnorm\frac{1}{\sqrt{n}}X_{n}-\frac{1}{\sqrt{n}}\hat{X}_{n}\rnorm_{2}^{2}
	\end{align*}
	so it is sufficient to prove $\E\lnorm\frac{1}{\sqrt{n}}X_{n}-\frac{1}{\sqrt{n}}\hat{X}_{n}\rnorm_{2}^{2}=o\left(n^{-2\varepsilon}\right)$. By the triangle inequality, 
	\begin{align*}
	\E\lnorm\frac{1}{\sqrt{n}}X_{n}-\frac{1}{\sqrt{n}}\hat{X}_{n}\rnorm_{2}^{2} & \ll \E\left[\lnorm\frac{1}{\sqrt{n}}X_{n}-\frac{1}{\sqrt{n}}\tilde{X}_{n}\rnorm_{2}^{2}+\lnorm\frac{1}{\sqrt{n}}\tilde{X}_{n}-\frac{1}{\sqrt{n}}\hat{X}_{n}\rnorm_{2}^{2}\right]
	\end{align*}
	and we may deal with the two terms on the right hand side of the above expression separately. First, since $\left|\E\left[\xi\indicator{|\xi|> n^{1/2-\varepsilon}}\right]\right|=\left|\E[\xi\indicator{|\xi|\leq n^{1/2-\varepsilon}}]\right|$ and by (\ref{Equ:4thMomentToZero}), we have 
	\begin{align*}
	\E\lnorm\frac{1}{\sqrt{n}}X_{n}-\frac{1}{\sqrt{n}}\tilde{X}_{n}\rnorm_{2}^{2} & = \frac{1}{n}\E\lnorm X_{n}-\tilde{X}_{n}\rnorm_{2}^{2}\\
	& \leq \frac{1}{n}\sum_{j,k=1}^{n}\E\left|(X_{n})_{(j,k)}-(\tilde{X}_{n})_{(j,k)}\right|^{2}\\
	& \leq \frac{4}{n}\sum_{j,k=1}^{n}\E\left[|\xi|^{2}\frac{|\xi|^{2}}{(n^{1/2-\varepsilon})^{2}}\indicator{|\xi|> n^{1/2-\varepsilon}}\right]\\
	& \leq  \frac{4n^{2\varepsilon}}{n^{4\varepsilon}}n^{4\varepsilon}\E\left[|\xi|^{4}\indicator{|\xi|> n^{1/2-\varepsilon}}\right]\\
	&=o(n^{-2\varepsilon}).
	\end{align*}
	\begin{Details}
		Next, we consider 
		\begin{equation*}
		\E\lnorm\frac{1}{\sqrt{n}}\tilde{X}_{n}-\frac{1}{\sqrt{n}}\hat{X}_{n}\rnorm_{2}^{2}.
		\end{equation*}
	\end{Details}
	Observe that by Lemma \ref{Lem:TruncateLeveln}, one has
	\begin{align*}
	\E\lnorm\frac{1}{\sqrt{n}}\tilde{X}_{n}-\frac{1}{\sqrt{n}}\hat{X}_{n}\rnorm_{2}^{2}& \leq \frac{1}{n}\sum_{j,k=1}^{n}\E\left|\hat{X}_{n (j,k)}\right|^{2}\left|\sqrt{\Var(\tilde{X}_{n,(j,k)})}-1\right|^{2}\\
	& \leq n \left|\Var(\tilde{\xi})-1\right|^{2}\\
	& =o(n^{-1-4\varepsilon})
	\end{align*}
	which concludes the proof.
\end{proof}


\begin{lemma}
	Let $X_{n}$ be an $n\times n$ iid random matrix with atom variable $\xi$ which satisfies Assumption \ref{assump:4PlusTau} with unit variance. Let $\dot{X}_{n}$ and $\hat{X}_{n}$ be the truncated matrices as defined in \eqref{def:Xdot} and \eqref{def:Xhat} respectively. Then 
	\[\E\lnorm \hat{X}_{n}-\dot{X}_{n}\rnorm_{2}^{2}= o(1).\]
	\label{Lem:XhatcloseXdot}
\end{lemma}

\begin{proof}
	Let $\tilde{\xi}$ be as defined in \eqref{Equ:TruncateLeveln} and observe that by the proof of Lemma \ref{Lem:TruncateLeveln} \ref{item:Lem:TruncateLeveln:ii}, $(\Var(\tilde{\xi}))^{-1/2}\leq 2$ for $n$ sufficiently large. Therefore
	\begin{align*}
	\E\lnorm \hat{X}_{n}-\dot{X}_{n}\rnorm_{2}^{2}&=\E\left[\sum_{i,j=1}^{n}\left|\hat{X}_{n,(i,j)}-\dot{X}_{n,(i,j)}\right|^{2}\right]\\
	&\leq n^{2}\E\left|\frac{\xi\indicator{|\xi|\leq n^{1/2-\varepsilon}}(1-(\Var(\tilde{\xi}))^{1/2})-\E\left[\xi\indicator{|\xi|\leq n^{1/2-\varepsilon}}\right]}{(\Var(\tilde{\xi}))^{1/2}}\right|^{2}\\
	&\ll n^{2}\E\left|\xi\indicator{|\xi|\leq n^{1/2-\varepsilon}}(1-(\Var(\tilde{\xi}))^{1/2})-\E\left[\xi\indicator{|\xi|\leq n^{1/2-\varepsilon}}\right]\right|^{2}\\
	&\ll n^{2}\left|1-\Var(\tilde{\xi})\right|^{2}\E\left[|\xi|^{2}\indicator{|\xi|\leq n^{1/2-\varepsilon}}\right]+n^{2}\left|\E\left[\xi\indicator{|\xi|\leq n^{1/2-\varepsilon}}\right]\right|^{2}
	\end{align*} 
	for $n$ sufficiently large. By Lemma \ref{Lem:TruncateLeveln} \ref{item:Lem:TruncateLeveln:i}, we have \[n^{2}\left|1-\Var(\tilde{\xi})\right|^{2}\E\left[|\xi|^{2}\indicator{|\xi|\leq n^{1/2-\varepsilon}}\right]=o(n^{-4\varepsilon}).\] 
	Next, observe that since $\left|\E\left[\xi\indicator{|\xi|\leq n^{1/2-\varepsilon}}\right]\right|=\left|\E\left[\xi\indicator{|\xi|> n^{1/2-\varepsilon}}\right]\right|$, we have 
	\begin{align*}
	n^{2}\left|\E\left[\xi\indicator{|\xi|\leq n^{1/2-\varepsilon}}\right]\right|^{2}&=n^{2}\left|\E\left[\xi\indicator{|\xi|> n^{1/2-\varepsilon}}\right]\right|^{2}\\
	&\leq n^{-1+6\varepsilon}\left(\E\left[|\xi|^{4}\indicator{|\xi|> n^{1/2-\varepsilon}}\right]\right)^{2}\\
	&=o(1).
	\end{align*} 
\end{proof}

\begin{lemma}
	Let $\hat{X}_{n}$ be an $n\times n$ iid random matrix with atom variable $\hat{\xi}$ which has mean zero, variance one, $\E|\hat{\xi}|^{4}=O(1)$, and satisfies $|\hat{\xi}|\ll n^{1/2-\varepsilon}$ almost surely for some $\varepsilon>0$. Then $\E\lnorm \hat{X}_{n}\rnorm^{2}=O(n)$ where $\lnorm \cdot\rnorm$ denotes the operator norm. 
	\label{Lem:TruncatedOpNormSquaredBounded}
\end{lemma}

\begin{proof}
	Observe that, for any constant $C>0$
	\begin{align*}
	\E\lnorm \hat{X}_{n}\rnorm^{2} &\ll \E\lnorm \hat{X}_{n}\indicator{\lnorm \hat{X}_{n}\rnorm \leq C\sqrt{n}}\rnorm^{2}+\E\lnorm \hat{X}_{n}\indicator{\lnorm \hat{X}_{n}\rnorm > C\sqrt{n}}\rnorm^{2}\\
	&\ll n+n^{3-2\varepsilon}\P\left(\lnorm\hat{X}_{n}\rnorm>C\sqrt{n}\right)
	\end{align*}
	where the power of $n$ came from bounding the operator norm by the Frobenious norm. By \cite[Theorem 5.9]{BSbook}, there exists $C>0$ sufficiently large so that $\P\left(\lnorm\hat{X}_{n}\rnorm>C\sqrt{n}\right)=O_{\alpha}(n^{-\alpha})$ for any $\alpha>0$. By selecting $\alpha$ sufficiently large, we arrive at the desired result.
\end{proof}

\begin{lemma}
	Let $X_{n}$ be an $n\times n$ iid random matrix with atom variable $\xi$ which has mean zero, variance one, and finite $4+\tau$ moment for some $\tau>0$ and define $\dot{X}_{n}$ as in \eqref{def:Xdot}. Then $\E\lnorm \dot{X}_{n}\rnorm^{2}=O(n)$ where $\lnorm \cdot\rnorm$ denotes the operator norm. 
	\label{Lem:OpNormJustTruncatedSquaredBounded}
\end{lemma}

\begin{proof}
	Let $\hat{X}_{n}$ be the truncated $n\times n$ iid random matrix with entries as defined in \eqref{def:Xhat} and observe that 
	\[\E\lnorm \dot{X}_{n}\rnorm^{2}\ll \E\lnorm \dot{X}_{n}-\hat{X}_{n}\rnorm^{2}+\E\lnorm \hat{X}_{n}\rnorm^{2}.\] 
	The proof follows by Lemmas \ref{Lem:TruncatedOpNormSquaredBounded} and \ref{Lem:XhatcloseXdot}.
\end{proof}

\begin{lemma}
	Let $X_{n}$ be an $n\times n$ iid random matrix with atom variable $\xi$ which has mean zero, variance one, and finite $4+\tau$ moment for some $\tau>0$. Then $\E\lnorm X_{n}\rnorm^{2}=O(n)$ where $\lnorm \cdot\rnorm$ denotes the operator norm. 
	\label{Lem:OpNormSquaredBounded}
\end{lemma}
\begin{proof}
	Let $\hat{X}_{n}$ be the truncated $n\times n$ iid random matrix with entries defined by \eqref{def:Xhat} and observe that by the triangle inequality we have
	\begin{equation}
	\E\lnorm X_{n}\rnorm^{2}\ll\E\lnorm X_{n}-\hat{X}_{n}\rnorm^{2}+\E\lnorm \hat{X}_{n}\rnorm^{2}.
	\label{Equ:Lem:OpNormSquaredBounded:1}
	\end{equation}
	Both terms in the right hand side of \eqref{Equ:Lem:OpNormSquaredBounded:1} are $O(n)$ by Lemmas \ref{Lem:XcloseXhat} and \ref{Lem:TruncatedOpNormSquaredBounded} as desired.
\end{proof}

\begin{lemma}
	Let $X_{n,i}$ be as defined in Theorem \ref{thm:mainCLT} with $\sigma_i = 1$, and for each $1\leq i\leq m$, define $\hat{X}_{n,i}$ as in \eqref{def:Xhat}. Define the product $P_{n}$ as in \eqref{Def:P_n} and define the truncated product
	\begin{equation}
	\hat{P}_{n}=n^{-m/2}\hat{X}_{n,1}\cdots\hat{X}_{n,m}.
	\label{Def:TruncatedProducts}
	\end{equation} 
	Then
	\begin{equation*}
	\E\lnorm P_{n}-\hat{P}_{n}\rnorm_{2}^{2}=o(n^{-2\varepsilon})\; \text{ and }\;
	\P\left(\lnorm P_{n}-\hat{P}_{n}\rnorm > n^{-\varepsilon}\right)=o(1).
	\end{equation*}
	\label{Lem:ProductsCloseToTruncatedProducts}
\end{lemma}
\begin{proof}
	By Markov's inequality,
	\begin{equation*}
	\P\left(\lnorm P_{n}-\hat{P}_{n}\rnorm>n^{-\varepsilon}\right) \leq{n^{2\varepsilon}}\E\lnorm P_{n}-\hat{P}_{n}\rnorm^{2} \leq n^{2\varepsilon}\E\lnorm P_{n}-\hat{P}_{n}\rnorm_{2}^{2}\\
	\end{equation*}
	so it is sufficient to prove the first bound. To this end, note that by the triangle inequality, independence, and Lemma \ref{Lem:BoundOnTraceByNorms}, we have 
	\begin{align*}
	&\E\lnorm P_{n}-\hat{P}_{n}\rnorm_{2}^{2}=\E\lnorm n^{-m/2}X_{n,1}\cdots X_{n,m}-n^{-m/2}\hat{X}_{n,1}\cdots \hat{X}_{n,m}\rnorm_{2}^{2}\\
	&\quad\quad\ll n^{-m}\left(\E\lnorm X_{n,1}-\hat{X}_{n,1}\rnorm_{2}^{2}\E\lnorm X_{n,2}\rnorm^{2}\cdots \E \lnorm X_{n,m-1}\rnorm^{2} \E\lnorm X_{n,m}\rnorm^{2}+\dots \right.\\
	&\quad\quad\quad\quad\quad\quad\quad\quad\left.+\E\lnorm \hat{X}_{n,1}\rnorm^{2}\E\lnorm \hat{X}_{n,2}\rnorm^{2}\cdots\E\lnorm  \hat{X}_{n,m-1}\rnorm^{2}\E\lnorm X_{n,m}-\hat{X}_{n,m}\rnorm_{2}^{2}\right).	
	\end{align*}
	By Lemmas \ref{Lem:TruncatedOpNormSquaredBounded} and \ref{Lem:OpNormSquaredBounded}, $\E\lnorm \hat{X}_{n,k}\rnorm^{2} = O(n)$ and $\E\lnorm X_{n,k}\rnorm^{2} = O(n)$ for all $1\leq k\leq m$. Therefore, by this observation and Lemma \ref{Lem:XcloseXhat}, 	
	\begin{align*}
	\E\lnorm P_{n}-\hat{P}_{n}\rnorm_{2}^{2}& \ll n^{-1}\left(\E\lnorm X_{n,1}-\hat{X}_{n,1}\rnorm_{2}^{2}+\dots+\E\lnorm X_{n,m}-\hat{X}_{n,m}\rnorm_{2}^{2}\right)\\	
	& = o(n^{-2\varepsilon}).
	\end{align*}
\end{proof}
\begin{Details}
	Lemma \ref{Lem:ProductsCloseToTruncatedProducts} lets us work with products of matrices with truncated entries, provided we can control the smallest singular value. 	
\begin{lemma}
	Let $X_{n,i}$ be as defined in Theorem \ref{thm:mainCLT}, $\hat{X}_{n,i}$ as defined in \eqref{Equ:TruncateLeveln}, and $\hat{P}_{n}$ as in \eqref{Def:TruncatedProducts}. Then the event $\hat{E}_{n}$ as defined in \eqref{Def:EventTruncatedProductSingVal} holds with overwhelming probability.
\end{lemma}
With this result in hand, we proceed to the following theorem.
\end{Details}

\begin{lemma}
	Let $X_{n,i}$ be as defined in Theorem \ref{thm:mainCLT} with $\sigma_i = 1$, and for each $1\leq i\leq m$, define $\dot{X}_{n,i}$ and $\hat{X}_{n,i}$ as in \eqref{def:Xdot} and \eqref{def:Xhat} respectively. Define $\hat{P}_{n}$ as in \eqref{Def:TruncatedProducts} and define the product \begin{equation}
	\dot{P}_{n}=n^{-m/2}\dot{X}_{n,1}\cdots\dot{X}_{n,m}.
	\label{Def:DotTruncatedProducts}
	\end{equation} 
	Then \[\E\lnorm \dot{P}_{n}-\hat{P}_{n}\rnorm_{2}^{2}=o(n^{-1}).\]
	\label{Lem:XdotProductsCloseToTruncatedProducts}
\end{lemma}
\begin{proof}
	By the triangle inequality, independence, and Lemma \ref{Lem:BoundOnTraceByNorms}, we have 
	\begin{align*}
	&\E\lnorm \dot{P}_{n}-\hat{P}_{n}\rnorm_{2}^{2}=\E\lnorm n^{-m/2}\dot{X}_{n,1}\cdots \dot{X}_{n,m}-n^{-m/2}\hat{X}_{n,1}\cdots \hat{X}_{n,m}\rnorm_{2}^{2}\\
	&\quad\quad\ll n^{-m}\left(\E\lnorm \dot{X}_{n,1}-\hat{X}_{n,1}\rnorm_{2}^{2}\E\lnorm \dot{X}_{n,2}\rnorm^{2}\cdots \E \lnorm \dot{X}_{n,m-1}\rnorm^{2} \E\lnorm \dot{X}_{n,m}\rnorm^{2}+\dots \right.\\
	&\quad\quad\quad\quad\quad\quad\quad\quad\left.+\E\lnorm \hat{X}_{n,1}\rnorm^{2}\E\lnorm \hat{X}_{n,2}\rnorm^{2}\cdots\E\lnorm  \hat{X}_{n,m-1}\rnorm^{2}\E\lnorm \dot{X}_{n,m}-\hat{X}_{n,m}\rnorm_{2}^{2}\right).	
	\end{align*}
	By Lemmas \ref{Lem:TruncatedOpNormSquaredBounded} and \ref{Lem:OpNormJustTruncatedSquaredBounded}, $\E\lnorm \hat{X}_{n,k}\rnorm^{2} = O(n)$ and $\E\lnorm \dot{X}_{n,k}\rnorm^{2} = O(n)$ for all $1\leq k\leq m$. By this observation and Lemma \ref{Lem:XhatcloseXdot}, 
	\begin{align*}
	\E\lnorm \dot{P}_{n}-\hat{P}_{n}\rnorm_{2}^{2}& \ll n^{-1}\left(\E\lnorm \dot{X}_{n,1}-\hat{X}_{n,1}\rnorm_{2}^{2}+\dots+\E\lnorm \dot{X}_{n,m}-\hat{X}_{n,m}\rnorm_{2}^{2}\right)\\	
	& = o(n^{-1}).
	\end{align*}
\end{proof}

With the preceding norm lemmas complete, we now show it is sufficient to consider a version of Theorem \ref{thm:firstReduction} in which all entries in the matrices are truncated. We now reduce to the case where we can consider the truncated product $\hat{P}_n$.  
\begin{theorem}
	Let $X_{n,i}$ be as defined in Theorem \ref{thm:mainCLT} with $\sigma_i = 1$, $\hat{X}_{n,i}$ as defined in \eqref{Equ:TruncateLeveln}, and $\hat{P}_{n}$ as in \eqref{Def:TruncatedProducts}.  Let $\delta > 0$, and let $f$ be a function which is analytic in some neighborhood containing the disk $D_{\delta}$ and bounded otherwise. Then there exists a constant $c>0$ such that the event 
	\begin{equation}
	\hat{E}_{n}:=\left\{\inf_{|z|>1+\delta/2}s_{n}(\hat{P}_{n}-zI)\geq c\right\}
	\label{Def:EventTruncatedProductSingVal}
	\end{equation}
	holds with overwhelming probability and 
	\begin{equation}
	\tr f(\hat{P}_{n})\oindicator{\hat{E}_{n}}- \E[\tr f(\hat{P}_{n})\oindicator{\hat{E}_{n}}]
	\label{Equ:SingleTruncatedTerm}
	\end{equation}
	converges in distribution to a mean-zero Gaussian random variable $F(f)$ with covariance structure 
	\[\E\left[\left(F(f)\right)^{2}\right]=-\frac{1}{4\pi^{2}}\oint_{\mathcal{C}}\oint_{\mathcal{C}}f(z)f(w)(zw-1)^{-2}dzdw\]
	and  
	\[\E\left[F(f)\overline{F(f)}\right] =\frac{1}{4\pi^{2}}\oint_{\mathcal{C}}\oint_{\mathcal{C}}f(z)\overline{f(w)}(z\bar{w}-1)^{-2}dzd\bar{w}.\]
	\label{Thm:TruncatedProductCLT}
\end{theorem}
We now prove Theorem \ref{thm:firstReduction} assuming Theorem \ref{Thm:TruncatedProductCLT}.
\begin{proof}[Proof of Theorem \ref{thm:firstReduction}]
Suppose the conclusion of Theorem \ref{Thm:TruncatedProductCLT} holds. We define 
\begin{equation}
\mathcal{A}_{n}(f):=\E\left[\tr f(\hat{P}_{n})\oindicator{\hat{E}_{n}}\right].
\label{Def:A_n}
\end{equation}
There exists $c > 0$ such that $\hat{E}_{n}$ holds with overwhelming probability by Lemma \ref{Lem:E_nHatOverwhelming}, and $E_{n}$ holds with probability $1-o(1)$ by Lemma \ref{Lem:E_nOverwhelming}.  Thus, we may work on the intersection of these events, and in order to show that $\tr f(P_{n})\oindicator{E_{n}}-\mathcal{A}_{n}(f)$ converges to a mean-zero Gaussian random variable with variance as in \eqref{Equ:FirstReductionVar1} and \eqref{Equ:FirstReductionVar2}, it is sufficient to show that for any $\eta>0$,  
	\begin{equation*}
	\P\left(\left|\tr f(\hat{P}_{n})\oindicator{E_{n}\cap\hat{E}_{n}}-\tr f(P_{n})\oindicator{E_{n}\cap\hat{E}_{n}}\right|>\eta\right) =o(1).
	\end{equation*} 
	To this end, define $\dot{X}_{n,k}$ as in \eqref{def:Xdot} for each $1\leq k\leq m$ and $\dot{P}_{n}$ as in \eqref{Def:DotTruncatedProducts}. Observe that for any $1\leq k\leq m$, by \eqref{Equ:4thMomentToZero} 
	\begin{align*}
	\P(X_{n,k}\neq \dot{X}_{n,k})&=\P\left(\bigcup_{i,j}\left\{|(X_{n,k})_{ij}|>n^{1/2-\varepsilon}\right\}\right)\\
	&\leq n^{2}\E\left[\indicator{|\xi_{k}|>n^{1/2-\varepsilon}}\right]\\
	&\leq n^{4\varepsilon}\E\left[|\xi_{k}|^{4}\indicator{|\xi|>n^{1/2-\varepsilon}}\right]\\
	&=o(1).
	\end{align*}
	By a union bound over all $1\leq k\leq m$, we have that $\P\left(P_{n}\neq \dot{P}\right)=o(1)$ as well. Therefore, it is sufficient to show that 
	\begin{equation}
	\P\left(\left|\tr f(\hat{P}_{n})\oindicator{E_{n}\cap\hat{E}_{n}}-\tr f(\dot{P}_{n})\oindicator{E_{n}\cap\hat{E}_{n}}\right|>\eta\right) =o(1).
	\label{Equ:ConvInProb}
	\end{equation}
	Define the event $\dot{E}_{n} :=\{P_{n}=\dot{P}_{n}\}\cap E_{n}$ and observe that $\dot{E}_{n}$ holds with probability $1-o(1)$. By this fact and \eqref{Equ:ConvInProb}, it is sufficient to prove
	\begin{equation}
	\P\left(\left|\tr f(\hat{P}_{n})\oindicator{\dot{E}_{n}\cap\hat{E}_{n}}-\tr f(\dot{P}_{n})\oindicator{\dot{E}_{n}\cap\hat{E}_{n}}\right|>\eta\right) =o(1).
	\label{Equ:ConvInProb2}
	\end{equation}
	 On the events $\dot{E}_n$ and $\hat{E}_n$, the eigenvalues of $\hat{P}_n$ and $\dot{P}_n$ are contained in the interior of the contour $\mathcal{C}$, which is defined as the boundary of the disk $D_\delta$. Therefore, on $\dot{E}_{n}$, for any $z\in\mathcal{C}$, the least singular values of $\dot{P}_{n}-zI$ is bounded away from zero. Thus, by Cauchy's integral formula,
	\begin{align*}
	&\E\left|(\tr f(\hat{P}_{n})-\tr f(\dot{P}_{n}))\oindicator{\dot{E}_{n}\cap\hat{E}_{n}}\right|^{2}\\
	&\quad\quad=\E\left|-\frac{1}{2\pi i}\oint_{\mathcal{C}}f(z)\left(\tr(\hat{P}_{n}-zI)^{-1}-\tr(\dot{P}_{n}-zI)^{-1}\right)\oindicator{\dot{E}_{n}\cap\hat{E}_{n}}dz\right|^{2}\\
	&\quad\quad \ll_{f} \E\left[\sup_{z\in\mathcal{C}}\left|\tr(\hat{P}_{n}-zI)^{-1}-\tr(\dot{P}_{n}-zI)^{-1}\right|^{2}\oindicator{\dot{E}_{n}\cap\hat{E}_{n}}\right].
	\end{align*}
	Since $f$ is assumed to be analytic on the disk and bounded otherwise, by applying Markov's inequality to the left-hand side of \eqref{Equ:ConvInProb}, it is sufficient to show that 
	\begin{equation*}
	\E\left[\sup_{z\in\mathcal{C}}\left|\tr(\hat{P}_{n}-zI)^{-1}-\tr(\dot{P}_{n}-zI)^{-1}\right|^{2}\oindicator{\dot{E}_{n}\cap\hat{E}_{n}}\right]=o(1).
	\end{equation*}
	
	By the resolvent identity, Lemma \ref{Lem:BoundOnTraceByNorms}, and Lemma \ref{Lem:XdotProductsCloseToTruncatedProducts}, we have
	\begin{align*}
	&\E\left[\sup_{z\in\mathcal{C}}\left|(\tr(\hat{P}_{n}-zI)^{-1}-\tr(\dot{P}_{n}-zI)^{-1})\right|^{2}\oindicator{\dot{E}_{n}\cap\hat{E}_{n}}\right]\\
	&\quad\quad=\E\left[\sup_{z\in\mathcal{C}}\left|\tr\left((\hat{P}_{n}-zI)^{-1}(\dot{P}_{n}-\hat{P}_{n})(\dot{P}_{n}-zI)^{-1}\right)\right|^{2}\oindicator{\dot{E}_{n}\cap\hat{E}_{n}}\right]\\
	&\quad\quad \ll n\E\lnorm \dot{P}_{n}-\hat{P}_{n}\rnorm_{2}^{2}\\
	&\quad\quad =o(1)
	\end{align*}
	since the spectral norms of the resolvents are bounded uniformly by a constant (Proposition \ref{Prop:LargeAndSmallSingVals}) on their respective events.
\end{proof}


\begin{Details}
Also, since this is convergence in distribution, finding the variance of this truncated random variable will be sufficient since the variance of this will match the variance of the original un-truncated random variable. 
\end{Details} 

\subsection{Linearization of the Product}

We now wish to linearize the product matrix $\hat{P}_n$ so that we can work with an $mn\times mn$ block matrix instead. Define the $mn\times mn$ matrix
\begin{equation}
\blmat{Y}{n} :=n^{-1/2}\left[\begin{array}{ccccc}
0 & \hat{X}_{n,1} & 0 & \cdots & 0\\
0 & 0 & \hat{X}_{n,2} & \dots & 0\\
\vdots & \vdots & \vdots & \ddots & \vdots\\
0 & 0 & 0 & \cdots & \hat{X}_{n,m-1}\\
\hat{X}_{n,m} & 0 & 0 & \dots & 0
\end{array}\right].
\label{Def:Y_n}
\end{equation}
Recall that by Proposition \ref{prop:linear}, $\blmat{Y}{n}^{m}$ has the same eigenvalues as $\hat{P}_{n}=n^{-m/2}\hat{X}_{n,1}\cdots \hat{X}_{n,m}$, each with multiplicity $m$. Therefore, the eigenvalues of the product $\hat{P}_{n}$ are completely determined by the eigenvalues of the linearized matrix $\blmat{Y}{n}$.

\begin{theorem}
	Let $\blmat{Y}{n}$ be the linearized matrix defined in \eqref{Def:Y_n} where $X_{n,i}$ are under the assumptions of Theorem \ref{thm:mainCLT} with $\sigma_i = 1$ and the entries of $\hat{X}_{n,i}$ are truncated as defined in \eqref{def:Xhat}. For every $\delta > 0$, there exists $c > 0$ such that the following holds.  The event 
	\begin{equation}
	\Omega_{n} :=\left\{\inf_{|z| >1+\delta/2}s_{mn}(\blmat{Y}{n}-zI)\geq c\right\}
	\label{Def:Omega_n}
	\end{equation}
	holds with overwhelming probability, and for any function $g$ which is analytic in a neighborhood of the disk $D_\delta$ and bounded otherwise, the random variable
	\begin{equation*}
	\tr g(\blmat{Y}{n})\oindicator{\Omega_{n}} -\E\left[\tr g(\blmat{Y}{n})\oindicator{\Omega_{n}} \right]
	\end{equation*}
	converges to a mean zero Gaussian random variable $F(g)$ with covariance structure
	\begin{equation}
	\E\left[(F(g))^{2}\right]=-\frac{1}{4\pi^{2}}\oint_{\mathcal{C}}\oint_{\mathcal{C}}g(z)g(w)\frac{m^{2}(zw)^{m-1}}{((zw)^{m}-1)^{2}}dzdw
	\label{Equ:Thm:LinTruncCLT:1}
	\end{equation}
	and  
	\begin{equation}
	\E\left[F(g)\overline{F(g)}\right] =\frac{1}{4\pi^{2}}\oint_{\mathcal{C}}\oint_{\mathcal{C}}g(z)\overline{g(w)}\frac{m^{2}(z\bar{w})^{m-1}}{((z\bar{w})^{m}-1)^{2}}dzd\bar{w},
	\label{Equ:Thm:LinTruncCLT:2}
	\end{equation}
	where $\mathcal{C}$ is the contour around the boundary of $D_\delta$.  
	\begin{Details}
		Then it follows that for any function $f$ which is analytic on the disk $\{z\in \C\;:\;|z|>1+\delta\}$ and bounded otherwise, the random variable 
		\begin{equation*}
		\tr f(\hat{P}_{n})\oindicator{\hat{E}_{n}} -\E\left[\tr f(\hat{P}_{n})\oindicator{\hat{E}_{n}} \right]
		\end{equation*}
		converges to a mean zero Gaussian random variable $F(f)$ with covariance structure
		\[\E\left[\left(F(f)\right)^{2}\right]=-\frac{1}{4\pi^{2}}\oint_{\mathcal{C}}\oint_{\mathcal{C}}f(z)f(w)(zw-1)^{2}dzdw\]
		and  
		\[\E\left[F(f)\overline{F(f)}\right] =\frac{1}{4\pi^{2}}\oint_{\mathcal{C}}\oint_{\mathcal{C}}f(z)\overline{f(w)}(z\bar{w}-1)^{-2}dzd\bar{w}.\]
	\end{Details}
\label{Thm:LinearizedTruncatedCLT}
\end{theorem}

Note that the convergence in \eqref{Equ:Thm:LinTruncCLT:1} and \eqref{Equ:Thm:LinTruncCLT:2} depend on $m$. We prove Theorem \ref{Thm:TruncatedProductCLT} assuming Theorem \ref{Thm:LinearizedTruncatedCLT}.

\begin{proof}[Proof of Theorem \ref{Thm:TruncatedProductCLT}]
	\begin{Details}
	Assume that for any function $g$ which is analytic on the disk $\{z\in \C\;:\;|z|<(1+\delta)^{1/m}\}$ and bounded otherwise, the random variable
	\begin{equation*}
	\tr g(\blmat{Y}{n})\oindicator{\Omega_{n}} -\E\left[\tr g(\blmat{Y}{n})\oindicator{\Omega_{n}} \right]
	\end{equation*}
	converges to a mean zero Gaussian random variable with covariance  structure
	\[\E\left[(F(g))^{2}\right]=-\frac{1}{4\pi^{2}}\oint_{\mathcal{C}}\oint_{\mathcal{C}}g(z)g(w)\frac{m^{2}(zw)^{m-1}}{((zw)^{m}-1)^{2}}dzdw\]
	and  
	\[\E\left[F(g)\overline{F(g)}\right] =\frac{1}{4\pi^{2}}\oint_{\mathcal{C}}\oint_{\mathcal{C}}g(z)\overline{g(w)}\frac{m^{2}(z\bar{w})^{m-1}}{((z\bar{w})^{m}-1)^{2}}dzd\bar{w}.\]
	\end{Details}
	First begin by observing that by assumption, there exists a $c>0$ such that $\Omega_{n}$ holds with overwhelming probability and by Lemma \ref{Lem:E_nHatOverwhelming}, there exists another constant $c'>0$ such that $\hat{E}_{n}$ holds with overwhelming probability as well. Let $f$ be any function which is analytic on the disk $D_\delta$ and bounded otherwise.  Define the function $g(z):=\frac{1}{m}f(z^{m})$ and note that this function $g$ is analytic on the disk $\{z\in \C\;:\;|z|\leq (1+\delta)^{1/m}\} = D_{\delta'}$ for some $\delta' > 0$ and bounded otherwise.  By Proposition \ref{prop:linear}, 
	\begin{equation*}
	\tr f(\hat{P}_{n}) =\sum_{i=1}^{n}f(\lambda_{i}(\hat{P}_{n}))=\sum_{i=1}^{mn}\frac{1}{m}f(\lambda_{i}(\blmat{Y}{n}^{m}))=\sum_{i=1}^{mn}g(\lambda_{i}(\blmat{Y}{n}))=\tr g(\blmat{Y}{n}).
	\end{equation*}
	\begin{Details}
	\begin{align*}
	\tr f(\hat{P}_{n}) &=\sum_{i=1}^{n}f(\lambda_{i}(\hat{P}_{n}))\\
	&=\sum_{i=1}^{mn}\frac{1}{m}f(\lambda_{i}(\blmat{Y}{n}^{m}))\\
	&=\sum_{i=1}^{mn}g(\lambda_{i}(\blmat{Y}{n}))\\
	&=\tr g(\blmat{Y}{n}).
	\end{align*}
	\end{Details}
	By assumption, $\tr g(\blmat{Y}{n})\oindicator{\Omega_{n}} -\E\left[\tr g(\blmat{Y}{n})\oindicator{\Omega_{n}} \right]$	converges to a mean-zero Gaussian with covariance structure given in \eqref{Equ:Thm:LinTruncCLT:1} and \eqref{Equ:Thm:LinTruncCLT:2}. We will show that 
	\begin{equation*}
	\E\left|\tr g(\blmat{Y}{n})\oindicator{\Omega_{n}} -\E\left[\tr g(\blmat{Y}{n})\oindicator{\Omega_{n}} \right]-\left(\tr f(\hat{P}_{n})\oindicator{\hat{E}_{n}} -\E\left[\tr f(\hat{P}_{n})\oindicator{\hat{E}_{n}} \right]\right)\right|=o(1).
	\end{equation*}
	To this end, observe that 
	\begin{align*}
	&\E\left|\tr g(\blmat{Y}{n})\oindicator{\Omega_{n}} -\E\left[\tr g(\blmat{Y}{n})\oindicator{\Omega_{n}} \right]-\left(\tr f(\hat{P}_{n})\oindicator{\hat{E}_{n}} -\E\left[\tr f(\hat{P}_{n})\oindicator{\hat{E}_{n}} \right]\right)\right|\\
	&\quad\quad= \E\left|\tr g(\blmat{Y}{n})\oindicator{\Omega_{n}}-\tr f(\hat{P}_{n})\oindicator{\hat{E}_{n}} -\E\left[\tr g(\blmat{Y}{n})\oindicator{\Omega_{n}} -\tr f(\hat{P}_{n})\oindicator{\hat{E}_{n}} \right]\right|\\
	&\quad\quad\leq 2\E\left|\left(\tr g(\blmat{Y}{n})-\tr f(\hat{P}_{n})\right)\oindicator{\hat{E}_{n}\cap\Omega_{n}}\right|\\
	&\quad\quad\quad\quad\quad\quad\quad\quad\quad+2\E\left|\tr g(\blmat{Y}{n})\oindicator{\Omega_{n}\cap\hat{E}_{n}^{c}}\right|+2\E\left|\tr f(\hat{P}_{n})\oindicator{\hat{E}_{n}\cap\Omega_{n}^{c}}\right|\\
	&\quad\quad \ll_{f,g}0+ n\P\left(\Omega_{n}\cap\hat{E}_{n}^{c}\right)+n\P\left(\hat{E}_{n}\cap\Omega_{n}^{c}\right)\\
	&\quad\quad=o(1)
	\end{align*}
	since $\Omega_{n}$ and $\hat{E}_{n}$ both hold with overwhelming probability by assumption and Lemma \ref{Lem:E_nHatOverwhelming} respectively. To see that the variance follows as claimed, observe that by letting $z=re^{\theta_{1}\sqrt{-1}}$ and $w=re^{\theta_{2}\sqrt{-1}}$ where $r=1+\delta$, we have
	\begin{align*}
	&\frac{1}{4\pi^{2}}\oint_{\mathcal{C}}\oint_{\mathcal{C}}\frac{1}{m^{2}}f(z^{m})\overline{f(w^{m})}\frac{m^{2}(z\bar{w})^{m-1}}{((z\bar{w})^{m}-1)^{2}}dzd\bar{w}\\
	&\quad=\frac{1}{4\pi^{2}}\int_{0}^{2\pi}\int_{0}^{2\pi}f(r^{m}e^{m\theta_{1}\sqrt{-1}})\overline{f(r^{m}e^{m\theta_{2}\sqrt{-1}})}\frac{r^{2m}e^{m\theta_{1}\sqrt{-1}}e^{-m\theta_{2}\sqrt{-1}}}{(r^{2m}e^{m\theta_{1}\sqrt{-1}}e^{-m\theta_{2}\sqrt{-1}}-1)^{2}}d\theta_{1}d\theta_{2}.
	\end{align*}
	Next by the substitution $m\theta_{1}=\tau_{1}$ and $m\theta_{2}=\tau_{2}$, and by noting that this substitution wraps around the contour $m$ times, 
	\begin{align*} 
	&\frac{1}{4\pi^{2}}\int_{0}^{2\pi}\int_{0}^{2\pi}f(r^{m}e^{m\theta_{1}\sqrt{-1}})\overline{f(r^{m}e^{m\theta_{2}\sqrt{-1}})}\frac{r^{2m}e^{m\theta_{1}\sqrt{-1}}e^{-m\theta_{2}\sqrt{-1}}}{(r^{2m}e^{m\theta_{1}\sqrt{-1}}e^{-m\theta_{2}\sqrt{-1}}-1)^{2}}d\theta_{1}d\theta_{2}\\
	&\quad\quad=\frac{1}{4\pi^{2}}\int_{0}^{2\pi}\int_{0}^{2\pi}f(r^{m}e^{\tau_{1}\sqrt{-1}})\overline{f(r^{m}e^{\tau_{2}\sqrt{-1}})}\frac{r^{2m}e^{\tau_{1}\sqrt{-1}}e^{-\tau_{2}\sqrt{-1}}}{(r^{2m}e^{\tau_{1}\sqrt{-1}}e^{-\tau_{2}\sqrt{-1}}-1)^{2}}d\tau_{1}d\tau_{2}
	\end{align*}
	and finally, by letting $z'=r^{m}e^{\tau_{1}\sqrt{-1}}$ and $w'=r^{m}e^{\tau_{2}\sqrt{-1}}$, we have the claimed variance. 
	\begin{Details} 
	\begin{align*}
	&\frac{1}{4\pi^{2}}\oint_{\mathcal{C}}\oint_{\mathcal{C}}f(r^{m}e^{\tau_{1}i})\overline{f(r^{m}e^{\tau_{2}i})}\frac{r^{2m}e^{\tau_{1}i}e^{-\tau_{2}i}}{(r^{2m}e^{\tau_{1}i}e^{-\tau_{2}i}-1)^{2}}d\tau_{1}d\tau_{2}\\
	&\quad=\frac{1}{4\pi^{2}}\oint_{\mathcal{C}}\oint_{\mathcal{C}}f(z')\overline{f(w')}(z'\bar{w}'-1)^{-2}dz'd\bar{w}'\\
	\end{align*}
	as claimed. The same calculation shows that 
	\begin{align*}
	-\frac{1}{4\pi^{2}}\oint_{\mathcal{C}}\oint_{\mathcal{C}}\frac{1}{m^{2}}&f(z^{m})f(w^{m})\frac{m^{2}(zw)^{m-1}}{((zw)^{m}-1)^{2}}dzdw\\
	&=-\frac{1}{4\pi^{2}}\oint_{\mathcal{C}}\oint_{\mathcal{C}}f(z')f(w')(z'w'-1)^{-2}dz'dw',
	\end{align*} 
	which concludes the variance calculation.
	\end{Details}
\end{proof}

\begin{Details} 
We have reduced the proof of the main theorem to prove the convergence of 
\begin{equation*}
\tr f(\blmat{Y}{n})\oindicator{\Omega_{n}} -\E\left[\tr f(\blmat{Y}{n})\oindicator{\Omega_{n}} \right].
\end{equation*}
\end{Details}

\section{Proof of Theorem \ref{Thm:LinearizedTruncatedCLT}}
\label{Sec:ProofOfReduction}
\begin{Details}
	At this point, note that Theorem \ref{thm:mainCLT} will follow from proving Theorem \ref{Thm:LinearizedTruncatedCLT}, where we are now working with a linearized truncated $mn\times mn$ matrix on the event $\Omega_{n}$ where the least singular value is bounded away from zero.
\end{Details}

It remains to prove Theorem \ref{Thm:LinearizedTruncatedCLT}. Define the resolvent
\begin{equation}
	\mathcal{G}_{n}(z):=(\blmat{Y}{n}-zI)^{-1}
	\label{Def:BlockResolvent}
\end{equation}
where $\blmat{Y}{n}$ is defined in \eqref{Def:Y_n} and for $z$ not an eigenvalue of $\blmat{Y}{n}$. Also define 
\begin{equation}
\Xi_{n}(z):=\tr \mathcal{G}_{n}(z)\oindicator{\Omega_{n}}-\E\left[\tr \mathcal{G}_{n}(z)\oindicator{\Omega_{n}}\right].
\label{Def:StochasticProcessXi}
\end{equation}

By Lemma \ref{Lem:Omega_nOverwhelming}, $\Omega_{n}$ holds with overwhelming probability. On the event $\Omega_n$, the eigenvalues of $\blmat{Y}{n}$ are contained in the interior of the disk $D_\delta$, and so, by Cauchy's integral formula, we have
\begin{align*}
\tr g(\blmat{Y}{n})\oindicator{\Omega_{n}} -&\E\left[\tr g(\blmat{Y}{n})\oindicator{\Omega_{n}} \right]\\ 
&= \sum_{i=1}^{mn}g(\lambda_{i}(\blmat{Y}{n}))\oindicator{\Omega_{n}}-\E\left[\sum_{i=1}^{mn}g(\lambda_{i}(\blmat{Y}{n}))\oindicator{\Omega_{n}}\right]\\
&= \sum_{i=1}^{mn}-\frac{1}{2\pi i}\oint_{\mathcal{C}}\frac{g(z)\oindicator{\Omega_{n}}}{\lambda_{i}(\blmat{Y}{n})-z}dz-\E\left[\sum_{i=1}^{mn}-\frac{1}{2\pi i}\oint_{\mathcal{C}}\frac{g(z)\oindicator{\Omega_{n}}}{\lambda_{i}(\blmat{Y}{n})-z}dz\right]\\
&= -\frac{1}{2\pi i}\oint_{\mathcal{C}}g(z)\left(\tr\mathcal{G}_{n}(z)\oindicator{\Omega_{n}}-\E\left[\tr\mathcal{G}_{n}(z)\oindicator{\Omega_{n}}\right]\right)dz\\
&= -\frac{1}{2\pi i}\oint_{\mathcal{C}}g(z)\Xi_{n}(z)dz
\end{align*}
where for the remainder of the paper we let $\mathcal{C}$ denote the contour on the boundary of $D_{\delta}$. 
\begin{Details} 
	It is useful to note that since $\Xi_{n}(z)$ depends on the eigenvalues of $\blmat{Y}{n}$ it is random variable. Since it is also a function of $z$ on the contour $\mathcal{C}$, $\{\Xi_{n}(z)\}_{z\in\mathcal{C}}$ is a sequence of random functions. 
\end{Details}
\begin{Details} 
	In particular, for a fixed $z_{0}$, $\Xi_{n}(z_{0})$ is a random variable taking values in $\C$ for each $n$ and for a fixed realization of the $m$ random matrices of a fixed size, $\Xi_{N_{0}}(z)$ is a continuous function on the contour $\mathcal{C}$. Thus, $\{\Xi_{n}(z)\}_{n=1}^{\infty}$ is a sequence of stochastic processes. 
\end{Details}  

We will reduce the proof of Theorem \ref{Thm:LinearizedTruncatedCLT} to showing the convergence of the resolvent process $(\Xi_n(z))_{z \in \mathcal{C}}$ to the limiting Gaussian process defined in the following lemma.

\begin{lemma}
	Let $\{\Xi(z)\}_{z\in\mathcal{C}}$ denote the mean-zero Gaussian process with covariance structure defined by
	\begin{equation}
	\E\left[\Xi(z)\overline{\Xi(w)}\right]=\frac{m^{2}(z\bar{w})^{m-1}}{((z\bar{w})^{m}-1)^{2}}
	\label{Equ:CovarianceForProcess}
	\end{equation} 
	and with the property that $\overline{\Xi(z)} = \Xi(\bar{z})$ for all $z \in \mathcal{C}$.  
	If $g$ be a function which is analytic on some neighborhood of the disk $D_{\delta}$ and bounded otherwise, then
	\begin{equation} 
	-\frac{1}{2\pi i}\oint_{\mathcal{C}}g(z)\Xi(z)dz
	\label{Equ:IntegralOfProcess}
	\end{equation}
	is a mean-zero Gaussian random variable with covariance structure
	\begin{equation}
	\E\left[\left(\frac{1}{2\pi i}\oint_{\mathcal{C}}g(z)\Xi(z)dz\right)^{2}\right]=-\frac{1}{4\pi^{2}}\oint_{\mathcal{C}}\oint_{\mathcal{C}}g(z)g(w)\frac{m^{2}(zw)^{m-1}}{((zw)^{m}-1)^{2}}dzdw \label{Equ:Lem:IntOfGaussIsGauss:1}
	\end{equation}
	and  
	\begin{equation}\E\left[\frac{1}{2\pi i}\oint_{\mathcal{C}}g(z)\Xi(z)dz\cdot\overline{\frac{1}{2\pi i}\oint_{\mathcal{C}}g(z)\Xi(z)dz}\right] =\frac{1}{4\pi^{2}}\oint_{\mathcal{C}}\oint_{\mathcal{C}}g(z)\overline{g(w)}\frac{m^{2}(z\bar{w})^{m-1}}{((z\bar{w})^{m}-1)^{2}}dzd\bar{w}.\label{Equ:Lem:IntOfGaussIsGauss:2}
	\end{equation}
	\label{Lem:IntOfGaussIsGauss}
\end{lemma}
\begin{proof}
	The proof follows by a number of standard techniques.  For instance, one can deduce the conclusion by computing moments (with an application of Fubini's theorem); we omit the details.  
	\begin{Details}
		\begin{align*}
		&\E\left[\frac{1}{2\pi i}\oint_{\mathcal{C}}g(z)\Xi(z)dz\cdot\frac{1}{2\pi i}\oint_{\mathcal{C}}g(w)\Xi(w)dw\right]\\
		&\quad\quad=-\frac{1}{4\pi^{2}}\oint_{\mathcal{C}}\oint_{\mathcal{C}}g(z)g(w)\E\left[\Xi(z)\Xi(w)\right]dzdw\\
		&\quad\quad=-\frac{1}{4\pi^{2}}\oint_{\mathcal{C}}\oint_{\mathcal{C}}g(z)g(w)\frac{m^{2}(zw)^{m-1}}{((zw)^{m}-1)^{2}}dzdw\\
		\end{align*}
		and 
		\begin{align*}
		&\E\left[\frac{1}{2\pi i}\oint_{\mathcal{C}}g(z)\Xi(z)dz\cdot\overline{\frac{1}{2\pi i}\oint_{\mathcal{C}}g(w)\Xi(w)dw}\right]\\
		&\quad\quad=\frac{1}{4\pi^{2}}\oint_{\mathcal{C}}\oint_{\mathcal{C}}g(z)\overline{g(w)}\E\left[\Xi(z)\overline{\Xi(w)}\right]dzd\bar{w}\\
		&\quad\quad=\frac{1}{4\pi^{2}}\oint_{\mathcal{C}}\oint_{\mathcal{C}}g(z) \overline{g(w)}\frac{m^{2}(z\bar{w})^{m-1}}{((z\bar{w})^{m}-1)^{2}}dzd\bar{w}\\
		\end{align*}
	\begin{equation*}
	\E\left[\frac{1}{2\pi i}\oint_{\mathcal{C}}g(z)\Xi(z)dz\cdot\frac{1}{2\pi i}\oint_{\mathcal{C}}g(w)\Xi(w)dw\right]=-\frac{1}{4\pi^{2}}\oint_{\mathcal{C}}\oint_{\mathcal{C}}g(z)g(w)\frac{m^{2}(zw)^{m-1}}{((zw)^{m}-1)^{2}}dzdw
	\end{equation*}
	and 
	\begin{equation*}
	\E\left[\frac{1}{2\pi i}\oint_{\mathcal{C}}g(z)\Xi(z)dz\cdot\overline{\frac{1}{2\pi i}\oint_{\mathcal{C}}g(w)\Xi(w)dw}\right]=\frac{1}{4\pi^{2}}\oint_{\mathcal{C}}\oint_{\mathcal{C}}g(z)\overline{g(w)}\frac{m^{2}(z\bar{w})^{m-1}}{((z\bar{w})^{m}-1)^{2}}dzd\bar{w}
\end{equation*}
	as desired. 
\end{Details}
\end{proof}
\begin{Details} 
	\begin{remark}
		Note that since this is a complex random variable, we need to know both of the above quantities in order to fully characterize the distribution. 
		The second expectation gives the second moment of the real part plus the second moment of the imaginary part. The second expectation gives the covariance between the real and imaginary parts, and allows us to distinguish between the second moment of the real part and the second moment of the imaginary part.
		Indeed, if $X$ and $Y$ are real random variables, and if we define $Z=X+Yi$, then $\E[Z\bar{Z}]=\E[X^{2}]+\E[Y^{2}]$ and $\E[Z^{2}]=\E[X^{2}]-\E[Y^{2}]+2i\E[XY]$.  
	\end{remark}
\end{Details}

The following result shows that $\{\Xi(z)\}_{z \in \mathcal{C}}$ is indeed the limiting distribution of the resolvent process  $\{\Xi_{n}(z)\}_{z\in\mathcal{C}}$.

\begin{theorem}
	Let $\{\Xi_{n}(z)\}_{z\in\mathcal{C}}$ be the sequence of stochastic processes defied in \eqref{Def:StochasticProcessXi} for $z$ on the contour $\mathcal{C}$ around the boundary of the disk $D_{\delta}$. Then $\{\Xi_{n}(z)\}_{z\in\mathcal{C}}$ converges in distribution to the mean-zero Gaussian process $\{\Xi(z)\}_{z\in\mathcal{C}}$ defined in Lemma \ref{Lem:IntOfGaussIsGauss}.  
	\label{Thm:StochasticProcessConv}
\end{theorem}

The bulk of the paper is devoted to the proof of Theorem \ref{Thm:StochasticProcessConv}.  Before doing so, let us complete the proof of Theorem \ref{Thm:LinearizedTruncatedCLT} assuming Theorem \ref{Thm:StochasticProcessConv}.

First note that by Lemma \ref{Lem:Omega_nOverwhelming}, there exists $c>0$ such that $\Omega_{n}$ holds with overwhelming probability. Next, observe that $\Xi_{n}(z)$ and $\Xi(z)$ are random elements in the space of continuous functions on the contour $\mathcal{C}$, which is a metric space with respect to the supremum norm. Since the map 
\begin{equation}
\Xi_{n}(z)\mapsto \frac{1}{2\pi i}\oint_{\mathcal{C}}g(z)\Xi_{n}(z)dz
\label{Equ:MapContMappingThm}
\end{equation}
is continuous in this metric space, the continuous mapping theorem (see \cite[Theorem 25.7]{Bill}) and Lemma \ref{Lem:IntOfGaussIsGauss} show that Theorem \ref{Thm:StochasticProcessConv} implies Theorem \ref{Thm:LinearizedTruncatedCLT}.  Indeed, if $\{\Xi_{n}(z)\}_{z\in\mathcal{C}}$ converges in distribution to $\{\Xi(z)\}_{z\in\mathcal{C}}$, then $\frac{1}{2\pi i}\oint_{\mathcal{C}}g(z)\Xi_{n}(z)dz$ converges in distribution to $\frac{1}{2\pi i}\oint_{\mathcal{C}}g(z)\Xi(z)dz$ as desired.

In order to prove Theorem \ref{Thm:StochasticProcessConv}, we will use the following characterization of convergence, which is a result of Theorems 7.5 and 12.3 from \cite{BillOld}.

\begin{theorem}
	Suppose that $\{\Xi(z)\}_{z\in\mathcal{C}}, \{\Xi_{n}(z)\}_{z\in\mathcal{C}}$ for $n\geq 1$ are stochastic processes on the contour $\mathcal{C}=\{z\in\C\;:\;|z|=1+\delta\}$. Suppose $\Xi(z)$, $(\Xi_{n}(z))_{n=1}^{\infty}$ satisfy
	\begin{equation}
	(\Xi_{n}(z_{1}),\Xi_{n}(z_{2}),\ldots,\Xi_{n}(z_{L}))\longrightarrow (\Xi(z_{1}),\Xi(z_{2}),\ldots,\Xi(z_{L}))
	\label{Equ:FinDimConv}
	\end{equation}
	in distribution as $n\rightarrow\infty$ for any fixed positive integer $L$ and any $z_{1},\ldots,z_{L}\in\mathcal{C}$, and suppose that there exists a constant $c>0$ such that
	\begin{equation}
	\sup_{z,w \in \mathcal{C}, z \neq w} \E\left|\frac{\Xi_{n}(z)-\Xi_{n}(w)}{z-w}\right|^{2}\leq c
	\label{Equ:TightnessCharacterization}
	\end{equation}
	 for all $n$. Then $\{\Xi_{n}(z)\}_{z\in\mathcal{C}}$ converges in distribution to $\{\Xi(z)\}_{z\in\mathcal{C}}$ as $n\rightarrow \infty$. 
	\label{Thm:ConvergenceCharacterization}
\end{theorem}

\begin{Details}
	\begin{theorem}[Theorem 12.3 from \cite{BillOld}]
		The sequence $\{X_{n}\}$ of random elements of the space of continuous functions on $[0,1]$ is tight if it satisfies the following two conditions.
		\begin{enumerate}[label=\emph{(\roman*)}]
			\item The sequence $\{X_{n}(0)\}$ is tight
	\item There exist constants $\gamma\geq 0$ and $\alpha>1$, and a nondecreasing function $F$ on $[0,1]$ such that
	\begin{equation*}
	\P(|X_{n}(t_{2})-X_{n}(t_{1})|\geq \lambda)\leq \frac{1}{\lambda^{\gamma}}|F(t_{2})-F(t_{1})|^{\alpha}
	\end{equation*}
	holds for all $t_{1},t_{2}$ and $n$, and for all positive $\lambda$.
\end{enumerate}
\end{theorem}
\end{Details}

\begin{Details} 
\begin{theorem}[Theorem 7.5 from \cite{BillOld}]
	Suppose that $X, X^{1},X^{2},\ldots$ are random functions. If 
	\begin{equation}
	(X_{t_{1}}^{n},X_{t_{2}}^{n},...,X_{y_{k}}^{n})\rightarrow (X_{t_{1}},X_{t_{2}},...,X_{t_{k}})
	\label{Equ:SecondBillingsleyConditionFinDim}
	\end{equation}
	for any fixed positive integer $k$ and any $t_{1},...,t_{k}$, and if
	\begin{equation}
	\lim_{\delta\rightarrow 0}\limsup_{n\rightarrow\infty}\P\left(\sup_{|s-t|<\delta}\left|X_{t}^{n}-X_{s}^{n}\right|>\varepsilon \right)=0
	\label{Equ:SecondBillingsleyCondition}
	\end{equation}
	for each positive $\varepsilon$, then $X_{n}$ converges in distribution to $X$ as $n\rightarrow \infty$. 
\end{theorem}
 The first condition states that finite dimensional distributions converge, and the second condition is a tightness condition. In fact, the following theorem characterizes \eqref{Equ:SecondBillingsleyCondition} in a slightly different way.

 \begin{theorem}[Theorem 12.3 from \cite{BillOld}]
 	The sequence $\{X_{n}\}$ of random elements of the space of continuous functions on $[0,1]$ is tight if it satisfies the following two conditions:
 	\begin{enumerate}[label=\emph{(\roman*)}]
 		\item The sequence $\{X_{n}(0)\}$ is tight
 		\item There exist constants $\gamma\geq 0$ and $\alpha>1$, and a nondecreasing function $F$ on $[0,1]$ such that
 		\begin{equation*}
 		\P(|X_{n}(t_{2})-X_{n}(t_{1})|\geq \lambda)\leq \frac{1}{\lambda^{\gamma}}|F(t_{2})-F(t_{1})|^{\alpha}
 		\end{equation*}
 		holds for all $t_{1},t_{2}$ and $n$, and for all positive $\lambda$.
 	\end{enumerate}
 \label{Thm:TightnessCharacterization}
 \end{theorem}
\end{Details}

\begin{Details}
 \begin{remark}
Note we are considering a stochastic process which is a function of complex numbers of the contour $\mathcal{C}=\{z\in\C\;:\;|z|=1+\delta\}$, but by the parameterization $z=(1+\delta)e^{2\pi\theta\sqrt{-1}}$, we can consider our stochastic process to be a sequence of random functions on the interval $[0,1]$.
\end{remark}
\end{Details}

The proof of Theorem \ref{Thm:StochasticProcessConv} reduces to showing that the two conditions from Theorem \ref{Thm:ConvergenceCharacterization} are satisfied.  In Section \ref{Sec:FiniteDimDist} we prove the convergence of the finite dimensional distributions \eqref{Equ:FinDimConv}. Section \ref{Sec:Tightness} contains the proof of the tightness of the sequence of stochastic processes \eqref{Equ:TightnessCharacterization}.

\section{Convergence of Finite Dimensional Distributions}
\label{Sec:FiniteDimDist}
This section is devoted to proving the convergence of the finite dimensional distributions of the stochastic process $\{\Xi_{n}(z)\}_{z\in\mathcal{C}}$. In particular, this section will be devoted to the proof of the following theorem.

\begin{theorem}
	For a fixed positive integer $L$ and any collection $(z_{1},z_{2},\dots,z_{L})$ such that $|z_{i}|=1+\delta$ for $1\leq i\leq L$, the random vector $\left(\Xi_{n}(z_{1}),\Xi_{n}(z_{2}),\dots,\Xi_{n}(z_{L})\right)$ converges in distribution to the random vector $	\left(\Xi(z_{1}),\Xi(z_{2}),\dots,\Xi(z_{L})\right)$ where $\{\Xi(z)\}_{z\in\mathcal{C}}$ is defined in Lemma \ref{Lem:IntOfGaussIsGauss}.
	\label{Thm:FiniteDimDist}
\end{theorem}

To prove Theorem \ref{Thm:FiniteDimDist}, we first make a sequence of reductions inspired by the proofs in \cite{RiS,OR:CLT}. First, recall that by the Cramer--Wold theorem, it is sufficient to prove the convergence of an arbitrary linear combination of the components of the vector in question. Ergo, by the Cramer--Wold theorem, it is sufficient to show that 
\begin{equation}
\sum_{l=1}^{L}(\alpha_{l}\Xi_{n}(z_{l})+\beta_{l}\overline{\Xi_{n}(z_{l})})
\label{Equ:LinearCombo}
\end{equation}
converges in distribution to 
\begin{equation}
\sum_{l=1}^{L}(\alpha_{l}\Xi(z_{l})+\beta_{l}\overline{\Xi(z_{l})})
\label{Equ:LinearComboLimit}
\end{equation}
for $\alpha_{l},\beta_{l}\in\C$ such that \eqref{Equ:LinearCombo} is real. As stated in Lemma \ref{Lem:IntOfGaussIsGauss}, since $\overline{\Xi(z)}=\Xi(\overline{z})$, it is sufficient to compute $\E[\Xi(z_{i})\overline{\Xi(z_{j})}]$ in order to characterize the covariance structure of $	\left(\Xi(z_{1}),\Xi(z_{2}),\dots,\Xi(z_{L})\right)$.
\begin{Details} 
\[\E\left[\Xi_{n}(z_{i})\Xi_{n}(z_{i})\right]\text{ and }\E\left[\Xi_{n}(z_{i})\overline{\Xi_{n}(z_{i})}\right]\]
in order to fully characterize the variance of this component. However, since the entries in the original matrices are assumed to be real, $\overline{\Xi_{n}(z_{i})}=\Xi_{n}(\overline{z_{i}})$. In addition, since $z_{i}$ is a point on the contour $\mathcal{C}$ around the boundary of the disk or radius $1+\delta$, $\overline{z_{i}}$ is just another point on the contour. Therefore, computing $\E\left[\Xi_{n}(z_{i})\overline{\Xi_{n}(z_{i})}\right]$ is equivalent to computing $\E\left[\Xi_{n}(z_{i})\Xi_{n}(z_{j})\right]$. Finally, since setting $i=j$ covers the first case as well, it is sufficient to calculate only $\E\left[\Xi_{n}(z_{i})\Xi_{n}(z_{j})\right]$ in order to characterize the distribution of $\Xi_{n}(z_{i})$. To characterize the distribution of the entire vector, we also must calculate the covariance between two components, which comes down to computing
\[\E\left[\Xi_{n}(z_{i})\Xi_{n}(z_{j})\right]\text{ and }\E\left[\Xi_{n}(z_{i})\overline{\Xi_{n}(z_{j})}\right]\]
for some fixed $1\leq i,j\leq L$. However, by the same argument as before, this reduces again to computing only $\E\left[\Xi_{n}(z_{i})\Xi_{n}(z_{j})\right]$.
\end{Details}

\begin{remark}
	\label{Remark:ComplexCase2}
	In the case where the atom random variables are complex-valued, we would need to compute $\E[\Xi(z_{i})\Xi(z_{j})],$ and $\E[\Xi(z_{i})\overline{\Xi(z_{j})}]$ in order to characterize the covariance. 
\end{remark}
In order to prove Theorem \ref{Thm:FiniteDimDist}, we will express the sum in \eqref{Equ:LinearCombo} as a martingale difference sequence. Let $c_{k}$ denote the $k$th column of $\blmat{Y}{n}$ and define the $\sigma$-algebras
\begin{equation}
\mathcal{F}_{k}:=\sigma(c_{1},\dots,c_{k},c_{n+1},\dots,c_{n+k},\dots,c_{(m-1)n+1},\dots,c_{(m-1)n+k})
\label{Def:F_{k}}
\end{equation}
 for $1\leq k\leq n$. Note that $\mathcal{F}_{k}$ is the $\sigma$-algebra generated by the first $k$ columns of each of the $n\times n$ blocks of $\blmat{Y}{n}$. 
 \begin{Details} 
 	Note that each subsequent $\sigma$-algebra is generated with an additional $m$ columns than the previous one.
 \end{Details}
Define $\mathcal{F}_{0}$ to be the trivial $\sigma$-algebra and note that $\mathcal{F}_{0}\subseteq \mathcal{F}_{1}\subseteq \dots \subseteq \mathcal{F}_{n}$. Then define the conditional expectation $\E_{k}[\;\cdot\;]:=\E[\;\cdot\;|\mathcal{F}_{k}]$ and observe that by definition of the $\sigma$-algebras, $\E_{0}[\;\cdot\;]=\E[\;\cdot\;]$ and $\E_{n}[\blmat{Y}{n}]=\blmat{Y}{n}$.

Also define $\blmat{Y}{n}^{(k)}$ to be the matrix $\blmat{Y}{n}$ with the columns $c_{k},c_{n+k},c_{2n+k},\dots, c_{(m-1)n+k}$ replaced with zeros. Note that $\blmat{Y}{n}^{(k)}$ can be viewed as the matrix $\blmat{Y}{n}$ with the $k$th column in each block replaced by zeros. By Corollary \ref{Cor:Omega_nkOverwhelming}, for every $\delta>0$, there exists some $c>0$ such that the event 
\begin{equation}
\Omega_{n,k}:=\left\{\inf_{|z|> 1+\delta/2}s_{mn}\left(\blmat{Y}{n}^{(k)}-zI\right)\geq c\right\}
\label{Def:Omega_{n,k}}
\end{equation} 
holds with overwhelming probability. Finally, define the resolvent
\begin{equation}
\mathcal{G}_{n}^{(k)}(z):=\left(\blmat{Y}{n}^{(k)}-zI\right)^{-1}.
\label{Def:G^{(k)}_{n}}
\end{equation} 

The following lemma follows from an application of Proposition \ref{Prop:LargeAndSmallSingVals}.

\begin{lemma}
	Define the events $\Omega_{n}$ and $\Omega_{n,k}$ as in \eqref{Def:Omega_n} and \eqref{Def:Omega_{n,k}} respectively. Then there exist a constant $C>0$ such that, for all $z\in\mathcal{C}$, $\lnorm \mathcal{G}_{n}(z)\rnorm \leq C$  surely on $\Omega_{n}$ and $\lnorm \mathcal{G}^{(k)}_{n}(z)\rnorm \leq C$ surely on $\Omega_{n,k}$. 
	\label{Lem:G_nkBounded}
	\label{Lem:G_nBounded}
\end{lemma}

With these definitions, we may write 
\begin{align*}
\Xi_{n}(z) &= \tr \mathcal{G}_{n}(z)\oindicator{\Omega_{n}}-\E\left[\tr \mathcal{G}_{n}(z)\oindicator{\Omega_{n}}\right]\\
&=\sum_{k=1}^{n}\left(\E_{k}[\tr \mathcal{G}_{n}(z)\oindicator{\Omega_{n}}]-\E_{k-1}[\tr \mathcal{G}_{n}(z)\oindicator{\Omega_{n}}]\right)\\
&=\sum_{k=1}^{n}Z_{n,k}(z)
\end{align*}
where we define
\begin{equation}
Z_{n,k}(z):=(\E_{k}-\E_{k-1})[\tr \mathcal{G}_{n}(z)\oindicator{\Omega_{n}}].
\label{Def:Z_{n,k}}
\end{equation}
With this notation, we can rewrite the linear combination from \eqref{Equ:LinearCombo} as 
\begin{Details}
\begin{align*}
\sum_{l=1}^{L}(\alpha_{l}\Xi_{n}(z_{l})+\beta_{l}\overline{\Xi_{n}(z_{l})})
&=\sum_{l=1}^{L}\left(\alpha_{l}\sum_{k=1}^{n}Z_{n,k}(z_{l})+\beta_{l}\overline{\sum_{k=1}^{n}Z_{n,k}(z_{l})}\right)\\
&=\sum_{k=1}^{n}\sum_{l=1}^{L}\left(\alpha_{l}Z_{n,k}(z_{l})+\beta_{l}\overline{Z_{n,k}(z_{l})}\right).
\end{align*}
\end{Details} 
\begin{equation*}
\sum_{l=1}^{L}(\alpha_{l}\Xi_{n}(z_{l})+\beta_{l}\overline{\Xi_{n}(z_{l})})=\sum_{k=1}^{n}\sum_{l=1}^{L}\left(\alpha_{l}Z_{n,k}(z_{l})+\beta_{l}\overline{Z_{n,k}(z_{l})}\right).
\end{equation*}
Let $M_{n,k}:=\sum_{l=1}^{L}\left(\alpha_{l}Z_{n,k}(z_{l})+\beta_{l}\overline{Z_{n,k}(z_{l})}\right)$ for any fixed integer $L>0$, and $z_{i}\in\mathcal{C}$, and any $\alpha_{i},\beta_{i}\in\C$ such that $M_{n,k}$ is real and denote
\begin{equation} 
M_{n}:=\sum_{k=1}^{n}M_{n,k}.
\label{Def:M_{n}}
\end{equation} 
\begin{Details} 
	Notice that in fact $M_{n}$ is the linear combination of the finite dimensional distributions, so the goal of this section is to prove that $M_{n}$ converges to a mean-zero Gaussian with appropriate variance. 
\end{Details}
In order to simplify computations, it will be beneficial to work with a slightly different expression in which some reductions are made.

\begin{lemma}
	Define $M_{n}$ as in \eqref{Def:M_{n}}, define $U_{k}$ to be the $mn\times m$ matrix which contains as its columns $c_{k},c_{n+k},\dots,c_{(m-1)n+k}$, and define $V_{k}$ to be the $mn\times m$ matrix which contains as its columns $e_{k},e_{n+k},\dots, e_{(m-1)n+k}$ where $e_{1},\dots ,e_{mn}$ denote the standard basis elements of $\C^{mn}$. Define the martingale difference sequence 
	\begin{equation*}
	\breve{M}_{n}:=\sum_{k=1}^{n}\breve{M}_{n,k}=\sum_{k=1}^{n}\left(\sum_{l=1}^{L}\alpha_{l}\breve{Z}_{n,k}(z_{l})+\beta_{l}\overline{\breve{Z}_{n,k}(z_{l})}\right),
	\end{equation*}
	and
	\begin{equation}
	\breve{Z}_{n,k}(z):=-\E_{k}\left[\tr\left(V_{k}^{T}(\mathcal{G}_{n}^{(k)}(z))^{2}U_{k}\right)\oindicator{\Omega_{n,k}}\right].
	\label{Def:breveZ}
	\end{equation}
	Then, as $n\rightarrow\infty$, if $\breve{M}_{n}$ converges in distribution, then $M_{n}$ also converges to the same distributional limit. 
	\label{Lem:ReductionZ}
\end{lemma}

\begin{Details}
	Before proving this lemma, note that the difference between $M_{n}$ and $\breve{M}_{n}$ relies completely on the difference between $Z_{n,k}(z)$ and $\breve{Z}_{n,k}(z)$. If one can show that the difference of $M_{n}$ and $\breve{M}_{n}$ converges to zero in probability, then these two random variables will converge in distribution to the same limit. 
	It is sufficient to prove that \[\E\left|M_{n}-\breve{M}_{n}\right|^{2}=o(1).\] 
	The purpose of Lemma \ref{Lem:ReductionZ} is to introduce independence in each term. Namely, observe that $\mathcal{G}_{n}^{(k)}(z)\oindicator{\Omega_{n,k}}$ is independent of the columns $c_{k},c_{n+k},\dots, c_{(m-1)n+k}$. This independence will be helpful in proving Theorem \ref{Thm:FiniteDimDist}. 
\end{Details}

We first develop some results we will need in the proof of Lemma \ref{Lem:ReductionZ}. Define the event
\begin{equation}
Q_{n,k}(z):=\left\{\lnorm V_{k}^{T}\mathcal{G}_{n}^{(k)}(z)U_{k}\oindicator{\Omega_{n,k}} \rnorm\leq 1/2\right\}.
\label{Def:EventQ}
\end{equation}
We will also need the following lemma. 
\begin{lemma}
	\label{Lem:QOverwhelming}
	Define the event $Q_{n,k}(z)$ as in \eqref{Def:EventQ}. Then, uniformly for any $z\in \mathcal{C}$, $Q_{n,k}$ holds with overwhelming probability 
\end{lemma}
\begin{proof}
	Let $\alpha>0$ be arbitrary. We will prove that the complementary event holds with probability at most $O_{\alpha}(n^{-\alpha})$ uniformly for any $z\in\mathcal{C}$. By Markov's inequality and the forthcoming Lemma \ref{Lem:BoundingVGU}, uniformly for any $z\in\mathcal{C}$ and for any $p\geq 2$, 
	\begin{align*}
	\P\left(\lnorm V_{k}^{T}\mathcal{G}_{n}^{(k)}(z)U_{k}\oindicator{\Omega_{n,k}}\rnorm \geq 1/2\right)&\leq \frac{\E\lnorm V_{k}^{T}\mathcal{G}_{n}^{(k)}(z)U_{k}\oindicator{\Omega_{n,k}}\rnorm^{2p}}{(1/2)^{2p}}\\
	&\ll_{p} n^{-2\varepsilon p+4\varepsilon-2}.
	\end{align*}
	Selecting $p$ sufficiently large completes the proof. 
\end{proof}


\begin{lemma} \label{Lem:conjLessThanConstant}
	Let $A$ be an $mn\times mn$ Hermitian positive semidefinite matrix with rank at most $d$ for some positive constant $d$. Suppose that $\xi$ is a complex-valued random variable with mean zero, unit variance, $\E|\xi|^{4}=O(1)$, and which satisfies $|\xi| \leq n^{1/2-\varepsilon}$ almost surely for some constant $0<\varepsilon<1/2$. Let $S\subseteq [mn]$, and let $w = (w_i)_{i=1}^{mn}$ be a vector with the following properties:  
	\begin{enumerate}[label=(\roman*)]
		\item $\{w_i : i \in S \}$ is a collection of iid copies of $\xi$, 
		\item $w_{i}=0$ for $i\not\in S$.
	\end{enumerate} 
	Then for any $p\geq 2$,
	\begin{equation}
	\E\left|w^{*}Aw\right|^{p} \ll_{d,p}n^{(1-2\varepsilon)p+4\varepsilon-2}\lnorm A\rnorm ^{p}.
	\end{equation} 
\end{lemma}

\begin{proof}
	Let $w_{S}$ denote the $|S|$-vector which contains entries $w_{i}$ for $i\in S$, and let $A_{S\times S}$ denote the $|S|\times|S|$ matrix which has entries $A_{(i,j)}$ for $i,j\in S$. Then we observe
	\begin{equation*}
	w^{*}Aw = \sum_{i,j}\bar{w}_{i}A_{(i,j)}w_{j}=  w_{S}^{*}A_{S\times S}w_{S}.
	\end{equation*}
	By Lemma \ref{Lem:BilinearForms}, we get
	\begin{align*}
	\E\left|w^{*}Aw\right|^{p} &\ll_{p}\left( \tr A_{S\times S}\right)^{p}+\E|\xi|^{2p}\tr A^{p}_{S\times S}\\ 
	&= \left(\tr A_{S \times S} \right)^{p}+\E\left[|\xi|^{4}|\xi|^{2p-4}\right]\text{tr}A_{S \times S}^{p}\\
	&\leq \left(\tr A_{S \times S} \right)^{p}+n^{(1-2\varepsilon)p+4\varepsilon-2}\E|\xi|^{4}\text{tr}A_{S \times S}^{p}.
	\end{align*}
	Since the rank of $A_{S \times S}$ is at most $d$, $\tr A_{S \times S} \ll_{d} \| A \|$ and $\tr A_{S \times S}^p\ll_{d} \|A \|^p,$	where we used the fact that the operator norm of a matrix bounds the operator norm of any sub-matrix. We conclude that 
	\begin{equation*}
	\E\left|w^{*}Aw\right|^{p} \ll_{d,p} \lnorm A\rnorm^{p}+n^{(1-2\varepsilon)p+4\varepsilon-2}\E|\xi|^{4} \lnorm A\rnorm^p  \ll_{d,p} n^{(1-2\varepsilon)p+4\varepsilon-2}\E|\xi|^{4}\lnorm A\rnorm ^{p},
	\end{equation*}
	as desired.
\end{proof}


\begin{lemma} \label{lem:BilinearFormWithDifferentVectors}
	Let $A$ be a deterministic complex $mn\times mn$ matrix for some fixed $m>0$. Suppose that $\xi$ is a complex-valued random variable with mean zero, unit variance, finite moments of all orders. Let $S,R\subseteq [mn]$, and let $w = (w_i)_{i=1}^{mn}$ and $t = (t_i)_{i=1}^{mn}$ be independent vectors with the following properties:  
	\begin{enumerate}[label=(\roman*)]
		\item $\{w_i : i \in S \}$ and $\{t_j : j \in R \}$ are independent collections of iid copies of $\xi$, 
		\item $w_{i}=0$ for $i\not\in S$, and  $t_{j}=0$ for $j\not\in R$.
	\end{enumerate} 
	Then for any $p\geq 1$,
	\begin{equation}
	\E\left|w^{*}At\right|^{2p}\ll_{p}\E|\xi|^{4p}(\tr(A^{*}A))^{p}.
	\end{equation}
\end{lemma}
\begin{proof}
	Let $w_{S}$ denote the $|S|$-vector which contains entries $w_{i}$ for $i\in S$, and let $t_{R}$ denote the $|R|$-vector which contains entries $t_{j}$ for $j\in R$.  For an $N \times N$ matrix $B$, we let $B_{S \times S}$ denote the $|S| \times |S|$ matrix with entries $B_{(i,j)}$ for $i,j \in S$.  Similarly, we let $B_{R \times R}$ denote the $|R| \times |R|$ matrix with entries $B_{(i,j)}$ for $i,j \in R$.  
	
	Since $w$ is independent of $t$, Lemma \ref{Lem:BilinearForms} implies that
	\begin{align*}
	\E|w^{*}At|^{2p} &= \E|w^{*}Att^{*}A^{*}w|^{p}\\
	&= \E\left|w^{*}_{S}(Att^{*}A^{*})_{S\times S}w_{S}\right|^{p}\\
	&\ll_{p}\E\left[\left(\tr(Att^{*}A^{*})_{S\times S}\right)^{p}+\E|\xi|^{2p}\tr(Att^{*}A^{*})_{S\times S}^{p}\right].
	\end{align*}
	Recall that for any matrix $B$, $\tr(B^{*}B)^{p}\leq (\tr(B^{*}B))^{p}$. By this and the fact that for a Hermitian positive semidefinite matrix, the partial trace is less than or equal to the full trace, we observe that 
	\begin{equation*} 
	\E\left[\left(\tr(Att^{*}A^{*})_{S\times S}\right)^{p}+\E|\xi|^{2p}\tr(Att^{*}A^{*})_{S\times S}^{p}\right]\ll_{p}\E|\xi|^{2p}\E\left[(\tr(Att^{*}A^{*}))^{p}\right].
	\end{equation*}
	By a cyclic permutation of the trace, we have 
	\begin{equation*}
	\E\left[(\tr(Att^{*}A^{*}))^{p}\right] = \E\left[(t^{*}A^{*}At)^{p}\right]\leq\E\left|t^{*}A^{*}At\right|^{p}.
	\end{equation*}
	By Lemma \ref{Lem:BilinearForms}, and a similar argument as above, we have
	\begin{align*}
	\E\left|t^{*}A^{*}At\right|^{p}&=\E\left|t_{R}^{*}(A^{*}A)_{R\times R}t_{R}\right|^{p}\\
	&\ll_{p}(\tr(A^{*}A)_{R\times R})^{p}+\E|\xi|^{2p}\tr(A^{*}A)_{R\times R}^{p}\\
	&\ll_{p}\E|\xi|^{2p}(\tr(A^{*}A))^{p}, 
	\end{align*}
	and thus by Jensen's inequality, we have 
	\[\E|w^{*}At|^{2p}\ll_{p}\E|\xi|^{2p}\E\left[(\tr(Att^{*}A^{*}))^{p}\right]\ll_{p}\E|\xi|^{4p}(\tr(A^{*}A))^{p} \]
	completing the proof.
\end{proof}
\begin{remark}
	\label{Rem:lem:BilinearFormWithDifferentVectors}
	Note that if $p\geq 1$ and we also assume that $\E|\xi|^{4}=O(1)$ and $|\xi|<n^{1/2-\varepsilon}$ surely for some $\varepsilon>0$, then we may write 
	\begin{align*}
	\E|w^{*}At|^{2p}&\ll_{p}\E|\xi|^{4p}(\tr(A^{*}A))^{p}\\
	&=\E\left[|\xi|^{4}|\xi|^{4p-4}\right](\tr(A^{*}A))^{p}\\
	&\ll n^{(2-4\varepsilon)p+4\varepsilon-2}\E|\xi|^{4}(\tr(A^{*}A))^{p}.
	\end{align*}
\end{remark}

\begin{lemma}
	Let $U_{k}$ be the $mn\times m$ matrix which contains as its columns the columns  $c_{k},c_{n+k},\dots,c_{(m-1)n+k}$ of $\blmat{Y}{n}$ and define $V_{k}$ to be the $mn\times m$ matrix which contains as its columns $e_{k},e_{n+k},\dots, e_{(m-1)n+k}$ where $e_{1},\dots ,e_{mn}$ denote the standard basis elements of $\C^{mn}$. Let $\mathcal{G}_{n}^{(k)}(z)$ be defined as in \eqref{Def:G^{(k)}_{n}}. Then 
	\begin{equation*}
	\E\lnorm V_{k}^{T}\mathcal{G}_{n}^{(k)}(z)U_{k}\oindicator{\Omega_{n,k}}\rnorm^{2} \ll n^{-1}
	\end{equation*}
	and for any $p\geq2$, 
	\begin{equation*}
	\E\lnorm V_{k}^{T}\mathcal{G}_{n}^{(k)}(z)U_{k}\oindicator{\Omega_{n,k}}\rnorm^{2p} \ll_{p} n^{-2\varepsilon p+4\varepsilon-2}.
	\end{equation*}
	\label{Lem:BoundingVGU}
\end{lemma}
\begin{proof}
	Begin by observing that 
	\begin{align*}
	&\E\lnorm V_{k}^{T}\mathcal{G}_{n}(z)U_{k}\oindicator{\Omega_{n,k}}\rnorm^{2p}\\
	&\quad\quad \ll \max_{1\leq i,j\leq m} \E\left| (V_{k}^{T}\mathcal{G}_{n}^{(k)}(z)U_{k})_{(i,j)}\oindicator{\Omega_{n,k}}\right|^{2p}\\
	&\quad\quad=\max_{1\leq i,j\leq m}
	\E\left|e_{(i-1)n+k}\mathcal{G}_{n}^{(k)}(z)c_{(j-1)n+k}\oindicator{\Omega_{n,k}}\right|^{2p}\\
	\end{align*}
	In the case when $p=1$, since the rank of $(\mathcal{G}_{n}^{(k)}(z))^{*}e_{(i-1)n+k}e_{(i-1)n+k}^{T}\mathcal{G}_{n}^{(k)}(z)$ is at most 1, for any $1\leq j\leq m$ we have 
	\begin{align*}
	&\max_{1\leq i,j\leq m}
	\E\left|e_{(i-1)n+k}\mathcal{G}_{n}^{(k)}(z)c_{(j-1)n+k}\oindicator{\Omega_{n,k}}\right|^{2p}\\
	&\quad\quad\ll\E\left[c_{(j-1)n+k}^{*}(\mathcal{G}_{n}^{(k)}(z))^{*}e_{(i-1)n+k}e_{(i-1)n+k}^{T}\mathcal{G}_{n}(z)c_{(j-1)n+k}\oindicator{\Omega_{n,k}}\right]\\
	&\quad\quad\ll n^{-1}\lnorm(\mathcal{G}_{n}^{(k)}(z))^{*}e_{(i-1)n+k}e_{(i-1)n+k}^{T}\mathcal{G}_{n}(z)\oindicator{\Omega_{n,k}}\rnorm\\
	&\quad\quad \ll n^{-1}
	\end{align*}
	by Lemma \ref{Lem:G_nBounded}.
	In the case where $p\geq 2$, we have 
	\begin{align*}
	&\max_{1\leq i,j\leq m}
	\E\left|e_{(i-1)n+k}\mathcal{G}_{n}^{(k)}(z)c_{(j-1)n+k}\oindicator{\Omega_{n,k}}\right|^{2p}\\
	&\quad\quad = \max_{1\leq i,j\leq m} \E\left|c_{(j-1)n+k}^{*}(\mathcal{G}_{n}^{(k)}(z))^{*}e_{(i-1)n+k}e_{(i-1)n+k}^{T}\mathcal{G}^{(k)}_{n}(z)c_{(j-1)n+k}\oindicator{\Omega_{n,k}}\right|^{p}.
	\end{align*}
	Note that by definition of $\blmat{Y}{n}$ in \eqref{Def:Y_n}, each entry in $c_{(j-1)n+k}$ has been scaled by $n^{-1/2}$. By this observation, Lemma \ref{Lem:G_nBounded}, and Lemma \ref{Lem:conjLessThanConstant},
	\begin{align*}
	&\E\left|c_{(j-1)n+k}^{*}(\mathcal{G}_{n}^{(k)}(z))^{*}e_{(i-1)n+k}e_{(i-1)n+k}^{T}\mathcal{G}^{(k)}_{n}(z)c_{(j-1)n+k}\oindicator{\Omega_{n,k}}\right|^{p}\\
	&\quad\quad\ll_{p} n^{-p} n^{(1-2\varepsilon)p+4\varepsilon-2}\lnorm(\mathcal{G}_{n}^{(k)}(z))^{*}e_{(i-1)n+k}e_{(i-1)n+k}^{T}\mathcal{G}^{(k)}_{n}(z)\oindicator{\Omega_{n,k}}\rnorm^{p}\\
	&\quad\quad\ll n^{-2\varepsilon p+4\varepsilon-2}
	\end{align*}
	for any $1\leq j\leq m$ since the rank of $(\mathcal{G}_{n}^{(k)}(z))^{*}e_{(i-1)n+k}e_{(i-1)n+k}^{T}\mathcal{G}_{n}^{(k)}(z)$ is at most 1.
\end{proof}
\begin{remark}
	The same argument as in Lemma \ref{Lem:BoundingVGU} also shows that\\ $\E\lnorm V_{k}^{T}(\mathcal{G}_{n}^{(k)}(z))^{2}U_{k}\oindicator{\Omega_{n,k}}\rnorm ^{2}\ll n^{-1}$ and $\E\lnorm V_{k}^{T}(\mathcal{G}_{n}^{(k)}(z))^{2}U_{k}\oindicator{\Omega_{n,k}}\rnorm ^{2p}\ll_{p} n^{-2\varepsilon p+4\varepsilon-2}$. 
	\label{Remark:ImprovementOnVGU}
\end{remark}

We now proceed with the proof of Lemma \ref{Lem:ReductionZ}.
\begin{proof}[Proof of Lemma \ref{Lem:ReductionZ}]
	To begin, note that the result will follow if we prove that $\E| M_{n}-\breve{M}_{n}|^{2} =o(1).$ Since the only difference between these two expressions is the difference between $Z_{n,k}(z)$ and $\breve{Z}_{n,k}(z)$, it will be sufficient to prove that for any $z$ on the contour, $\E\left|\sum_{k=1}^{n}\left(Z_{n,k}(z)-\breve{Z}_{n,k}(z)\right)\right|^{2}=o(1).$ Since $Z_{n,k}$ and $\breve{Z}_{n,k}$ are martingale difference sequences, we will prove 
	\[\E\left|Z_{n,k}(z)-\breve{Z}_{n,k}(z)\right|^{2}=o\left(n^{-1}\right)\] 
	uniformly for any $0<k\leq n$ and any $z$ on the contour $\mathcal{C}$. To do so, we will make a sequence of comparisons, each of which differs from the previous expression by error terms which is $o(n^{-1})$. To begin, observe that since $\blmat{Y}{n}^{(k)}$ has columns $k,n+k,\dots, (m-1)n+k$ replaced with zeros, we have $(\E_{k}-\E_{k-1})[\tr\mathcal{G}_{n}^{(k)}(z)\oindicator{\Omega_{n,k}}]=0.$ Thus, we can rewrite
	\begin{align*}
	Z_{n,k}(z) &= (\E_{k}-\E_{k-1})[\tr\mathcal{G}_{n}(z)\oindicator{\Omega_{n}}]\\
	&= (\E_{k}-\E_{k-1})[\tr\mathcal{G}_{n}(z)\oindicator{\Omega_{n}}-\tr\mathcal{G}_{n}^{(k)}(z)\oindicator{\Omega_{n,k}}]\\
	&= (\E_{k}-\E_{k-1})[(\tr\mathcal{G}_{n}(z)-\tr\mathcal{G}_{n}^{(k)}(z))\oindicator{\Omega_{n}\cap\Omega_{n,k}}]\\
	&\quad\quad+(\E_{k}-\E_{k-1})[\tr\mathcal{G}_{n}(z)\oindicator{\Omega_{n}\cap\Omega_{n,k}^{c}}]\\
	&\quad\quad-(\E_{k}-\E_{k-1})[\tr\mathcal{G}_{n}^{(k)}(z)\oindicator{\Omega_{n,k}\cap\Omega_{n}^{c}}].
	\end{align*}
	Note that, uniformly for $z$ with $|z|=1+\delta$, by Lemma \ref{Lem:G_nBounded} and since $\Omega_{n,k}$ holds with overwhelming probability,
	\begin{align*}
	\E\left|(\E_{k}-\E_{k-1})[\tr\mathcal{G}_{n}(z)\oindicator{\Omega_{n}\cap\Omega_{n,k}^{c}}]\right|^{2}&\ll \E\left|\tr\mathcal{G}_{n}(z)\oindicator{\Omega_{n}\cap\Omega_{n,k}^{c}}\right|^{2}\\
	&\ll n^{2}\E\left[\lnorm\mathcal{G}_{n}(z)\rnorm^{2}\oindicator{\Omega_{n}\cap\Omega_{n,k}^{c}}\right]\\
	&\ll_{\alpha} n^{2-\alpha}
	\end{align*}
	for any $\alpha>0$. Since $\Omega_{n,k}$ holds with overwhelming probability, the same argument shows that $\E\left|(\E_{k}-\E_{k-1})[\tr\mathcal{G}_{n}^{(k)}(z)\oindicator{\Omega_{n,k}\cap\Omega_{n}^{c}}]\right|^{2}\ll_{\alpha} n^{2-\alpha}$ for any $\alpha>0$. 
	Ergo, we have reduced from working with $Z_{n,k}(z)$ to working with
	$(\E_{k}-\E_{k-1})[(\tr\mathcal{G}_{n}(z)-\tr\mathcal{G}_{n}^{(k)}(z))\oindicator{\Omega_{n}\cap\Omega_{n,k}}]$. Next, observe that by linearity and cyclic permutation of the trace, and by the resolvent identity \eqref{Equ:ResolventIndentity},
	\begin{align}
	\tr\mathcal{G}_{n}(z)-\tr\mathcal{G}_{n}^{(k)}(z)&=\tr\left(\mathcal{G}_{n}(z)\left(\blmat{Y}{n}^{(k)}-\blmat{Y}{n}\right)\mathcal{G}_{n}^{(k)}(z)\right)\notag\\ 
	&=-\tr\left(\mathcal{G}_{n}(z)U_{k}V_{k}^{T}\mathcal{G}_{n}^{(k)}(z)\right)\notag\\
	&=-\tr\left(V_{k}^{T}\mathcal{G}_{n}^{(k)}(z)\mathcal{G}_{n}(z)U_{k}\right).\label{Equ:ReductionZStep}
	\end{align}
	 To guarantee that $I_{m}+V_{k}^{T}\mathcal{G}_{n}^{(k)}U_{k}$ is invertible, we wish to work on the event $Q_{n,k}$ defined in \eqref{Def:EventQ}. By Lemma \ref{Lem:QOverwhelming}, $Q_{n,k}$ hold with overwhelming probability so that by Lemma \ref{Lem:G_nBounded}, the Cauchy--Schwarz inequality, and bounding the spectral norm by the Frobenius norm, we have 
	\begin{align*}
	&\E\left|(\E_{k}-\E_{k-1})[\tr(V_{k}^{T}\mathcal{G}_{n}^{(k)}(z)\mathcal{G}_{n}(z)U_{k})\oindicator{\Omega_{n}\cap\Omega_{n,k}}]\right.\\
	&\quad\quad\quad\quad\quad\quad\quad\left.-(\E_{k}-\E_{k-1})[\tr(V_{k}^{T}\mathcal{G}_{n}^{(k)}(z)\mathcal{G}_{n}(z)U_{k})\oindicator{\Omega_{n}\cap\Omega_{n,k}\cap Q_{n,k}}]\right|^2\\
	&\quad\quad\ll \E\left[\lnorm V_{k}^{T}\mathcal{G}_{n}^{(k)}(z)\mathcal{G}_{n}(z)U_{k}\oindicator{\Omega_{n}\cap\Omega_{n,k}}\rnorm^{2}\oindicator{Q_{n,k}^{c}}\right]\\
	&\quad\quad\ll n^4\E\left[\lnorm U_{k}\rnorm^{2}\oindicator{Q_{n,k}^{c}}\right]\\
	&\quad\quad\ll_{\alpha} n^{6-\alpha}
	\end{align*}
	for any $\alpha>0$. By selecting $\alpha$ sufficiently large, we can justify working with
	\[-(\E_{k}-\E_{k-1})[\tr(V_{k}^{T}\mathcal{G}_{n}^{(k)}(z)\mathcal{G}_{n}(z)U_{k})\oindicator{\Omega_{n}\cap\Omega_{n,k}\cap Q_{n,k}}]\] 
	instead of $Z_{n,k}(z)$. By the Sherman--Morrison--Woodbury formula \eqref{Equ:ShermanMorrisonWoodbury}, we have
	\begin{align*}
	&-(\E_{k}-\E_{k-1})[\tr(V_{k}^{T}\mathcal{G}_{n}^{(k)}(z)\mathcal{G}_{n}(z)U_{k})\oindicator{\Omega_{n}\cap\Omega_{n,k}\cap Q_{n,k}}]\notag \\ 
	&\quad\quad=-(\E_{k}-\E_{k-1})[\tr(V_{k}^{T}(\mathcal{G}_{n}^{(k)}(z))^{2}U_{k}(I_{m}+V_{k}^{T}\mathcal{G}_{n}^{(k)}U_{k})^{-1})\oindicator{\Omega_{n}\cap\Omega_{n,k}\cap Q_{n,k}}].
	\end{align*}
	Since $\mathcal{G}_{n}(z)$ is no longer present, we may drop the event $\Omega_{n}$ gaining a sufficiently small error, and the same argument justifies working with \[-(\E_{k}-\E_{k-1})[\tr(V_{k}^{T}(\mathcal{G}_{n}^{(k)}(z))^{2}U_{k}(I_{m}+V_{k}^{T}\mathcal{G}_{n}^{(k)}U_{k})^{-1})\oindicator{\Omega_{n,k}\cap Q_{n,k}}]\] 
	instead of $Z_{n,k}(z)$. At this point, we wish to replace $(I_{m}+V_{k}^{T}\mathcal{G}_{n}^{(k)}U_{k})^{-1}$ with $I_{m}$. To justify this, observe that
	\begin{align}
	&\E\left|(\E_{k}-\E_{k-1})[\tr(V_{k}^{T}(\mathcal{G}_{n}^{(k)}(z))^{2}U_{k}(I_{m}+V_{k}^{T}\mathcal{G}_{n}^{(k)}U_{k})^{-1})\oindicator{\Omega_{n,k}\cap Q_{n,k}}]\right.\notag \\
	&\quad\quad\quad\quad\quad\quad\quad\quad\quad\left.-(\E_{k}-\E_{k-1})[\tr(V_{k}^{T}(\mathcal{G}_{n}^{(k)}(z))^{2}U_{k}I_{m})\oindicator{\Omega_{n,k}\cap Q_{n,k}}]\right|^{2}\notag\\
	&\quad\quad \ll \E\lnorm  V_{k}^{T}(\mathcal{G}_{n}^{(k)}(z))^{2}U_{k}((I_{m}+V_{k}^{T}\mathcal{G}_{n}^{(k)}U_{k})^{-1}-I_{m})\oindicator{\Omega_{n,k}\cap Q_{n,k}}\rnorm^{2}.\label{Equ:ReductionZStep3}
	\end{align}
	Note that by the resolvent identity \eqref{Equ:ResolventIndentity}, 
	\begin{Details} 
		\[(I_{m}+V_{k}^{T}\mathcal{G}_{n}^{(k)}U_{k})^{-1}-I_{m}=-(I_{m}+V_{k}^{T}\mathcal{G}_{n}^{(k)}U_{k})^{-1}V_{k}^{T}\mathcal{G}_{n}^{(k)}(z)U_{k}\]
	which can be rearranged to see
	\end{Details}
	\begin{equation}(I_{m}+V_{k}^{T}\mathcal{G}_{n}^{(k)}U_{k})^{-1}=I_{m}-(I_{m}+V_{k}^{T}\mathcal{G}_{n}^{(k)}U_{k})^{-1}V_{k}^{T}\mathcal{G}_{n}^{(k)}(z)U_{k}.\label{Equ:ExpansionOfMatrixErrorTerm}
	\end{equation}
	By iterating this twice, we have
	\[(I_{m}+V_{k}^{T}\mathcal{G}_{n}^{(k)}U_{k})^{-1}-I_{m}=-V_{k}^{T}\mathcal{G}_{n}^{(k)}(z)U_{k}+(I_{m}+V_{k}^{T}\mathcal{G}_{n}^{(k)}U_{k})^{-1}(V_{k}^{T}\mathcal{G}_{n}^{(k)}(z)U_{k})^{2}.\]
	Inserting this into the last line of \eqref{Equ:ReductionZStep3}, we get
	\begin{align}
	& \E\lnorm V_{k}^{T}(\mathcal{G}_{n}^{(k)}(z))^{2}U_{k}((I_{m}+V_{k}^{T}\mathcal{G}_{n}^{(k)}U_{k})^{-1}-I_{m})\oindicator{\Omega_{n,k}\cap Q_{n,k}}\rnorm^{2}\notag\\
	&\quad\quad \ll \E\lnorm V_{k}^{T}(\mathcal{G}_{n}^{(k)}(z))^{2}U_{k}(V_{k}^{T}\mathcal{G}_{n}^{(k)}(z)U_{k})\oindicator{\Omega_{n,k}}\rnorm^{2}\label{Equ:ReductionZStep4}\\
	&\quad\quad\quad\quad+\E\lnorm V_{k}^{T}(\mathcal{G}_{n}^{(k)}(z))^{2}U_{k}(I_{m}+V_{k}^{T}\mathcal{G}_{n}^{(k)}U_{k})^{-1}(V_{k}^{T}\mathcal{G}_{n}^{(k)}(z)U_{k})\oindicator{\Omega_{n,k}\cap Q_{n,k}}\rnorm^{2}.\label{Equ:ReductionZStep5}
	\end{align}
	We will bound each of the above terms separately. First, we begin with term \eqref{Equ:ReductionZStep4}. Note that by Cauchy--Schwarz inequality, Lemma \ref{Lem:BoundingVGU}, and Remark \ref{Remark:ImprovementOnVGU}, we have 
	\begin{align*}
	&\E\lnorm V_{k}^{T}(\mathcal{G}_{n}^{(k)}(z))^{2}U_{k}(V_{k}^{T}\mathcal{G}_{n}^{(k)}(z)U_{k})\oindicator{\Omega_{n,k}}\rnorm^{2}\\
	&\quad\quad \ll\left(\E\lnorm V_{k}^{T}(\mathcal{G}_{n}^{(k)}(z))^{2}U_{k}\oindicator{\Omega_{n,k}}\rnorm^{4}\E\lnorm V_{k}^{T}\mathcal{G}_{n}^{(k)}(z)U_{k}\oindicator{\Omega_{n,k}}\rnorm^{4}\right)^{1/2}\\
	&\quad\quad =o(n^{-1}).
	\end{align*}
	It remains to show that term \eqref{Equ:ReductionZStep5} is also $o(n^{-1})$. To this end, observe that by the Cauchy--Schwarz inequality, Lemma \ref{Lem:BoundingVGU}, and Remark \ref{Remark:ImprovementOnVGU},
	\begin{align*}
	&\E\lnorm V_{k}^{T}(\mathcal{G}_{n}^{(k)}(z))^{2}U_{k}(I_{m}+V_{k}^{T}\mathcal{G}_{n}^{(k)}U_{k})^{-1}(V_{k}^{T}\mathcal{G}_{n}^{(k)}(z)U_{k})^{2}\oindicator{\Omega_{n,k}\cap Q_{n,k}}\rnorm^{2}\\
	&\quad\quad \leq \left(\E\lnorm V_{k}^{T}(\mathcal{G}_{n}^{(k)}(z))^{2}U_{k}\oindicator{\Omega_{n,k}\cap Q_{n,k}}\rnorm^{8}\E\lnorm(I_{m}+V_{k}^{T}\mathcal{G}_{n}^{(k)}U_{k})^{-1}\oindicator{\Omega_{n,k}\cap Q_{n,k}}\rnorm^{8}\right.\\
	&\quad\quad\quad\quad\quad\quad\quad\quad\quad\quad\quad\quad\quad\quad\quad\quad\quad\quad \times \left.\E\lnorm(V_{k}^{T}\mathcal{G}_{n}^{(k)}(z)U_{k})^{2}\oindicator{\Omega_{n,k}\cap Q_{n,k}}\rnorm^{8}\right)^{1/4}\\
	&\quad\quad\ll \left(\E\lnorm V_{k}^{T}(\mathcal{G}_{n}^{(k)}(z))^{2}U_{k}\oindicator{\Omega_{n,k}}\rnorm^{8}
	\E \lnorm (V_{k}^{T}\mathcal{G}_{n}^{(k)}(z)U_{k})^{2}\oindicator{\Omega_{n,k}} \rnorm^{8}\right)^{1/4}\\
	&\quad\quad\ll \left(\E\lnorm V_{k}^{T}(\mathcal{G}_{n}^{(k)}(z))^{2}U_{k}\oindicator{\Omega_{n,k}}\rnorm^{8}
	\E \lnorm V_{k}^{T}\mathcal{G}_{n}^{(k)}(z)U_{k}\oindicator{\Omega_{n,k}} \rnorm^{16}\right)^{1/4}\\
	&\quad\quad \ll \left(n^{-2\varepsilon\cdot 4+4\varepsilon-2}\cdot n^{-2\varepsilon\cdot 8+4\varepsilon-2}\right)^{1/4}\\
	&\quad\quad= \left(n^{-16\varepsilon-4}\right)^{1/4}\\
	&\quad\quad \ll n^{-4\varepsilon-1}.
	\end{align*}
	\begin{Details} 
		Therefore
		\begin{align*}
		&\sum_{k=1}^{n}\E\left|(\E_{k}-\E_{k-1})[\tr(V_{k}^{T}(\mathcal{G}_{n}^{(k)}(z))^{2}U_{k}(I_{m}+V_{k}^{T}\mathcal{G}_{n}^{(k)}U_{k})^{-1})\oindicator{\Omega_{n,k}\cap Q_{n,k}}]\right.\\
		&\quad\quad\quad\quad\quad\quad\quad\quad\quad\quad\quad\quad\quad\left.-(\E_{k}-\E_{k-1})[\tr(V_{k}^{T}(\mathcal{G}_{n}^{(k)}(z))^{2}U_{k}I_{m})\oindicator{\Omega_{n,k}\cap Q_{n,k}}]\right|\\
		&\quad\quad\leq \sum_{k=1}^{2} Cn^{-2}\\
		&\quad\quad =o(1).
		\end{align*}
	\end{Details}
	Since the above term is also $o(n^{-1})$, we may proceed working with the term
	\[-(\E_{k}-\E_{k-1})[\tr(V_{k}^{T}(\mathcal{G}_{n}^{(k)}(z))^{2}U_{k})\oindicator{\Omega_{n,k}\cap Q_{n,k}}].\]
	Next, we will justify removing the event $Q_{n,k}$. Observe that by Remark \ref{Remark:ImprovementOnVGU} and repeating the same argument as above, 
	\begin{Details}
		\begin{align*}
		&\E\left|(\E_{k}-\E_{k-1})[\tr(V_{k}^{T}(\mathcal{G}_{n}^{(k)}(z))^{2}U_{k})\oindicator{\Omega_{n,k}}]\right.\\
		&\quad\quad\quad\quad\left.-(\E_{k}-\E_{k-1})[\tr(V_{k}^{T}(\mathcal{G}_{n}^{(k)}(z))^{2}U_{k})\oindicator{\Omega_{n,k}\cap Q_{n,k}}]\right|^{2}\\
		&\quad\quad \ll \E\left|\tr(V_{k}^{T}(\mathcal{G}_{n}^{(k)}(z))^{2}U_{k})\oindicator{\Omega_{n,k}}(\oindicator{Q_{n,k}^{c}})\right|^{2}\\
		&\quad\quad \ll n^{(-\alpha-1)/2}.
		\end{align*}
	\end{Details}
	\begin{align*}
	&\E\left|(\E_{k}-\E_{k-1})[\tr(V_{k}^{T}(\mathcal{G}_{n}^{(k)}(z))^{2}U_{k})\oindicator{\Omega_{n,k}}]\right.\\
	&\quad\quad\quad\quad\left.-(\E_{k}-\E_{k-1})[\tr(V_{k}^{T}(\mathcal{G}_{n}^{(k)}(z))^{2}U_{k})\oindicator{\Omega_{n,k}\cap Q_{n,k}}]\right|^{2} \ll_{\alpha} n^{-1-\alpha/2}.
	\end{align*}
	By selecting $\alpha$ sufficiently large in the above expression, we can proceed with 
	\[-(\E_{k}-\E_{k-1})[\tr(V_{k}^{T}(\mathcal{G}_{n}^{(k)}(z))^{2}U_{k})\oindicator{\Omega_{n,k}}].\]
	Finally, note that $U_{k}$ is independent of $\mathcal{G}_{n}^{(k)}(z)$ and $\Omega_{n,k}$, so that
	\begin{align*}
	&\E_{k-1}[\tr(V_{k}^{T}(\mathcal{G}_{n}^{(k)}(z))^{2}U_{k})\oindicator{\Omega_{n,k}}]\\
	&\quad\quad = \sum_{i=1}^{m}\sum_{a,b=1}^{mn}(V_{k}^{T})_{i,a}\E_{k-1}\left[(\mathcal{G}_{n}^{(k)}(z))_{a,b}^{2}\oindicator{\Omega_{n,k}}\right]\E_{k-1}[(U_{k})_{b,i}]\\
	&\quad\quad = 0.
	\end{align*}
	This completes the proof.
	\begin{Details}
		 Therefore, we can work with 
		 \[-\E_{k}[\tr(V_{k}^{T}(\mathcal{G}_{n}^{(k)}(z))^{2}U_{k})\oindicator{\Omega_{n,k}}]\]
		 instead of $Z_{n,k}(z)$, completing the proof.
	\end{Details}
\end{proof}

To prove that $\breve{M}_{n}$ converges to a mean-zero Gaussian, we will use the following martingale difference sequence central limit theorem.

\begin{theorem}[Theorem 35.12 of \cite{Bill}]
	For each $N$, suppose $Z_{N_{1}},Z_{N_{2}},\dots,Z_{N_{r_{N}}}$ is a real martingale difference sequence with respect to the increasing $\sigma$-field $\{\mathcal{F}_{N_{j}}\}$ having second moments. Suppose, for any $\eta>0$ and a positive constant $\nu^{2}$, 
	\begin{equation}
	\lim_{N\rightarrow\infty}\P\left(\left|\sum_{j=1}^{r_{N}}\E\left(Z_{N_{j}}^{2}|\mathcal{F}_{N_{j-1}}\right)-\nu^{2}\right|>\eta\right)=0
	\label{Equ:MDSVarianceCondition}
	\end{equation}
	and 
	\begin{equation}
	\lim_{N\rightarrow\infty}\sum_{j=1}^{r_{N}}\E\left(Z_{N_{j}}^{2}\indicator{|Z_{N_{j}}|\geq\eta}\right)=0.
	\label{Equ:MDSLindinbergCondition}
	\end{equation}
	Then as $N\rightarrow\infty$, the distribution of $\sum_{j=1}^{r_{N}}Z_{N_{j}}$ converges weakly to a Gaussian distribution with mean zero and variance $\nu^{2}$.
	\label{Thm:MDS_CLT}
\end{theorem}

We will apply this result to $\{\breve{M}_{n,k}\}_{k=1}^{n}$ and the corresponding $\sigma$-algebras are $\{\mathcal{F}_{k}\}$. 
\begin{Details} 
To utilize this theorem, first note that $\{\breve{M}_{n,k}\}$ is a real martingale difference sequence by choice of $\alpha_{l}$ and $\beta_{l}$, and it has finite second moments. 
The calculation for the mean follows easily. Indeed,
\begin{align*}
\E\left[\breve{M}_{n,k}\right] &= \E\left[\sum_{l=1}^{L}\alpha_{l}\breve{Z}_{n,k}(z_{l})+\beta_{l}\overline{\breve{Z}_{n,k}(z_{l})}\right]\\
&=-\sum_{l=1}^{L}\alpha_{l}\E\left[\tr\left(V_{k}^{T}(\mathcal{G}_{n}^{(k)}(z_{l}))^{2}U_{k}\right)\oindicator{\Omega_{n,k}}\right]\\
&\quad\quad\quad+\beta_{l}\E\left[\overline{\tr\left(V_{k}^{T}(\mathcal{G}_{n}^{(k)}(z_{l}))^{2}U_{k}\right)\oindicator{\Omega_{n,k}}}\right]\\
&=-\sum_{l=1}^{L}\alpha_{l}\sum_{i=1}^{m}\sum_{a,b=1}^{mn}(V_{k}^{T})_{(i,a)}\E\left[(\mathcal{G}_{n}^{(k)}(z_{l}))_{(a,b)}^{2}\oindicator{\Omega_{n,k}}\right]\E_{k}[(U_{k})_{(b,i)}]\\
&\quad\quad\quad+\beta_{l}\sum_{i=1}^{m}\sum_{a,b=1}^{mn}\E[(U_{k}^{*})_{(i,b)}]\E\left[(\mathcal{G}_{n}^{(k_{l})}(z))_{(b,a)}^{2*}\oindicator{\Omega_{n,k}}\right](V_{k})_{(a,i)}\\
&=-\sum_{l=1}^{L}\alpha_{l}\sum_{i=1}^{m}\sum_{a,b=1}^{mn}(V_{k}^{T})_{(i,a)}\E\left[(\mathcal{G}_{n}^{(k)}(z_{l}))_{(a,b)}^{2}\oindicator{\Omega_{n,k}}\right]\cdot 0\\
&\quad\quad\quad+\beta_{l}\sum_{i=1}^{m}\sum_{a,b=1}^{mn}0\cdot\E\left[(\mathcal{G}_{n}^{(k)}(z_{l}))_{(b,a)}^{2*}\oindicator{\Omega_{n,k}}\right](V_{k})_{(a,i)}\\
&=0.
\end{align*}
For finite variance, let $\kappa:=\max_{1 \leq l\leq L} \{|\alpha_{l}|,|\beta_{l}|\}$ and observe that
\begin{align*}
\E\left|\breve{M}_{n,k}\right|^{2}&=\E\left|\sum_{l=1}^{L}\alpha_{l}\breve{Z}_{n,k}(z_{l})+\beta_{l}\overline{\breve{Z}_{n,k}(z_{l})}\right|^{2}\\
&\ll_{L}\sum_{l=1}^{L}\E\left|\alpha_{l}\breve{Z}_{n,k}(z_{l})+\beta_{l}\overline{\breve{Z}_{n,k}(z_{l})}\right|^{2}\\
&\ll_{\kappa,L}\sum_{l=1}^{L}\E\left|\breve{Z}_{n,k}(z_{l})\right|^{2}\\
&\ll_{\kappa,L}\sum_{l=1}^{L}\E\lnorm V_{k}^{T}(\mathcal{G}_{n}^{(k)}(z_{l}))^{2}U_{k}\oindicator{\Omega_{n,k}}\rnorm^{2}\\&\ll_{\kappa,L}\sum_{l=1}^{L}n^{-1}
\end{align*}
by independence of $U_{k}$ from $\mathcal{G}_{n}^{(k)}(z)$ and by Remark \ref{Remark:ImprovementOnVGU}.
To apply Theorem \ref{Thm:MDS_CLT}, we must verify that the hypotheses \eqref{Equ:MDSVarianceCondition} and \eqref{Equ:MDSLindinbergCondition} hold. 
\end{Details} 
Verifying \eqref{Equ:MDSVarianceCondition} for $\{\breve{M}_{n,k}\}_{k=1}^{n}$ is lengthy and will require new notation, so we begin with verifying \eqref{Equ:MDSLindinbergCondition} for $\{\breve{M}_{n,k}\}_{k=1}^{n}$. Let $\eta>0$ and observe that by Remark \ref{Remark:ImprovementOnVGU}, we have
\begin{align*}
\sum_{k=1}^{n}\E\left[\breve{M}_{n,k}^{2}\indicator{|\breve{M}_{n,k}|>\eta}\right]& \ll \sum_{k=1}^{n}\E\left[\frac{\breve{M}_{n,k}^{4}}{\eta^{2}}\indicator{|\breve{M}_{n,k}|>\eta}\right]\\
&\ll_{\eta,L}\sum_{k=1}^{n}\sum_{l=1}^{L}\E\lnorm V_{k}^{T}(\mathcal{G}_{n}^{(k)}(z_{l}))^{2}U_{k}\oindicator{\Omega_{n,k}}\rnorm^{4}\\
&\ll_{\eta,L}n^{-1}.
\end{align*} 

Condition \eqref{Equ:MDSVarianceCondition} will follow from the following lemma. 
\begin{lemma}
	The martingale difference sequence 
	\begin{equation*}
	\{\breve{M}_{n,k}\}=\left\{\sum_{l=1}^{L}\alpha_{l}\breve{Z}_{n,k}(z_{l})+\beta_{l}\overline{\breve{Z}_{n,k}(z_{l})}\right\}
	\end{equation*}
	has finite second moments and satisfies
	\begin{align}
	\sum_{k=1}^{n}\E_{k-1}[\breve{M}_{n,k}^{2}]\rightarrow&\sum_{1\leq i,j\leq L}\alpha_{i}\alpha_{j}\frac{m^{2}(z_{i}{z_{j}})^{m-1}}{((z_{i}{z_{j}})^{m}-1)^{2}}+\alpha_{i}\beta_{j}\frac{m^{2}(z_{i}\bar{z_{j}})^{m-1}}{((z_{i}\bar{z_{j}})^{m}-1)^{2}}\notag\\
	&\quad\quad\quad+\beta_{i}\alpha_{j}\frac{m^{2}(\bar{z_{i}}z_{j})^{m-1}}{((\bar{z_{i}}{z_{j}})^{m}-1)^{2}}+\beta_{i}\beta_{j}\frac{m^{2}(\bar{z_{i}}\bar{z_{j}})^{m-1}}{((\bar{z_{i}}\bar{z_{j}})^{m}-1)^{2}}\label{Equ:Lem:MDSReductionVariance}
	\end{align}
	in probability as $n\rightarrow\infty$.
	\label{Lem:MDSReduction}
\end{lemma}

To prove Lemma \ref{Lem:MDSReduction}, we will need some definitions and results. We develop these now before proceeding to the proof.\\

Define $\blmat{Y}{n}^{(k,s)}$ to be the matrix $\blmat{Y}{n}$ with columns $c_{k},c_{n+k},\dots, c_{(m-1)n+k}$, and $c_{s}$ filled with zeros and define the resolvent $\mathcal{G}_{n}^{(k,s)}(z):=\left(\blmat{Y}{n}^{(k,s)}-zI\right)^{-1}$. By Corollary \ref{Cor:Omega_nksOverwhelming}, for any $\delta>0$ there exists a constant $c>0$ depending only on $\delta$ such that the event
\begin{equation}
\Omega_{n,k,s}:=\left\{\inf_{|z|> 1+\delta/2}s_{mn}\left(\blmat{Y}{n}^{(k,s)}-zI\right)\geq c\right\}
\label{Def:Omega_{n,k,s}}
\end{equation}
holds with overwhelming probability. By the Sherman-Morrison formula \eqref{equ:ShermanMorrison2}, provided $1+e_{s}^{T}\mathcal{G}_{n}^{(k,s)}(z)c_{s}$ is not zero, we may write 
\begin{equation} \mathcal{G}_{n}^{(k)}(z)c_{s}=\frac{\left(\blmat{Y}{n}^{(k,s)}-zI\right)^{-1}c_{s}}{1+e_{s}^{T}\left(\blmat{Y}{n}^{(k,s)}-zI\right)^{-1}c_{s}}=\mathcal{G}_{n}^{(k,s)}(z)c_{s}\delta_{k,s}(z)\label{Equ:RemovingColS}
\end{equation}
\begin{Details}
	\begin{align}
	\mathcal{G}_{n}^{(k)}(z)c_{s}&=\left(\blmat{Y}{n}^{(k)}-zI\right)^{-1}c_{s}\notag\\
	&=\left(\blmat{Y}{n}^{(k,s)}+c_{s}e_{s}^{T}-zI\right)^{-1}c_{s}\notag\\
	&=\frac{\left(\blmat{Y}{n}^{(k,s)}-zI\right)^{-1}c_{s}}{1+e_{s}^{T}\left(\blmat{Y}{n}^{(k,s)}-zI\right)^{-1}c_{s}}\notag\\
	&=\mathcal{G}_{n}^{(k,s)}(z)c_{s}\delta_{k,s}(z)
	\end{align}
\end{Details}
where 
\begin{equation} 
\delta_{k,s}(z):=(1+e_{s}^{T}\mathcal{G}_{n}^{(k,s)}c_{s})^{-1}.
\label{Def:delta_{n,k}}
\end{equation}
By the same formula,
\begin{equation}
c_{s}^{*}(\mathcal{G}_{n}^{(k)}(w))^{*}=(\delta_{k,s}(w))^{*}c_{s}^{*}(\mathcal{G}_{n}^{(k,s)}(w))^{*}\label{Equ:RemovingColS2}.
\end{equation}
To ensure that these quantities exist, we introduce the event 
\begin{equation}
Q'_{n,k,s}(z):=\left\{\left|e_{s}^{T}\mathcal{G}_{n}^{(k,s)}(z)c_{s}\oindicator{\Omega_{n,k,s}}\right|\leq 1/2\right\}.
\label{Def:Q'}
\end{equation}

\begin{lemma}
	\label{Lem:Q'Overwhelming}
	Define the event $Q'_{n,k,s}(z)$ as in \eqref{Def:Q'}. Then uniformly for any $z\in\mathcal{C}$, $Q'_{n,k,s}(z)$ holds with overwhelming probability 
\end{lemma}

\begin{proof}
	Let $\alpha>0$ be arbitrary. We will show the complement event holds with probability at most $O_{\alpha}(n^{-\alpha})$ uniformly for any $z\in\mathcal{C}$. Observe that by Markov's inequality and by Lemma \ref{Lem:conjLessThanConstant}, uniformly for any $z\in\mathcal{C}$ and for any $p\geq 2$, 
	\begin{align*}
	&\P\left(\left|e_{s}^{T}\mathcal{G}_{n}^{(k,s)}(z)c_{s}\oindicator{\Omega_{n,k,s}}\right|\geq 1/2\right)\\
	&\quad\quad \ll_{p} \E\left|e_{s}^{T}\mathcal{G}_{n}^{(k,s)}(z)c_{s}\oindicator{\Omega_{n,k,s}}\right|^{2p}\\
	&\quad\quad= \E\left|c_{s}^{*}(\mathcal{G}_{n}^{(k,s)}(z))^{*}e_{s}e_{s}^{T}\mathcal{G}_{n}^{(k,s)}(z)c_{s}\oindicator{\Omega_{n,k,s}}\right|^{p}\\
	&\quad\quad\ll_{p,\alpha} n^{-2\varepsilon p+4\varepsilon-2}.
	\end{align*} 
	Selecting $p$ sufficiently large concludes the proof.
\end{proof}
The next Lemma follows by an application of Proposition \ref{Prop:LargeAndSmallSingVals}.
\begin{lemma}
	On the event $\Omega_{n,k,s}$, there exists a constant $C>0$ such that  $\lnorm\mathcal{G}_{n}^{(k,s)}(z)\rnorm\leq C$ almost surely uniformly for any $z$ on the contour $\mathcal{C}$. There exists a constant $C>0$ such that $\left|\delta_{k,s}(z)\oindicator{Q'_{n,k,s}}\right|\leq C$ almost surely uniformly for any $z$ on the contour $\mathcal{C}$.
	\label{Lem:DeltaBoundedOnEvent}
	\label{Lem:G_nksBounded}
\end{lemma}

With these definitions and results in hand, we proceed with the proof of Lemma \ref{Lem:MDSReduction}. In the proof of Lemma \ref{Lem:MDSReduction}, we make some reductions, each of which produces error terms which are sufficiently small in $L^{2}$-norm. In particular, the proof of Lemma \ref{Lem:MDSReduction} uses techniques of expanding using a resolvent identity and invoking Vitali's Theorem to get a self consistent equation, allowing us to solve for the variance. Unlike the proof for a single matrix (see \cite[Lemma 3.2]{RiS} or \cite [Theorem 5.2]{OR:CLT}), we iterate the process $m$ times before recovering a system of self consistent equations.

\begin{proof}[Proof of Lemma \ref{Lem:MDSReduction}]
	We may begin by expanding 
	\begin{align*}
	\E_{k-1}[\breve{M}_{n,k}^{2}]&=\E_{k-1}\left[\left(\sum_{l=1}^{L}\alpha_{l}\breve{Z}_{n,k}(z_{l})+\beta_{l}\overline{\breve{Z}_{n,k}(z_{l})}\right)^{2}\right]\\
	&=\E_{k-1}\left[\sum_{i,j=1}^{L}\alpha_{i}\alpha_{j}\breve{Z}_{n,k}(z_{i})\breve{Z}_{n,k}(z_{j})\right]+\E_{k-1}\left[\sum_{i,j=1}^{L}\alpha_{i}\beta_{j}\breve{Z}_{n,k}(z_{i})\overline{\breve{Z}_{n,k}(z_{j})}\right]\\
	&\quad+\E_{k-1}\left[\sum_{i,j=1}^{L}\beta_{i}\alpha_{j}\overline{\breve{Z}_{n,k}(z_{i})}\breve{Z}_{n,k}(z_{j})\right]+\E_{k-1}\left[\sum_{i,j=1}^{L}\beta_{i}\beta_{j}\overline{\breve{Z}_{n,k}(z_{i})}\overline{\breve{Z}_{n,k}(z_{j})}\right]
	\end{align*}
	where $\breve{Z}_{n,k}(z)$ was defined in \eqref{Def:breveZ}, and therefore 
	\begin{align}
	&\sum_{k=1}^{n}\E_{k-1}\left[\breve{M}_{n,k}^{2}\right]\notag\\
	&\quad=\sum_{k=1}^{n}\sum_{i,j=1}^{L}\alpha_{i}\alpha_{j}\E_{k-1}\left[\breve{Z}_{n,k}(z_{i})\breve{Z}_{n,k}(z_{j})\right]\label{Equ:VarExpansion1}\\
	&\quad\quad+\sum_{k=1}^{n}\sum_{i,j=1}^{L}\alpha_{i}\beta_{j}\E_{k-1}\left[\breve{Z}_{n,k}(z_{i})\overline{\breve{Z}_{n,k}(z_{j})}\right]\label{Equ:VarExpansion2}\\
	&\quad\quad+\sum_{k=1}^{n}\sum_{i,j=1}^{L}\beta_{i}\alpha_{j}\E_{k-1}\left[\overline{\breve{Z}_{n,k}(z_{i})}\breve{Z}_{n,k}(z_{j})\right]\label{Equ:VarExpansion3}\\
	&\quad\quad+\sum_{k=1}^{n}\sum_{i,j=1}^{L}\beta_{i}\beta_{j}\E_{k-1}\left[\overline{\breve{Z}_{n,k}(z_{i})}\overline{\breve{Z}_{n,k}(z_{j})}\right].\label{Equ:VarExpansion4}
	\end{align}
	We analyze each of these terms separately. Note that since the entries in the matrix $\blmat{Y}{n}$ are real, $\overline{\breve{Z}_{n,k}(z_{j})}=\breve{Z}_{n,k}(\overline{z_{j}})$ so the calculations for all terms will be the same. 
	\begin{Details} We show the calculation for a single term and the remaining terms follow in a similar manner. Looking at term \eqref{Equ:VarExpansion2}, we have
	\begin{align*}
	&\sum_{k=1}^{n}\sum_{i,j=1}^{L}\alpha_{i}\beta_{j}\E_{k-1}\left[\breve{Z}_{n,k}(z_{i})\overline{\breve{Z}_{n,k}(z_{j})}\right]\\
	&\quad=\sum_{i,j=1}^{L}\alpha_{i}\beta_{j}\sum_{k=1}^{n}\E_{k-1}\left[\E_{k}\left[\tr\left(V_{k}^{T}(\mathcal{G}_{n}^{(k)}(z_{i}))^{2}U_{k}\right)\oindicator{\Omega_{n,k}}\right]\right.\\
	&\quad\quad\quad\quad\quad\quad\quad\quad\quad\quad\quad\quad\left.\times\overline{\E_{k}\left[\tr\left(V_{k}^{T}(\mathcal{G}_{n}^{(k)}(z_{j}))^{2}U_{k}\right)\oindicator{\Omega_{n,k}}\right]}\right].
	\end{align*}
\end{Details}
	Therefore it suffices to show that 
	\begin{align*}
	\sum_{k=1}^{n}\sum_{i,j=1}^{L}\alpha_{i}\beta_{j}\E_{k-1}\left[\breve{Z}_{n,k}(z)\overline{\breve{Z}_{n,k}(w)}\right]\rightarrow \sum_{1\leq i,j,\leq L}\alpha_{i}\beta_{j}\frac{m^{2}(z\bar{w})^{m-1}}{((z\bar{w})^{m}-1)^{2}}
	\end{align*}
	in probability for fixed $z,w\in\mathcal{C}$. For now, we focus on the sum over $k$. Observe that
	\begin{align}
	&\sum_{k=1}^{n}\E_{k-1}\left[\breve{Z}_{n,k}(z)\overline{\breve{Z}_{n,k}(w)}\right]\\
	&\quad\quad =\sum_{k=1}^{n}\E_{k-1}\left[\E_{k}\left[\sum_{i=1}^{m}\sum_{a,b=1}^{mn}(V_{k}^{T})_{(i,a)}(\mathcal{G}_{n}^{(k)}(z))^{2}_{(a,b)}(U_{k})_{(b,i)}\oindicator{\Omega_{n,k}}\right]\right.\notag\\
	&\quad\quad\quad\quad\quad\quad\quad\quad\quad\times\left.\E_{k}\left[\sum_{j=1}^{m}\sum_{c,d=1}^{mn}(U_{k}^{*})_{(j,d)}(\mathcal{G}_{n}^{(k)}(w))_{(d,c)}^{2*}(V_{k})_{(c,j)}\oindicator{\Omega_{n,k}}\right]\right]\notag\\
	&\quad\quad =\sum_{k=1}^{n}\sum_{i,j=1}^{m}\sum_{a,b,c,d=1}^{mn}\E_{k-1}\left[(V_{k}^{T})_{(i,a)}\E_{k}\left[(\mathcal{G}_{n}^{(k)}(z))^{2}_{(a,b)}\oindicator{\Omega_{n,k}}\right](U_{k})_{(b,i)}\right.\notag\\
	&\quad\quad\quad\quad\quad\quad\quad\quad\quad\quad\quad\quad\quad\quad\times\left.(U_{k}^{*})_{(j,d)}\E_{k}\left[(\mathcal{G}_{n}^{(k)}(w))_{(d,c)}^{2*}\oindicator{\Omega_{n,k}}\right](V_{k})_{(c,j)}\right].\label{Equ:VarExpansion5}
	\end{align}
	At this point, we may exploit the block structure of these matrices in order to reduce the number of terms in the above sums. First, since $V_{k}$ is a $mn\times m$ matrix which contains columns $e_{k},e_{n+k}\dots e_{(m-1)n+k}$, $(V_{k}^{T})_{(i,a)}=0$ unless $a=(i-1)n+k$. The same argument shows that $(V_{k})_{(c,j)}=0$ unless $c=(j-1)n+k$. Since $U_{k}$ is independent of $\mathcal{G}_{n}^{(k)}(z)$, we can factor this out of the expectation and rewrite \eqref{Equ:VarExpansion5} as 
	\begin{Details}
	\begin{align*}
	&\sum_{k=1}^{n}\sum_{i,j=1}^{m}\sum_{a,b,c,d=1}^{mn}\E_{k-1}\left[(V_{k}^{T})_{(i,a)}\E_{k}\left[(\mathcal{G}_{n}^{(k)}(z))^{2}_{(a,b)}\oindicator{\Omega_{n,k}}\right](U_{k})_{(b,i)}\right.\\
	&\quad\quad\quad\quad\quad\quad\quad\quad\quad\quad\quad\times\left.(U_{k}^{*})_{(j,d)}\E_{k}\left[(\mathcal{G}_{n}^{(k)}(w))_{(d,c)}^{2*}\oindicator{\Omega_{n,k}}\right](V_{k})_{(c,j)}\right]\\
	&=\sum_{k=1}^{n}\sum_{i,j=1}^{m}\sum_{b,d=1}^{mn}\E_{k-1}\left[(V_{k}^{T})_{(i,(i-1)n+k)}\E_{k}\left[(\mathcal{G}_{n}^{(k)}(z))^{2}_{((i-1)n+k,b)}\oindicator{\Omega_{n,k}}\right](U_{k})_{(b,i)}\right.\\
	&\quad\quad\quad\quad\quad\quad\quad\quad\quad\quad\quad\times\left.(U_{k}^{*})_{(j,d)}\E_{k}\left[(\mathcal{G}_{n}^{(k)}(w))_{(d,(j-1)n+k)}^{2*}\oindicator{\Omega_{n,k}}\right](V_{k})_{((j-1)n+k,j)}\right]\\
	&=\sum_{k=1}^{n}\sum_{i,j=1}^{m}\sum_{b,d=1}^{mn}\E_{k-1}\left[\E_{k}\left[(\mathcal{G}_{n}^{(k)}(z))^{2}_{((i-1)n+k,b)}\oindicator{\Omega_{n,k}}\right](U_{k})_{(b,i)}\right.\\
	&\quad\quad\quad\quad\quad\quad\quad\quad\quad\quad\quad\times\left.(U_{k}^{*})_{(j,d)}\E_{k}\left[(\mathcal{G}_{n}^{(k)}(w))_{(d,(j-1)n+k)}^{2*}\oindicator{\Omega_{n,k}}\right]\right]\\
	&=\sum_{k=1}^{n}\sum_{i,j=1}^{m}\sum_{b,d=1}^{mn}\E_{k-1}\left[\E_{k}\left[(\mathcal{G}_{n}^{(k)}(z))^{2}_{((i-1)n+k,b)}\oindicator{\Omega_{n,k}}\right]\right.\\
	&\quad\quad\quad\quad\quad\quad\quad\quad\quad\quad\quad\times\left.\E_{k}\left[(\mathcal{G}_{n}^{(k)}(w))_{(d,(j-1)n+k)}^{2*}\oindicator{\Omega_{n,k}}\right]\right]\E[(U_{k})_{(b,i)}(U_{k}^{*})_{(j,d)}].
	\end{align*}	
	\end{Details}
	\begin{align*}
	&\sum_{k=1}^{n}\sum_{i,j=1}^{m}\sum_{b,d=1}^{mn}\E_{k-1}\left[\E_{k}\left[(\mathcal{G}_{n}^{(k)}(z))^{2}_{((i-1)n+k,b)}\oindicator{\Omega_{n,k}}\right]\right.\\
	&\quad\quad\quad\quad\quad\quad\quad\quad\quad\quad\quad\times\left.\E_{k}\left[(\mathcal{G}_{n}^{(k)}(w))_{(d,(j-1)n+k)}^{2*}\oindicator{\Omega_{n,k}}\right]\right]\E[(U_{k})_{(b,i)}(U_{k}^{*})_{(j,d)}].
	\end{align*}
	Now, if $i\neq j$ or $b\neq d$, then $(U_{k})_{(b,i)}$ and $(U_{k}^{*})_{(j,d)}$ come from different columns or are different entries in the same column of $\blmat{Y}{n}$ and hence are independent. Therefore, the only non-zero terms are those in which $i=j$ and $b=d$. Thus the sum in \eqref{Equ:VarExpansion5} can be further reduced to 
	\begin{equation*}
	\sum_{k=1}^{n}\sum_{i=1}^{m}\sum_{b=1}^{mn}\E_{k}\left[(\mathcal{G}_{n}^{(k)}(z))^{2}_{((i-1)n+k,b)}\oindicator{\Omega_{n,k}}\right]\E_{k}\left[(\mathcal{G}_{n}^{(k)}(w))_{(b,(i-1)n+k)}^{2*}\oindicator{\Omega_{n,k}}\right]\E\left|(U_{k})_{(b,i)}\right|^{2}.
	\end{equation*}
	Now, since $U_{k}$ is filled with columns $c_{k},c_{n+k},\dots c_{(m-1)n+k}$, we may analyze the structure of these columns to evaluate $\E\left|(U_{k})_{(b,i)}\right|^{2}$. Since column $c_{(i-1)n+k}$ comes from the $i$th block of $\blmat{Y}{n}$, $(U_{k})_{(b,i)}=0$ unless $(i-2)n-1\leq b\leq (i-1)n$ where these subscripts are reduced modulo $m$ and we use the convention that $-1n\equiv (m-1)n$ and $0n\equiv mn$. For $b$ in such a range, we have $\E|U_{(b,i)}|^{2}=\frac{1}{n}$. Therefore we can simplify the sum in \eqref{Equ:VarExpansion5} further as 
	\begin{equation*}
	\frac{1}{n}\sum_{k=1}^{n}\sum_{i=1}^{m}\sum_{b=(i-2)n+1}^{(i-1)n}\E_{k}\left[(\mathcal{G}_{n}^{(k)}(z))^{2}_{((i-1)n+k,b)}\oindicator{\Omega_{n,k}}\right]\E_{k}\left[(\mathcal{G}_{n}^{(k)}(w))_{(b,(i-1)n+k)}^{2*}\oindicator{\Omega_{n,k}}\right].\\
	\end{equation*}
    Define the diagonal $mn\times mn$ matrix $\mathcal{D}_{p}$ with entries  
	\begin{equation}
	\left(\mathcal{D}_{p}\right)_{(i,j)}:=\begin{cases}
	1 & \text{ if } i=j,\;(p-1)n+1\leq i \leq pn\\
	0 & \text{ otherwise}
	\end{cases}
	\label{Def:ProjectionMatrix}
	\end{equation}
	for $1\leq p\leq m$ and $1\leq i,j\leq mn$. Note that $\mathcal{D}_{p}$ is nonzero only on the diagonal of the $p$th block. 
	Then we have 
	\begin{align}
	&\frac{1}{n}\sum_{k=1}^{n}\sum_{i=1}^{m}\sum_{b=(i-2)n+1}^{(i-1)n}\E_{k}\left[(\mathcal{G}_{n}^{(k)}(z))^{2}_{((i-1)n+k,b)}\oindicator{\Omega_{n,k}}\right]\E_{k}\left[(\mathcal{G}_{n}^{(k)}(w))_{(b,(i-1)n+k)}^{2*}\oindicator{\Omega_{n,k}}\right]\notag\\
	&=\frac{1}{n}\sum_{k=1}^{n}\sum_{i=1}^{m}e_{(i-1)n+k}^{T}\E_{k}\left[(\mathcal{G}_{n}^{(k)}(z))^{2}\oindicator{\Omega_{n,k}}\right]\mathcal{D}_{i-1}\E_{k}\left[(\mathcal{G}_{n}^{(k)}(w))^{2*}\oindicator{\Omega_{n,k}}\right]e_{(i-1)n+k}\label{Equ:Lem:FinDimDist:1}
	\end{align}
	where the subscript on $\mathcal{D}_{i-1}$ is reduced modulo $m$ and in the range $\{1,\dots,m\}$.
	 \begin{Details}
	 	 Next, instead of proving the convergence of \eqref{Equ:Lem:FinDimDist:1}, we move to proving the convergence of its anti-derivative. By Vitali's theorem (see for instance \cite[Lemma 2.14]{BSbook}) convergence of \eqref{Equ:Lem:FinDimDist:1} will follow provided that the primitive satisfies the hypothesis of Vitali's theorem.
	 	If $\{f_{N}(z)\}$ is a sequence of uniformly bounded analytic functions on a domain $\mathbb{D}\subset\mathcal{C}$ converging at a collection $\{z\}$ containing a limit point in $\mathbb{D}$, then there exists an analytic function $f(z)$ such that $f_{N}(z)\rightarrow f(z)$ and $f'_{N}(z)\rightarrow f'(z)$ throughout $\mathbb{D}$.
	 	if we show that \[e_{(i-1)n+k}^{T}\E_{k}[(\mathcal{G}_{n}^{(k)}(z))^{2}\oindicator{\Omega_{n,k}}]\mathcal{D}_{i-1}\E_{k}[(\mathcal{G}_{n}^{(k)}(w))^{2*}\oindicator{\Omega_{n,k}}]e_{(i-1)n+k}.\] 
	 	Provided this primitive function satisfies the hypothesis of Vitali's theorem, then the convergence of the original function will follow.
	 \end{Details}  
	 Next, if we can show that
	 \begin{align}
	 &\frac{1}{n}\sum_{k=1}^{n}\sum_{i=1}^{m}e_{(i-1)n+k}^{T}\E_{k}\left[\mathcal{G}_{n}^{(k)}(z)\oindicator{\Omega_{n,k}}\right]\mathcal{D}_{i-1}\E_{k}\left[(\mathcal{G}_{n}^{(k)}(w))^{*}\oindicator{\Omega_{n,k}}\right]e_{(i-1)n+k} \label{Equ:MDSReductionAntiderivative}\\
	 &\quad\quad\quad\quad\quad\quad\quad\quad\quad\longrightarrow-\ln\left(1-\frac{1}{(z\bar{w})^{m}}\right),\notag
	\end{align}
	in probability as $n\rightarrow\infty$, then by Vitali's theorem (see for instance \cite[Lemma 2.14]{BSbook}), it will follow that 
	\begin{align}
	&\frac{1}{n}\sum_{k=1}^{n}\sum_{i=1}^{m}e_{(i-1)n+k}^{T}\E_{k}\left[(\mathcal{G}_{n}^{(k)}(z))^{2}\oindicator{\Omega_{n,k}}\right]\mathcal{D}_{i-1}\E_{k}\left[(\mathcal{G}_{n}^{(k)}(w))^{2*}\oindicator{\Omega_{n,k}}\right]e_{(i-1)n+k}\notag\\
	&\quad\quad\quad\quad\quad\quad\quad\quad\quad\longrightarrow \frac{m^{2}(z\bar{w})^{m-1}}{((z\bar{w})^{m}-1)^{2}}.\label{Equ:ConvergenceOfOrigVitali}
	\end{align}
	Vitali's theorem is justified because \eqref{Equ:MDSReductionAntiderivative} is bounded and analytic in the region where $|z|,|w|>1+\delta/2$ and this region has an accumulation point. Note that here we apply Vitali's theorem twice, once in the variable $z$ and once in the variable $\bar{w}$. 
	\begin{Details}
		satisfies the assumptions of Vitali's theorem. To this end, note that \eqref{Equ:MDSReductionAntiderivative} is analytic in the variable $z$ over the region $\{z\in\C\;:\;|z|>1+\delta/2\}$ for any fixed $|\overline{w}|>1+\delta/2$. \eqref{Equ:MDSReductionAntiderivative} is also analytic in the variable $\bar{w}$ over the region $\{\bar{w}\in\C\;:\;|w|>1+\delta/2\}$ for any fixed $|z|>1+\delta/2$, and is uniformly bounded. Finally, the collection $\{z,\overline{w}\in\C\;:\;|z|,|\overline{w}|>1+\delta/2\}$ contains an accumulation point, which verifies the hypothesis of Vitali's theorem. With this, we are justified in the application of Vitali's theorem twice, once in $z$ and then once in $\overline{w}$. After doing so, we arrive at the conclusion that it suffices to show 
		\begin{align}
		&\frac{1}{n}\sum_{k=1}^{n}\sum_{i=1}^{m}e_{(i-1)n+k}^{T}\E_{k}\left[\mathcal{G}_{n}^{(k)}(z)\oindicator{\Omega_{n,k}}\right]\mathcal{D}_{i-1}\E_{k}\left[(\mathcal{G}_{n}^{(k)}(w))^{*}\oindicator{\Omega_{n,k}}\right]e_{(i-1)n+k}\notag\\
		&\quad\quad\quad\quad\quad\quad\quad\rightarrow -\ln\left(1-\frac{1}{(z\bar{w})^{m}}\right)
		\label{Equ:ConvergenceByVitali}
		\end{align}
		in probability
		\begin{remark}
			Note that the goal of applying Vitali's theorem was to work with the resolvents $\mathcal{G}_{n}^{(k)}(z)$ and $\left(\mathcal{G}_{n}^{(k)}(w)\right)^{*}$ instead of their squares $\left(\mathcal{G}_{n}^{(k)}(z)\right)^{2}$ and $\left(\mathcal{G}_{n}^{(k)}(w)\right)^{2*}$. 
		\end{remark}
	\end{Details}

	 To analyze the limit of \eqref{Equ:MDSReductionAntiderivative}, we will focus on a fixed term in the sum. Define 
	 \begin{equation}
	 \mathcal{T}_{n,k}(z,\overline{w}) := \sum_{i=1}^{m}e_{(i-1)n+k}^{T}\E_{k}\left[\mathcal{G}_{n}^{(k)}(z)\oindicator{\Omega_{n,k}}\right]\mathcal{D}_{i-1}\E_{k}\left[(\mathcal{G}_{n}^{(k)}(w))^{*}\oindicator{\Omega_{n,k}}\right]e_{(i-1)n+k}.\label{Def:T_{n,k}}
	 \end{equation}
	 
	 Provided the resolvent is defined, we have the matrix identity (see for example, \cite[Equation (3.15)]{RiS}),
	 \begin{equation}
	 \mathcal{G}_{n}^{(k)}(z)=-\frac{1}{z}I+\frac{1}{z}\sum_{t\neq* n+k}\mathcal{G}_{n}^{(k)}(z)c_{t}e_{t}^{T}
	 \label{Equ:ResolventExpansion}
	 \end{equation}
	 where the notation $t\neq* n+k$ indicates that the sum is over all $1\leq t\leq mn$ such that $t\neq k, n+k, \dots, (m-1)n+k$. We may use this to expand the term $\mathcal{T}_{n,k}(z,\overline{w})$. If we expand a generic term in the sum \eqref{Def:T_{n,k}}, we can show the following lemma holds.
	 
	 	 \begin{lemma}
	 	Define all quantities as in Lemma \ref{Lem:MDSReduction}. Let $b$ be fixed with $1\leq b< m$. Then under the assumptions of Lemma \ref{Lem:MDSReduction},
	 	\begin{align*}
	 	&\E\left|\frac{1}{n}\sum_{k=1}^{n}\sum_{i=1}^{m}\left(e_{(i-1)n+k}^{T}\E_{k}\left[\mathcal{G}_{n}^{(k)}(z)\oindicator{\Omega_{n,k}}\right]\mathcal{D}_{i-b}\E_{k}\left[(\mathcal{G}_{n}^{(k)}(w))^{*}\oindicator{\Omega_{n,k}}\right]e_{(i-1)n+k}\right.\right.\\
	 	&\quad\quad-\frac{1}{z\bar{w}}\frac{k-1}{n}e_{(i-1)n+k}^{T}\E_{k}\left[\mathcal{G}_{n}^{(k)}(z)\oindicator{\Omega_{n,k}}\right]\mathcal{D}_{i-b-1}\\
	 	&\quad\quad\quad\quad\quad\quad\quad\quad\quad\quad\quad\quad\left.\left.\times\E_{k}\left[(\mathcal{G}_{n}^{(k)}(w))^{*}\oindicator{\Omega_{n,k}}\right]e_{(i-1)n+k}\right)\right|^{2}=o(1)
	 	\end{align*}
	 	where $i-b$ and $i-b-1$ are reduced modulo $m$ with representatives in $\{1,\dots,m\}$.
	 	\label{Lem:Iteration}
	 \end{lemma}
 The use of Lemma \ref{Lem:Iteration} is one of the main technical components of this paper. This lemma allows us to iterate the techniques used in previous results (see for example \cite{OR:CLT,RiS}) which results in a system of self consistent equations. Due to the iterative process used in the proof, the two terms in the difference in Lemma \ref{Lem:Iteration} differ from $\mathcal{T}_{n,k}(z,\overline{w})$ defined in \eqref{Def:T_{n,k}} because of the subscripts on $\mathcal{D}_{p}$. We prove Lemma \ref{Lem:Iteration} now, but several of the technical calculations are done separately for clarity. These calculations are presented in lemmas at the end of Section \ref{Sec:FiniteDimDist}.  
 
 \begin{proof}[Proof of Lemma \ref{Lem:Iteration}]
 	To begin, we may use the matrix identity from equation \eqref{Equ:ResolventExpansion} to expand $e^{T}_{(i-1)n+k}\E_{k}\left[\mathcal{G}_{n}^{(k)}(z)\oindicator{\Omega_{n,k}}\right]\mathcal{D}_{i-b}\E_{k}\left[(\mathcal{G}_{n}^{(k)}(w))^{*}\oindicator{\Omega_{n,k}}\right]e_{(i-1)n+k}.$ Doing so, we have that
 	\begin{Details}
 		\begin{align*}
 		&\E_{k}[\mathcal{G}_{n}^{(k)}(z)\oindicator{\Omega_{n,k}}]\mathcal{D}_{i-b}\E_{k}[(\mathcal{G}_{n}^{(k)}(w))^{*}\oindicator{\Omega_{n,k}}]\\
 		&=\frac{1}{z\bar{w}}\mathcal{D}_{i-b}(\P_{k}(\Omega_{n,k}))^{2}-\frac{1}{z\overline{w}}\sum_{s\neq * n+ k}\mathcal{D}_{i-1}\E_{k}\left[e_{s}c_{s}^{*}(\mathcal{G}_{n}^{(k)}(w))^{*}\oindicator{\Omega_{n,k}}\right]\P_{k}(\Omega_{n,k})\\
 		&\quad-\frac{1}{z\overline{w}}\sum_{t\neq * n+k}\E_{k}\left[\mathcal{G}_{n}^{(k)}(z)c_{t}e_{t}^{T}\oindicator{\Omega_{n,k}}\right]\mathcal{D}_{i-b}\P_{k}(\Omega_{n,k})\\
 		&\quad+\frac{1}{z\bar{w}}\sum_{s\neq * n+k}\sum_{t\neq * n+k}\E_{k}\left[\mathcal{G}_{n}^{(k)}(z)c_{t}e_{t}^{T}\oindicator{\Omega_{n,k}}\right]\mathcal{D}_{i-b}\E_{k}\left[e_{s}c_{s}^{*}(\mathcal{G}_{n}^{(k)}(w))^{*}\oindicator{\Omega_{n,k}}\right]
 		\end{align*}
 	\end{Details}
 	\begin{align}
 	&e_{(i-1)n+k}^{T}\E_{k}[\mathcal{G}_{n}^{(k)}(z)\oindicator{\Omega_{n,k}}]\mathcal{D}_{i-b}\E_{k}[(\mathcal{G}_{n}^{(k)}(w))^{*}\oindicator{\Omega_{n,k}}]e_{(i-1)n+k}\notag\\
 	&=e_{(i-1)n+k}^{T}\mathcal{D}_{i-b}\frac{1}{z\bar{w}}(\P_{k}(\Omega_{n,k}))^{2}e_{(i-1)n+k}\label{Equ:ExpansionOfSingleTerm1}\\
 	&\quad-e_{(i-1)n+k}^{T}\frac{1}{z\overline{w}}\sum_{s\neq * n+ k}\mathcal{D}_{i-b}\E_{k}\left[e_{s}c_{s}^{*}(\mathcal{G}_{n}^{(k)}(w))^{*}\oindicator{\Omega_{n,k}}\right]\P_{k}(\Omega_{n,k})e_{(i-1)n+k}\label{Equ:ExpansionOfSingleTerm2}\\
 	&\quad-e_{(i-1)n+k}^{T}\frac{1}{z\overline{w}}\sum_{t\neq * n+k}\E_{k}\left[\mathcal{G}_{n}^{(k)}(z)c_{t}e_{t}^{T}\oindicator{\Omega_{n,k}}\right]\mathcal{D}_{i-b}\P_{k}(\Omega_{n,k})e_{(i-1)n+k}\label{Equ:ExpansionOfSingleTerm3}\\
 	&\quad+e_{(i-1)n+k}^{T}\frac{1}{z\bar{w}}\sum_{s\neq * n+ k}\sum_{t\neq * n+ k}\E_{k}\left[\mathcal{G}_{n}^{(k)}(z)c_{t}e_{t}^{T}\oindicator{\Omega_{n,k}}\right]\mathcal{D}_{i-b}\notag\\
 	&\quad\quad\quad\quad\quad\quad\quad\quad\quad\quad\quad\quad\times\E_{k}\left[e_{s}c_{s}^{*}(\mathcal{G}_{n}^{(k)}(w))^{*}\oindicator{\Omega_{n,k}}\right]e_{(i-1)n+k}.\label{Equ:ExpansionOfSingleTerm4}
 	\end{align}
 	where $\P_{k}$ denotes the conditional probability with respect to $\mathcal{F}_{k}$ and we assume that the subscript on $\mathcal{D}_{i-b}$ is reduced modulo $m$. We simplify each term separately. For any $1\leq i\leq m$, assuming $b\leq m$, term \eqref{Equ:ExpansionOfSingleTerm1} equals zero since the only nonzero elements in $\mathcal{D}_{i-b}$ are in block $i-b$, and \eqref{Equ:ExpansionOfSingleTerm1} selects an element from the $i$th block.
 	\begin{Details} 
 		\begin{align*}
 		&e_{(i-1)n+k}^{T}\mathcal{D}_{i-b}\frac{1}{z\bar{w}}(\P_{k}(\Omega_{n,k}))^{2}e_{(i-1)n+k}=0\\
 		\end{align*}
 		since the only nonzero elements in $\mathcal{D}_{i-b}$ are in block $i-b$, and the above product selects and element from the $i$th block.
 	\end{Details} 
 	For terms \eqref{Equ:ExpansionOfSingleTerm2} and \eqref{Equ:ExpansionOfSingleTerm3}, since $s\neq jn+k$ for any $0\leq j\leq m-1$, we have the expression $e_{(i-1)n+k}^{T}\mathcal{D}_{i-b}e_{s}$, which results in an off diagonal element of $\mathcal{D}_{i-b}$. Ergo, 
 	\begin{align*}
 	&-e_{(i-1)n+k}^{T}\frac{1}{z\overline{w}}\sum_{s\neq * n+ k}\mathcal{D}_{i-b}\E_{k}\left[e_{s}c_{s}^{*}(\mathcal{G}_{n}^{(k)}(w))^{*}\oindicator{\Omega_{n,k}}\right]\P_{k}(\Omega_{n,k})e_{(i-1)n+k}\\
 	&\quad\quad=-\frac{1}{z\overline{w}}\sum_{s\neq *  n+ k}e_{(i-1)n+k}^{T}\mathcal{D}_{i-b}e_{s}\E_{k}\left[c_{s}^{*}(\mathcal{G}_{n}^{(k)}(w))^{*}\oindicator{\Omega_{n,k}}\right]\P_{k}(\Omega_{n,k})e_{(i-1)n+k}\\
 	&\quad\quad=0
 	\end{align*}  
 	since $(\mathcal{D}_{i-b})_{f,g}=0$ unless $(i-b-1)n+1\leq f,g\leq (i-b)n$ and $f=g$. Similarly, since $t\neq jn+k$ for any $1\leq j\leq m$, we have $e_{t}^{T}\mathcal{D}_{i-b}e_{(i-1)n+k}=0$. Thus terms \eqref{Equ:ExpansionOfSingleTerm2} and \eqref{Equ:ExpansionOfSingleTerm3} are zero. Note that $e_{t}^{T}\mathcal{D}_{i-b}e_{s}=(\mathcal{D}_{i-b})_{(t,s)}.$ This is zero unless $t=s$ and $(i-b-1)n+1\leq s\leq (i-b)n$. Therefore term \eqref{Equ:ExpansionOfSingleTerm4} can be simplified to
 	\begin{align*}
 	&e_{(i-1)n+k}^{T}\frac{1}{z\bar{w}}\sum_{s\neq * n+k}\sum_{t\neq * n+k}\E_{k}\left[\mathcal{G}_{n}^{(k)}(z)c_{t}e_{t}^{T}\oindicator{\Omega_{n,k}}\right]\mathcal{D}_{i-b}\E_{k}\left[e_{s}c_{s}^{*}(\mathcal{G}_{n}^{(k)}(w))^{*}\oindicator{\Omega_{n,k}}\right]e_{(i-1)n+k}\\
 	&\quad=e_{(i-1)n+k}^{T}\frac{1}{z\bar{w}}\sum_{\substack{ s=(i-b-1)n+1\\ s\neq * n+k}}^{(i-b)n}\E_{k}\left[\mathcal{G}_{n}^{(k)}(z)c_{s}\oindicator{\Omega_{n,k}}\right]\E_{k}\left[c_{s}^{*}(\mathcal{G}_{n}^{(k)}(w))^{*}\oindicator{\Omega_{n,k}}\right]e_{(i-1)n+k}.
 	\end{align*}
 	Ergo, we have 
 	\begin{align}
 	&e_{(i-1)n+k}^{T}\E_{k}[\mathcal{G}_{n}^{(k)}(z)\oindicator{\Omega_{n,k}}]\mathcal{D}_{i-b}\E_{k}[(\mathcal{G}_{n}^{(k)}(w))^{*}\oindicator{\Omega_{n,k}}]e_{(i-1)n+k} \notag\\
 	&=\frac{1}{z\bar{w}}\sum_{\substack{ s=(i-b-1)n+1,\\ s\neq * n+k}}^{(i-b)n}e_{(i-1)n+k}^{T}\E_{k}\left[\mathcal{G}_{n}^{(k)}(z)c_{s}\oindicator{\Omega_{n,k}}\right]\E_{k}\left[c_{s}^{*}(\mathcal{G}_{n}^{(k)}(w))^{*}\oindicator{\Omega_{n,k}}\right]e_{(i-1)n+k}.\label{Equ:T_nk}
 	\end{align}
 	We now remove the $s$th column from the resolvent. In order to remove this column, we need to work on the appropriate events. Since $\Omega_{n,k,s}$ defined in \eqref{Def:Omega_{n,k,s}} holds with overwhelming probability by Corollary \ref{Cor:Omega_nksOverwhelming}, we may insert the event with as sufficiently small $L^{2}$ norm error. This is verified in Lemma \ref{Lem:FiniteDimError1}. Since Lemma \ref{Lem:Q'Overwhelming} proves that $Q'_{n,k,s}(z)$ (defined in $\eqref{Def:Q'}$) also holds with overwhelming probability, a very similar argument shows this event can be inserted as well. For ease of notation, we will drop the dependence on $z$ in $Q'_{n,k,s}(z)$ and write $Q'_{n,k,s}$. We proceed with 
 	\begin{align*}
 	&e_{(i-1)n+k}^{T}\E_{k}\left[\mathcal{G}_{n}^{(k)}(z)c_{s}\oindicator{\Omega_{n,k}\cap\Omega_{n,k,s}\cap Q'_{n,k,s}}\right]\\
 	&\quad\quad\quad\quad\quad\quad\quad\times\E_{k}\left[c_{s}^{*}(\mathcal{G}_{n}^{(k)}(w))^{*}\oindicator{\Omega_{n,k}\cap\Omega_{n,k,s}\cap Q'_{n,k,s}}\right]e_{(i-1)n+k}.
 	\end{align*}
 	Then by \eqref{Equ:RemovingColS} and \eqref{Equ:RemovingColS2}, we have
 	\begin{align*}
 	&e_{(i-1)n+k}^{T}\E_{k}\left[\mathcal{G}_{n}^{(k)}(z)c_{s}\oindicator{\Omega_{n,k}\cap\Omega_{n,k,s}\cap Q'_{n,k,s}}\right]\\
 	&\quad\quad\quad\quad\quad\quad\times\E_{k}\left[c_{s}^{*}(\mathcal{G}_{n}^{(k)}(w))^{*}\oindicator{\Omega_{n,k}\cap\Omega_{n,k,s}\cap Q'_{n,k,s}}\right]e_{(i-1)n+k}\\
 	&=e_{(i-1)n+k}^{T}\E_{k}\left[\mathcal{G}_{n}^{(k,s)}(z)c_{s}\delta_{k,s}(z)\oindicator{\Omega_{n,k}\cap\Omega_{n,k,s}\cap Q'_{n,k,s}}\right]\\
 	&\quad\quad\quad\quad\quad\quad\times\E_{k}\left[(\delta_{k,s}(w))^{*}c_{s}^{*}(\mathcal{G}_{n}^{(k,s)}(w))^{*}\oindicator{\Omega_{n,k}\cap\Omega_{n,k,s}\cap Q'_{n,k,s}}\right]e_{(i-1)n+k}.
 	\end{align*}
 	Since $G_{n}^{(k)}(z)$ is no longer present in the expression, the same argument as above shows that we can now remove the event $\Omega_{n,k}$ and work with
 	\begin{Details} 
 		\begin{align*}
 		&e_{(i-1)n+k}^{T}\E_{k}\left[\mathcal{G}_{n}^{(k,s)}(z)c_{s}\delta_{k,s}(z)\oindicator{\Omega_{n,k}\cap\Omega_{n,k,s}\cap Q'_{n,k,s}}\right]\\
 		&\quad\quad\quad\quad\quad\quad\cdot\E_{k}\left[(\delta_{k,s}(w))^{*}c_{s}^{*}(\mathcal{G}_{n}^{(k,s)}(w))^{*}\oindicator{\Omega_{n,k}\cap\Omega_{n,k,s}\cap Q'_{n,k,s}}\right]e_{(i-1)n+k}\\
 		& = e_{(i-1)n+k}^{T}\E_{k}\left[\mathcal{G}_{n}^{(k,s)}(z)c_{s}\delta_{k,s}(z)\oindicator{\Omega_{n,k,s}\cap Q'_{n,k,s}}\right]\\
 		&\quad\quad\quad\quad\quad\quad\cdot\E_{k}\left[(\delta_{k,s}(w))^{*}c_{s}^{*}(\mathcal{G}_{n}^{(k,s)}(w))^{*}\oindicator{\Omega_{n,k,s}\cap Q'_{n,k,s}}\right]e_{(i-1)n+k}\\
 		&\quad -e_{(i-1)n+k}^{T}\E_{k}\left[\mathcal{G}_{n}^{(k,s)}(z)c_{s}\delta_{k,s}(z)\oindicator{\Omega_{n,k,s}\cap Q'_{n,k,s}}\right]\\
 		&\quad\quad\quad\quad\quad\quad\quad\cdot\E_{k}\left[(\delta_{k,s}(w))^{*}c_{s}^{*}(\mathcal{G}_{n}^{(k,s)}(w))^{*}\oindicator{\Omega_{n,k}^{c}\cap\Omega_{n,k,s}\cap Q'_{n,k,s}}\right]e_{(i-1)n+k}\\
 		&\quad -e_{(i-1)n+k}^{T}\E_{k}\left[\mathcal{G}_{n}^{(k,s)}(z)c_{s}\delta_{k,s}(z)\oindicator{\Omega_{n,k}^{c}\cap\Omega_{n,k,s}\cap Q'_{n,k,s}}\right]\\
 		&\quad\quad\quad\quad\quad\quad\quad\cdot\E_{k}\left[(\delta_{k,s}(w))^{*}c_{s}^{*}(\mathcal{G}_{n}^{(k,s)}(w))^{*}\oindicator{\Omega_{n,k,s}\cap Q'_{n,k,s}}\right]e_{(i-1)n+k}\\
 		&\quad+e_{(i-1)n+k}^{T}\E_{k}\left[\mathcal{G}_{n}^{(k,s)}(z)c_{s}\delta_{k,s}(z)\oindicator{\Omega_{n,k}^{c}\cap\Omega_{n,k,s}\cap Q'_{n,k,s}}\right]\\
 		&\quad\quad\quad\quad\quad\quad\quad\cdot\E_{k}\left[(\delta_{k,s}(w))^{*}c_{s}^{*}(\mathcal{G}_{n}^{(k,s)}(w))^{*}\oindicator{\Omega_{n,k}^{c}\cap\Omega_{n,k,s}\cap Q'_{n,k,s}}\right]e_{(i-1)n+k}\\
 		&=e_{(i-1)n+k}^{T}\E_{k}\left[\mathcal{G}_{n}^{(k,s)}(z)c_{s}\delta_{k,s}(z)\oindicator{\Omega_{n,k,s}\cap Q'_{n,k,s}}\right]\\
 		&\quad\quad\quad\quad\quad\quad\quad\cdot\E_{k}\left[(\delta_{k,s}(w))^{*}c_{s}^{*}(\mathcal{G}_{n}^{(k,s)}(w))^{*}\oindicator{\Omega_{n,k,s}\cap Q'}\right]e_{(i-1)n+k}+o(n^{-1}).
 		\end{align*}
 	\end{Details}
 	\begin{align*}
 	&e_{(i-1)n+k}^{T}\E_{k}\left[\mathcal{G}_{n}^{(k,s)}(z)c_{s}\delta_{k,s}(z)\oindicator{\Omega_{n,k,s}\cap Q'_{n,k,s}}\right]\\
 	&\quad\quad\quad\quad\quad\quad\quad\times\E_{k}\left[(\delta_{k,s}(w))^{*}c_{s}^{*}(\mathcal{G}_{n}^{(k,s)}(w))^{*}\oindicator{\Omega_{n,k,s}\cap Q'_{n,k,s}}\right]e_{(i-1)n+k}
 	\end{align*}
 	with a sufficiently small $L^{2}$-norm error. Next, we wish to replace $\delta_{k,s}(z)$ and $(\delta_{k,s}(w))^{*}$ with 1. Observe that 
 	\begin{align*}
 	&\E\Bigg|\frac{1}{n}\sum_{k=1}^{n}\frac{1}{z\bar{w}}\sum_{i=1}^{m}\sum_{\substack{ s=(i-b-1)n+1,\\ s\neq * n+k}}^{(i-b)n}e_{(i-1)n+k}^{T}\left(\E_{k}\left[\mathcal{G}_{n}^{(k,s)}(z)c_{s}\delta_{k,s}(z)\oindicator{\Omega_{n,k,s}\cap Q'_{n,k,s}}\right]\right.\\
 	&\quad\quad\quad\quad\quad\quad\quad\quad\quad\quad\times\left.\left.\E_{k}\left[(\delta_{k,s}(w))^{*}c_{s}^{*}(\mathcal{G}_{n}^{(k,s)}(w))^{*}\oindicator{\Omega_{n,k,s}\cap Q'_{n,k,s}}\right]e_{(i-1)n+k}\right.\right.\\
 	&\quad - \left.\E_{k}\left[\mathcal{G}_{n}^{(k,s)}(z)c_{s}\oindicator{\Omega_{n,k,s}\cap Q'_{n,k,s}}\right]\E_{k}\left[c_{s}^{*}(\mathcal{G}_{n}^{(k,s)}(w))^{*}\oindicator{\Omega_{n,k,s}\cap Q'_{n,k,s}}\right]\right)e_{(i-1)n+k} \Bigg|^{2}\\	
 	&\ll \max_{\substack{1\leq k\leq n\\
 			1\leq i\leq m}}\E\Bigg|\sum_{\substack{ s=(i-b-1)n+1,\\ s\neq * n+k}}^{(i-b)n}e_{(i-1)n+k}^{T}\left(\E_{k}\left[\mathcal{G}_{n}^{(k,s)}(z)c_{s}\delta_{k,s}(z)\oindicator{\Omega_{n,k,s}\cap Q'_{n,k,s}}\right]\right.\\
 	&\quad\quad\quad\quad\quad\quad\quad\quad\quad\quad\times\left.\E_{k}\left[(\delta_{k,s}(w))^{*}c_{s}^{*}(\mathcal{G}_{n}^{(k,s)}(w))^{*}\oindicator{\Omega_{n,k,s}\cap Q'_{n,k,s}}\right]e_{(i-1)n+k}\right.\\
 	&\quad - \left.\E_{k}\left[\mathcal{G}_{n}^{(k,s)}(z)c_{s}\oindicator{\Omega_{n,k,s}\cap Q'_{n,k,s}}\right]\E_{k}\left[c_{s}^{*}(\mathcal{G}_{n}^{(k,s)}(w))^{*}\oindicator{\Omega_{n,k,s}\cap Q'_{n,k,s}}\right]\right)e_{(i-1)n+k} \Bigg|^{2}.	
 	\end{align*} 
 	Therefore, it is sufficient to show that
 	\begin{align*}
 	&\E\left|e_{(i-1)n+k}^{T}\E_{k}\left[\mathcal{G}_{n}^{(k,s)}(z)c_{s}\delta_{k,s}(z)\oindicator{\Omega_{n,k,s}\cap Q'_{n,k,s}}\right]\right.\\
 	&\quad\quad\quad\quad\quad\quad\quad\quad\quad\quad\quad\quad\left.\times\E_{k}\left[(\delta_{k,s}(w))^{*}c_{s}^{*}(\mathcal{G}_{n}^{(k,s)}(w))^{*}\oindicator{\Omega_{n,k,s}\cap Q'_{n,k,s}}\right]e_{(i-1)n+k}\right.\\
 	&\quad\quad\left.-e_{(i-1)n+k}^{T}\E_{k}\left[\mathcal{G}_{n}^{(k,s)}(z)c_{s}\oindicator{\Omega_{n,k,s}\cap Q'_{n,k,s}}\right]\right.\\
 	&\quad\quad\quad\quad\quad\quad\quad\quad\quad\quad\quad\quad\left.\times\E_{k}\left[c_{s}^{*}(\mathcal{G}_{n}^{(k,s)}(w))^{*}\oindicator{\Omega_{n,k,s}\cap Q'_{n,k,s}}\right]e_{(i-1)n+k}]\right|^{2}=o(n^{-2})
 	\end{align*}
 	uniformly in $i$, $k$, and $s$. This is done in Lemma \ref{Lem:FiniteDimError2}. Since $\delta_{n,k}$ is no longer present, we can justify dropping the event $Q'_{n,k,s}$ by an argument similar to Lemma \ref{Lem:FiniteDimError1}. Thus, we can continue from here working with 
 	\[\frac{1}{z\bar{w}}\sum_{i=1}^{m}\sum_{\substack{s=(i-b-1)n+1,\\ s\neq * n+k}}^{(i-b)n}e_{(i-1)n+k}^{T}\E_{k}\left[\mathcal{G}_{n}^{(k,s)}(z)c_{s}\oindicator{\Omega_{n,k,s}}\right]\E_{k}\left[c_{s}^{*}(\mathcal{G}_{n}^{(k,s)}(w))^{*}\oindicator{\Omega_{n,k,s}}\right]e_{(i-1)n+k}.\]
 	Next, for any $1\leq i\leq m$, by independence of $c_{s}$ from $\mathcal{G}_{n}^{(k,s)}(z)\oindicator{\Omega_{n,k,s}}$, we can factor the above as
 	\begin{align}
 	&\sum_{\substack{s=(i-b-1)n+1,\\ s\neq * n+k}}^{(i-b)n}e_{(i-1)n+k}^{T}\E_{k}\left[\mathcal{G}_{n}^{(k,s)}(z)c_{s}\oindicator{\Omega_{n,k,s}}\right]\E_{k}\left[c_{s}^{*}(\mathcal{G}_{n}^{(k,s)}(w))^{*}\oindicator{\Omega_{n,k,s}}\right]e_{(i-1)n+k}\notag\\
 	&\quad=\sum_{\substack{s=(i-b-1)n+1,\\ s\neq * n+k}}^{(i-b)n}e_{(i-1)n+k}^{T}\E_{k}\left[\mathcal{G}_{n}^{(k,s)}(z)\oindicator{\Omega_{n,k,s}}\right]\E_{k}[c_{s}]\notag\\
 	&\quad\quad\quad\quad\quad\quad\quad\quad\quad\quad\quad\quad\times\E_{k}[c_{s}^{*}]\E_{k}\left[(\mathcal{G}_{n}^{(k,s)}(w))^{*}\oindicator{\Omega_{n,k,s}}\right]e_{(i-1)n+k}.\label{Equ:Lem:FinDimDist:Ref1}
 	\end{align}
 	Observe that the value of $\E_{k}[c_{s}]$ depends on whether or not the column $c_{s}$ has been conditioned on. If the column has been conditioned on, then the expectation returns the column itself and otherwise the expectation is zero. Since $\E_{k}[\cdot]$ conditions on the first $k$ columns in each block, we can simplify \eqref{Equ:Lem:FinDimDist:Ref1} to
 	\begin{equation*}
 	\sum_{s=(i-b-1)n+1}^{(i-b-1)n+k-1}e_{(i-1)n+k}^{T}\E_{k}\left[\mathcal{G}_{n}^{(k,s)}(z)\oindicator{\Omega_{n,k,s}}\right]c_{s}c_{s}^{*}\E_{k}\left[(\mathcal{G}_{n}^{(k,s)}(w))^{*}\oindicator{\Omega_{n,k,s}}\right]e_{(i-1)n+k}.
 	\end{equation*}
 	Now, consider $c_{s}c_{s}^{*}$. Based on the block structure of $\blmat{Y}{n}$, if $(i-b-1)n+1\leq s < (i-b-1)n+k$, then we have 
 	\[\E[(c_{s}c_{s}^{*})_{(f,g)}]=\begin{cases}
 	\frac{1}{n} &\text{ if }(i-b-2)n+1\leq f=g\leq (i-b-1)n\\ 
 	0 & \text{ otherwise }\\
 	\end{cases}.\]
 	Therefore, for a fixed term $i$ with $1\leq i\leq m$ and since $(i-b-1)n+1\leq s<(i-b-1)n+k$, we have $\E\left[c_{s}c_{s}^{*}\right]=\frac{1}{n}\mathcal{D}_{i-b-1}.$ We wish now to replace $c_{s}c_{s}^{*}$ with its expectation. Observe that the terms 
 	\begin{align*}
 	&e_{(i-1)n+k}^{T}\E_{k}\left[\mathcal{G}_{n}^{(k,s)}(z)\oindicator{\Omega_{n,k,s}}\right]\left(c_{s}c_{s}^{*}-\frac{1}{n}\mathcal{D}_{i-b-1}\right)\times\E_{k}\left[(\mathcal{G}_{n}^{(k,s)}(w))^{*}\oindicator{\Omega_{n,k,s}}\right]e_{(i-1)n+k}
 	\end{align*}
 	satisfy the conditions of a martingale difference sequence in $s$. Therefore we have 
 	\begin{align*}
 	&\E\left|\frac{1}{n}\sum_{k=1}^{n}\frac{1}{z\overline{w}}\sum_{i=1}^{m}\sum_{s=(i-b-1)n+1}^{(i-b-1)n+k-1}e_{(i-1)n+k}^{T}\E_{k}\left[\mathcal{G}_{n}^{(k,s)}(z)\oindicator{\Omega_{n,k,s}}\right]\right.\\
 	&\quad\quad\quad\quad\quad\quad\quad\quad\quad\quad\quad\left.\times\left(c_{s}c_{s}^{*}-\frac{1}{n}\mathcal{D}_{i-b-1}\right)\E_{k}\left[(\mathcal{G}_{n}^{(k,s)}(w))^{*}\oindicator{\Omega_{n,k,s}}\right]e_{(i-1)n+k}\right|^{2}\\
 	&\ll\max_{\substack{1\leq k\leq n\\	1\leq i\leq m}} \sum_{s=(i-b-1)n+1}^{(i-b-1)n+k-1}\E\left|e_{(i-1)n+k}^{T}\E_{k}\left[\mathcal{G}_{n}^{(k,s)}(z)\oindicator{\Omega_{n,k,s}}\right]\right.\\
 	&\quad\quad\quad\quad\quad\quad\quad\quad\quad\quad\quad\left.\times\left(c_{s}c_{s}^{*}-\frac{1}{n}\mathcal{D}_{i-b-1}\right)\E_{k}\left[(\mathcal{G}_{n}^{(k,s)}(w))^{*}\oindicator{\Omega_{n,k,s}}\right]e_{(i-1)n+k}\right|^{2}.
 	\end{align*}
 	By Lemma \ref{Lem:ReplaceColumnsWithExpectation},
 	\begin{align*} 
 	&\E\left|e_{(i-1)n+k}^{T}\E_{k}\left[\mathcal{G}_{n}^{(k,s)}(z)\oindicator{\Omega_{n,k,s}}\right]\left(c_{s}c_{s}^{*}-\frac{1}{n}\mathcal{D}_{i-b-1}\right)\right.\\
 	&\quad\quad\quad\quad\quad\quad\quad\quad\quad\quad\quad\quad\times\left.\E_{k}\left[(\mathcal{G}_{n}^{(k,s)}(w))^{*}\oindicator{\Omega_{n,k,s}}\right]e_{(i-1)n+k}\right|^{2}=o(n^{-1})
 	\end{align*}
 	uniformly in $i$, $k$, and $s$, and therefore we may proceed with 
 	\[\sum_{s=(i-b-1)n+1}^{(i-b-1)n+k-1}e_{(i-1)n+k}^{T}\E_{k}\left[\mathcal{G}_{n}^{(k,s)}(z)\oindicator{\Omega_{n,k,s}}\right]\frac{1}{n}\mathcal{D}_{i-b-1}\E_{k}\left[(\mathcal{G}_{n}^{(k,s)}(w))^{*}\oindicator{\Omega_{n,k,s}}\right]e_{(i-1)n+k}.\]
 	Next, we wish to add back in column $c_{s}$ to $\mathcal{G}_{n}^{(k,s)}$ and $\Omega_{n,k,s}$. Since the events $\Omega_{n,k}$ and $\Omega_{n,k,s}$ both hold with overwhelming probability, an argument similar to Lemma \ref{Lem:FiniteDimError1} shows that we can insert or drop these events with a sufficiently small $L^{2}$-norm error. This is achieved by showing 
 	\begin{align*}
 	&\E\Bigg|\frac{1}{n}\sum_{s=(i-b-1)n+1}^{(i-b-1)n+k-1}e_{(i-1)n+k}^{T}\left(\E_{k}\left[\mathcal{G}_{n}^{(k,s)}(z)\oindicator{\Omega_{n,k,s}}\right]\mathcal{D}_{i-b-1}\E_{k}\left[(\mathcal{G}_{n}^{(k,s)}(w))^{*}\oindicator{\Omega_{n,k,s}}\right]\right.\\
 	&\quad\quad\quad \quad \left.-\E_{k}\left[\mathcal{G}_{n}^{(k)}(z)\oindicator{\Omega_{n,k}}\right]\mathcal{D}_{i-b-1}\E_{k}\left[(\mathcal{G}_{n}^{(k)}(w))^{*}\oindicator{\Omega_{n,k}}\right]\right)e_{(i-1)n+k}\Bigg|^{2}=o(1),
 	\end{align*}
 	which is done in Lemma \ref{Lem:ReplaceColS}. Since 
 	\begin{align*}
 	&\E\Bigg|\frac{1}{n}\sum_{k=1}^{n}\sum_{i=1}^{m}\frac{1}{n}\sum_{s=(i-b-1)n+1}^{(i-b-1)n+k-1}\left(e_{(i-1)n+k}^{T}\left(\E_{k}\left[\mathcal{G}_{n}^{(k,s)}(z)\oindicator{\Omega_{n,k,s}}\right]\mathcal{D}_{i-b-1}\E_{k}\left[(\mathcal{G}_{n}^{(k,s)}(w))^{*}\oindicator{\Omega_{n,k,s}}\right]\right.\right.\\
 	&\quad\quad\quad \quad \left.\left.-\E_{k}\left[\mathcal{G}_{n}^{(k)}(z)\oindicator{\Omega_{n,k}}\right]\mathcal{D}_{i-b-1}\E_{k}\left[(\mathcal{G}_{n}^{(k)}(w))^{*}\oindicator{\Omega_{n,k}}\right]\right)e_{(i-1)n+k}\right)\Bigg|^{2}=o(1),
 	\end{align*}
 	column $c_{s}$ may be reinserted. With this column replaced, in each term we now have 
 	\begin{align*}
 	&\frac{1}{n}\sum_{s=(i-b-1)n+1}^{(i-b-1)n+k-1}e_{(i-1)n+k}^{T}\E_{k}\left[\mathcal{G}_{n}^{(k)}(z)\oindicator{\Omega_{n,k}}\right]\mathcal{D}_{i-b-1}\\
 	&\quad\quad\quad\quad\quad\quad\quad\quad\quad\quad\quad\quad\quad\times\E_{k}\left[(\mathcal{G}_{n}^{(k)}(w))^{*}\oindicator{\Omega_{n,k}}\right]e_{(i-1)n+k}\\
 	&=\frac{k-1}{n}e_{(i-1)n+k}^{T}\E_{k}\left[\mathcal{G}_{n}^{(k)}(z)\oindicator{\Omega_{n,k}}\right]\mathcal{D}_{i-b-1}\E_{k}\left[(\mathcal{G}_{n}^{(k)}(w))^{*}\oindicator{\Omega_{n,k}}\right]e_{(i-1)n+k}
 	\end{align*}
 	since there were $k-1$ terms in the above sum, and none of them depended on $s$. This concludes the proof of Lemma \ref{Lem:Iteration}.
 	\begin{Details}
 		This proof follows by the same arguments as the first expansion in the proof of Lemma \ref{Lem:MDSReduction}. However, now the index on $\mathcal{D}_{i-b}$ is arbitrary, assuming $b\neq m$. We leave out details which are identical to the arguments in Lemma \ref{Lem:MDSReduction}. Begin by observing that provided the resolvent is defined, we can expand using the resolvent expansion \eqref{Equ:ResolventExpansion} to get  
 		\begin{align}
 		&e_{(i-1)n+k}^{T}\E_{k}\left[\mathcal{G}_{n}^{(k)}(z)\oindicator{\Omega_{n,k}}\right]\mathcal{D}_{i-b}\E_{k}\left[(\mathcal{G}_{n}^{(k)}(w))^{*}\oindicator{\Omega_{n,k}}\right]e_{(i-1)n+k}\notag\\
 		&= e_{(i-1)n+k}^{T}\mathcal{D}_{i-b}\frac{1}{z\bar{w}}(\P_{k}(\Omega_{n,k}))^{2}e_{(i-1)n+k}\label{Equ:IterationsExpansion1}\\
 		&\quad-e_{(i-1)n+k}^{T}\frac{1}{z\overline{w}}\sum_{s\neq * n+k}\mathcal{D}_{i-b}\E_{k}\left[e_{s}c_{s}^{*}(\mathcal{G}_{n}^{(k)}(w))^{*}\oindicator{\Omega_{n,k}}\right]\P_{k}(\Omega_{n,k})e_{(i-1)n+k}\label{Equ:IterationsExpansion2}\\
 		&\quad-e_{(i-1)n+k}^{T}\frac{1}{z\overline{w}}\sum_{t\neq * n+k}\E_{k}\left[\mathcal{G}_{n}^{(k)}(z)c_{t}e_{t}^{T}\oindicator{\Omega_{n,k}}\right]\mathcal{D}_{i-b}\P_{k}(\Omega_{n,k})e_{(i-1)n+k}\label{Equ:IterationsExpansion3}\\
 		&\quad+e_{(i-1)n+k}^{T}\frac{1}{z\bar{w}}\sum_{s\neq * n+ k}\sum_{t\neq * n+k}\E_{k}\left[\mathcal{G}_{n}^{(k)}(z)c_{t}e_{t}^{T}\oindicator{\Omega_{n,k}}\right]\mathcal{D}_{i-b}\notag\\
 		&\quad\quad\quad\quad\quad\quad\quad\quad\quad\quad\quad\times\E_{k}\left[e_{s}c_{s}^{*}(\mathcal{G}_{n}^{(k)}(w))^{*}\oindicator{\Omega_{n,k}}\right]e_{(i-1)n+k}\label{Equ:IterationsExpansion4}
 		\end{align}
 		where the notation $t\neq* n+k$ indicates that the sum is over all $1\leq t\leq mn$ such that $t\neq k, n+k, \dots (m-1)n+k$. We will handle each term individually. For term \eqref{Equ:IterationsExpansion1}, observe that 
 		\begin{equation*}
 		e_{(i-1)n+k}^{T}\mathcal{D}_{i-b}\frac{1}{z\bar{w}}(\P_{k}(\Omega_{n,k}))^{2}e_{(i-1)n+k}=0\\
 		\end{equation*}
 		since the only nonzero elements in $\mathcal{D}_{i-b}$ are in block $i-b$, and the above product selects an element from the $i$th block. This is zero unless $b=m$ since we are reducing mod $m$.
 		
 		For terms \eqref{Equ:IterationsExpansion2} and \eqref{Equ:IterationsExpansion3}, note that since any off diagonal element of $\mathcal{D}_{i-b}$ will be zero,
 		\begin{equation*}
 		e_{(i-1)n+k}^{T}\frac{1}{z\overline{w}}\sum_{s\neq * n+k}\mathcal{D}_{i-b}\E_{k}\left[e_{s}c_{s}^{*}(\mathcal{G}_{n}^{(k)}(w))^{*}\oindicator{\Omega_{n,k}}\right]\P_{k}(\Omega_{n,k})e_{(i-1)n+k}=0
 		\end{equation*}
 		and 
 		\begin{equation*}
 		e_{(i-1)n+k}^{T}\frac{1}{z\overline{w}}\sum_{t\neq * n+k}\E_{k}\left[\mathcal{G}_{n}^{(k)}(z)c_{t}e_{t}^{T}\oindicator{\Omega_{n,k}}\right]\mathcal{D}_{i-b}\P_{k}(\Omega_{n,k})e_{(i-1)n+k}=0.
 		\end{equation*}
 		Therefore, 
 		\begin{align*}
 		&e_{(i-1)n+k}^{T}\E_{k}\left[\mathcal{G}_{n}^{(k)}(z)\oindicator{\Omega_{n,k}}\right]\mathcal{D}_{i-b}\E_{k}\left[(\mathcal{G}_{n}^{(k)}(w))^{*}\oindicator{\Omega_{n,k}}\right]e_{(i-1)n+k}\\
 		&\quad\quad =e_{(i-1)n+k}^{T}\frac{1}{z\bar{w}}\sum_{\substack{s=(i-b-1)n+1\\ s\neq * n+k}}^{(i-b)n}\E_{k}\left[\mathcal{G}_{n}^{(k)}(z)c_{s}\oindicator{\Omega_{n,k}}\right]\\
 		&\quad\quad\quad\quad\quad\quad\quad\quad\quad\quad\quad\quad\quad\quad\quad\times\E_{k}\left[c_{s}^{*}(\mathcal{G}_{n}^{(k)}(w))^{*}\oindicator{\Omega_{n,k}}\right]e_{(i-1)n+k}.
 		\end{align*}
 		Again, from here we wish to remove column $c_{s}$ from $\mathcal{G}_{n}^{(k)}(z)$ and $(\mathcal{G}_{n}^{(k)}(w))^{*}$. By arguments similar to Lemma \ref{Lem:FiniteDimError1}, we may insert events $\Omega_{n,k,s}$ and $Q'_{n,k,s}$ since they both hold with overwhelming probability by Corollary \ref{Cor:Omega_nksOverwhelming} and Lemma \ref{Lem:Q'Overwhelming}. By equations \eqref{Equ:RemovingColS} and \eqref{Equ:RemovingColS2}, we can write
 		\begin{align*}
 		&e_{(i-1)n+k}^{T}\frac{1}{z\bar{w}}\sum_{\substack{s=(i-b-1)n+1\\ s\neq * n+k}}^{(i-b)n}\E_{k}\left[\mathcal{G}_{n}^{(k)}(z)c_{s}\oindicator{\Omega_{n,k}\cap\Omega_{n,k,s}\cap Q'_{n,k,s}}\right]\E_{k}\left[c_{s}^{*}(\mathcal{G}_{n}^{(k)}(w))^{*}\oindicator{\Omega_{n,k}\cap\Omega_{n,k,s}\cap Q'_{n,k,s}}\right]e_{(i-1)n+k}\\
 		&\quad\quad = e_{(i-1)n+k}^{T}\frac{1}{z\bar{w}}\sum_{\substack{s=(i-b-1)n+1\\ s\neq * n+k}}^{(i-b)n}\E_{k}\left[\mathcal{G}_{n}^{(k,s)}(z)c_{s}\delta_{k,s}(z)\oindicator{\Omega_{n,k}\cap\Omega_{n,k,s}\cap Q'_{n,k,s}}\right]\\
 		&\quad\quad\quad\quad\quad\quad\times\E_{k}\left[(\delta_{k,s}(w))^{*}c_{s}^{*}(\mathcal{G}_{n}^{(k,s)}(w))^{*}\oindicator{\Omega_{n,k}\cap\Omega_{n,k,s}\cap Q'_{n,k,s}}\right]e_{(i-1)n+k}.
 		\end{align*}
 		By Lemma \ref{Lem:FiniteDimError2}, we can replace $\delta_{k,s}(z)$ and $(\delta_{k,s}(w))^{*}$ with $1$ with a sufficiently small error, so we can proceed with 
 		\begin{align*}
 		&\quad\quad e_{(i-1)n+k}^{T}\frac{1}{z\bar{w}}\sum_{\substack{s=(i-b-1)n+1\\ s\neq * n+k}}^{(i-b)n}\E_{k}\left[\mathcal{G}_{n}^{(k,s)}(z)c_{s}\oindicator{\Omega_{n,k}\cap\Omega_{n,k,s}\cap Q'_{n,k,s}}\right]\\
 		&\quad\quad\quad\quad\quad\quad\quad\quad\quad\times\E_{k}\left[c_{s}^{*}(\mathcal{G}_{n}^{(k,s)}(w))^{*}\oindicator{\Omega_{n,k}\cap\Omega_{n,k,s}\cap Q'_{n,k,s}}\right]e_{(i-1)n+k}.
 		\end{align*}
 		We can now drop the events $\Omega_{n,k}$ and $Q'_{n,k,s}$ since they both hold with overwhelming probability. Now, by independence we have
 		\begin{align*}
 		&e_{(i-1)n+k}^{T}\frac{1}{z\bar{w}}\sum_{\substack{s=(i-b-1)n+1\\ s\neq * n+k}}^{(i-b)n}\E_{k}\left[\mathcal{G}_{n}^{(k,s)}(z)c_{s}\oindicator{\Omega_{n,k,s}}\right]\\
 		&\quad\quad\quad\quad\quad\quad\quad\quad\quad\quad\quad\quad\times\E_{k}\left[c_{s}^{*}(\mathcal{G}_{n}^{(k,s)}(w))^{*}\oindicator{\Omega_{n,k,s}}\right]e_{(i-1)n+k}\\
 		&\quad\quad=e_{(i-1)n+k}^{T}\frac{1}{z\bar{w}}\sum_{\substack{s=(i-b-1)n+1\\ s\neq * n+k}}^{(i-b)n}\E_{k}\left[\mathcal{G}_{n}^{(k,s)}(z)\oindicator{\Omega_{n,k,s}}\right]\E_{k}[c_{s}]\E_{k}[c_{s}^{*}]\\
 		&\quad\quad\quad\quad\quad\quad\quad\quad\quad\quad\quad\quad\quad\quad\quad\times\E_{k}\left[(\mathcal{G}_{n}^{(k,s)}(w))^{*}\oindicator{\Omega_{n,k,s}}\right]e_{(i-1)n+k}\\
 		&\quad\quad=e_{(i-1)n+k}^{T}\frac{1}{z\bar{w}}\sum_{s=(i-b-1)n+1}^{(i-b)n+k-1}\E_{k}\left[\mathcal{G}_{n}^{(k,s)}(z)\oindicator{\Omega_{n,k,s}}\right]c_{s}c_{s}^{*}\\
 		&\quad\quad\quad\quad\quad\quad\quad\quad\quad\quad\quad\quad\quad\quad\quad\times\E_{k}\left[(\mathcal{G}_{n}^{(k,s)}(w))^{*}\oindicator{\Omega_{n,k,s}}\right]e_{(i-1)n+k}
 		\end{align*}
 		and since the terms in the above sum are a martingale difference sequence in $s$, by Lemma \ref{Lem:ReplaceColumnsWithExpectation} we have
 		\begin{align*}
 		&\E\Bigg|\frac{1}{n}\frac{1}{z\bar{w}}\sum_{k=1}^{n}\sum_{i=1}^{m}\sum_{s=(i-b-1)n+1}^{(i-b)n+k-1}\left(e_{(i-1)n+k}^{T}\E_{k}\left[\mathcal{G}_{n}^{(k,s)}(z)\oindicator{\Omega_{n,k,s}}\right]c_{s}c_{s}^{*}\right.\\
 		&\quad\quad\quad\quad\quad\quad\quad\quad\quad\quad\quad\quad\quad\quad\quad\times\E_{k}\left[(\mathcal{G}_{n}^{(k,s)}(w))^{*}\oindicator{\Omega_{n,k,s}}\right]e_{(i-1)n+k}\\
 		&\quad\quad-\frac{1}{n}e_{(i-1)n+k}^{T}\E_{k}\left[\mathcal{G}_{n}^{(k,s)}(z)\oindicator{\Omega_{n,k,s}}\right]\mathcal{D}_{i-(b+1)}\\
 		&\quad\quad\quad\quad\quad\quad\quad\quad\quad\quad\quad\quad\quad\quad\quad\left.\times\E_{k}\left[(\mathcal{G}_{n}^{(k,s)}(w))^{*}\oindicator{\Omega_{n,k,s}}\right]e_{(i-1)n+k}\right)\Bigg|^{2}=o(1).
 		\end{align*}
 		We may reinsert the event $\Omega_{n,k}$ since it holds with overwhelming probability, and thus by Lemma \ref{Lem:ReplaceColS} we can replace column $c_{s}$ and we have
 		\begin{align*}
 		&\E\Bigg|\frac{1}{n^{2}}\frac{1}{z\bar{w}}\sum_{k=1}^{n}\sum_{i=1}^{m}\sum_{s=(i-b-1)n+1}^{(i-b)n+k-1}\left(e_{(i-1)n+k}^{T}\E_{k}\left[\mathcal{G}_{n}^{(k,s)}(z)\oindicator{\Omega_{n,k}\cap\Omega_{n,k,s}}\right]\mathcal{D}_{i-(b+1)}\right.\\
 		&\quad\quad\quad\quad\quad\quad\quad\quad\quad\quad\quad\quad\quad\times\E_{k}\left[(\mathcal{G}_{n}^{(k,s)}(w))^{*}\oindicator{\Omega_{n,k}\cap\Omega_{n,k,s}}\right]e_{(i-1)n+k}\\
 		&\quad\quad-e_{(i-1)n+k}^{T}\E_{k}\left[\mathcal{G}_{n}^{(k)}(z)\oindicator{\Omega_{n,k}\cap\Omega_{n,k,s}}\right]\mathcal{D}_{i-(b+1)}\\
 		&\quad\quad\quad\quad\quad\quad\quad\quad\quad\quad\quad\quad\quad\left.\times\E_{k}\left[(\mathcal{G}_{n}^{(k)}(w))^{*}\oindicator{\Omega_{n,k}\cap\Omega_{n,k,s}}\right]e_{(i-1)n+k}\right)\Bigg|^{2}=o(1).
 		\end{align*}
 		Dropping the event $\Omega_{n,k,s}$ is justified since it holds with overwhelming probability, and after doing so, we see that the terms 
 		\[e_{(i-1)n+k}^{T}\E_{k}\left[\mathcal{G}_{n}^{(k)}(z)\oindicator{\Omega_{n,k}}\right]\mathcal{D}_{i-(b+1)}\E_{k}\left[(\mathcal{G}_{n}^{(k)}(w))^{*}\oindicator{\Omega_{n,k}}\right]e_{(i-1)n+k}\]
 		no longer depend on $s$. Summing over $s$ results in a factor of $k-1$, which concludes the proof of the lemma.
 		By Corollary \ref{Cor:Omega_nkOverwhelming}, we have 
 		\begin{align*}
 		&e_{(i-1)n+k}^{T}\frac{1}{z\bar{w}}\sum_{\substack{s=(i-b-1)n+1\\ s\neq jn+k}}^{(i-b)n}\E_{k}\left[\mathcal{G}_{n}^{(k)}(z)c_{s}\oindicator{\Omega_{n,k}}\right]\\
 		&\quad\quad\quad\quad\quad\quad\quad\quad\quad\quad\quad\quad\times\E_{k}\left[c_{s}^{*}(\mathcal{G}_{n}^{(k)}(w))^{*}\oindicator{\Omega_{n,k}}\right]e_{(i-1)n+k}\\
 		&\quad\quad = e_{(i-1)n+k}^{T}\frac{1}{z\bar{w}}\sum_{\substack{s=(i-b-1)n+1\\ s\neq jn+k}}^{(i-b)n}\E_{k}\left[\mathcal{G}_{n}^{(k,s)}(z)c_{s}\delta_{k,s}(z)\oindicator{\Omega_{n,k}\cap\Omega_{n,k,s}\cap Q'_{n,k,s}}\right]\\
 		&\quad\quad\quad\quad\quad\quad\quad\quad\quad\quad\times\E_{k}\left[(\delta_{k,s}(w))^{*}c_{s}^{*}(\mathcal{G}_{n}^{(k,s)}(w))^{*}\oindicator{\Omega_{n,k}\cap\Omega_{n,k,s}\cap Q'_{n,k,s}}\right]e_{(i-1)n+k}+o(n^{-\alpha})\\
 		&\quad\quad = e_{(i-1)n+k}^{T}\frac{1}{z\bar{w}}\sum_{\substack{s=(i-b-1)n+1\\ s\neq jn+k}}^{(i-b)n}\E_{k}\left[\mathcal{G}_{n}^{(k,s)}(z)c_{s}\oindicator{\Omega_{n,k,s}\cap Q'_{n,k,s}}\right]\\
 		&\quad\quad\quad\quad\quad\quad\quad\quad\quad\times\E_{k}\left[c_{s}^{*}(\mathcal{G}_{n}^{(k,s)}(w))^{*}\oindicator{\Omega_{n,k,s}\cap Q'_{n,k,s}}\right]e_{(i-1)n+k}+o(n^{-1})\\
 		&\quad\quad = e_{(i-1)n+k}^{T}\frac{1}{z\bar{w}}\sum_{\substack{s=(i-b-1)n+1\\ s\neq jn+k}}^{(i-b)n}\E_{k}\left[\mathcal{G}_{n}^{(k,s)}(z)c_{s}\oindicator{\Omega_{n,k,s}}\right]\\
 		&\quad\quad\quad\quad\quad\quad\quad\quad\quad\times\E_{k}\left[c_{s}^{*}(\mathcal{G}_{n}^{(k,s)}(w))^{*}\oindicator{\Omega_{n,k,s}}\right]e_{(i-1)n+k}+o(n^{-1})
 		\end{align*}
 		with probability $1-o(1)$. Now, by independence and Lemma \ref{Lem:ReplaceColumnsWithExpectation}, we have
 		\begin{align*}
 		&e_{(i-1)n+k}^{T}\frac{1}{z\bar{w}}\sum_{\substack{s=(i-b-1)n+1\\ s\neq jn+k}}^{(i-b)n}\E_{k}\left[\mathcal{G}_{n}^{(k,s)}(z)c_{s}\oindicator{\Omega_{n,k,s}}\right]\\
 		&\quad\quad\quad\quad\quad\quad\quad\quad\quad\quad\quad\quad\times\E_{k}\left[c_{s}^{*}(\mathcal{G}_{n}^{(k,s)}(w))^{*}\oindicator{\Omega_{n,k,s}}\right]e_{(i-1)n+k}\\
 		&\quad\quad=e_{(i-1)n+k}^{T}\frac{1}{z\bar{w}}\sum_{\substack{s=(i-b-1)n+1\\ s\neq jn+k}}^{(i-b)n}\E_{k}\left[\mathcal{G}_{n}^{(k,s)}(z)\oindicator{\Omega_{n,k,s}}\right]\E_{k}[c_{s}]\E_{k}[c_{s}^{*}]\\
 		&\quad\quad\quad\quad\quad\quad\quad\quad\quad\quad\quad\quad\quad\quad\quad\times\E_{k}\left[(\mathcal{G}_{n}^{(k,s)}(w))^{*}\oindicator{\Omega_{n,k,s}}\right]e_{(i-1)n+k}\\
 		&\quad\quad=\frac{1}{n}e_{(i-1)n+k}^{T}\frac{1}{z\bar{w}}\sum_{s=(i-b-1)n+1 }^{(i-b-1)n+k-1}\E_{k}\left[\mathcal{G}_{n}^{(k,s)}(z)\oindicator{\Omega_{n,k,s}}\right]\mathcal{D}_{i-b-1}\\
 		&\quad\quad\quad\quad\quad\quad\quad\quad\quad\quad\quad\quad\quad\quad\quad\quad\times\E_{k}\left[(\mathcal{G}_{n}^{(k,s)}(w))^{*}\oindicator{\Omega_{n,k,s}}\right]e_{(i-1)n+k}+o(1)
 		\end{align*}
 		with probability $1-o(1)$. By Lemma \ref{Lem:ReplaceColS} we have
 		\begin{align*}
 		&\frac{1}{n}e_{(i-1)n+k}^{T}\frac{1}{z\bar{w}}\sum_{s=(i-b-1)n+1 }^{(i-b-1)n+k-1}\E_{k}\left[\mathcal{G}_{n}^{(k,s)}(z)\oindicator{\Omega_{n,k,s}}\right]\mathcal{D}_{i-b-1}\\
 		&\quad\quad\quad\quad\quad\quad\quad\quad\quad\quad\quad\quad\quad\times\E_{k}\left[(\mathcal{G}_{n}^{(k,s)}(w))^{*}\oindicator{\Omega_{n,k,s}}\right]e_{(i-1)n+k}\\
 		&\quad\quad=\frac{k-1}{n}e_{(i-1)n+k}^{T}\frac{1}{z\bar{w}}\E_{k}\left[\mathcal{G}_{n}^{(k)}(z)\oindicator{\Omega_{n,k}}\right]\mathcal{D}_{i-b-1}\\
 		&\quad\quad\quad\quad\quad\quad\quad\quad\quad\quad\quad\quad\quad\times\E_{k}\left[(\mathcal{G}_{n}^{(k)}(w))^{*}\oindicator{\Omega_{n,k}}\right]e_{(i-1)n+k}+o(1)
 		\end{align*}
 		with probability $1-o(1)$. This concludes the proof of the lemma.
 	\end{Details}
 \end{proof}
 
 With the proof of Lemma \ref{Lem:Iteration} complete, we continue with the proof of Lemma \ref{Lem:MDSReduction}. Applying Lemma \ref{Lem:Iteration} in the base when $b=1$ gives 
	 \begin{align*}
	 &\E\left|\frac{1}{n}\sum_{k=1}^{n}\Bigg(\mathcal{T}_{n,k}(z,\overline{w})-\frac{1}{z\bar{w}}\frac{k-1}{n}\sum_{i=1}^{m}e_{(i-1)n+k}^{T}\E_{k}\left[\mathcal{G}_{n}^{(k)}(z)\oindicator{\Omega_{n,k}}\right]\mathcal{D}_{i-2}\right.\\
	 &\left.\quad\quad\quad\quad\quad\quad\quad\quad\quad\quad\quad\quad\quad\quad\quad\quad\times\E_{k}\left[(\mathcal{G}_{n}^{(k)}(w))^{*}\oindicator{\Omega_{n,k}}\right]e_{(i-1)n+k}\Bigg)\right|^{2}=o(1).
	 \end{align*}
	 The goal is to iterate this process until $\mathcal{T}_{n,k}(z,\overline{w})$ reappears. Lemma \ref{Lem:Iteration} verifies that, for any $b\neq m$,  
	 \begin{align*}
	 &\E\left|\frac{1}{n}\sum_{k=1}^{n}\sum_{i=1}^{m}\Bigg(e_{(i-1)n+k}^{T}\E_{k}\left[\mathcal{G}_{n}^{(k)}(z)\oindicator{\Omega_{n,k}}\right]\mathcal{D}_{i-b}\E_{k}\left[(\mathcal{G}_{n}^{(k)}(w))^{*}\oindicator{\Omega_{n,k}}\right]e_{(i-1)n+k}\right.\\
	 &\quad\quad-\frac{1}{z\bar{w}}\frac{k-1}{n}e_{(i-1)n+k}^{T}\E_{k}\left[\mathcal{G}_{n}^{(k)}(z)\oindicator{\Omega_{n,k}}\right]\mathcal{D}_{i-(b+1)}\\
	 &\quad\quad\quad\quad\quad\quad\quad\quad\quad\quad\quad\quad\times\left.\E_{k}\left[(\mathcal{G}_{n}^{(k)}(w))^{*}\oindicator{\Omega_{n,k}}\right]e_{(i-1)n+k}\Bigg)\right|^{2}=o(1)
	 \end{align*} 
	 where $i-b$ is reduced modulo $m$. After iterating twice, we have  

	 \begin{align*}
	 &\E\left|\frac{1}{n}\sum_{k=1}^{n}\Bigg(\mathcal{T}_{n,k}(z,\overline{w}) -\left(\frac{1}{z\bar{w}}\frac{k-1}{n}\right)^{2}\sum_{i=1}^{m} e_{(i-1)n+k}^{T}\E_{k}\left[\mathcal{G}_{n}^{(k)}(z)\oindicator{\Omega_{n,k}}\right]\mathcal{D}_{i-3}\right.\\
	 &\quad\quad\quad\quad\quad\quad\quad\quad\quad\quad\quad\quad\quad\quad\quad\times\left.\E_{k}\left[(\mathcal{G}_{n}^{(k)}(w))^{*}\oindicator{\Omega_{n,k}}\right]e_{(i-1)n+k}\Bigg)\right|^{2}=o(1).
	 \end{align*}
	 After iterating $m-1$ times, we have 
	 \begin{align*}
	 &\E\left|\frac{1}{n}\sum_{k=1}^{n}\Bigg(\mathcal{T}_{n,k}(z,\overline{w}) -\left(\frac{1}{z\bar{w}}\frac{k-1}{n}\right)^{m-1}\sum_{i=1}^{m} e_{(i-1)n+k}^{T}\E_{k}\left[\mathcal{G}_{n}^{(k)}(z)\oindicator{\Omega_{n,k}}\right]\mathcal{D}_{i-m}\right.\\
	 &\quad\quad\quad\quad\quad\quad\quad\quad\quad\quad\quad\quad\quad\quad\quad\times\left.\E_{k}\left[(\mathcal{G}_{n}^{(k)}(w))^{*}\oindicator{\Omega_{n,k}}\right]e_{(i-1)n+k}\Bigg)\right|^{2}=o(1).
	 \end{align*}
	 To recover $\mathcal{T}_{n,k}(z,\overline{w})$, we will iterate one final time. Due to the block structure of $\mathcal{D}_{p}$ for $1\leq p\leq m$, the $m$th iteration will result in less cancellation than in previous iterations. Consider the expansion due to \eqref{Equ:ResolventExpansion},
	 \begin{align*}
	 &e_{(i-1)n+k}^{T}\E_{k}\left[\mathcal{G}_{n}^{(k)}(z)\oindicator{\Omega_{n,k}}\right]\mathcal{D}_{i-m}\E_{k}\left[(\mathcal{G}_{n}^{(k)}(w))^{*}\oindicator{\Omega_{n,k}}\right]e_{(i-1)n+k}\\
	 &=e_{(i-1)n+k}^{T}\mathcal{D}_{i-m}\frac{1}{z\bar{w}}(\P_{k}(\Omega_{n,k}))^{2}e_{(i-1)n+k}\\
	 &\quad-e_{(i-1)n+k}^{T}\frac{1}{z\overline{w}}\sum_{s\neq * n+k}\mathcal{D}_{i-m}\E_{k}\left[e_{s}c_{s}^{*}(\mathcal{G}_{n}^{(k)}(w))^{*}\oindicator{\Omega_{n,k}}\right]\P_{k}(\Omega_{n,k})e_{(i-1)n+k}\\
	 &\quad-e_{(i-1)n+k}^{T}\frac{1}{z\overline{w}}\sum_{t\neq * n+k}\E_{k}\left[\mathcal{G}_{n}^{(k)}(z)c_{t}e_{t}^{T}\oindicator{\Omega_{n,k}}\right]\mathcal{D}_{i-m}\P_{k}(\Omega_{n,k})e_{(i-1)n+k}\\
	 &\quad+e_{(i-1)n+k}^{T}\frac{1}{z\bar{w}}\sum_{s\neq * n+k}\sum_{t\neq * n+k}\E_{k}\left[\mathcal{G}_{n}^{(k)}(z)c_{t}e_{t}^{T}\oindicator{\Omega_{n,k}}\right]\mathcal{D}_{i-m}\\
	 &\quad\quad\quad\quad\quad\quad\quad\quad\quad\quad\quad\quad\quad\quad\quad\quad\quad\times\E_{k}\left[e_{s}c_{s}^{*}(\mathcal{G}_{n}^{(k)}(w))^{*}\oindicator{\Omega_{n,k}}\right]e_{(i-1)n+k}
	 \end{align*}
	 where the notation $t\neq * n+k$ indicates that the sum is over all $1\leq t\leq mn$ such that $t\neq k, n+k, \dots, (m-1)n+k$ and $\mathbb{P}_{k}$ denotes the conditional probability on the $\sigma$-algebra $\mathcal{F}_{k}$. We have 
	 \[e_{(i-1)n+k}^{T}\mathcal{D}_{i-m}\frac{1}{z\bar{w}}(\P_{k}(\Omega_{n,k}))^{2}e_{(i-1)n+k}=(\mathcal{D}_{i-m})_{((i-1)n+k,(i-1)n+k)}\frac{1}{z\bar{w}}(\P_{k}(\Omega_{n,k}))^{2}.\]
	 Note that after reducing modulo $m$, $\mathcal{D}_{i-m}$ is nonzero in the $(i-1)$st block. Therefore 
	 \begin{equation*}
	 e_{(i-1)n+k}^{T}\mathcal{D}_{i-m}\frac{1}{z\bar{w}}(\P_{k}(\Omega_{n,k}))^{2}e_{(i-1)n+k}=\frac{1}{z\bar{w}}(\P_{k}(\Omega_{n,k}))^{2}.
	 \end{equation*}
	 By arguments similar to those used in the proof of Lemma \ref{Lem:Iteration}, we can calculate that 
	 \[e_{(i-1)n+k}^{T}\frac{1}{z\overline{w}}\sum_{s\neq * n+k}\mathcal{D}_{i-m}\E_{k}\left[e_{s}c_{s}^{*}(\mathcal{G}_{n}^{(k)}(w))^{*}\oindicator{\Omega_{n,k}}\right]\P_{k}(\Omega_{n,k})e_{(i-1)n+k}=0\]
	 and
	 \[e_{(i-1)n+k}^{T}\frac{1}{z\overline{w}}\sum_{t\neq * n+k}\E_{k}\left[\mathcal{G}_{n}^{(k)}(z)c_{t}e_{t}^{T}\oindicator{\Omega_{n,k}}\right]\mathcal{D}_{i-m}\P_{k}(\Omega_{n,k})e_{(i-1)n+k}=0.\]
	 By the forthcoming Lemma \ref{Lem:FiniteDimError3}, we can see that $\E\left|(z\bar{w})^{-1}(1-\P_{k}(\Omega_{n,k})^{2})\right|=o_{\alpha}(n^{-\alpha})$ for any $\alpha>0$, and therefore we have 
	 \begin{align*}
	 &\E\left|e_{(i-1)n+k}^{T}\E_{k}\left[\mathcal{G}_{n}^{(k)}(z)\oindicator{\Omega_{n,k}}\right]\mathcal{D}_{i-m}\E_{k}\left[(\mathcal{G}_{n}^{(k)}(w))^{*}\oindicator{\Omega_{n,k}}\right]e_{(i-1)n+k}\right.\\
	 &\quad\quad\quad\quad-\frac{1}{z\bar{w}}-\frac{1}{z\bar{w}}e_{(i-1)n+k}^{T}\sum_{\substack{s= (i-1)n+1,\\ s\neq * n+k}}^{in}\E_{k}\left[\mathcal{G}_{n}^{(k)}(z)c_{s}\oindicator{\Omega_{n,k}}\right]\\
	 &\quad\quad\quad\quad\quad\quad\quad\quad\quad\quad\quad\quad\quad\quad\quad\quad\left.\times\E_{k}\left[c_{s}^{*}(\mathcal{G}_{n}^{(k)}(w))^{*}\oindicator{\Omega_{n,k}}\right]e_{(i-1)n+k}\right|^{2}=o_{\alpha}(n^{-\alpha}).
	 \end{align*}
	 \begin{Details}
	 \begin{align*}
	 &e_{(i-1)n+k}^{T}\E_{k}\left[\mathcal{G}_{n}^{(k)}(z)\oindicator{\Omega_{n,k}}\right]\mathcal{D}_{i-m}\E_{k}\left[(\mathcal{G}_{n}^{(k)}(w))^{*}\oindicator{\Omega_{n,k}}\right]e_{(i-1)n+k}\\
	 &\quad=\frac{1}{z\bar{w}}(\P_{k}(\Omega_{n,k}))^{2}+e_{(i-1)n+k}^{T}\frac{1}{z\bar{w}}\sum_{s\neq jn+k}\sum_{t\neq jn+k}\E_{k}\left[\mathcal{G}_{n}^{(k)}(z)c_{t}e_{t}^{T}\oindicator{\Omega_{n,k}}\right]\mathcal{D}_{i-m}\\
	 &\quad\quad\quad\quad\quad\quad\quad\quad\quad\quad\quad\quad\quad\quad\quad\quad\quad\quad\quad\times\E_{k}\left[e_{s}c_{s}^{*}(\mathcal{G}_{n}^{(k)}(w))^{*}\oindicator{\Omega_{n,k}}\right]e_{(i-1)n+k}\\
	 &\quad=\frac{1}{z\bar{w}}+\frac{1}{z\bar{w}}e_{(i-1)n+k}^{T}\sum_{\substack{s= (i-1)n+1,\\ s\neq jn+k}}^{(i-m)n}\E_{k}\left[\mathcal{G}_{n}^{(k)}(z)c_{s}\oindicator{\Omega_{n,k}}\right]\\
	 &\quad\quad\quad\quad\quad\quad\quad\quad\quad\quad\quad\quad\quad\quad\quad\quad\times\E_{k}\left[c_{s}^{*}(\mathcal{G}_{n}^{(k)}(w))^{*}\oindicator{\Omega_{n,k}}\right]e_{(i-1)n+k}+o(1)
	 \end{align*}
	\end{Details}
	By the same argument as in Lemma \ref{Lem:Iteration}, we have 
	 \begin{align*}
	 &\E\Bigg|\frac{1}{n}\sum_{k=1}^{n}\sum_{i=1}^{m}\left(e_{(i-1)n+k}^{T}\E_{k}\left[\mathcal{G}_{n}^{(k)}(z)\oindicator{\Omega_{n,k}}\right]\mathcal{D}_{i-m}\E_{k}\left[(\mathcal{G}_{n}^{(k)}(w))^{*}\oindicator{\Omega_{n,k}}\right]e_{(i-1)n+k}\right.\\
	 &\quad\quad\quad-\frac{1}{z\bar{w}}-\frac{1}{z\bar{w}}\frac{k-1}{n}e_{(i-1)n+k}^{T}\E_{k}\left[\mathcal{G}_{n}^{(k)}(z)\oindicator{\Omega_{n,k}}\right]\mathcal{D}_{i-1}\\
	 &\quad\quad\quad\quad\quad\quad\quad\quad\quad\quad\quad\quad\quad\quad\quad\left.\times\E_{k}\left[(\mathcal{G}_{n}^{(k)}(w))^{*}\oindicator{\Omega_{n,k}}\right]e_{(i-1)n+k}\right)\Bigg|^{2}=o(1).
	\end{align*}
	\begin{Details}
	 \begin{align*}
	 &\frac{1}{z\bar{w}}+\frac{1}{z\bar{w}}e_{(i-1)n+k}^{T}\sum_{\substack{s= (i-1)n+1,\\ s\neq jn+k}}^{(i-m)n}\E_{k}\left[\mathcal{G}_{n}^{(k)}(z)c_{s}\oindicator{\Omega_{n,k}}\right]\\
	 &\quad\quad\quad\quad\quad\quad\quad\quad\quad\quad\quad\quad\quad\quad\times \E_{k}\left[c_{s}^{*}(\mathcal{G}_{n}^{(k)}(w))^{*}\oindicator{\Omega_{n,k}}\right]e_{(i-1)n+k}+o(1)\\
	 &=\frac{1}{z\bar{w}}+\frac{1}{z\bar{w}}e_{(i-1)n+k}^{T}\sum_{\substack{s= (i-1)n+1,\\ s\neq jn+k}}^{(i-m)n}\E_{k}\left[\mathcal{G}_{n}^{(k,s)}(z)c_{s}\oindicator{\Omega_{n,k,s}}\right]\\
	 &\quad\quad\quad\quad\quad\quad\quad\quad\quad\quad\quad\quad\quad\quad\quad\times \E_{k}\left[c_{s}^{*}(\mathcal{G}_{n}^{(k,s)}(w))^{*}\oindicator{\Omega_{n,k,s}}\right]e_{(i-1)n+k}+o(1)\\
	 &=\frac{1}{z\bar{w}}+\frac{1}{z\bar{w}}e_{(i-1)n+k}^{T}\sum_{\substack{s= (i-1)n+1,\\ s\neq jn+k}}^{(i-m)n}\E_{k}\left[\mathcal{G}_{n}^{(k,s)}(z)\oindicator{\Omega_{n,k,s}}\right]\E_{k}\left[c_{s}\right]\\
	 &\quad\quad\quad\quad\quad\quad\quad\quad\quad\quad\quad\quad\quad\quad\quad\times\E_{k}\left[c_{s}^{*}\right]\E_{k}\left[(\mathcal{G}_{n}^{(k,s)}(w))^{*}\oindicator{\Omega_{n,k,s}}\right]e_{(i-1)n+k}+o(1)\\
	 &=\frac{1}{z\bar{w}}+\frac{1}{z\bar{w}}e_{(i-1)n+k}^{T}\sum_{s= (i-1)n+1}^{(i-1)n+k-1}\E_{k}\left[\mathcal{G}_{n}^{(k,s)}(z)\oindicator{\Omega_{n,k,s}}\right]c_{s}\\
	 &\quad\quad\quad\quad\quad\quad\quad\quad\quad\quad\quad\quad\quad\quad\quad\times c_{s}^{*}\E_{k}\left[(\mathcal{G}_{n}^{(k,s)}(w))^{*}\oindicator{\Omega_{n,k,s}}\right]e_{(i-1)n+k}+o(1)\\
	 &=\frac{1}{z\bar{w}}+\frac{1}{z\bar{w}}\frac{1}{n}e_{(i-1)n+k}^{T}\sum_{s= (i-1)n+1}^{(i-1)n+k-1}\E_{k}\left[\mathcal{G}_{n}^{(k,s)}(z)\oindicator{\Omega_{n,k,s}}\right]\mathcal{D}_{i-1}\\
	 &\quad\quad\quad\quad\quad\quad\quad\quad\quad\quad\quad\quad\quad\quad\quad\times\E_{k}\left[(\mathcal{G}_{n}^{(k,s)}(w))^{*}\oindicator{\Omega_{n,k,s}}\right]e_{(i-1)n+k}+o(1)\\
	 &=\frac{1}{z\bar{w}}+\frac{1}{z\bar{w}}\frac{1}{n}e_{(i-1)n+k}^{T}\sum_{s= (i-1)n+1}^{(i-1)n+k-1}\E_{k}\left[\mathcal{G}_{n}^{(k)}(z)\oindicator{\Omega_{n,k}}\right]\mathcal{D}_{i-1}\\
	 &\quad\quad\quad\quad\quad\quad\quad\quad\quad\quad\quad\quad\quad\quad\quad\times\E_{k}\left[(\mathcal{G}_{n}^{(k)}(w))^{*}\oindicator{\Omega_{n,k}}\right]e_{(i-1)n+k}+o(1)\\
	 &=\frac{1}{z\bar{w}}+\frac{1}{z\bar{w}}\frac{k-1}{n}e_{(i-1)n+k}^{T}\E_{k}\left[\mathcal{G}_{n}^{(k)}(z)\oindicator{\Omega_{n,k}}\right]\mathcal{D}_{i-1}\\
	 &\quad\quad\quad\quad\quad\quad\quad\quad\quad\quad\quad\quad\quad\quad\quad\times\E_{k}\left[(\mathcal{G}_{n}^{(k)}(w))^{*}\oindicator{\Omega_{n,k}}\right]e_{(i-1)n+k}+o(1)
	 \end{align*}
	 \end{Details}
	 By recognizing that we have recovered $\mathcal{T}_{n,k}(z,\bar{w})$ in the previous expression, and by putting this together with the previous iterations of the process, we in total get 
	 \begin{equation}
	 \E\left|\frac{1}{n}\sum_{k=1}^{n}\left(\mathcal{T}_{n,k}(z,\overline{w})-\frac{m}{(z\bar{w})^{m}}\left(\frac{k-1}{n}\right)^{m-1}-\left(\frac{1}{z\bar{w}}\frac{k-1}{n}\right)^{m}\mathcal{T}_{n,k}(z,\bar{w})\right)\right|^{2}=o(1).\label{Equ:TnkInTotal}
	 \end{equation}
	 The goal now is to regroup in order to compare the object of study, $\frac{1}{n}\sum_{k=1}^{n}\mathcal{T}_{n,k}(z,\bar{w})$, to an appropriate Riemann sum. The sum we will compare to is
	 \begin{equation}
	 \frac{1}{n}\sum_{k=1}^{n}\frac{m(k-1)^{m-1}}{n^{m-1}}\left(\left(z\bar{w}\right)^{m}-\left(\frac{k-1}{n}\right)^{m}\right)^{-1}.
	 \label{Equ:RiemannSumForConv}
	 \end{equation}
	 As $n\rightarrow \infty$, \eqref{Equ:RiemannSumForConv} is the Riemann sum for $\int_{0}^{1} mx^{m-1}\left((z\bar{w})^{m}-x^{m}\right)^{-1}dx$  which, by a substitution of variables, is equal to $-\ln\left(1-\frac{1}{(z\bar{w})^{m}}\right).$
	 \begin{Details}
	 \begin{align*}
	 -\ln\left((z\bar{w})^{m}-1\right)+\ln\left((z\bar{w})^{m}\right)&=-\ln\left(1-\frac{1}{(z\bar{w})^{m}}\right).
	 \end{align*}
	\end{Details}
	 By regrouping the quantities inside the sum in \eqref{Equ:TnkInTotal}, we can write 
	 \begin{equation*}
	 \mathcal{T}_{n,k}(z,\bar{w})\left(1-\frac{(k-1)^{m}}{n^{m}(z\bar{w})^{m}}\right)=\frac{m}{(z\bar{w})^{m}}\left(\frac{k-1}{n}\right)^{m-1}+\mathcal{E}_{n,k}(z,\bar{w})
	 \end{equation*}
	 where $\mathcal{E}_{n,k}(z,\bar{w})$ is an error term which satisfies $\E\left|\frac{1}{n}\sum_{k=1}^{n}\mathcal{E}_{n,k}(z,\bar{w})\right|^{2}=o(1)$. This implies
	 \begin{equation*}
	 \mathcal{T}_{n,k}(z,\bar{w})\left(\frac{n^{m}(z\bar{w})^{m}-(k-1)^{m}}{n^{m}(z\bar{w})^{m}}\right)=\frac{m}{(z\bar{w})^{m}}\left(\frac{k-1}{n}\right)^{m-1}+\mathcal{E}_{n,k}(z,\bar{w})
	 \end{equation*}
	 and thus
	 \begin{equation}
	 \mathcal{T}_{n,k}(z,\bar{w})=\frac{mn(k-1)^{m-1}}{n^{m}(z\bar{w})^{m}-(k-1)^{m}}+\left(\frac{n^{m}(z\bar{w})^{m}}{n^{m}(z\bar{w})^{m}-(k-1)^{m}}\right)\times \mathcal{E}_{n,k}(z,\bar{w}).\label{Equ:RearrangementTerm}
	 \end{equation}
	 \begin{Details} 
	 \begin{align*}
	 \mathcal{T}_{n,k}(z,\bar{w})&=\frac{mn(k-1)^{m-1}}{n^{m}(z\bar{w})^{m}-(k-1)^{m}}+\left(\frac{n^{m}(z\bar{w})^{m}}{n^{m}(z\bar{w})^{m}-(k-1)^{m}}\right)\times \mathcal{E}_{n,k}(z,\bar{w}). 
	 \end{align*}
	\end{Details}
	 We are now ready to compare $\frac{1}{n}\sum_{k=1}^{n}\mathcal{T}_{n,k}$ to the Riemann sum in \eqref{Equ:RiemannSumForConv}. By rearranging \eqref{Equ:RearrangementTerm}, we have 
	 \begin{align*}
	 &\E\left|\frac{1}{n}\sum_{k=1}^{n}\mathcal{T}_{n,k}(z\bar{w})-\frac{1}{n}\sum_{k=1}^{n}\frac{m(k-1)^{m-1}}{n^{m-1}}\left(\left(z\bar{w}\right)^{m}-\left(\frac{k-1}{n}\right)^{m}\right)^{-1}\right|^{2}\\
	 &\ll \E\left|\frac{1}{n}\sum_{k=1}^{n} \mathcal{E}_{n,k}(z,\bar{w})\right|^{2}\\
	 &=o(1).
	 \end{align*}
	 Therefore $\frac{1}{n}\sum_{k=1}^{n}\mathcal{T}_{n,k}(z,\bar{w})$ converges to $-\ln\left(1-\frac{1}{(z\bar{w})^{m}}\right)$ in probability as $n\rightarrow\infty$ as claimed in \eqref{Equ:MDSReductionAntiderivative}. 
	 \begin{Details} 
	 	Because we applied Vitali's theorem twice, we only need to take a derivative with respect to $z$ then with respect to $\bar{w}$. Note that 
	 \begin{equation*}
	 \frac{d^{2}}{d\bar{w}dz}\left[-\ln\left(1-\frac{1}{(z\bar{w})^{m}}\right)\right]=\frac{m^{2}(z\bar{w})^{m-1}}{((z\bar{w})^{m}-1)^{2}}.
	 \end{equation*}
	\end{Details}
	 \begin{Details}
	 \begin{equation*}
	 \frac{d}{dz}\left[-\ln\left(1-\frac{1}{(z\bar{w})^{m}}\right)\right]=-m\left(z^{m+1}\bar{w}^{m}-z\right)^{-1}
	 \end{equation*}
	 and 
	 \begin{equation*}
	 \frac{d}{d\bar{w}}\left[-m\left(z^{m+1}\bar{w}^{m}-z\right)^{-1}\right]=\frac{m^{2}(z\bar{w})^{m-1}}{((z\bar{w})^{m}-1)^{2}}.
	 \end{equation*}
	 Taking a derivative with respect to $z$ then $\bar{w}$ shows that
	 \begin{equation*}
	 \sum_{i,j=1}^{L}\sum_{k=1}^{n}\alpha_{i}\beta_{j}\E_{k}\left[\breve{Z}_{n,k}(z)\overline{\breve{Z}_{n,k}(w)}\right]\rightarrow \alpha_{i}\beta_{j}\frac{m^{2}(z\bar{w})^{m-1}}{((z\bar{w})^{m}-1)^{2}}
	 \end{equation*}
	 in probability as $n\rightarrow\infty.$ 
	\end{Details}
	 By recalling that we invoked Vitali's theorem, this implies 
	 \begin{equation*}
	 \sum_{k=1}^{n}\alpha_{i}\beta_{j}\E_{k}\left[\breve{Z}_{n,k}(z)\overline{\breve{Z}_{n,k}(w)}\right]\rightarrow \alpha_{i}\beta_{j}\frac{m^{2}(z\bar{w})^{m-1}}{((z\bar{w})^{m}-1)^{2}}
	 \end{equation*}
	 in probability as $n\rightarrow\infty.$ This concludes the proof of Lemma \ref{Lem:MDSReduction}. 
	 \begin{Details} 
	 \begin{equation*}
	 \sum_{i,j=1}^{L}\sum_{k=1}^{n}\alpha_{i}\alpha_{j}\E_{k}\left[\breve{Z}_{n,k}(z)\breve{Z}_{n,k}(w)\right]\rightarrow \alpha_{i}\alpha_{j}\frac{m^{2}(zw)^{m-1}}{((zw)^{m}-1)^{2}}.
	 \end{equation*}
	 in probability as $n\rightarrow\infty$, which concludes the arguments for terms \eqref{Equ:VarExpansion1} and \eqref{Equ:VarExpansion4}.
	 \end{Details}
	 \begin{Details} 
	  we have 
	 \begin{equation*}
	 \sum_{i,j=1}^{L}\sum_{k=1}^{n}\beta_{i}\beta_{j}\E_{k}\left[\overline{\breve{Z}_{n,k}(z_{i})}\overline{\breve{Z}_{n,k}(z_{j})}\right]\rightarrow \beta_{i}\beta_{j}\frac{m^{2}(\bar{z_{i}}\bar{z_{j}})^{m-1}}{((\bar{z_{i}}\bar{z_{j}})^{m}-1)^{2}}
	 \end{equation*}
	 in probability as $n\rightarrow\infty$. 
	 \begin{equation*}
	 \sum_{i,j=1}^{L}\sum_{k=1}^{n}\beta_{i}\alpha_{j}\E_{k}\left[\overline{\breve{Z}_{n,k}(z_{i})}\breve{Z}_{n,k}(z_{j})\right]\rightarrow \beta_{i}\alpha_{j}\frac{m^{2}(\bar{z_{i}}z_{j})^{m-1}}{((\bar{z_{i}}z_{j})^{m}-1)^{2}}
	 \end{equation*}
	 in probability as $n\rightarrow\infty$.
	 This concludes the proof of Lemma \ref{Lem:MDSReduction}.
	\end{Details} 
 \end{proof}

\begin{remark}
	\label{Remark:ComplexCase3}
	In the case where the atom variables are complex-valued, the limit of $\sum_{k=1}^{n}\E_{k-1}[\breve{M}_{n,k}^{2}]$ would differ from that of Lemma \ref{Lem:MDSReduction}. In this case, the calculations for terms of the form $\beta_{i}\alpha_{j}\E_{k}\left[\overline{\breve{Z}_{n,k}(z_{i})}\breve{Z}_{n,k}(z_{j})\right]$ would differ from those of the form $\beta_{i}\alpha_{j}\E_{k}\left[\breve{Z}_{n,k}(z_{i})\breve{Z}_{n,k}(z_{j})\right]$ due to the fact that the conjugation would need to be carried throughout the entire calculation.
\end{remark}

\begin{Details} 
Theorem \ref{Thm:MDS_CLT} together with Lemma \ref{Lem:ReductionZ} implies $M_{n}$ converges to the mean zero Gaussian with variance and covariance terms determined in Lemma \ref{Lem:MDSReduction}. Specifically,
\begin{equation*}
\sum_{l=1}^{L}\left(\alpha_{l}\Xi_{n}(z_{l})+\beta_{l}\overline{\Xi_{n}(z_{l})}\right)
\end{equation*}
converges to a real mean-zero Gaussian with variance
\begin{align*}
&\sum_{i,j=1}^{L}\left(\alpha_{i}\beta_{j}\frac{m^{2}(z_{i}\bar{z_{j}})^{m-1}}{((z_{i}\bar{z_{j}})^{m}-1)^{2}}+\beta_{i}\alpha_{j}\frac{m^{2}(\bar{z_{i}}z_{j})^{m-1}}{((\bar{z_{i}}z_{j})^{m}-1)^{2}}\right.\\
&\quad\quad\quad\quad\quad\quad\left.+\alpha_{i}\alpha_{j}\frac{m^{2}(z_{i}z_{j})^{m-1}}{((z_{i}z_{j})^{m}-1)^{2}}+\beta_{i}\beta_{j}\frac{m^{2}(\bar{z_{i}}\bar{z_{j}})^{m-1}}{((\bar{z_{i}}\bar{z_{j}})^{m}-1)^{2}}\right).
\end{align*}
\end{Details}
The Cramer--Wold theorem implies the convergence of finite dimensional distributions which completes the proof of Theorem \ref{Thm:FiniteDimDist}. The remainder of this section is devoted to proving the technical lemmas needed for Lemma \ref{Lem:Iteration}.

\begin{lemma} Define all quantities as in Lemma \ref{Lem:MDSReduction}. Then under the assumptions of Lemma \ref{Lem:MDSReduction}, for any $\alpha>0$, we have 
	\begin{align*} 
	&\E\left|e_{(i-1)n+k}^{T}\E_{k}\left[\mathcal{G}_{n}^{(k)}(z)c_{s}\oindicator{\Omega_{n,k}}\right]\E_{k}\left[c_{s}^{*}(\mathcal{G}_{n}^{(k)}(w))^{*}\oindicator{\Omega_{n,k}}\right]e_{(i-1)n+k}\right.\\
	&\quad\left.- e_{(i-1)n+k}^{T}\E_{k}\left[\mathcal{G}_{n}^{(k)}(z)c_{s}\oindicator{\Omega_{n,k}\cap\Omega_{n,k,s}}\right]\E_{k}\left[c_{s}^{*}(\mathcal{G}_{n}^{(k)}(w))^{*}\oindicator{\Omega_{n,k}\cap\Omega_{n,k,s}}\right]e_{(i-1)n+k}\right|^{2}\\
	&\quad\quad\quad\quad\quad\quad\quad\quad\quad\quad\quad\quad\quad\quad\quad\quad\quad\quad\quad\quad\quad\quad\quad\quad\quad\quad\quad\quad\quad=o_{\alpha}(n^{4-\alpha/4})
	\end{align*}
	uniformly in $i$ and $k$.
	\label{Lem:FiniteDimError1}
\end{lemma}
\begin{proof}
	Observe that
	\begin{align} 
	&e_{(i-1)n+k}^{T}\E_{k}\left[\mathcal{G}_{n}^{(k)}(z)c_{s}\oindicator{\Omega_{n,k}}\right]\E_{k}\left[c_{s}^{*}(\mathcal{G}_{n}^{(k)}(w))^{*}\oindicator{\Omega_{n,k}}\right]e_{(i-1)n+k}\notag\\
	&\quad=e_{(i-1)n+k}^{T}\E_{k}\left[\mathcal{G}_{n}^{(k)}(z)c_{s}\oindicator{\Omega_{n,k}\cap\Omega_{n,k,s}}\right]\E_{k}\left[c_{s}^{*}(\mathcal{G}_{n}^{(k)}(w))^{*}\oindicator{\Omega_{n,k}\cap\Omega_{n,k,s}}\right]e_{(i-1)n+k}\notag\\ 
	&\quad\quad+e_{(i-1)n+k}^{T}\E_{k}\left[\mathcal{G}_{n}^{(k)}(z)c_{s}\oindicator{\Omega_{n,k}\cap\Omega_{n,k,s}^{c}}\right]\E_{k}\left[c_{s}^{*}(\mathcal{G}_{n}^{(k)}(w))^{*}\oindicator{\Omega_{n,k}\cap\Omega_{n,k,s}}\right]e_{(i-1)n+k}\label{Equ:Lem:FiniteDimError1:1}\\
	&\quad\quad+e_{(i-1)n+k}^{T}\E_{k}\left[\mathcal{G}_{n}^{(k)}(z)c_{s}\oindicator{\Omega_{n,k}\cap\Omega_{n,k,s}}\right]\E_{k}\left[c_{s}^{*}(\mathcal{G}_{n}^{(k)}(w))^{*}\oindicator{\Omega_{n,k}\cap\Omega_{n,k,s}^{c}}\right]e_{(i-1)n+k}\label{Equ:Lem:FiniteDimError1:2}\\
	&\quad\quad+e_{(i-1)n+k}^{T}\E_{k}\left[\mathcal{G}_{n}^{(k)}(z)c_{s}\oindicator{\Omega_{n,k}\cap\Omega_{n,k,s}^{c}}\right]\E_{k}\left[c_{s}^{*}(\mathcal{G}_{n}^{(k)}(w))^{*}\oindicator{\Omega_{n,k}\cap\Omega_{n,k,s}^{c}}\right]e_{(i-1)n+k}.\label{Equ:Lem:FiniteDimError1:3}
	\end{align}
	Therefore, we must show terms \eqref{Equ:Lem:FiniteDimError1:1}, \eqref{Equ:Lem:FiniteDimError1:2}, and \eqref{Equ:Lem:FiniteDimError1:3} are sufficiently small in the $L^{2}$-norm. The argument for all three terms is very similar. We use Jensen's inequality and the Cauchy--Schwarz inequality to separate the inner conditional expectations. For resolvent terms where a complement event is not present, we bound by a constant and in terms where a complement event is present, we bound by $O(n^{-\alpha})$ since each event holds with overwhelming probability. We show the calculation for term \eqref{Equ:Lem:FiniteDimError1:1}, and the other terms follow in a similar manner. Observe 
	\begin{align*}
	&\E\left|e_{(i-1)n+k}^{T}\E_{k}\left[\mathcal{G}_{n}^{(k)}(z)c_{s}\oindicator{\Omega_{n,k}\cap\Omega_{n,k,s}^{c}}\right]\E_{k}\left[c_{s}^{*}(\mathcal{G}_{n}^{(k)}(w))^{*}\oindicator{\Omega_{n,k}\cap\Omega_{n,k,s}}\right]e_{(i-1)n+k}\right|^{2}\\
	&\quad\leq\E\left[\E_{k}\left[\lnorm\mathcal{G}_{n}^{(k)}(z)\oindicator{\Omega_{n,k}\cap\Omega_{n,k,s}^{c}}\rnorm^{2}\lnorm c_{s}\rnorm^{2}\right]\E_{k}\left[\lnorm c_{s}^{*}\rnorm^{2}\lnorm(\mathcal{G}_{n}^{(k)}(w))^{*}\oindicator{\Omega_{n,k}\cap\Omega_{n,k,s}}\rnorm^{2}\right] \right] \\
	&\quad\leq\left(\E\lnorm\mathcal{G}_{n}^{(k)}(z)\oindicator{\Omega_{n,k}\cap\Omega_{n,k,s}^{c}}\rnorm^{4}\E\lnorm c_{s}\rnorm^{4}\E\lnorm c_{s}^{*}\rnorm^{4}\E\lnorm(\mathcal{G}_{n}^{(k)}(w))^{*}\oindicator{\Omega_{n,k}\cap\Omega_{n,k,s}}\rnorm^{4}\right)^{1/2}\\
	&\quad\ll\left(\E\lnorm\mathcal{G}_{n}^{(k)}(z)\oindicator{\Omega_{n,k}\cap\Omega_{n,k,s}^{c}}\rnorm^{4}\E\lnorm(\mathcal{G}_{n}^{(k)}(w))^{*}\oindicator{\Omega_{n,k}\cap\Omega_{n,k,s}}\rnorm^{4}\right)^{1/2}\\
	&\quad\ll \P(\Omega_{n,k,s}^{c})^{1/2}\\
	&\quad\ll n^{-\alpha/2}.
	\end{align*}
	\begin{Details}
		\begin{align*}
		&\E\left|e_{(i-1)n+k}^{T}\E_{k}\left[\mathcal{G}_{n}^{(k)}(z)c_{s}\oindicator{\Omega_{n,k}}\right]\E_{k}\left[c_{s}^{*}(\mathcal{G}_{n}^{(k)}(w))^{*}\oindicator{\Omega_{n,k}}\right]e_{(i-1)n+k}\right.\\
		&\quad\quad\left.- e_{(i-1)n+k}^{T}\E_{k}\left[\mathcal{G}_{n}^{(k)}(z)c_{s}\oindicator{\Omega_{n,k}\cap\Omega_{n,k,s}}\right]\E_{k}\left[c_{s}^{*}(\mathcal{G}_{n}^{(k)}(w))^{*}\oindicator{\Omega_{n,k}\cap\Omega_{n,k,s}}\right]e_{(i-1)n+k}\right|^{2}\\
		&\quad\leq \E\left|e_{(i-1)n+k}^{T}\E_{k}\left[\mathcal{G}_{n}^{(k)}(z)c_{s}\oindicator{\Omega_{n,k}\cap\Omega_{n,k,s}^{c}}\right]\E_{k}\left[c_{s}^{*}(\mathcal{G}_{n}^{(k)}(w))^{*}\oindicator{\Omega_{n,k}\cap\Omega_{n,k,s}}\right]e_{(i-1)n+k}\right|^{2}\\
		&\quad\quad+\E\left|e_{(i-1)n+k}^{T}\E_{k}\left[\mathcal{G}_{n}^{(k)}(z)c_{s}\oindicator{\Omega_{n,k}\cap\Omega_{n,k,s}}\right]\E_{k}\left[c_{s}^{*}(\mathcal{G}_{n}^{(k)}(w))^{*}\oindicator{\Omega_{n,k}\cap\Omega_{n,k,s}^{c}}\right]e_{(i-1)n+k}\right|^{2}\\
		&\quad\quad+\E\left|e_{(i-1)n+k}^{T}\E_{k}\left[\mathcal{G}_{n}^{(k)}(z)c_{s}\oindicator{\Omega_{n,k}\cap\Omega_{n,k,s}^{c}}\right]\E_{k}\left[c_{s}^{*}(\mathcal{G}_{n}^{(k)}(w))^{*}\oindicator{\Omega_{n,k}\cap\Omega_{n,k,s}^{c}}\right]e_{(i-1)n+k}\right|^{2}\\
		&\quad\leq\E\left[\E_{k}\left[\lnorm\mathcal{G}_{n}^{(k)}(z)\oindicator{\Omega_{n,k}\cap\Omega_{n,k,s}^{c}}\rnorm^{2}\lnorm c_{s}\rnorm^{2}\right]\E_{k}\left[\lnorm c_{s}^{*}\rnorm^{2}\lnorm(\mathcal{G}_{n}^{(k)}(w))^{*}\oindicator{\Omega_{n,k}\cap\Omega_{n,k,s}}\rnorm^{2}\right] \right] \\
		&\quad\quad+\E\left[\E_{k}\left[\lnorm\mathcal{G}_{n}^{(k)}(z)\oindicator{\Omega_{n,k}\cap\Omega_{n,k,s}}\rnorm^{2}\lnorm c_{s}\rnorm^{2}\right]\E_{k}\left[\lnorm c_{s}^{*}\rnorm^{2}\lnorm(\mathcal{G}_{n}^{(k)}(w))^{*}\oindicator{\Omega_{n,k}\cap\Omega_{n,k,s}^{c}}\rnorm^{2}\right]\right]\\
		&\quad\quad+\E\left[\E_{k}\left[\lnorm\mathcal{G}_{n}^{(k)}(z)\oindicator{\Omega_{n,k}\cap\Omega_{n,k,s}^{c}}\rnorm^{2}\lnorm c_{s}\rnorm^{2}\right]\E_{k}\left[\lnorm c_{s}^{*}\rnorm^{2}\lnorm(\mathcal{G}_{n}^{(k)}(w))^{*}\oindicator{\Omega_{n,k}\cap\Omega_{n,k,s}^{c}}\rnorm^{2}\right]\right]\\
		&\quad\leq\left(\E\lnorm\mathcal{G}_{n}^{(k)}(z)\oindicator{\Omega_{n,k}\cap\Omega_{n,k,s}^{c}}\rnorm^{8}\E\lnorm c_{s}\rnorm^{8}\E\lnorm c_{s}^{*}\rnorm^{8}\E\lnorm(\mathcal{G}_{n}^{(k)}(w))^{*}\oindicator{\Omega_{n,k}\cap\Omega_{n,k,s}}\rnorm^{8}\right)^{1/4}\\
		&\quad\quad+\left(\E\lnorm\mathcal{G}_{n}^{(k)}(z)\oindicator{\Omega_{n,k}\cap\Omega_{n,k,s}}\rnorm^{8}\E\lnorm c_{s}\rnorm^{8}\E\lnorm c_{s}^{*}\rnorm^{8}\E\lnorm(\mathcal{G}_{n}^{(k)}(w))^{*}\oindicator{\Omega_{n,k}\cap\Omega_{n,k,s}^{c}}\rnorm^{8}\right)^{1/4}\\
		&\quad\quad+\left(\E\lnorm\mathcal{G}_{n}^{(k)}(z)\oindicator{\Omega_{n,k}\cap\Omega_{n,k,s}^{c}}\rnorm^{8}\E\lnorm c_{s}\rnorm^{8}\E\lnorm c_{s}^{*}\rnorm^{8}\E\lnorm(\mathcal{G}_{n}^{(k)}(w))^{*}\oindicator{\Omega_{n,k}\cap\Omega_{n,k,s}^{c}}\rnorm^{8}\right)^{1/4}\\
		&\quad\leq n^{4}\E\left[\E_{k}\left[\lnorm\mathcal{G}_{n}^{(k)}(z)\oindicator{\Omega_{n,k}}\rnorm^{2}\oindicator{\Omega_{n,k,s}^{c}}\right]\E_{k}\left[\lnorm(\mathcal{G}_{n}^{(k)}(w))^{*}\oindicator{\Omega_{n,k}}\rnorm^{2}\oindicator{\Omega_{n,k,s}}\right] \right] \\
		&\quad\quad+n^{4}\E\left[\E_{k}\left[\lnorm\mathcal{G}_{n}^{(k)}(z)\oindicator{\Omega_{n,k}}\rnorm^{2}\oindicator{\Omega_{n,k,s}}\right]\E_{k}\left[\lnorm(\mathcal{G}_{n}^{(k)}(w))^{*}\oindicator{\Omega_{n,k}}\rnorm^{2}\oindicator{\Omega_{n,k,s}^{c}}\right]\right]\\
		&\quad\quad+n^{4}\E\left[\E_{k}\left[\lnorm\mathcal{G}_{n}^{(k)}(z)\oindicator{\Omega_{n,k}}\rnorm^{2}\oindicator{\Omega_{n,k,s}^{c}}\right]\E_{k}\left[\lnorm(\mathcal{G}_{n}^{(k)}(w))^{*}\oindicator{\Omega_{n,k}}\rnorm^{2}\oindicator{\Omega_{n,k,s}^{c}}\right]\right]\\
		&\quad\ll n^{4}\P(\Omega_{n,k,s}^{c})+n^{4}\P(\Omega_{n,k,s}^{c})+n^{4}\P(\Omega_{n,k,s}^{c})\\
		&\quad\ll n^{4}n^{-\alpha}.
		\end{align*}
	\end{Details}
\end{proof}

\begin{lemma} Define all quantities as in Lemma \ref{Lem:MDSReduction}. Then, for any $i$ with $1\leq i\leq m$ where subscripts are reduced modulo $m$, under the assumptions of Lemma \ref{Lem:MDSReduction}, we have 
	\begin{align*}
	&\E\left|e_{(i-1)n+k}^{T}\E_{k}\left[\mathcal{G}_{n}^{(k,s)}(z)c_{s}\delta_{k,s}(z)\oindicator{\Omega_{n,k,s}\cap Q'_{n,k,s}}\right]\right.\\
	&\quad\quad\quad\quad\quad\quad\left.\times\E_{k}\left[(\delta_{k,s}(w))^{*}c_{s}^{*}(\mathcal{G}_{n}^{(k,s)}(w))^{*}\oindicator{\Omega_{n,k,s}\cap Q'_{n,k,s}}\right]e_{(i-1)n+k}\right.\\
	&\quad\quad\left.-e_{(i-1)n+k}^{T}\E_{k}\left[\mathcal{G}_{n}^{(k,s)}(z)c_{s}\oindicator{\Omega_{n,k,s}\cap Q'_{n,k,s}}\right]\right.\\
	&\quad\quad\quad\quad\quad\quad\quad\quad\left.\times\E_{k}\left[c_{s}^{*}(\mathcal{G}_{n}^{(k,s)}(w))^{*}\oindicator{\Omega_{n,k,s}\cap Q'_{n,k,s}}\right]e_{(i-1)n+k}]\right|^{2}=o(n^{-2})
	\end{align*}
	uniformly in $i$, $s$, and $k$.
	\label{Lem:FiniteDimError2}
\end{lemma}

\begin{proof}
	Observe that by the triangle inequality, 
	\begin{align}
	&\E\left|e_{(i-1)n+k}^{T}\E_{k}\left[\mathcal{G}_{n}^{(k,s)}(z)c_{s}\delta_{k,s}(z)\oindicator{\Omega_{n,k,s}\cap Q'_{n,k,s}}\right]\right.\notag\\
	&\quad\quad\quad\quad\quad\quad\left.\times\E_{k}\left[(\delta_{k,s}(w))^{*}c_{s}^{*}(\mathcal{G}_{n}^{(k,s)}(w))^{*}\oindicator{\Omega_{n,k,s}\cap Q'_{n,k,s}}\right]e_{(i-1)n+k}\right.\notag\\
	&\quad\quad\left.-e_{(i-1)n+k}^{T}\E_{k}\left[\mathcal{G}_{n}^{(k,s)}(z)c_{s}\oindicator{\Omega_{n,k,s}\cap Q'_{n,k,s}}\right]\right.\notag\\
	&\quad\quad\quad\quad\quad\quad\quad\quad\left.\times\E_{k}\left[c_{s}^{*}(\mathcal{G}_{n}^{(k,s)}(w))^{*}\oindicator{\Omega_{n,k,s}\cap Q'_{n,k,s}}\right]e_{(i-1)n+k}]\right|^{2}\notag\\
	&\ll \E\left|e_{(i-1)n+k}^{T}\E_{k}\left[\mathcal{G}_{n}^{(k,s)}(z)c_{s}\delta_{k,s}(z)\oindicator{\Omega_{n,k,s}\cap Q'_{n,k,s}}\right]\right.\notag\\
	&\quad\quad\quad\quad\quad\quad\quad\left.\times\E_{k}\left[((\delta_{k,s}(w))^{*}-1)c_{s}^{*}(\mathcal{G}_{n}^{(k,s)}(w))^{*}\oindicator{\Omega_{n,k,s}\cap Q'_{n,k,s}}\right]e_{(i-1)n+k}\right|^{2}\label{Equ:lem:FiniteDimError2:Equ1}\\
	&\quad+\E\left|e_{(i-1)n+k}^{T}\E_{k}\left[\mathcal{G}_{n}^{(k,s)}(z)c_{s}(\delta_{k,s}(z)-1)\oindicator{\Omega_{n,k,s}\cap Q'_{n,k,s}}\right]\right.\notag\\
	&\quad\quad\quad\quad\quad\quad\quad\quad\left.\times\E_{k}\left[c_{s}^{*}(\mathcal{G}_{n}^{(k,s)}(w))^{*}\oindicator{\Omega_{n,k,s}\cap Q'_{n,k,s}}\right]e_{(i-1)n+k}\right|^{2}.\label{Equ:lem:FiniteDimError2:Equ2}
	\end{align}
	We will show each of these terms are $o(n^{-2})$. We begin with \eqref{Equ:lem:FiniteDimError2:Equ1}. By the generalized H\"{o}lders inequality, Lemma \ref{Lem:conjLessThanConstant}, and Lemma \ref{Lem:DeltaBoundedOnEvent}, we have 
	\begin{Details}
		\begin{align}
		&\E\left|e_{(i-1)n+k}^{T}\E_{k}\left[\mathcal{G}_{n}^{(k,s)}(z)c_{s}\delta_{k,s}(z)\oindicator{\Omega_{n,k,s}\cap Q'_{n,k,s}}\right]\right.\notag\\
		&\quad\quad\quad\quad\quad\quad\left.\times\E_{k}\left[((\delta_{k,s}(w))^{*}-1)c_{s}^{*}(\mathcal{G}_{n}^{(k,s)}(w))^{*}\oindicator{\Omega_{n,k,s}\cap Q'_{n,k,s}}\right]e_{(i-1)n+k}\right|^{2}\notag \\
		&\leq\E\left|\E_{k}\left[e_{(i-1)n+k}^{T}\mathcal{G}_{n}^{(k,s)}(z)c_{s}\delta_{k,s}(z)\oindicator{\Omega_{n,k,s}\cap Q'_{n,k,s}}\right]\right.\notag\\
		&\quad\quad\quad\quad\quad\quad\times\left.\E_{k}\left[((\delta_{k,s}(w))^{*}-1)c_{s}^{*}(\mathcal{G}_{n}^{(k,s)}(w))^{*}e_{(i-1)n+k}\oindicator{\Omega_{n,k,s}\cap Q'_{n,k,s}}\right]\right|^{2}\notag \\ 
		&\leq\left(\E\lnorm\E_{k}\left[\mathcal{G}_{n}^{(k,s)}(z)c_{s}\delta_{k,s}(z)\oindicator{\Omega_{n,k,s}\cap Q'_{n,k,s}}\right]\rnorm^{4}\right.\notag \\ 
		&\quad\quad\quad\quad\left.\times\E\lnorm\E_{k}\left[((\delta_{k,s}(w))^{*}-1)c_{s}^{*}(\mathcal{G}_{n}^{(k,s)}(w))^{*}\oindicator{\Omega_{n,k,s}\cap Q'}\right]\rnorm^{4}\right)^{1/2}\notag \\
		&\leq\left(\E\left| e_{(i-1)n+k}^{T}\mathcal{G}_{n}^{(k,s)}(z)c_{s}\delta_{k,s}(z)\oindicator{\Omega_{n,k,s}\cap Q'_{n,k,s}}\right|^{4}\right.\notag \\
		&\quad\quad\quad\quad\left.\times\E\left|((\delta_{k,s}(w))^{*}-1)c_{s}^{*}(\mathcal{G}_{n}^{(k,s)}(w))^{*}e_{(i-1)n+k}\oindicator{\Omega_{n,k,s}\cap Q'_{n,k,s}}\right|^{4}\right)^{1/2}\notag\\ 
		&\leq\left(\E\left[|\delta_{k,s}(z)|^{4}\lnorm\mathcal{G}_{n}^{(k,s)}(z)c_{s}\rnorm^{4}\oindicator{\Omega_{n,k,s}\cap Q'_{n,k,s}}\right]\right.\notag\\
		&\quad\quad\quad\quad\left.\times\E\lnorm((\delta_{k,s}(w))^{*}-1)c_{s}^{*}(\mathcal{G}_{n}^{(k,s)}(w))^{*}\oindicator{\Omega_{n,k,s}\cap Q'_{n,k,s}}\rnorm^{4}\right)^{1/2}\notag\\
		&\ll\left(\E\left[\left| c_{s}^{*}(\mathcal{G}_{n}^{(k,s)}(z))^{*}e_{(i-1)n+k}e_{(i-1)n+k}^{T}\mathcal{G}_{n}^{(k,s)}(z)c_{s}\right|^{2}\oindicator{\Omega_{n,k,s}\cap Q'_{n,k,s}}\right]\right.\notag\\
		&\quad\quad\quad\quad\left.\times\E\left|((\delta_{k,s}(w))^{*}-1)c_{s}^{*}(\mathcal{G}_{n}^{(k,s)}(w))^{*}e_{(i-1)n+k}\oindicator{\Omega_{n,k,s}\cap Q'_{n,k,s}}\right|^{4}\right)^{1/2}\notag\\
		&\ll\left(n^{-2}\E\left[\lnorm (\mathcal{G}_{n}^{(k,s)}(z))^{*}\mathcal{G}_{n}^{(k,s)}(z)\rnorm^{2}\oindicator{\Omega_{n,k,s}\cap Q'_{n,k,s}}\right]\right.\notag\\
		&\quad\quad\quad\quad\quad\left.\times\E\lnorm((\delta_{k,s}(w))^{*}-1)c_{s}^{*}(\mathcal{G}_{n}^{(k,s)}(w))^{*}\oindicator{\Omega_{n,k,s}\cap Q'_{n,k,s}}\rnorm^{4}\right)^{1/2}\notag\\
		&\ll \frac{1}{n} \left(\E\left|((\delta_{k,s}(w))^{*}-1)c_{s}^{*}(\mathcal{G}_{n}^{(k,s)}(w))^{*}e_{(i-1)n+k}\oindicator{\Omega_{n,k,s}\cap Q'_{n,k,s}}\right|^{4}\right)^{1/2}\notag\\ 
		&\ll \frac{1}{n} \left(\E\left|((\delta_{k,s}(w))^{*}-1)\oindicator{\Omega_{n,k,s}\cap Q'_{n,k,s}}\right|^{8}\right.\notag\\
		&\quad\quad\quad\quad\times\left.\E\left| c_{s}^{*}(\mathcal{G}_{n}^{(k,s)}(w))^{*}e_{(i-1)n+k}e_{(i-1)n+k}^{T}\mathcal{G}_{n}^{(k,s)}(w)c_{s}\oindicator{\Omega_{n,k,s}\cap Q'_{n,k,s}}\right|^{4}\right)^{1/4}\notag\\ 
		&\leq Cn^{-1} \left(\E\left|((\delta_{k,s}(w))^{*}-1)\oindicator{\Omega_{n,k,s}\cap Q'_{n,k,s}}\right|^{8}\E\lnorm c_{s}^{*}(\mathcal{G}_{n}^{(k,s)}(w))^{*}\oindicator{\Omega_{n,k,s}\cap Q'_{n,k,s}}\rnorm^{8}\right)^{1/4}\notag\\ 
		&\leq Cn^{-1} \left(\E\left|((\delta_{k,s}(w))^{*}-1)\oindicator{\Omega_{n,k,s}\cap Q'_{n,k,s}}\right|^{8}\right.\notag\\ 
		&\quad\quad\quad\quad\quad\left.\times\E\lnorm c_{s}^{*}(\mathcal{G}_{n}^{(k,s)}(w))^{*}\mathcal{G}_{n}^{(k,s)}(z)c_{s}\oindicator{\Omega_{n,k,s}\cap Q'_{n,k,s}}\rnorm^{4}\right)^{1/4}\notag\\ 
		&\leq Cn^{-1} \left(\E\left|((\delta_{k,s}(w))^{*}-1)\oindicator{\Omega_{n,k,s}\cap Q'_{n,k,s}}\right|^{8}n^{-4\varepsilon-2}\right.\notag\\
		&\quad\quad\quad\quad\quad\left.\times\E\lnorm (\mathcal{G}_{n}^{(k,s)}(w))^{*}\mathcal{G}_{n}^{(k,s)}(z)\oindicator{\Omega_{n,k,s}\cap Q'_{n,k,s}}\rnorm^{4}\right)^{1/4}\notag\\
		&\ll n^{-1} \left(n^{-2-4\varepsilon}\E\left|((\delta_{k,s}(w))^{*}-1)\oindicator{\Omega_{n,k,s}\cap Q'_{n,k,s}}\right|^{8}\right)^{1/4}\notag\\
		&\ll n^{-3/2-\varepsilon} \left(\E\left|((\delta_{k,s}(w))^{*}-1)\oindicator{\Omega_{n,k,s}\cap Q'_{n,k,s}}\right|^{8}\right)^{1/4}.
		\end{align}
	\end{Details}
	\begin{align}
	&\E\left|e_{(i-1)n+k}^{T}\E_{k}\left[\mathcal{G}_{n}^{(k,s)}(z)c_{s}\delta_{k,s}(z)\oindicator{\Omega_{n,k,s}\cap Q'_{n,k,s}}\right]\right.\notag\\
	&\quad\quad\quad\quad\quad\quad\left.\times\E_{k}\left[((\delta_{k,s}(w))^{*}-1)c_{s}^{*}(\mathcal{G}_{n}^{(k,s)}(w))^{*}\oindicator{\Omega_{n,k,s}\cap Q'_{n,k,s}}\right]e_{(i-1)n+k}\right|^{2}\notag \\ 
	&\leq\left(\E\left| e_{(i-1)n+k}^{T}\mathcal{G}_{n}^{(k,s)}(z)c_{s}\delta_{k,s}(z)\oindicator{\Omega_{n,k,s}\cap Q'_{n,k,s}}\right|^{4}\right)^{1/2}\notag \\
	&\quad\times\left(\E\left|((\delta_{k,s}(w))^{*}-1)\oindicator{Q'_{n,k,s}}\right|^{8}\right)^{1/4}\left(\E\left|c_{s}^{*}(\mathcal{G}_{n}^{(k,s)}(w))^{*}e_{(i-1)n+k}\oindicator{\Omega_{n,k,s}}\right|^{8}\right)^{1/4}\notag\\ 
	&\ll\left(\E\left[\left| c_{s}^{*}(\mathcal{G}_{n}^{(k,s)}(z))^{*}e_{(i-1)n+k}e_{(i-1)n+k}^{T}\mathcal{G}_{n}^{(k,s)}(z)c_{s}\right|^{2}\oindicator{\Omega_{n,k,s}\cap Q'_{n,k,s}}\right]\right)^{1/2}\notag\\
	&\quad\times\left(\E\left|((\delta_{k,s}(w))^{*}-1)\oindicator{Q'_{n,k,s}}\right|^{8}\right)^{1/4}\notag\\
	&\quad\times\left(\E\left| c_{s}^{*}(\mathcal{G}_{n}^{(k,s)}(w))^{*}e_{(i-1)n+k}e_{(i-1)n+k}^{T}\mathcal{G}_{n}^{(k,s)}(w)c_{s}\oindicator{\Omega_{n,k,s}\cap Q'_{n,k,s}}\right|^{4}\right)^{1/4}\notag\\
	&\ll n^{-3/2-\varepsilon} \left(\E\left|((\delta_{k,s}(w))^{*}-1)\oindicator{\Omega_{n,k,s}\cap Q'_{n,k,s}}\right|^{8}\right)^{1/4}.\label{Equ:FinDimError2:Equ1}
	\end{align}
	Now, recall $\delta_{k,s}(z)$ defined in \eqref{Def:delta_{n,k}}. By the resolvent identity \eqref{Equ:ResolventIndentity}, we have $1-\delta_{k,s}(z)=1-\left(1+e_{s}^{T}\mathcal{G}_{n}^{(k,s)}(z)c_{s}\right)^{-1}=\left(e_{s}^{T}\mathcal{G}_{n}^{(k,s)}(z)c_{s}\right)\delta_{k,s}(z).$ 
	This gives 
	 \[\delta_{k,s}(z)=1-(e_{s}^{T}\mathcal{G}_{n}^{(k,s)}(z)c_{s})+(e_{s}^{T}\mathcal{G}_{n}^{(k,s)}(z)c_{s})^{2}\delta_{k,s}(z).\]
	\begin{Details}
	\begin{align*}
	\delta_{k,s}(z)&=1-(e_{s}^{T}\mathcal{G}_{n}^{(k,s)}(z)c_{s})\delta_{k,s}(z)\\
	&=1-(e_{s}^{T}\mathcal{G}_{n}^{(k,s)}(z)c_{s})+(e_{s}^{T}\mathcal{G}_{n}^{(k,s)}(z)c_{s})^{2}\delta_{k,s}(z).
	\end{align*}
	\end{Details}
	Ergo,
	\begin{equation}
	((\delta_{k,s}(w))^{*}-1=-(e_{s}^{T}\mathcal{G}_{n}^{(k,s)}(w)c_{s})^{*}+(e_{s}^{T}\mathcal{G}_{n}^{(k,s)}(w)c_{s})^{2*}(\delta_{k,s}(w))^{*}.\label{Equ:ExpansionOfDelta}
	\end{equation}
	We replace $((\delta_{k,s}(w))^{*}-1)$ in \eqref{Equ:FinDimError2:Equ1} with the expression on the right hand side of \eqref{Equ:ExpansionOfDelta} and use Lemma \ref{Lem:conjLessThanConstant} to see 
	\begin{align*}
	&n^{-3/2-\varepsilon}\left(\E\left|((\delta_{k,s}(w))^{*}-1)\oindicator{\Omega_{n,k,s}\cap Q'_{n,k,s}}\right|^{8} \right)^{1/4}\\ 
	&\ll n^{-3/2-\varepsilon}\left(\E\left|(-(e_{s}^{T}\mathcal{G}_{n}^{(k,s)}(w)c_{s})^{*}+(\delta_{k,s}(w))^{*}(e_{s}^{T}\mathcal{G}_{n}^{(k,s)}(w)c_{s})^{2*})\oindicator{\Omega_{n,k,s}\cap Q'_{n,k,s}}\right|^{8} \right)^{1/4}\\
	&\ll  n^{-3/2-\varepsilon}\left(\E\left|c_{s}^{*}(\mathcal{G}_{n}^{(k,s)}(w))^{*}e_{s}e_{s}^{T}\mathcal{G}_{n}^{(k,s)}(w)c_{s}\oindicator{\Omega_{n,k,s}\cap Q'_{n,k,s}}\right|^{4}\right.\\ 
	&\quad\quad\quad\quad\quad\quad\quad\left.+\E\left[\left|(\delta_{k,s}(w))^{*}\oindicator{\Omega_{n,k,s}\cap Q'_{n,k,s}}\right|^{8}\left|c_{s}^{*}(\mathcal{G}_{n}^{(k,s)}(w))^{*}e_{s}e_{s}^{T}\mathcal{G}_{n}^{(k,s)}(w)c_{s}\oindicator{\Omega_{n,k,s}\cap Q'_{n,k,s}}\right|^{8}\right] \right)^{1/4}\\ 
	&\ll n^{-3/2-\varepsilon}\left(n^{-4\varepsilon-2}\E\lnorm(\mathcal{G}_{n}^{(k,s)}(w))^{*}e_{s}e_{s}^{T}\mathcal{G}_{n}^{(k,s)}(w)\oindicator{\Omega_{n,k,s}\cap Q'_{n,k,s}}\rnorm^{4}\right.\\
	&\quad\quad\quad\quad\quad\quad\quad\left.+n^{-12\varepsilon-2}\E\lnorm(\mathcal{G}_{n}^{(k,s)}(w))^{*}e_{s}e_{s}^{T}\mathcal{G}_{n}^{(k,s)}(w)\oindicator{\Omega_{n,k,s}\cap Q'_{n,k,s}}\rnorm^{8} \right)^{1/4}\\ 
	&\ll n^{-3/2-\varepsilon}n^{-1/2-\varepsilon}.
	\end{align*}
	This shows term \eqref{Equ:lem:FiniteDimError2:Equ1} is $o(n^{-2})$. 
	\begin{Details} 
		Hence
		\begin{align*}
		&\E\left|e_{(i-1)n+k}^{T}\E_{k}\left[\mathcal{G}_{n}^{(k,s)}(z)c_{s}\delta_{k,s}(z)\oindicator{\Omega_{n,k,s}\cap Q'_{n,k,s}}\right]\right.\\
		&\quad\quad\quad\quad\quad\quad\left.\times\E_{k}\left[((\delta_{k,s}(w))^{*}-1)c_{s}^{*}(\mathcal{G}_{n}^{(k,s)}(w))^{*}\oindicator{\Omega_{n,k,s}\cap Q'_{n,k,s}}\right]e_{(i-1)n+k}\right|^{2}=o(n^{-1}).
		\end{align*}
	\end{Details}
	A very similar argument shows that term \eqref{Equ:lem:FiniteDimError2:Equ2} is $o(n^{-2})$ as well. We omit the details.
	\begin{Details}
		By the same argument as in the previous case, we can write
		\begin{align}
		&\E\left|e_{(i-1)n+k}^{T}\E_{k}\left[\mathcal{G}_{n}^{(k,s)}(z)c_{s}(\delta_{k,s}(z)-1)\oindicator{\Omega_{n,k,s}\cap Q'_{n,k,s}}\right]\right.\notag\\
		&\quad\quad\quad\quad\quad\quad\left.\times\E_{k}\left[c_{s}^{*}(\mathcal{G}_{n}^{(k,s)}(w))^{*}\oindicator{\Omega_{n,k,s}\cap Q'_{n,k,s}}\right]e_{(i-1)n+k}\right|^{2}\notag\\
		&\leq\left(\E\lnorm\mathcal{G}_{n}^{(k,s)}(z)c_{s}(\delta_{k,s}(z)-1)\oindicator{\Omega_{n,k,s}\cap Q'_{n,k,s}}\rnorm^{4}\E\lnorm c_{s}^{*}(\mathcal{G}_{n}^{(k,s)}(w))^{*}\oindicator{\Omega_{n,k,s}\cap Q'_{n,k,s}}\rnorm^{4}\right)^{1/2}\notag\\
		&\ll \frac{1}{n}\left(\E\lnorm\mathcal{G}_{n}^{(k,s)}(z)c_{s}(\delta_{k,s}(z)-1)\oindicator{\Omega_{n,k,s}\cap Q'_{n,k,s}}\rnorm^{4}\right)^{1/2}\notag\\
		&\ll n^{-3/2-\varepsilon}\left(\E\left|(\delta_{k,s}(z)-1)\oindicator{\Omega_{n,k,s}\cap Q'_{n,k,s}}\right|^{8}\right)^{1/4}.\label{Equ:FinDimError2:Equ2}
		\end{align}
		Again, by the expansion \eqref{Equ:ExpansionOfDelta}, we can see that 
		\begin{align*}
		&n^{-3/2-\varepsilon}\left(\E\left|(\delta_{k,s}(z)-1)\oindicator{\Omega_{n,k,s}\cap Q'_{n,k,s}}\right|^{8}\right)^{1/4}\\
		&\ll n^{-3/2-\varepsilon}\left(\E\left|(-(e_{s}^{T}\mathcal{G}_{n}^{(k,s)}(z)c_{s})+(e_{s}^{T}\mathcal{G}_{n}^{(k,s)}(z)c_{s})^{2}\delta_{k,s}(z))\oindicator{\Omega_{n,k,s}\cap Q'_{n,k,s}}\right|^{8}\right)^{1/4}\\
		&\ll n^{-2-2\varepsilon}\left(\E\left|c_{s}^{*}(\mathcal{G}_{n}^{(k,s)}(z))^{*}e_{s}e_{s}^{T}\mathcal{G}_{n}^{(k,s)}(z)c_{s}\oindicator{\Omega_{n,k,s}}\right|^{8}\right)^{1/4}\\
		&=o(n^{-1}).
		\end{align*}
	\end{Details}
	This completes the proof. 
\end{proof}

\begin{lemma}
	Define all quantities as in Lemma \ref{Lem:MDSReduction}. Then under the assumptions of Lemma \ref{Lem:MDSReduction}, for any $\alpha>0$ 
	\[\E\left|\frac{1}{z\bar{w}}(1-\P_{k}(\Omega_{n,k})^{2})\right|^{2}=o_{\alpha}(n^{-\alpha})\]
	uniformly in $k$.
	\label{Lem:FiniteDimError3}
\end{lemma}
\begin{proof}
	Observe that since $z,\bar{w}\in\mathcal{C}$ are fixed with $|z|=|\bar{w}|=1+\delta$, we know that 
	\begin{equation*}
	\left|\frac{1}{z\bar{w}}(1-\P_{k}(\Omega_{n,k})^{2})\right|^{2}\ll \left|1-\P_{k}(\Omega_{n,k})^{2}\right|^{2}\ll \left|1-\P_{k}(\Omega_{n,k})\right|^{2}.
	\end{equation*}
	Since $\Omega_{n,k}$ holds with overwhelming probability by Corollary \ref{Cor:Omega_nkOverwhelming},
	\[\E\left|1-\P_{k}(\Omega_{n,k})\right|^{2}=\E\left|\P_{k}(\Omega_{n,k}^{c})\right|^{2}\leq \P(\Omega_{n,k}^{c})=o_{\alpha}(n^{-\alpha})\]
	for any $\alpha>0$.
\end{proof}

\begin{lemma}
	Define all quantities as in Lemma \ref{Lem:MDSReduction} and assume $(i-b-1)n+1\leq s\leq (i-b)n$. Then under the assumptions of Lemma \ref{Lem:MDSReduction}, 
	\begin{align*}
	&\E\left|e_{(i-1)n+k}^{T}\E_{k}\left[\mathcal{G}_{n}^{(k,s)}(z)\oindicator{\Omega_{n,k,s}}\right]c_{s}c_{s}^{*}\E_{k}\left[(\mathcal{G}_{n}^{(k,s)}(w))^{*}\oindicator{\Omega_{n,k,s}}\right]e_{(i-1)n+k}\right.\\
	&\quad -\left.\frac{1}{n}e_{(i-1)n+k}^{T}\E_{k}\left[\mathcal{G}_{n}^{(k,s)}(z)\oindicator{\Omega_{n,k,s}}\right]\mathcal{D}_{i-b-1}\E_{k}\left[(\mathcal{G}_{n}^{(k,s)}(w))^{*}\oindicator{\Omega_{n,k,s}}\right]e_{(i-1)n+k}\right|^{2}=o(n^{-1})
	\end{align*}
	uniformly in $k$.
	\label{Lem:ReplaceColumnsWithExpectation}
\end{lemma}
\begin{proof}
	To begin, observe that by viewing the expression as a trace, and by cyclic permutation, we can rewrite
	\begin{align*}
	&e_{(i-1)n+k}^{T}\E_{k}\left[\mathcal{G}_{n}^{(k,s)}(z)\oindicator{\Omega_{n,k,s}}\right]c_{s}c_{s}^{*}\E_{k}\left[(\mathcal{G}_{n}^{(k,s)}(w))^{*}\oindicator{\Omega_{n,k,s}}\right]e_{(i-1)n+k}\\
	&\quad = \frac{1}{n}\left(\sqrt{n}c_{s}^{*}\E_{k}\left[(\mathcal{G}_{n}^{(k,s)}(w))^{*}\oindicator{\Omega_{n,k,s}}\right]e_{(i-1)n+k}e_{(i-1)n+k}^{T}\E_{k}\left[\mathcal{G}_{n}^{(k,s)}(z)\oindicator{\Omega_{n,k,s}}\right]\sqrt{n}c_{s}\right).
	\end{align*}
	For any complex-valued $N\times N$ matrix $A$ and any subset $S\subseteq[N]$, let $A_{S\times S}$ denote the $|S|\times |S|$ matrix which has entries $A_{(i,j)}$ for $i,j\in S$. Let $S_{b}=\{(i-b-2)n+1, (i-b-2)n+2,\dots, (i-b-1)n\}$. Then observe by cyclic permutation of the trace, we have 
	\begin{align*}
	&\frac{1}{n}e_{(i-1)n+k}^{T}\E_{k}\left[\mathcal{G}_{n}^{(k,s)}(z)\oindicator{\Omega_{n,k,s}}\right]\mathcal{D}_{i-b-1}\E_{k}\left[(\mathcal{G}_{n}^{(k,s)}(w))^{*}\oindicator{\Omega_{n,k,s}}\right]e_{(i-1)n+k}\\
	&\quad = \tr\left(\frac{1}{n}e_{(i-1)n+k}^{T}\E_{k}\left[\mathcal{G}_{n}^{(k,s)}(z)\oindicator{\Omega_{n,k,s}}\right]\mathcal{D}_{i-b-1}\E_{k}\left[(\mathcal{G}_{n}^{(k,s)}(w))^{*}\oindicator{\Omega_{n,k,s}}\right]e_{(i-1)n+k}\right)\\
	&\quad =\frac{1}{n} \tr\left(\mathcal{D}_{i-b-1}\E_{k}\left[(\mathcal{G}_{n}^{(k,s)}(w))^{*}\oindicator{\Omega_{n,k,s}}\right]e_{(i-1)n+k}e_{(i-1)n+k}^{T}\E_{k}\left[\mathcal{G}_{n}^{(k,s)}(z)\oindicator{\Omega_{n,k,s}}\right]\right)\\
	&\quad =\frac{1}{n}\tr\left(\E_{k}\left[(\mathcal{G}_{n}^{(k,s)}(w))^{*}\oindicator{\Omega_{n,k,s}}\right]e_{(i-1)n+k}e_{(i-1)n+k}^{T}\E_{k}\left[\mathcal{G}_{n}^{(k,s)}(z)\oindicator{\Omega_{n,k,s}}\right]\right)_{S_{b}\times S_{b}}.
	\end{align*}
	By this observation and Lemma \ref{Lem:BlockBilinearFormsWithTrace} we have 
	\begin{align}
	&\E\left|e_{(i-1)n+k}^{T}\E_{k}\left[\mathcal{G}_{n}^{(k,s)}(z)\oindicator{\Omega_{n,k,s}}\right]c_{s}c_{s}^{*}\E_{k}\left[(\mathcal{G}_{n}^{(k,s)}(w))^{*}\oindicator{\Omega_{n,k,s}}\right]e_{(i-1)n+k}\right.\notag\\
	&\quad -\left.\frac{1}{n}e_{(i-1)n+k}^{T}\E_{k}\left[\mathcal{G}_{n}^{(k,s)}(z)\oindicator{\Omega_{n,k,s}}\right]\mathcal{D}_{i-b-1}\E_{k}\left[(\mathcal{G}_{n}^{(k,s)}(w))^{*}\oindicator{\Omega_{n,k,s}}\right]e_{(i-1)n+k}\right|^{2}\notag\\
	&=\frac{1}{n^{2}}\E\left|\sqrt{n}c_{s}^{*}\E_{k}\left[(\mathcal{G}_{n}^{(k,s)}(w))^{*}\oindicator{\Omega_{n,k,s}}\right]e_{(i-1)n+k}e_{(i-1)n+k}^{T}\E_{k}\left[\mathcal{G}_{n}^{(k,s)}(z)\oindicator{\Omega_{n,k,s}}\right]\sqrt{n}c_{s}\right.\notag\\
	&\quad -\left.\tr\left(\E_{k}\left[(\mathcal{G}_{n}^{(k,s)}(w))^{*}\oindicator{\Omega_{n,k,s}}\right]e_{(i-1)n+k}e_{(i-1)n+k}^{T}\E_{k}\left[\mathcal{G}_{n}^{(k,s)}(z)\oindicator{\Omega_{n,k,s}}\right]\right)_{S_{b}\times S_{b}}\right|^{2}\notag\\
	&\ll\frac{1}{n^{2}}\E\left[\tr\left(\left(\E_{k}\left[(\mathcal{G}_{n}^{(k,s)}(w))^{*}\oindicator{\Omega_{n,k,s}}\right]e_{(i-1)n+k}e_{(i-1)n+k}^{T}\E_{k}\left[\mathcal{G}_{n}^{(k,s)}(z)\oindicator{\Omega_{n,k,s}}\right]\right)^{*}\right.\right.\notag\\	&\quad\times\left.\left.\left(\E_{k}\left[(\mathcal{G}_{n}^{(k,s)}(w))^{*}\oindicator{\Omega_{n,k,s}}\right]e_{(i-1)n+k}e_{(i-1)n+k}^{T}\E_{k}\left[\mathcal{G}_{n}^{(k,s)}(z)\oindicator{\Omega_{n,k,s}}\right]\right)\right)\right].\label{Equ:RepColsWExpectation:Equ1}
	\end{align}
	Observe that the rank is at most 1 so by bounding the trace by the rank times the norm, and by the fact that each term in the above expression can be bounded in norm by a constant, we can bound \eqref{Equ:RepColsWExpectation:Equ1} by $O(n^{-2})$ which completes the proof.
\end{proof}

\begin{lemma}
	Define all quantities as in Lemma \ref{Lem:MDSReduction}. Then under the assumptions of Lemma \ref{Lem:MDSReduction}, we have 
	\begin{align*}
	&\E\left|\frac{1}{n}\sum_{s=(i-b-1)n+1}^{(i-b-1)n+k-1}e_{(i-1)n+k}^{T}\E_{k}\left[\mathcal{G}_{n}^{(k,s)}(z)\oindicator{\Omega_{n,k,s}}\right]\mathcal{D}_{i-b-1}\E_{k}\left[(\mathcal{G}_{n}^{(k,s)}(w))^{*}\oindicator{\Omega_{n,k,s}}\right]e_{(i-1)n+k}\right.\\
	&\quad\left.-\frac{1}{n}\sum_{s=(i-2)n+1}^{(i-2)n+k-1}e_{(i-1)n+k}^{T}\E_{k}\left[\mathcal{G}_{n}^{(k)}(z)\oindicator{\Omega_{n,k}}\right]\mathcal{D}_{i-b-1}\E_{k}\left[(\mathcal{G}_{n}^{(k)}(w))^{*}\oindicator{\Omega_{n,k}}\right]e_{(i-1)n+k}\right|^{2}=o(1)
	\end{align*}
	uniformly in $k$. 
	\label{Lem:ReplaceColS}
\end{lemma}
\begin{proof}
	To begin, observe that 
	\begin{align*}
	&\E\left|\frac{1}{n}\sum_{s=(i-b-1)n+1}^{(i-b-1)n+k-1}e_{(i-1)n+k}^{T}\E_{k}\left[\mathcal{G}_{n}^{(k,s)}(z)\oindicator{\Omega_{n,k,s}}\right]\mathcal{D}_{i-b-1}\E_{k}\left[(\mathcal{G}_{n}^{(k,s)}(w))^{*}\oindicator{\Omega_{n,k,s}}\right]e_{(i-1)n+k}\right.\\
	&\quad\quad\left.-\frac{1}{n}\sum_{s=(i-b-1)n+1}^{(i-b-1)n+k-1}e_{(i-1)n+k}^{T}\E_{k}\left[\mathcal{G}_{n}^{(k)}(z)\oindicator{\Omega_{n,k}}\right]\mathcal{D}_{i-b-1}\E_{k}\left[(\mathcal{G}_{n}^{(k)}(w))^{*}\oindicator{\Omega_{n,k}}\right]e_{(i-1)n+k}\right|^{2}\\
	&\ll\frac{1}{n} \sum_{s=(i-b-1)n+1}^{(i-b-1)n+k-1}\E\left|\E_{k}\left[e_{(i-1)n+k}^{T}\mathcal{G}_{n}^{(k,s)}(z)\oindicator{\Omega_{n,k,s}}\right]\mathcal{D}_{i-b-1}\E_{k}\left[(\mathcal{G}_{n}^{(k,s)}(w))^{*}e_{(i-1)n+k}\oindicator{\Omega_{n,k,s}}\right]\right.\\
	&\quad\quad\quad\quad\quad\quad\quad\quad\quad\quad\left.-\E_{k}\left[e_{(i-1)n+k}^{T}\mathcal{G}_{n}^{(k)}(z)\oindicator{\Omega_{n,k}}\right]\mathcal{D}_{i-b-1}\E_{k}\left[(\mathcal{G}_{n}^{(k)}(w))^{*}e_{(i-1)n+k}\oindicator{\Omega_{n,k}}\right]\right|^{2}
	\end{align*}
	so it suffices to prove that 
	\begin{align*}
	&\E\left|\E_{k}\left[e_{(i-1)n+k}^{T}\mathcal{G}_{n}^{(k,s)}(z)\oindicator{\Omega_{n,k,s}}\right]\mathcal{D}_{i-b-1}\E_{k}\left[(\mathcal{G}_{n}^{(k,s)}(w))^{*}e_{(i-1)n+k}\oindicator{\Omega_{n,k,s}}\right]\right.\\
	&\quad\quad\quad\left.-\E_{k}\left[e_{(i-1)n+k}^{T}\mathcal{G}_{n}^{(k)}(z)\oindicator{\Omega_{n,k}}\right]\mathcal{D}_{i-b-1}\E_{k}\left[(\mathcal{G}_{n}^{(k)}(w))^{*}e_{(i-1)n+k}\oindicator{\Omega_{n,k}}\right]\right|^{2}=o(1)
	\end{align*}
	uniformly in $i$, $k$ and $s$. We can expand this difference using the triangle inequality to get differences which only vary by an event or a resolvent. We begin by observing that by the Cauchy--Schwarz inequality and Lemma \ref{Lem:G_nksBounded}, for any $\alpha>0$,
	\begin{align*}
	&\E\left|\E_{k}\left[e_{(i-1)n+k}^{T}\mathcal{G}_{n}^{(k,s)}(z)\oindicator{\Omega_{n,k,s}}\right]\mathcal{D}_{i-b-1}\E_{k}\left[(\mathcal{G}_{n}^{(k,s)}(w))^{*}e_{(i-1)n+k}\oindicator{\Omega_{n,k,s}}\right]\right.\notag\\
	&\quad\quad\quad\left.-\E_{k}\left[e_{(i-1)n+k}^{T}\mathcal{G}_{n}^{(k,s)}(z)\oindicator{\Omega_{n,k,s}}\right]\mathcal{D}_{i-b-1}\E_{k}\left[(\mathcal{G}_{n}^{(k,s)}(w))^{*}e_{(i-1)n+k}\oindicator{\Omega_{n,k,s}\cap\Omega_{n,k}}\right]\right|^{2}\\
	&\ll\E\lnorm\E_{k}\left[\mathcal{G}_{n}^{(k,s)}(z)\oindicator{\Omega_{n,k,s}}\right]\mathcal{D}_{i-b-1}\E_{k}\left[(\mathcal{G}_{n}^{(k,s)}(w))^{*}\oindicator{\Omega_{n,k,s}}\right]\right.\\
	&\quad\quad\quad\left.-\E_{k}\left[\mathcal{G}_{n}^{(k,s)}(z)\oindicator{\Omega_{n,k,s}}\right]\mathcal{D}_{i-b-1}\E_{k}\left[(\mathcal{G}_{n}^{(k,s)}(w))^{*}\oindicator{\Omega_{n,k,s}\cap\Omega_{n,k}}\right]\rnorm^{2}\\
	&=\E\lnorm\E_{k}\left[\mathcal{G}_{n}^{(k,s)}(z)\oindicator{\Omega_{n,k,s}}\right]\mathcal{D}_{i-b-1}\E_{k}\left[(\mathcal{G}_{n}^{(k,s)}(w))^{*}(\oindicator{\Omega_{n,k,s}}\oindicator{\Omega_{n,k}^{c}})\right]\rnorm^{2}\\
	&\leq \left(\E\left[\lnorm\mathcal{G}_{n}^{(k,s)}(z)\oindicator{\Omega_{n,k,s}}\rnorm^{4}\lnorm\mathcal{D}_{i-b-1}\rnorm^{4}\right]\E\left[\lnorm(\mathcal{G}_{n}^{(k,s)}(w))^{*}\oindicator{\Omega_{n,k,s}}\rnorm^{4}\oindicator{\Omega_{n,k}^{c}}\right]\right)^{1/2}\\
	&\ll_{\alpha} n^{-\alpha/2}.
	\end{align*}
	The same argument shows that each indicator can be replaced with $\oindicator{\Omega_{n,k}\cap\Omega_{n,k,s}}$. We also bound terms that differ by a resolvent. To this end, observe that by the resolvent identity \eqref{Equ:ResolventIndentity} and Lemma \ref{Lem:conjLessThanConstant},
	\begin{align*}
	&\E\left|\E_{k}\left[e_{(i-1)n+k}^{T}\mathcal{G}_{n}^{(k,s)}(z)\oindicator{\Omega_{n,k}\cap\Omega_{n,k,s}}\right]\mathcal{D}_{i-b-1}\E_{k}\left[(\mathcal{G}_{n}^{(k,s)}(w))^{*}e_{(i-1)n+k}\oindicator{\Omega_{n,k}\cap\Omega_{n,k,s}}\right]\right.\\
	&\quad\left.-\E_{k}\left[e_{(i-1)n+k}^{T}\mathcal{G}_{n}^{(k,s)}(z)\oindicator{\Omega_{n,k}\cap\Omega_{n,k,s}}\right]\mathcal{D}_{i-b-1}\E_{k}\left[(\mathcal{G}_{n}^{(k)}(w))^{*}e_{(i-1)n+k}\oindicator{\Omega_{n,k}\cap\Omega_{n,k,s}}\right]\right|^{2}\\
	&=\E\left|\E_{k}\left[e_{(i-1)n+k}^{T}\mathcal{G}_{n}^{(k,s)}(z)\oindicator{\Omega_{n,k}\cap\Omega_{n,k,s}}\right]\mathcal{D}_{i-b-1}\right.\\
	&\quad\quad\quad\left.\times\E_{k}\left[(\mathcal{G}_{n}^{(k,s)}(w)(c_{s}e_{s}^{T})\mathcal{G}_{n}^{(k)}(w))^{*}e_{(i-1)n+k}\oindicator{\Omega_{n,k}\cap\Omega_{n,k,s}}\right]\right|^{2}\\
	&\leq\left(\E\left[\lnorm e_{(i-1)n+k}\rnorm^{4}\lnorm\mathcal{G}_{n}^{(k,s)}(z)\oindicator{\Omega_{n,k}\cap\Omega_{n,k,s}}\rnorm^{4}\lnorm\mathcal{D}_{i-b-1}\rnorm^{4}\right]\right.\\
	&\quad\quad\quad\quad\left.\times\E\left[\lnorm(\mathcal{G}_{n}^{(k)}(w))^{*}\oindicator{\Omega_{n,k}}\rnorm^{4}\lnorm e_{s}\rnorm^{4}\left|c_{s}^{*}(\mathcal{G}_{n}^{(k,s)}(w))^{*}e_{(i-1)n+k}\oindicator{\Omega_{n,k,s}}\right|^{4}\right]\right)^{1/2}\\
	&\ll \left(\E\left|c_{s}^{*}(\mathcal{G}_{n}^{(k,s)}(w))^{*}e_{(i-1)n+k}e_{(i-1)n+k}^{T}\mathcal{G}_{n}^{(k,s)}(w)c_{s}\oindicator{\Omega_{n,k,s}}\right|^{2}\right)^{1/2}\\
	&\ll n^{-1}.
	\end{align*}
	The same argument shows that all instances of $\mathcal{G}_{n}^{(k,s)}(z)$ or $(\mathcal{G}_{n}^{(k,s)}(w))^{*}$ can be replaced with $\mathcal{G}_{n}^{(k)}(z)$ or $(\mathcal{G}_{n}^{(k)}(w))^{*}$ respectively gaining an error that is $o(1)$ in $L^{2}$-norm. Finally, the same argument as before shows that 
	\begin{align*}
	&\E\left|\E_{k}\left[e_{(i-1)n+k}^{T}\mathcal{G}_{n}^{(k)}(z)\oindicator{\Omega_{n,k}\cap\Omega_{n,k,s}}\right]\mathcal{D}_{i-b-1}\E_{k}\left[(\mathcal{G}_{n}^{(k)}(w))^{*}e_{(i-1)n+k}\oindicator{\Omega_{n,k}\cap\Omega_{n,k,s}}\right]\right.\\
	&\left.-\E_{k}\left[e_{(i-1)n+k}^{T}\mathcal{G}_{n}^{(k)}(z)\oindicator{\Omega_{n,k}}\right]\mathcal{D}_{i-b-1}\E_{k}\left[(\mathcal{G}_{n}^{(k)}(w))^{*}e_{(i-1)n+k}\oindicator{\Omega_{n,k}}\right]\right|^{2}=o(1).
	\end{align*}
	Replacing all instances of $\oindicator{\Omega_{n,k}\cap\Omega_{n,k,s}}$ with $\oindicator{\Omega_{n,k}}$ completes the proof.
	\begin{Details}
		\begin{align}
		&\E\left|\E_{k}\left[e_{(i-1)n+k}^{T}\mathcal{G}_{n}^{(k,s)}(z)\oindicator{\Omega_{n,k,s}}\right]\mathcal{D}_{i-b-1}\E_{k}\left[(\mathcal{G}_{n}^{(k,s)}(w))^{*}e_{(i-1)n+k}\oindicator{\Omega_{n,k,s}}\right]\right.\notag\\
		&\quad\quad\quad\quad\left.-\E_{k}\left[e_{(i-1)n+k}^{T}\mathcal{G}_{n}^{(k)}(z)\oindicator{\Omega_{n,k}}\right]\mathcal{D}_{i-b-1}\E_{k}\left[(\mathcal{G}_{n}^{(k)}(w))^{*}e_{(i-1)n+k}\oindicator{\Omega_{n,k}}\right]\right|^{2}\notag\\
		&\ll\E\left|\E_{k}\left[e_{(i-1)n+k}^{T}\mathcal{G}_{n}^{(k,s)}(z)\oindicator{\Omega_{n,k,s}}\right]\mathcal{D}_{i-b-1}\E_{k}\left[(\mathcal{G}_{n}^{(k,s)}(w))^{*}e_{(i-1)n+k}\oindicator{\Omega_{n,k,s}}\right]\right.\notag\\
		&\quad\quad\left.-\E_{k}\left[e_{(i-1)n+k}^{T}\mathcal{G}_{n}^{(k,s)}(z)\oindicator{\Omega_{n,k,s}}\right]\mathcal{D}_{i-b-1}\E_{k}\left[(\mathcal{G}_{n}^{(k,s)}(w))^{*}e_{(i-1)n+k}\oindicator{\Omega_{n,k,s}\cap\Omega_{n,k}}\right]\right|^{2}\label{Equ:RepColS:Equ1}\\
		&\quad+\E\left|\E_{k}\left[e_{(i-1)n+k}^{T}\mathcal{G}_{n}^{(k,s)}(z)\oindicator{\Omega_{n,k,s}}\right]\mathcal{D}_{i-b-1}\E_{k}\left[(\mathcal{G}_{n}^{(k,s)}(w))^{*}e_{(i-1)n+k}\oindicator{\Omega_{n,k,s}\cap\Omega_{n,k}}\right]\right.\notag\\
		&\quad\quad\left.-\E_{k}\left[e_{(i-1)n+k}^{T}\mathcal{G}_{n}^{(k,s)}(z)\oindicator{\Omega_{n,k,s}}\right]\mathcal{D}_{i-b-1}\E_{k}\left[(\mathcal{G}_{n}^{(k)}(w))^{*}e_{(i-1)n+k}\oindicator{\Omega_{n,k,s}\cap\Omega_{n,k}}\right]\right|^{2}\label{Equ:RepColS:Equ2}\\
		&\quad+\E\left|\E_{k}\left[e_{(i-1)n+k}^{T}\mathcal{G}_{n}^{(k,s)}(z)\oindicator{\Omega_{n,k,s}}\right]\mathcal{D}_{i-b-1}\E_{k}\left[(\mathcal{G}_{n}^{(k)}(w))^{*}e_{(i-1)n+k}\oindicator{\Omega_{n,k,s}\cap\Omega_{n,k}}\right]\right.\notag\\
		&\quad\quad\left.-\E_{k}\left[e_{(i-1)n+k}^{T}\mathcal{G}_{n}^{(k,s)}(z)\oindicator{\Omega_{n,k,s}}\right]\mathcal{D}_{i-b-1}\E_{k}\left[(\mathcal{G}_{n}^{(k)}(w))^{*}e_{(i-1)n+k}\oindicator{\Omega_{n,k}}\right]\right|^{2}\label{Equ:RepColS:Equ3}\\
		&\quad+\E\left|\E_{k}\left[e_{(i-1)n+k}^{T}\mathcal{G}_{n}^{(k,s)}(z)\oindicator{\Omega_{n,k,s}}\right]\mathcal{D}_{i-b-1}\E_{k}\left[(\mathcal{G}_{n}^{(k)}(w))^{*}e_{(i-1)n+k}\oindicator{\Omega_{n,k}}\right]\right.\notag\\
		&\quad\quad\left.-\E_{k}\left[e_{(i-1)n+k}^{T}\mathcal{G}_{n}^{(k,s)}(z)\oindicator{\Omega_{n,k,s}\cap\Omega_{n,k}}\right]\mathcal{D}_{i-b-1}\E_{k}\left[(\mathcal{G}_{n}^{(k)}(w))^{*}e_{(i-1)n+k}\oindicator{\Omega_{n,k}}\right]\right|^{2}\label{Equ:RepColS:Equ4}\\
		&\quad+\E\left|\E_{k}\left[e_{(i-1)n+k}^{T}\mathcal{G}_{n}^{(k,s)}(z)\oindicator{\Omega_{n,k,s}\cap\Omega_{n,k}}\right]\mathcal{D}_{i-b-1}\E_{k}\left[(\mathcal{G}_{n}^{(k)}(w))^{*}e_{(i-1)n+k}\oindicator{\Omega_{n,k}}\right]\right.\notag\\
		&\quad\quad\left.-\E_{k}\left[e_{(i-1)n+k}^{T}\mathcal{G}_{n}^{(k)}(z)\oindicator{\Omega_{n,k,s}\cap\Omega_{n,k}}\right]\mathcal{D}_{i-b-1}\E_{k}\left[(\mathcal{G}_{n}^{(k)}(w))^{*}e_{(i-1)n+k}\oindicator{\Omega_{n,k}}\right]\right|^{2}\label{Equ:RepColS:Equ5}\\
		&\quad+\E\left|\E_{k}\left[e_{(i-1)n+k}^{T}\mathcal{G}_{n}^{(k)}(z)\oindicator{\Omega_{n,k,s}\cap\Omega_{n,k}}\right]\mathcal{D}_{i-b-1}\E_{k}\left[(\mathcal{G}_{n}^{(k)}(w))^{*}e_{(i-1)n+k}\oindicator{\Omega_{n,k}}\right]\right.\notag\\
		&\quad\quad\left.-\E_{k}\left[e_{(i-1)n+k}^{T}\mathcal{G}_{n}^{(k)}(z)\oindicator{\Omega_{n,k}}\right]\mathcal{D}_{i-b-1}\E_{k}\left[(\mathcal{G}_{n}^{(k)}(w))^{*}e_{(i-1)n+k}\oindicator{\Omega_{n,k}}\right]\right|^{2}.\label{Equ:RepColS:Equ6}
		\end{align}
	We will bound each of these terms separately. We begin with \eqref{Equ:RepColS:Equ1}. 
	\begin{align*}
		&\E\lnorm\E_{k}\left[\mathcal{G}_{n}^{(k,s)}(z)\oindicator{\Omega_{n,k,s}}\right]\mathcal{D}_{i-b-1}\E_{k}\left[(\mathcal{G}_{n}^{(k,s)}(w))^{*}\oindicator{\Omega_{n,k,s}}\right]\right.\\
		&\quad\quad\quad\quad\quad\quad\left.-\E_{k}\left[\mathcal{G}_{n}^{(k,s)}(z)\oindicator{\Omega_{n,k,s}}\right]\mathcal{D}_{i-b-1}\E_{k}\left[(\mathcal{G}_{n}^{(k,s)}(w))^{*}\oindicator{\Omega_{n,k,s}\cap\Omega_{n,k}}\right]\rnorm^{2}.
	\end{align*}
	Very similar arguments show that \eqref{Equ:RepColS:Equ3}, \eqref{Equ:RepColS:Equ4}, and \eqref{Equ:RepColS:Equ6} are all $o(1)$.
	\begin{align*} 
		&\E\lnorm\E_{k}\left[\mathcal{G}_{n}^{(k,s)}(z)\oindicator{\Omega_{n,k,s}}\right]\mathcal{D}_{i-2}\E_{k}\left[(\mathcal{G}_{n}^{(k)}(w))^{*}\oindicator{\Omega_{n,k,s}\cap\Omega_{n,k}}\right]\right.\\
		&\quad\quad\quad\quad\quad\quad\left.-\E_{k}\left[\mathcal{G}_{n}^{(k,s)}(z)\oindicator{\Omega_{n,k,s}}\right]\mathcal{D}_{i-2}\E_{k}\left[(\mathcal{G}_{n}^{(k)}(w))^{*}\oindicator{\Omega_{n,k}}\right]\rnorm^{2}=o(1),
	\end{align*}
	\begin{align*} 
		&\E\lnorm\E_{k}\left[\mathcal{G}_{n}^{(k,s)}(z)\oindicator{\Omega_{n,k,s}}\right]\mathcal{D}_{i-2}\E_{k}\left[(\mathcal{G}_{n}^{(k)}(w))^{*}\oindicator{\Omega_{n,k}}\right]\right.\\
		&\quad\quad\quad\quad\quad\quad\left.-\E_{k}\left[\mathcal{G}_{n}^{(k,s)}(z)\oindicator{\Omega_{n,k,s}\cap\Omega_{n,k}}\right]\mathcal{D}_{i-2}\E_{k}\left[(\mathcal{G}_{n}^{(k)}(w))^{*}\oindicator{\Omega_{n,k}}\right]\rnorm^{2}=o(1),
	\end{align*}
	and
	\begin{align*} 
		&\E\lnorm\E_{k}\left[\mathcal{G}_{n}^{(k)}(z)\oindicator{\Omega_{n,k,s}\cap\Omega_{n,k}}\right]\mathcal{D}_{i-2}\E_{k}\left[(\mathcal{G}_{n}^{(k)}(w))^{*}\oindicator{\Omega_{n,k}}\right]\right.\\
		&\quad\quad\quad\quad\quad\quad\left.-\E_{k}\left[\mathcal{G}_{n}^{(k)}(z)\oindicator{\Omega_{n,k}}\right]\mathcal{D}_{i-2}\E_{k}\left[(\mathcal{G}_{n}^{(k)}(w))^{*}\oindicator{\Omega_{n,k}}\right]\rnorm^{2}=o(1).
	\end{align*}
	This leaves us with only the two terms which differ by resolvents. We will first deal with \eqref{Equ:RepColS:Equ2}. 
	\begin{align*}
		&\E\lnorm\E_{k}\left[\mathcal{G}_{n}^{(k,s)}(z)\oindicator{\Omega_{n,k,s}}\right]\mathcal{D}_{i-2}\E_{k}\left[(\mathcal{G}_{n}^{(k,s)}(w))^{*}\oindicator{\Omega_{n,k,s}\cap\Omega_{n,k}}\right]\right.\\
		&\quad\quad\quad\quad\quad\left.-\E_{k}\left[\mathcal{G}_{n}^{(k,s)}(z)\oindicator{\Omega_{n,k,s}}\right]\mathcal{D}_{i-2}\E_{k}\left[(\mathcal{G}_{n}^{(k)}(w))^{*}\oindicator{\Omega_{n,k,s}\cap\Omega_{n,k}}\right]\rnorm^{2}.
	\end{align*}
	Observe that by the resolvent identity \eqref{Equ:ResolventIndentity} and Lemma \ref{Lem:conjLessThanConstant},
	\begin{align*}
	&\E\left|\E_{k}\left[e_{(i-1)n+k}^{T}\mathcal{G}_{n}^{(k,s)}(z)\oindicator{\Omega_{n,k,s}}\right]\mathcal{D}_{i-b-1}\E_{k}\left[(\mathcal{G}_{n}^{(k,s)}(w))^{*}e_{(i-1)n+k}\oindicator{\Omega_{n,k,s}\cap\Omega_{n,k}}\right]\right.\\
	&\quad\quad\quad\left.-\E_{k}\left[e_{(i-1)n+k}^{T}\mathcal{G}_{n}^{(k,s)}(z)\oindicator{\Omega_{n,k,s}}\right]\mathcal{D}_{i-b-1}\E_{k}\left[(\mathcal{G}_{n}^{(k)}(w))^{*}e_{(i-1)n+k}\oindicator{\Omega_{n,k,s}\cap\Omega_{n,k}}\right]\right|^{2}\\
	&=\E\left|\E_{k}\left[e_{(i-1)n+k}^{T}\mathcal{G}_{n}^{(k,s)}(z)\oindicator{\Omega_{n,k,s}}\right]\mathcal{D}_{i-b-1}\right.\\
	&\quad\quad\quad\left.\times \E_{k}\left[(\mathcal{G}_{n}^{(k,s)}(w)(\mathcal{Y}_{n}^{(k)}-\blmat{Y}{n}^{(k,s)})\mathcal{G}_{n}^{(k)}(w))^{*}e_{(i-1)n+k}\oindicator{\Omega_{n,k,s}\cap\Omega_{n,k}}\right]\right|^{2}\\
	&=\E\left|\E_{k}\left[e_{(i-1)n+k}^{T}\mathcal{G}_{n}^{(k,s)}(z)\oindicator{\Omega_{n,k,s}}\right]\mathcal{D}_{i-b-1}\right.\\
	&\quad\quad\quad\left.\times\E_{k}\left[(\mathcal{G}_{n}^{(k,s)}(w)(c_{s}e_{s}^{T})\mathcal{G}_{n}^{(k)}(w))^{*}e_{(i-1)n+k}\oindicator{\Omega_{n,k,s}\cap\Omega_{n,k}}\right]\right|^{2}\\
	&\leq\left(\E\lnorm\E_{k}\left[e_{(i-1)n+k}^{T}\mathcal{G}_{n}^{(k,s)}(z)\oindicator{\Omega_{n,k,s}}\right]\mathcal{D}_{i-b-1}\rnorm^{4}\right.\\
	&\quad\quad\quad\quad\left.\times\E\lnorm\E_{k}\left[(\mathcal{G}_{n}^{(k,s)}(w)(c_{s}e_{s}^{T})\mathcal{G}_{n}^{(k)}(w))^{*}e_{(i-1)n+k}\oindicator{\Omega_{n,k,s}\cap\Omega_{n,k}}\right]\rnorm^{4}\right)^{1/2}\\
	&\leq\left(\E\left[\lnorm e_{(i-1)n+k}\rnorm^{4}\lnorm\mathcal{G}_{n}^{(k,s)}(z)\oindicator{\Omega_{n,k,s}}\rnorm^{4}\lnorm\mathcal{D}_{i-b-1}\rnorm^{4}\right]\right.\\
	&\quad\quad\quad\quad\left.\times\E\left[\lnorm(\mathcal{G}_{n}^{(k)}(w))^{*}\oindicator{\Omega_{n,k}}\rnorm^{4}\lnorm e_{s}\rnorm^{4}\left|c_{s}^{*}(\mathcal{G}_{n}^{(k,s)}(w))^{*}e_{(i-1)n+k}\oindicator{\Omega_{n,k,s}}\right|^{4}\right]\right)^{1/2}\\
	&\ll \left(\E\left|c_{s}^{*}(\mathcal{G}_{n}^{(k,s)}(w))^{*}e_{(i-1)n+k}e_{(i-1)n+k}^{T}\mathcal{G}_{n}^{(k,s)}(w)c_{s}\oindicator{\Omega_{n,k,s}}\right|^{2}\right)^{1/2}\\
	&\ll \left(n^{-2}\E\lnorm(\mathcal{G}_{n}^{(k,s)}(w))^{*}e_{(i-1)n+k}e_{(i-1)n+k}^{T}\mathcal{G}_{n}^{(k,s)}(w)\oindicator{\Omega_{n,k,s}}\rnorm^{2}\right)^{1/2}\\
	&\ll n^{-1}.
	\end{align*}
	A very similar argument also shows that \eqref{Equ:RepColS:Equ5} is $o(1)$ as well,
	\begin{align*}
		&\E\lnorm\E_{k}\left[\mathcal{G}_{n}^{(k,s)}(z)\oindicator{\Omega_{n,k,s}\cap\Omega_{n,k}}\right]\mathcal{D}_{i-2}\E_{k}\left[(\mathcal{G}_{n}^{(k)}(w))^{*}\oindicator{\Omega_{n,k}}\right]\right.\\
		&\quad\quad\quad\quad\quad\quad\left.-\E_{k}\left[\mathcal{G}_{n}^{(k)}(z)\oindicator{\Omega_{n,k,s}\cap\Omega_{n,k}}\right]\mathcal{D}_{i-2}\E_{k}\left[(\mathcal{G}_{n}^{(k)}(w))^{*}\oindicator{\Omega_{n,k}}\right]\rnorm^{2}=o(1)
	\end{align*}
	concluding the proof.
\end{Details}
\end{proof}

\section{Tightness}
\label{Sec:Tightness}
In order to extend the finite dimensional convergence proved in Section \ref{Sec:FiniteDimDist} to convergence of the stochastic process $\{\Xi_{n}(z)\}_{z\in\mathcal{C}}$, we must check that the sequence of stochastic processes $\{\Xi_{n}(z)\}_{z\in\mathcal{C}}$ is tight. Namely, recall that we must verify condition \eqref{Equ:TightnessCharacterization} in Theorem \ref{Thm:ConvergenceCharacterization}. To check this condition, it will be helpful to recenter $\Xi_{n}(z)$. 
\begin{Details} 
	there exists a $c>0$ such that 
\[\E\left|\frac{\Xi_{n}(t_{2})-\Xi_{n}(t_{1})}{t_{2}-t_{1}}\right|^{2}\leq c\]
for all $t_{1}$, $t_{2}$, and $n$. To this end, observe that in Section \ref{Sec:FiniteDimDist}, We proved that for any fixed $z_{i}$ on the contour $\mathcal{C}$, $(\Xi_{n}(z_{i}))_{i=1}^{L}$ converges to a mean-zero multivariate Gaussian. Therefore, by choosing only one point of the contour, $z=(1+\delta)e^{0\sqrt{-1}}$, we can conclude that the sequence $\{\Xi_{n}((1+\delta)e^{0\sqrt{-1}})\}$ converges in distribution to a mean-zero Gaussian random variable. Since convergence in distribution implies tightness, the first condition is satisfied.
	\[\Xi_{n}(z)=\tr(\mathcal{G}_{n}(z))\oindicator{\Omega_{n}}-\E[\tr(\mathcal{G}_{n}(z))\oindicator{\Omega_{n}}]=\sum_{k=1}^{n}(\E_{k}-\E_{k-1})[\tr(\mathcal{G}_{n}(z))\oindicator{\Omega_{n}}]\]
	and observe that 
	\[\sum_{k=1}^{n}(\E_{k}-\E_{k-1})[\tr(\mathcal{G}_{n}^{(k)})\oindicator{\Omega_{n,k}}]=0.\]
\end{Details}
Define the modified sequence 
\begin{equation}
\tilde{\Xi}_{n}(z)=\sum_{k=1}^{n}(\E_{k}-\E_{k-1})[(\tr(\mathcal{G}_{n}(z))-\tr(\mathcal{G}_{n}^{(k)}(z)))\oindicator{\Omega_{n}\cap\Omega_{n,k}}]
\label{Equ:Xi_tilde}
\end{equation}
which differs from $\Xi_{n}(z)$ by the fact that we have subtracted the trace of $\mathcal{G}_{n}^{(k)}(z)$ and multiplied by $\oindicator{\Omega_{n,k}}$. We wish to proceed from here working with $\tilde{\Xi}_{n}(z)$ instead of $\Xi_{n}(z)$. We will be justified in doing so after proving the following lemma.
\begin{lemma}
	Let $\Xi_{n}(z)$ be as defined in \eqref{Def:StochasticProcessXi} and $\hat{\Xi}_{n}(z)$ as in \eqref{Equ:Xi_tilde}. Then under the assumptions of Theorem \ref{Thm:LinearizedTruncatedCLT},
	\begin{equation*}
	\E\left|\frac{\Xi_{n}(z)-\Xi_{n}(w)}{z-w}\right|^{2}\leq c +\E\left|\frac{\tilde{\Xi}_{n}(z)-\tilde{\Xi}_{n}(w)}{z-w}\right|^{2}
	\end{equation*}
	for some constant $c>0$ independent of $n$ and of any choice of $z,w$ on the contour $\mathcal{C}$.
\end{lemma}	
\begin{proof}
	\begin{Details} 
		Recall that 
		\[\Xi_{n}(z)=\sum_{k=1}^{n}(\E_{k}-\E_{k-1})[\tr(\mathcal{G}_{n}(z))\oindicator{\Omega_{n}}]\]
		and
		\begin{align*}
		\tilde{\Xi}_{n}(z)&=\sum_{k=1}^{n}(\E_{k}-\E_{k-1})[(\tr(\mathcal{G}_{n}(z))-\tr(\mathcal{G}_{n}^{(k)}(z)))\oindicator{\Omega_{n}\cap\Omega_{n,k}}]\\
		&=\sum_{k=1}^{n}(\E_{k}-\E_{k-1})[\tr(\mathcal{G}_{n}(z))\oindicator{\Omega_{n}\cap\Omega_{n,k}}]
		\end{align*}
		since
		\[\sum_{k=1}^{n}(\E_{k}-\E_{k-1})[\tr(\mathcal{G}_{n}^{(k)})\oindicator{\Omega_{n,k}}]=0.\]
	\end{Details}
	We can see that 
	\begin{align*}
	&\E\left|\frac{\Xi_{n}(z)-\Xi_{n}(w)}{z-w}\right|^{2}\\
	&\ll \E\left|\frac{\Xi_{n}(z)-\Xi_{n}(w)}{z-w}-\frac{\tilde{\Xi}_{n}(z)-\tilde{\Xi}_{n}(w)}{z-w}\right|^{2}+\E\left|\frac{\tilde{\Xi}_{n}(z)-\tilde{\Xi}_{n}(w)}{z-w}\right|^{2}.
	\end{align*}
	Now note that by the resolvent identity \eqref{Equ:ResolventIndentity}, 
	\[\Xi_{n}(z)-\Xi_{n}(w) =  \sum_{k=1}^{n}(\E_{k}-\E_{k-1})[\tr(\mathcal{G}_{n}(z)(w-z)\mathcal{G}_{n}(w))\oindicator{\Omega_{n}}]\]
	\begin{Details}
		\begin{align*}
		\Xi_{n}(z)-\Xi_{n}(w) &=  \sum_{k=1}^{n}(\E_{k}-\E_{k-1})[\tr(\mathcal{G}_{n}(z)(w-z)\mathcal{G}_{n}(w))\oindicator{\Omega_{n}}]\\
		\end{align*}
	\end{Details}
	and 
	\begin{align*}
	&\tilde{\Xi}_{n}(z)-\tilde{\Xi}_{n}(w)\notag\\
	&\quad\quad=\sum_{k=1}^{n}(\E_{k}-\E_{k-1})[\tr(\mathcal{G}_{n}(z)(w-z)\mathcal{G}_{n}(w))\oindicator{\Omega_{n}\cap\Omega_{n,k}}\notag\\
	&\quad\quad\quad\quad\quad\quad\quad\quad\quad\quad\quad-\tr(\mathcal{G}_{n}^{(k)}(z)(w-z)\mathcal{G}_{n}^{(k)}(w))\oindicator{\Omega_{n}\cap\Omega_{n,k}}].
	\end{align*}
	Therefore, by cyclic permutation of the trace and since the covariance terms in a martingale difference sequence are zero, we have 
	\begin{align*}
	&\E\left|\frac{\Xi_{n}(z)-\Xi_{n}(w)}{z-w}-\frac{\tilde{\Xi}_{n}(z)-\tilde{\Xi}_{n}(w)}{z-w}\right|^{2}\\
	&=\E\left|\sum_{k=1}^{n}\left((\E_{k}-\E_{k-1})[\tr(\mathcal{G}_{n}(w)\mathcal{G}_{n}(z))(\oindicator{\Omega_{n}}-\oindicator{\Omega_{n}\cap\Omega_{n,k}})]\right.\right.\\
	&\quad\quad\quad\quad\quad\quad\quad\quad\quad\left.\left.-(\E_{k}-\E_{k-1})[\tr(\mathcal{G}_{n}^{(k)}(z)\mathcal{G}_{n}^{(k)}(w))\oindicator{\Omega_{n}\cap\Omega_{n,k}}]\right)\right|^{2}\\
	&\ll \sum_{k=1}^{n}\left(\E\left|\tr(\mathcal{G}_{n}(w)\mathcal{G}_{n}(z))\oindicator{\Omega_{n}}\oindicator{\Omega_{n,k}^{c}}\right|^{2}\right.\\
	&\quad\quad\quad\quad\quad\quad\quad\quad\quad\left.+\E\left|(\E_{k}-\E_{k-1})[\tr(\mathcal{G}_{n}^{(k)}(z)\mathcal{G}_{n}^{(k)}(w))\oindicator{\Omega_{n,k}}]\right|^{2}\right)\\
	&\ll_{\alpha} n^{3-\alpha}
	\end{align*}
	for any $\alpha>0$ since $(\E_{k}-\E_{k-1})[\tr(\mathcal{G}_{n}^{(k)}(z)\mathcal{G}_{n}^{(k)}(w))\oindicator{\Omega_{n,k}}]=0$. Note that any choice of $\alpha \geq 3$ suffices to show this term is bounded by a constant, concluding the proof.
	\begin{Details}
		 so 
		\begin{align*}
		&\E\left|\frac{\Xi_{n}(z)-\Xi_{n}(w)}{z-w}\right|^{2}\ll c+\E\left|\frac{\tilde{\Xi}_{n}(z)-\tilde{\Xi}_{n}(w)}{z-w}\right|^{2}
		\end{align*}
	\end{Details}
\end{proof}
	
\begin{Details} 
Indeed, with the preceding lemma, bounding 
\[\E\left|\frac{\tilde{\Xi}_{n}(z)-\tilde{\Xi}_{n}(w)}{z-w}\right|^{2}\] 
by a constant will in turn bound 
\[\E\left|\frac{\Xi_{n}(z)-\Xi_{n}(w)}{z-w}\right|^{2}\]
as desired. 
\end{Details} 
The tightness of $\{\Xi_{n}(z)\}_{z\in\mathcal{C}}$ will follow from the following lemma.
\begin{lemma}
	Let $\{\tilde{\Xi}_{n}(z)\}$ be the sequence of stochastic processes defined in \eqref{Equ:Xi_tilde}. It holds that
	\[\E\left|\frac{\tilde{\Xi}_{n}(z)-\tilde{\Xi}_{n}(w)}{z-w}\right|^{2}\leq c\]
	for a constant $c>0$ independent of $n$ and of any choice of $z,w$ on the contour $\mathcal{C}$.
	\label{Lem:LipschitzInequality}
\end{lemma}
\begin{Details}
	Note that since Theorem \ref{Thm:TightnessCharacterization} only requires the existence of constants $\gamma$ and $\alpha$ and a nondecreasing function $F$, Lemma \ref{Lem:LipschitzInequality} will imply Theorem \ref{Thm:TightnessCharacterization} holds with $\alpha=\gamma=2$ and $F$ being the identity function. 
\end{Details}
\begin{proof}
	The idea behind this proof is similar to what was done in the proof of Lemma \ref{Lem:MDSReduction} where we remove columns to achieve independence. First, observe that by definition of $\tilde{\Xi}_{n}(z)$, linearity of trace, and the resolvent identity \eqref{Equ:ResolventIndentity},
	\begin{align}
	&\frac{\tilde{\Xi}_{n}(z)-\tilde{\Xi}_{n}(w)}{z-w}\notag\\
	&=\sum_{k=1}^{n}\frac{(\E_{k}-\E_{k-1})[\tr(\mathcal{G}_{n}(z)-\mathcal{G}_{n}(w)-(\mathcal{G}_{n}^{(k)}(z)-\mathcal{G}_{n}^{(k)}(w)))\oindicator{\Omega_{n}\cap\Omega_{n,k}}]}{z-w}\notag\\
	&=-\sum_{k=1}^{n}(\E_{k}-\E_{k-1})[\tr(\mathcal{G}_{n}(z)\mathcal{G}_{n}(w)-\mathcal{G}_{n}^{(k)}(z)\mathcal{G}_{n}^{(k)}(w))\oindicator{\Omega_{n}\cap\Omega_{n,k}}].\label{Equ:LipschitzInequalityStep1}
	\end{align}
	Now note that
	\begin{align*}
	&(\mathcal{G}_{n}(z)-\mathcal{G}_{n}^{(k)}(z))(\mathcal{G}_{n}(w)-\mathcal{G}_{n}^{(k)}(w))\\
	&\quad=\mathcal{G}_{n}(z)\mathcal{G}_{n}(w)-\mathcal{G}_{n}(z)\mathcal{G}_{n}^{(k)}(w)-\mathcal{G}_{n}^{(k)}(z)\mathcal{G}_{n}(w)+\mathcal{G}_{n}^{(k)}(z)\mathcal{G}_{n}^{(k)}(w)
	\end{align*}
	which implies
	\begin{align*}
	&\mathcal{G}_{n}(z)\mathcal{G}_{n}(w)+\mathcal{G}_{n}^{(k)}(z)\mathcal{G}_{n}^{(k)}(w)\\
	&\quad=(\mathcal{G}_{n}(z)-\mathcal{G}_{n}^{(k)}(z))(\mathcal{G}_{n}(w)-\mathcal{G}_{n}^{(k)}(w))+\mathcal{G}_{n}(z)\mathcal{G}_{n}^{(k)}(w)+\mathcal{G}_{n}^{(k)}(z)\mathcal{G}_{n}(w).
	\end{align*}
	By subtracting $2\mathcal{G}_{n}^{(k)}(z)\mathcal{G}_{n}^{(k)}(w)$ from each side of the previous equality, regrouping,  and applying the resolvent identity \eqref{Equ:ResolventIndentity}, we have
	\begin{align*}
	&\mathcal{G}_{n}(z)\mathcal{G}_{n}(w)-\mathcal{G}_{n}^{(k)}(z)\mathcal{G}_{n}^{(k)}(w)\\
	&=(\mathcal{G}_{n}(z)(U_{k}V_{k}^{T})\mathcal{G}_{n}^{(k)}(z))(\mathcal{G}_{n}(w)(U_{k}V_{k}^{T}\mathcal{G}_{n}^{(k)}(w))\\
	&\quad+(\mathcal{G}_{n}(z)(U_{k}V_{k}^{T})\mathcal{G}_{n}^{(k)}(z))\mathcal{G}_{n}^{(k)}(w)\\
	&\quad+\mathcal{G}_{n}^{(k)}(z)(\mathcal{G}_{n}(w)(U_{k}V_{k}^{T})\mathcal{G}_{n}^{(k)}(w))
	\end{align*}
	where we recall that $U_{k}$ is the $mn\times n$ matrix which contains as its columns $c_{k},c_{n+k},\dots,c_{(m-1)n+k}$ and $V_{k}$ is the $mn\times m$ matrix which contains as its columns $e_{k},e_{n+k},\dots,e_{(m-1)n+k}$. 
	By the Sherman--Morrison--Woodbury formula \eqref{Equ:ShermanMorrisonWoodbury}, we know $\mathcal{G}_{n}(z)U_{k}=$ $\mathcal{G}_{n}^{(k)}(z)U_{k}(I_{m}+V_{k}^{T}\mathcal{G}_{n}^{(k)}(z)U_{k})^{-1}$ $=\mathcal{G}_{n}^{(k)}(z)U_{k}\Delta_{n,k}(z)$ where $\Delta_{n,k}(z):=(I_{m}+V_{k}^{T}\mathcal{G}_{n}^{(k)}(z)U_{k})^{-1}$ provided $I_{m}+V_{k}^{T}\mathcal{G}_{n}^{(k)}(z)U_{k}$ is invertible. Recall, as was done in Section \ref{Sec:FiniteDimDist}, we can guarantee that this matrix is invertible by working on the event $Q_{n,k}$ defined in \eqref{Def:EventQ}. Since $Q_{n,k}$ holds with overwhelming probability by Lemma \ref{Lem:QOverwhelming}, the same argument as in Section \ref{Sec:FiniteDimDist} shows that we can work on this event with error $o_{\alpha}(n^{-\alpha})$ for any $\alpha>0$, so we are justified doing so. Therefore we can continue on the event $\Omega_{n,k}\cap Q_{n,k}$, with
	\begin{align*}
	&\mathcal{G}_{n}(z)\mathcal{G}_{n}(w)-\mathcal{G}_{n}^{(k)}(z)\mathcal{G}_{n}^{(k)}(w)\\
	&=(\mathcal{G}_{n}^{(k)}(z)U_{k}\Delta_{n,k}(z)V_{k}^{T}\mathcal{G}_{n}^{(k)}(z))(\mathcal{G}_{n}^{(k)}(w)U_{k}\Delta_{n,k}(w)V_{k}^{T}\mathcal{G}_{n}^{(k)}(w))\\
	&\quad+(\mathcal{G}_{n}^{(k)}(z)U_{k}\Delta_{n,k}(z)V_{k}^{T}\mathcal{G}_{n}^{(k)}(z))\mathcal{G}_{n}^{(k)}(w)\\
	&\quad+\mathcal{G}_{n}^{(k)}(z)(\mathcal{G}_{n}^{(k)}(w)U_{k}\Delta_{n,k}(w)V_{k}^{T}\mathcal{G}_{n}^{(k)}(w)).
	\end{align*}
	Since $(\blmat{Y}{n}^{(k)}-zI)$ and $(\blmat{Y}{n}^{(k)}-wI)$ commute, we can interchange the order in which we multiply $\mathcal{G}_{n}^{(k)}(z)$ and $\mathcal{G}_{n}^{(k)}(w)$. By this observation and by cyclic permutation of the trace, we have
	\begin{align*}
	&\tr(\mathcal{G}_{n}(z)\mathcal{G}_{n}(w)-\mathcal{G}_{n}^{(k)}(z)\mathcal{G}_{n}^{(k)}(w))\\
	&=\tr\left((\mathcal{G}_{n}^{(k)}(z)U_{k}\Delta_{n,k}(z)V_{k}^{T}\mathcal{G}_{n}^{(k)}(z))(\mathcal{G}_{n}^{(k)}(w)U_{k}\Delta_{n,k}(w)V_{k}^{T}\mathcal{G}_{n}^{(k)}(w))\right)\\
	&\quad+\tr\left((\mathcal{G}_{n}^{(k)}(z)U_{k}\Delta_{n,k}(z)V_{k}^{T}\mathcal{G}_{n}^{(k)}(z))\mathcal{G}_{n}^{(k)}(w)\right)\\
	&\quad+\tr\left(\mathcal{G}_{n}^{(k)}(z)(\mathcal{G}_{n}^{(k)}(w)U_{k}\Delta_{n,k}(w)V_{k}^{T}\mathcal{G}_{n}^{(k)}(w))\right).
	\end{align*}
	Putting all of these observations together, we have shown that 
	\begin{align}
	&\E\left|\frac{\tilde{\Xi}_{n}(z)-\tilde{\Xi}_{n}(w)}{z-w}\right|^{2}\notag\\
	&\quad\quad\ll\sum_{k=1}^{n}\E\left|(\E_{k}-\E_{k-1})[\tr(\mathcal{G}_{n}(z)\mathcal{G}_{n}(w)-\mathcal{G}_{n}^{(k)}(z)\mathcal{G}_{n}^{(k)}(w))\oindicator{\Omega_{n}\cap\Omega_{n,k}}]\right|^{2}\notag\\
	&\quad\quad\leq \sum_{k=1}^{n}\E\left|\tr\left((\mathcal{G}_{n}^{(k)}(z)U_{k}\Delta_{n,k}(z)V_{k}^{T}\mathcal{G}_{n}^{(k)}(z))\right.\right.\notag\\
	&\quad\quad\quad\quad\quad\quad\quad\quad\quad\quad\times\left.\left.(\mathcal{G}_{n}^{(k)}(w)U_{k}\Delta_{n,k}(w)V_{k}^{T}\mathcal{G}_{n}^{(k)}(w))\right)\oindicator{\Omega_{n,k}\cap Q_{n,k}}\right|^{2}\label{Equ:Lem:Tightness:1}\\
	&\quad\quad\quad\quad+\sum_{k=1}^{n}\E\left|\tr\left((\mathcal{G}_{n}^{(k)}(z)U_{k}\Delta_{n,k}(z)V_{k}^{T}\mathcal{G}_{n}^{(k)}(z))\mathcal{G}_{n}^{(k)}(w)\right)\oindicator{\Omega_{n,k}\cap Q_{n,k}}\right|^{2}\label{Equ:Lem:Tightness:2}\\
	&\quad\quad\quad\quad+\sum_{k=1}^{n}\E\left|\tr\left(\mathcal{G}_{n}^{(k)}(z)(\mathcal{G}_{n}^{(k)}(w)U_{k}\Delta_{n,k}(w)V_{k}^{T}\mathcal{G}_{n}^{(k)}(w))\right)\oindicator{\Omega_{n,k}\cap Q_{n,k}}\right|^{2}+O(1).\label{Equ:Lem:Tightness:3}
	\end{align}
	Note that since $\mathcal{G}_{n}(z)$ is no longer present in \eqref{Equ:Lem:Tightness:1}, \eqref{Equ:Lem:Tightness:2}, and \eqref{Equ:Lem:Tightness:3}, we are justified dropping the event $\Omega_{n}$ as well. The $O(1)$ is due to the error from introducing the event $Q_{n,k}$ and dropping the event $\Omega_{n}$. Next, we show that we can replace $\Delta_{n,k}(z)$ and $\Delta_{n,k}(w)$ with $I_{m}$ in \eqref{Equ:Lem:Tightness:1}, \eqref{Equ:Lem:Tightness:2}, and \eqref{Equ:Lem:Tightness:3}. We begin by showing the calculation for term \eqref{Equ:Lem:Tightness:1}. Observe that by cyclic permutation of the trace, on the event $\Omega_{n,k}\cap Q_{n,k}$, we have 
	\begin{align*}
	&\tr\left((\mathcal{G}_{n}^{(k)}(z)U_{k}\Delta_{n,k}(z)V_{k}^{T}\mathcal{G}_{n}^{(k)}(z))(\mathcal{G}_{n}^{(k)}(w)U_{k}\Delta_{n,k}(w)V_{k}^{T}\mathcal{G}_{n}^{(k)}(w))\right)\\
	&\quad-\tr\left((\mathcal{G}_{n}^{(k)}(z)U_{k}V_{k}^{T}\mathcal{G}_{n}^{(k)}(z))(\mathcal{G}_{n}^{(k)}(w)U_{k}V_{k}^{T}\mathcal{G}_{n}^{(k)}(w))\right)\\ 
	&=\tr\left((\Delta_{n,k}(z)-I_{m})V_{k}^{T}\mathcal{G}_{n}^{(k)}(z)\mathcal{G}_{n}^{(k)}(w)U_{k}\Delta_{n,k}(w)V_{k}^{T}\mathcal{G}_{n}^{(k)}(w)\mathcal{G}_{n}^{(k)}(z)U_{k}\right)\\
	&\quad+\tr\left(V_{k}^{T}\mathcal{G}_{n}^{(k)}(z)\mathcal{G}_{n}^{(k)}(w)U_{k}(\Delta_{n,k}(w)-I_{m})V_{k}^{T}\mathcal{G}_{n}^{(k)}(w)\mathcal{G}_{n}^{(k)}(z)U_{k}\right).
	\end{align*}
	We use the generalized H\"{o}lders inequality to break the above expression into pieces which have bounded expectation. By bounding the trace by the rank times the norm, we have
	\begin{align}
	&\E\left|\tr\left((\mathcal{G}_{n}^{(k)}(z)U_{k}\Delta_{n,k}(z)V_{k}^{T}\mathcal{G}_{n}^{(k)}(z))(\mathcal{G}_{n}^{(k)}(w)U_{k}\Delta_{n,k}(w)V_{k}^{T}\mathcal{G}_{n}^{(k)}(w))\right)\oindicator{\Omega_{n,k}\cap Q_{n,k}}\right.\notag\\
	&\quad-\left.\tr\left((\mathcal{G}_{n}^{(k)}(z)U_{k}V_{k}^{T}\mathcal{G}_{n}^{(k)}(z))(\mathcal{G}_{n}^{(k)}(w)U_{k}V_{k}^{T}\mathcal{G}_{n}^{(k)}(w))\right)\oindicator{\Omega_{n,k}\cap Q_{n,k}}\right|^{2}\notag\\
	&\ll \left(\E\lnorm(\Delta_{n,k}(z)-I_{m})\oindicator{\Omega_{n,k}\cap Q_{n,k}}\rnorm^{4}\right)^{1/2}\left(\E\lnorm V_{k}^{T}\mathcal{G}_{n}^{(k)}(z)\mathcal{G}_{n}^{(k)}(w)U_{k}\oindicator{\Omega_{n,k}}\rnorm^{8}\right)^{1/4}\label{Equ:Lem:Tightness:4}\\
	&\quad\quad\times\left(\E\lnorm\Delta_{n,k}(w)\oindicator{Q_{n,k}}\rnorm^{16}\right)^{1/8}\left(\E\lnorm V_{k}^{T}\mathcal{G}_{n}^{(k)}(w)\mathcal{G}_{n}^{(k)}(z)U_{k}\oindicator{\Omega_{n,k}}\rnorm^{16}\right)^{1/8}\label{Equ:Lem:Tightness:5}\\
	&\quad+\left(\E\lnorm V_{k}^{T}\mathcal{G}_{n}^{(k)}(z)\mathcal{G}_{n}^{(k)}(w) U_{k}\oindicator{\Omega_{n,k}}\rnorm^{4}\right)^{1/2}\label{Equ:Lem:Tightness:6}\\
	&\quad\quad\times\left(\E\lnorm (\Delta_{n,k}(w)-I_{m})\oindicator{\Omega_{n,k}\cap Q_{n,k}}\rnorm^{8}\right)^{1/4}\left(\E\lnorm V_{k}^{T}\mathcal{G}_{n}^{(k)}(w)\mathcal{G}_{n}^{(k)}(z)U_{k}\oindicator{\Omega_{n,k}}\rnorm^{8}\right)^{1/4}\label{Equ:Lem:Tightness:7}
	\end{align}
	We bound each expectation in \eqref{Equ:Lem:Tightness:4}, \eqref{Equ:Lem:Tightness:5}, \eqref{Equ:Lem:Tightness:6}, and \eqref{Equ:Lem:Tightness:7}. We start by bounding the expectation of terms in which two resolvents appear. By the same argument as in Lemma \ref{Lem:BoundingVGU}, we can show that for $p\geq 4$
	\[\E\lnorm V_{k}^{T}\mathcal{G}_{n}^{(k)}(w)\mathcal{G}_{n}^{(k)}(z)U_{k}\oindicator{\Omega_{n,k}}\rnorm^{p}\ll_{p} n^{-\varepsilon(p-4)-2}\]
	which shows 
	\[\left(\E\lnorm V_{k}^{T}\mathcal{G}_{n}^{(k)}(w)\mathcal{G}_{n}^{(k)}(z)U_{k}\oindicator{\Omega_{n,k}}\rnorm^{16}\right)^{1/8}\ll n^{-3/2\varepsilon-1/4},\]
	\[\left(\E\lnorm V_{k}^{T}\mathcal{G}_{n}^{(k)}(z)\mathcal{G}_{n}^{(k)}(w) U_{k}\oindicator{\Omega_{n,k}}\rnorm^{4}\right)^{1/2}\ll n^{-1},\]
	and
	\[\left(\E\lnorm V_{k}^{T}\mathcal{G}_{n}^{(k)}(z)\mathcal{G}_{n}^{(k)}(w)U_{k}\oindicator{\Omega_{n,k}}\rnorm^{8}\right)^{1/4}\ll n^{-\varepsilon-1/2}.\]
	Since the second term in \eqref{Equ:Lem:Tightness:7} differs from the last line above only by the order in which we multiply resolvents, the same bound holds for the second term in \eqref{Equ:Lem:Tightness:7}. Now we bound terms involving $\Delta_{n,k}$. Recall that $\Delta_{n,k}$ is bounded by a constant almost surely on $Q_{n,k}$ so we need only to bound the expectations involving $\Delta_{n,k}(z)-I_{m}$ in terms \eqref{Equ:Lem:Tightness:4} and \eqref{Equ:Lem:Tightness:7}. Recall the expansion from \eqref{Equ:ExpansionOfMatrixErrorTerm} can be iterated to get 
	\[\Delta_{n,k}(z)=I_{m}-(V_{k}^{T}\mathcal{G}_{n}^{(k)}(z)U_{k})+(V_{k}^{T}\mathcal{G}_{n}^{(k)}(z)U_{k})^{2}\Delta_{n,k}(z).\]
	Using this fact, Lemma \ref{Lem:BoundingVGU}, and the Cauchy--Schwarz inequality, we get
	\begin{align*}
	&\left(\E\lnorm(\Delta_{n,k}(z)-I_{m})\oindicator{\Omega_{n,k}\cap Q_{n,k}}\rnorm^{4}\right)^{1/2}\\
	& \ll \left(\E\lnorm V_{k}^{T}\mathcal{G}_{n}^{(k)}(z)U_{k}\oindicator{\Omega_{n,k}\cap Q_{n,k}}\rnorm^{4}+\E\lnorm(V_{k}^{T}\mathcal{G}_{n}^{(k)}(z)U_{k})^{2}\Delta_{n,k}(z)\oindicator{\Omega_{n,k}\cap Q_{n,k}}\rnorm^{4}\right)^{1/2}\\
	&\ll n^{-2\varepsilon-1}+n^{-1}
	\end{align*}
	and
	\begin{align*}
	&\left(\E\lnorm(\Delta_{n,k}(w)-I_{m})\oindicator{\Omega_{n,k}\cap Q_{n,k}}\rnorm^{8}\right)^{1/4}\\
	& \ll \left(\E\lnorm V_{k}^{T}\mathcal{G}_{n}^{(k)}(z)U_{k}\oindicator{\Omega_{n,k}\cap Q_{n,k}}\rnorm^{8}+\E\lnorm(V_{k}^{T}\mathcal{G}_{n}^{(k)}(z)U_{k})^{2}\Delta_{n,k}(z)\oindicator{\Omega_{n,k}\cap Q_{n,k}}\rnorm^{8}\right)^{1/4}\\
	&\ll n^{-\varepsilon-1/2}+n^{-3\varepsilon -1/2.}
	\end{align*}
	Therefore, combining these bounds, we can bound \eqref{Equ:Lem:Tightness:4}, \eqref{Equ:Lem:Tightness:5}, \eqref{Equ:Lem:Tightness:6} and \eqref{Equ:Lem:Tightness:7} by $O(n^{-7/4})$. This concludes the argument to show that we may replace all $\Delta_{n,k}(z)$ and $\Delta_{n,k}(w)$ in term \eqref{Equ:Lem:Tightness:1} with $I_{m}$. A very similar argument shows that, for terms \eqref{Equ:Lem:Tightness:2} and \eqref{Equ:Lem:Tightness:3},
	\begin{equation*}
	\E\left|\tr\left((\mathcal{G}_{n}^{(k)}(z)U_{k}\Delta_{n,k}(z)V_{k}^{T}\mathcal{G}_{n}^{(k)}(z))\mathcal{G}_{n}^{(k)}(w)\right)-\tr\left((\mathcal{G}_{n}^{(k)}(z)U_{k}V_{k}^{T}\mathcal{G}_{n}^{(k)}(z))\mathcal{G}_{n}^{(k)}(w)\right)\right|^{2}
	\end{equation*}
	and 
	\begin{equation*}
	\E\left|\tr\left(\mathcal{G}_{n}^{(k)}(z)(\mathcal{G}_{n}^{(k)}(w)U_{k}\Delta_{n,k}(w)V_{k}^{T}\mathcal{G}_{n}^{(k)}(w))\right)-\tr\left(\mathcal{G}_{n}^{(k)}(z)(\mathcal{G}_{n}^{(k)}(w)U_{k}V_{k}^{T}\mathcal{G}_{n}^{(k)}(w))\right)\right|^{2}
	\end{equation*}
	are both $O(n^{-2})$. Ergo, we need to show only that the expression
	\begin{align}
	&\E\left|\tr\left((\mathcal{G}_{n}^{(k)}(z)U_{k}V_{k}^{T}\mathcal{G}_{n}^{(k)}(z))(\mathcal{G}_{n}^{(k)}(w)U_{k}V_{k}^{T}\mathcal{G}_{n}^{(k)}(w))\right)\right.\notag\\
	&\quad+\tr\left((\mathcal{G}_{n}^{(k)}(z)U_{k}V_{k}^{T}\mathcal{G}_{n}^{(k)}(z))\mathcal{G}_{n}^{(k)}(w)\right)+\left.\tr\left(\mathcal{G}_{n}^{(k)}(z)(\mathcal{G}_{n}^{(k)}(w)U_{k}V_{k}^{T}\mathcal{G}_{n}^{(k)}(w))\right)\right|^{2}\notag\\
	&\ll \E\left|\tr\left(V_{k}^{T}\mathcal{G}_{n}^{(k)}(z)\mathcal{G}_{n}^{(k)}(w)U_{k}\right)^{2}\right|^{2}\label{Equ:Lem:Tightness:8}\\
	&\quad+\E\left|\tr\left(V_{k}^{T}\mathcal{G}_{n}^{(k)}(w)(\mathcal{G}_{n}^{(k)}(z))^{2}U_{k}\right) \right|^{2}\label{Equ:Lem:Tightness:9}\\
	&\quad+\E\left|\tr\left(V_{k}^{T}\mathcal{G}_{n}^{(k)}(z)(\mathcal{G}_{n}^{(k)}(w))^{2}U_{k}\right)\right|^{2}\label{Equ:Lem:Tightness:10}
	\end{align}
	is bounded by $O(n^{-1})$, where we cyclically permuted the trace and reordered the product of resolvents again. We will bound each term separately. First consider term \eqref{Equ:Lem:Tightness:8}, and observe that
	\begin{align*}
	&\E\left|\tr\left((V_{k}^{T}\mathcal{G}_{n}^{(k)}(z)\mathcal{G}_{n}^{(k)}(w)U_{k})^{2}\right)\oindicator{\Omega_{n,k}}\right|^{2}\\
	&\quad\ll \E\lnorm U_{k}^{*}(\mathcal{G}_{n}^{(k)}(w))^{*}(\mathcal{G}_{n}^{(k)}(z))^{*} V_{k}V_{k}^{T}\mathcal{G}_{n}^{(k)}(z)\mathcal{G}_{n}^{(k)}(w)U_{k}\oindicator{\Omega_{n,k}}\rnorm^{2}
	\end{align*}
	Since this matrix is $m\times m$, if we can bound an arbitrary entry uniformly, then we can bound the norm. We now wish to bound
	\[\max_{1\leq i,j\leq m}\E\left| c_{(i-1)n+k}^{T}(\mathcal{G}_{n}^{(k)}(w))^{*}(\mathcal{G}_{n}^{(k)}(z))^{*} V_{k}V_{k}^{T}\mathcal{G}_{n}^{(k)}(z)\mathcal{G}_{n}^{(k)}(w)c_{(j-1)n+k}\oindicator{\Omega_{n,k}}\right|^{2}.\]
	Note that $(\mathcal{G}_{n}^{(k)}(w))^{*}(\mathcal{G}_{n}^{(k)}(z))^{*}V_{k}V_{k}^{T}\mathcal{G}_{n}^{(k)}(z)\mathcal{G}_{n}^{(k)}(w)$ is independent of the $k$th column of each block, and hence is independent from $c_{(i-1)n+k}$ and $c_{(j-1)n+k}$. It is also rank at most $m$ and it is Hermitian positive definite. Then by Lemma \ref{Lem:conjLessThanConstant} (when $i=j$) or Lemma \ref{lem:BilinearFormWithDifferentVectors} (when $i\neq j$), we have 
	\begin{align*}
	&\E\left| c_{(i-1)n+k}^{*}(\mathcal{G}_{n}^{(k)}(w))^{*}(\mathcal{G}_{n}^{(k)}(z))^{*}V_{k}V_{k}^{T}\mathcal{G}_{n}^{(k)}(z)\mathcal{G}_{n}^{(k)}(w)c_{(j-1)n+k}\oindicator{\Omega_{n,k}}\right|^{2}\\
	&\quad\ll n^{-2}\E\left[\tr\left(((\mathcal{G}_{n}^{(k)}(w))^{*}(\mathcal{G}_{n}^{(k)}(z))^{*}V_{k}V_{k}^{T}\mathcal{G}_{n}^{(k)}(z)\mathcal{G}_{n}^{(k)}(w)\oindicator{\Omega_{n,k}})^{2}\right)\right]\\
	&\quad\ll n^{-2}
	\end{align*}
	where we bounded an arbitrary element of $V_{k}^{T}\mathcal{G}_{n}^{(k)}(z)\mathcal{G}_{n}^{(k)}(w)(\mathcal{G}_{n}^{(k)}(w))^{*}(\mathcal{G}_{n}^{(k)}(z))^{*}V_{k}$ by a constant on the event $\Omega_{n,k}$. This concludes the argument for term \eqref{Equ:Lem:Tightness:8}. Since terms \eqref{Equ:Lem:Tightness:9} and \eqref{Equ:Lem:Tightness:10} are symmetric in $z$ and $w$, the argument will be the same for both terms. We show the argument for \eqref{Equ:Lem:Tightness:9}. Observe that
	\begin{align*}
	&\E\left|\tr\left(V_{k}^{T}\mathcal{G}_{n}^{(k)}(w)(\mathcal{G}_{n}^{(k)}(z))^{2}U_{k}\right)\oindicator{\Omega_{n,k}}\right|^{2}\\
	&\quad \ll \max_{1 \leq i \leq m}\E\left|e_{(i-1)n+k}^{T}\mathcal{G}_{n}^{(k)}(w)(\mathcal{G}_{n}^{(k)}(z))^{2}c_{(i-1)n+k}\oindicator{\Omega_{n,k}}\right|^{2}\\
	&\quad \ll \max_{1 \leq i \leq m}\E\left[c_{(i-1)n+k}^{*}(\mathcal{G}_{n}^{(k)}(z))^{2*}(\mathcal{G}_{n}^{(k)}(w))^{*}e_{(i-1)n+k}\right.\\
	&\quad\quad\quad\quad\quad\quad\quad\quad\quad\quad\quad\quad\left.\times e_{(i-1)n+k}^{T}\mathcal{G}_{n}^{(k)}(w)(\mathcal{G}_{n}^{(k)}(z))^{2}c_{(i-1)n+k}\oindicator{\Omega_{n,k}}\right]\\
	&\quad \ll n^{-1}.
	\end{align*}
	This concludes the argument for \eqref{Equ:Lem:Tightness:9} and the proof of Lemma \ref{Lem:LipschitzInequality}. 
	\begin{Details} 
		The argument for the third term will be the same as the argument for the second. Therefore, we have shown that
		\begin{align*}
		&\E\left|\frac{\Xi_{n}(z)-\Xi_{n}(w)}{z-w}\right|^{2}\\
		&\ll \sum_{k=1}^{n} \E\left|\tr\left((V_{k}^{T}\mathcal{G}_{n}^{(k)}(z)\mathcal{G}_{n}^{(k)}(w)U_{k})^{2}\right)\oindicator{\Omega_{n,k}}\right|^{2}\\
		&\quad\quad\quad +\E\left|\tr\left(V_{k}^{T}\mathcal{G}_{n}^{(k)}(w)(\mathcal{G}_{n}^{(k)}(z))^{2}U_{k}\right)\oindicator{\Omega_{n,k}}\right|^{2}\\
		&\quad\quad\quad+\E\left|\tr\left(V_{k}^{T}\mathcal{G}_{n}^{(k)}(z)(\mathcal{G}_{n}^{(k)}(w))^{2}U_{k}\right)\oindicator{\Omega_{n,k}}\right|^{2}+O(1)\\
		&\ll_{m}\sum_{k=1}^{n}( n^{-2} + n^{-1})+O(1)\\
		&\ll_{m}1.
		\end{align*}
		This concludes the proof of the lemma.
	\end{Details}
	\end{proof}

\appendix
\section{Truncation Arguments}
\label{Sec:AppendexTruncation}
This section is devoted to the proof of Lemma \ref{Lem:TruncateLeveln}.
\begin{proof}[Proof of Lemma \ref{Lem:TruncateLeveln}]
	First, we prove property \ref{item:Lem:TruncateLeveln:i}. Observe that
	\begin{align*}
	1&=\Var(\xi)\\
	&=\E[\xi^{2}\indicator{|\xi|\leq n^{1/2-\varepsilon}}]+\E[\xi^{2}\indicator{|\xi|> n^{1/2-\varepsilon}}]\\
	&=\Var(\tilde{\xi})+\left(\E[\xi\indicator{|\xi|\leq n^{1/2-\varepsilon}}]\right)^{2}+\E[\xi^{2}\indicator{|\xi|> n^{1/2-\varepsilon}}].
	\end{align*}
Also observe that 
	\begin{equation*}
	0=\E[\xi]=\E[\xi\indicator{|\xi|\leq n^{1/2-\varepsilon}}]+\E[\xi\indicator{|\xi|> n^{1/2-\varepsilon}}]\\
	\end{equation*}
which implies $\left|\E[\xi\indicator{|\xi|\leq n^{1/2-\varepsilon}}]\right|=\left|\E[\xi\indicator{|\xi|> n^{1/2-\varepsilon}}]\right|.$ Hence
	\begin{align*}
	|1-\Var(\tilde{\xi})|&=\left(\E[\xi\indicator{|\xi|\leq n^{1/2-\varepsilon}}]\right)^{2}+\E[\xi^{2}\indicator{|\xi|> n^{1/2-\varepsilon}}]\\
	&=\left|\E[\xi\indicator{|\xi|> n^{1/2-\varepsilon}}]\right|^{2}+\E[\xi^{2}\indicator{|\xi|> n^{1/2-\varepsilon}}]\\
	&\leq 2\E\left[\frac{|\xi|^{4}}{n^{1-2\varepsilon}}\indicator{|\xi|> n^{1/2-\varepsilon}}\right]\\
	&=o(n^{-1-2\varepsilon}).
	\end{align*}
	Next we move onto \ref{item:Lem:TruncateLeveln:ii}. By construction, $\E[\hat{\xi}]=0$ and $\Var(\hat{\xi})=1$ provided $n$ is sufficiently large. By part \ref{item:Lem:TruncateLeveln:i},
	\begin{equation*}
	1-\frac{C}{n^{1+2\varepsilon}}\leq \var(\tilde{\xi})
	\end{equation*}
	for some constant $C>0$ so choosing $N_{0}>\left(\frac{4C}{3}\right)^{1/(1+2\varepsilon)}$ ensures that $\frac{1}{4}\leq\Var(\tilde{\xi})$, which gives $2\geq\left(\Var(\tilde{\xi})\right)^{-1/2}$ for $n>N_{0}$. With such an $n>N_{0}$, 
	\begin{align*}
	\left|\hat{\xi}\right|&=\left|\frac{\xi\indicator{|\xi|\leq n^{1/2-\varepsilon}}-\E\left[\xi\indicator{|\xi|\leq n^{1/2-\varepsilon}}\right]}{\sqrt{\Var(\tilde{\xi})}}\right|\\
	&\leq 2\left|\xi\indicator{|\xi|\leq n^{1/2-\varepsilon}}\right|+2\left|\E\left[\xi\indicator{|\xi|\leq n^{1/2-\varepsilon}}\right]\right|\\
	&\leq 4n^{1/2-\varepsilon}
	\end{align*}
	almost surely. For part \ref{item:Lem:TruncateLeveln:iii}, we have
	\begin{align*}
	\E|\hat{\xi}|^{4}&=\E\left|\frac{\xi\indicator{|\xi|\leq n^{1/2-\varepsilon}}-\E\left[\xi\indicator{|\xi|\leq n^{1/2-\varepsilon}}\right]}{\sqrt{\Var(\tilde{\xi})}}\right|^{4}\\
	&\leq 2^{4}\E\left|\xi\indicator{|\xi|\leq n^{1/2-\varepsilon}}-\E\left[\xi\indicator{|\xi|\leq n^{1/2-\varepsilon}}\right]\right|^{4}\\
	&\leq 2^{8}\E\left|\xi\right|^{4}
	\end{align*} 
	completing the proof of the claim.
\end{proof}

\section{Largest and Smallest Singular Values}
\label{Sec:AppendexEvents}
In this section, we consider events concerning the largest and smallest singular values for the random matrices appearing in this paper. These results are included as an appendix because the methods used to prove them are slight modifications of those in \cite{COW,N,OR}. In order to prove these results, we need to introduce an intermediate truncation of the matrices. Specifically, let $\xi_{1},\xi_{2},\dots \xi_{m}$ be real-valued random variables each having mean zero, variance one, and finite $4+\tau$ moment for some $\tau>0$. Let $X_{n,1}X_{n,2},\dots X_{n,m}$ be independent iid $n\times n$ random matrices with atom random variables $\xi_{1},\xi_{2},\dots \xi_{m}$ respectively. For a fixed $\varepsilon>0$, and for each $1\leq k\leq m$, define truncated random variables (at $n^{1/2-\varepsilon}$) $\tilde{\xi}_{k}$ and $\hat{\xi}_{k}$ as in $\eqref{Equ:TruncateLeveln}$. Also define truncated matrices $\tilde{X}_{n,k}$ and $\hat{X}_{n,k}$ as in \eqref{def:Xtilde} and \eqref{def:Xhat} respectively. Define the linearized truncated matrix $\blmat{Y}{n}$ as in \eqref{Def:Y_n}. Also recall that $P_{n}=n^{-m/2}X_{n,1}X_{n,2}\cdots X_{n,m}$ and $\hat{P}_{n}=n^{-m/2}\hat{X}_{n,1}\hat{X}_{n,2}\cdots \hat{X}_{n,m}.$

Let $X$ be an $n\times n$ random matrix filled with iid copies of a random variable $\xi$ which has mean zero, unit variance, and finite $4+\tau$ moment. For a fixed constant $L>0$, define matrices $\mathring{X}$ and $\check{X}$ to be the $n\times n$ matrices with entries defined by 
\begin{equation}
\mathring{X}_{(i,j)}:=X_{(i,j)}\indicator{|X_{(i,j)}|\leq L/\sqrt{2}}-\E\left[X_{(i,j)}\indicator{|X_{(i,j)}|\leq L/\sqrt{2}}\right]\label{def:XLevelLIntermediate}
\end{equation}
and
\begin{equation}
\check{X}_{(i,j)}:= \frac{\mathring{X}_{(i,j)}}{\sqrt{\Var(\mathring{X}_{(i,j)})}}
\label{def:XLevelL}
\end{equation}
for $1 \leq i,j \leq n$. Define $\mathring{X}_{n,1},\mathring{X}_{n,2},\dots\mathring{X}_{n,m}$ and  $\check{X}_{n,1},\check{X}_{n,2},\dots\check{X}_{n,m}$ as in \eqref{def:XLevelLIntermediate} and \eqref{def:XLevelL} respectively. Finally, define the linearized truncated matrix
\begin{equation}
\check{\mathcal{Y}}_{n} :=n^{-1/2}\left[\begin{array}{ccccc}
0 & \check{X}_{n,1} & 0 & \cdots & 0\\
0 & 0 & \check{X}_{n,2} & \dots & 0\\
\vdots & \vdots & \vdots & \ddots & \vdots\\
0 & 0 & 0 & \cdots & \check{X}_{n,m-1}\\
\check{X}_{n,m} & 0 & 0 & \dots & 0
\end{array}\right].
\label{Def:Y_nLevelL}
\end{equation}

\begin{lemma}
	\label{Lem:Omega_nOverwhelming}
	Fix $\varepsilon>0$. For a fixed integer $m>0$, let $\xi_{1},\xi_{2},\dots \xi_{m}$ be real-valued random variables each mean zero, variance one, and finite $4+\tau$ moment for some $\tau>0$. Let $\hat{X}_{n,1},\hat{X}_{n,2},\dots,\hat{X}_{n,m}$ be independent iid random matrices with atom variables as defined in \eqref{def:Xhat}, and define $\blmat{Y}{n}$ as in \eqref{Def:Y_n}. For every $\delta>0$, there exists a constant $c>0$ depending only on $\delta$ such that
	\[\inf_{|z|> 1+\delta/2}s_{mn}\left(\blmat{Y}{n}-zI\right)\geq c\]
	with overwhelming probability. 
\end{lemma}
\begin{proof}
	Fix $\delta>0$ and define $\check{\mathcal{Y}}_{n}$ as in \eqref{Def:Y_nLevelL}. By \cite[Lemma 8.1]{COW}, which is based on techniques in \cite{N,N2}, we know that there exists a constant $c'>0$ which depends only on $\delta$ such that $\inf_{|z|> 1+\delta/2}s_{mn}\left(\check{\mathcal{Y}}_{n}-zI\right)\geq c'$ with overwhelming probability. Note that by Weyl's inequality \eqref{Equ:weyl}, 
	\begin{equation}\label{Equ:Lem:Omega_nOverwhelming}
	\sup_{z\in\mathcal{C}}\left|s_{mn}\left(\check{\mathcal{Y}}_{n}-zI\right)-s_{mn}\left(\blmat{Y}{n}-zI\right)\right|\leq \lnorm \check{\mathcal{Y}}_{n}-\blmat{Y}{n}\rnorm\leq \max_{1 \leq k \leq m}\frac{1}{\sqrt{n}}\lnorm \check{X}_{n,k}-\hat{X}_{n,k}\rnorm.
	\end{equation} 
	\begin{Details}
		Since this is uniform in $z$, it will be sufficient to prove that 
		\[\limsup_{n\rightarrow\infty}\max_{1 \leq k \leq m}\frac{1}{\sqrt{n}}\lnorm \check{X}_{n,k}-\hat{X}_{n,k}\rnorm <\frac{c'}{2}.\]
	\end{Details}
	Focusing on an arbitrary value of $k$, we have
	\begin{equation*}
	\frac{1}{\sqrt{n}}\lnorm \check{X}_{n,k}-\hat{X}_{n,k}\rnorm  \leq \frac{1}{\sqrt{n}}\lnorm \frac{\mathring{X}_{n,k}}{\sqrt{\Var((\mathring{X}_{n,k})_{(i,j)})}}-\frac{\tilde{X}_{n,k}}{\sqrt{\Var((\tilde{X}_{n,k})_{(i,j)})}}\rnorm\\
	\end{equation*}
	for any $1\leq i,j\leq n$. Observe that 
	\begin{equation*}
	\frac{1}{\sqrt{n}}\lnorm \frac{\mathring{X}_{n,k}}{\sqrt{\Var((\mathring{X}_{n,k})_{(i,j)})}}- \mathring{X}_{n,k}\rnorm=\frac{1}{\sqrt{n}}\lnorm \frac{\mathring{X}_{n,k}\left(1-\sqrt{\Var((\mathring{X}_{n,k})_{(i,j)})}\right)}{\sqrt{\Var((\mathring{X}_{n,k})_{(i,j)})}}\rnorm.
	\end{equation*}
	By \cite[Lemma 7.1]{COW}, $\left(\Var((\mathring{X}_{n,k})_{(i,j)})\right)^{-1/2}\leq 2$ for $L$ sufficiently large. Additionally, an argument similar to that of \cite[Lemma 7.1]{COW} shows that\\ $\left|1-\sqrt{\Var((\mathring{X}_{n,k})_{(i,j)})}\right|\leq \frac{C}{L^{2}}$ for any $1\leq i,j\leq n$ and some constant $C>0$. Therefore by \cite[Theorem 1.4]{Tout}, for $L$ sufficiently large,
	\[\frac{1}{\sqrt{n}}\lnorm \frac{\mathring{X}_{n,k}}{\sqrt{\Var((\mathring{X}_{n,k})_{(i,j)})}}- \mathring{X}_{n,k}\rnorm\leq \frac{C}{L^{2}\sqrt{n}}\lnorm \frac{\mathring{X}_{n,k}}{\sqrt{\Var((\mathring{X}_{n,k})_{(i,j)})}}\rnorm \leq \frac{c'}{16}\]
	with overwhelming probability. Similarly, 
	\begin{equation*}
	\frac{1}{\sqrt{n}}\lnorm \frac{\tilde{X}_{n,k}}{\sqrt{\Var((\tilde{X}_{n,k})_{(i,j)})}}- \tilde{X}_{n,k}\rnorm=\frac{1}{\sqrt{n}}\lnorm \frac{\tilde{X}_{n,k}\left(1-\sqrt{\Var((\tilde{X}_{n,k})_{(i,j)})}\right)}{\sqrt{\Var((\tilde{X}_{n,k})_{(i,j)})}}\rnorm.
	\end{equation*}
	By the arguments to prove part \ref{item:Lem:TruncateLeveln:ii} of Lemma \ref{Lem:TruncateLeveln}, $\left(\Var((\tilde{X}_{n,k})_{(i,j)})\right)^{-1/2}\leq 2$ for $n$ sufficiently large. Also, by part \ref{item:Lem:TruncateLeveln:i} of Lemma \ref{Lem:TruncateLeveln}, we can show that $\left|1-\sqrt{\Var((\tilde{X}_{n,k})_{(i,j)})}\right|=o(n^{-1+2\varepsilon})$. Therefore by \cite[Theorem 5.9]{BSbook},
	\[\frac{1}{\sqrt{n}}\lnorm \frac{\tilde{X}_{n,k}}{\sqrt{\Var((\tilde{X}_{n,k})_{(i,j)})}}- \tilde{X}_{n,k}\rnorm =o(n^{-1-2\varepsilon}) \frac{1}{\sqrt{n}}\lnorm \tilde{X}_{n,k}\rnorm \leq \frac{c'}{16}\]
	with overwhelming probability. Ergo, by the triangle inequality, for $L$ sufficiently large,
	\begin{align}
	\frac{1}{\sqrt{n}}\lnorm \check{X}_{n,k}-\hat{X}_{n,k}\rnorm&\leq \frac{1}{\sqrt{n}}\lnorm \frac{\mathring{X}_{n,k}}{\sqrt{\Var((\mathring{X}_{n,k})_{(i,j)})}}-\frac{\tilde{X}_{n,k}}{\sqrt{\Var((\tilde{X}_{n,k})_{(i,j)})}}\rnorm\notag\\
	&\leq \frac{c'}{8}+ \frac{1}{\sqrt{n}}\lnorm \mathring{X}_{n,k}-\tilde{X}_{n,k}\rnorm\label{Equ:Lem:OmeganOverwhelming:1}
	\end{align} 
	with overwhelming probability. 
	
	Now, recall that the entries of $\mathring{X}_{n,k}$ are truncated at level $L$ for a fixed $L>0$ so for sufficiently large $n$, $L\leq n^{1/2-\varepsilon}$. Note that if all entries are less than $L$ in absolute value, then the entries in $\mathring{X}_{n,k}$ and $\tilde{X}_{n}$ agree. Similarly, if all entries are greater than $n^{1/2-\varepsilon}$ then the entries in $\mathring{X}_{n,k}$ and $\tilde{X}_{n}$ agree. Ergo, we need only consider the case when there exists some entries $1\leq i,j\leq n$ such that $L\leq|(\tilde{X}_{n,k})_{i,j}|\leq n^{1/2-\varepsilon}$. For each $1\leq k\leq m$, define the random variables
	\begin{equation*}
	\dot{\xi}_{k} := \xi_{k}\indicator{L\leq |\xi_{k}|\leq n^{1/2-\varepsilon}}-\E\left[\xi_{k}\indicator{L\leq |\xi_{k}|\leq n^{1/2-\varepsilon}}\right]
	\end{equation*}
	and define $\dot{X}_{n,k}$ to be the matrix with entries
	\[(\dot{X}_{n,k})_{(i,j)} := (X_{n,k})_{(i,j)}\indicator{L\leq |(X_{n,k})_{(i,j)}|\leq n^{1/2-\varepsilon}}-\E\left[(X_{n,k})_{(i,j)}\indicator{L\leq |(X_{n,k})_{(i,j)}|\leq n^{1/2-\varepsilon}}\right].\]
	for $1\leq i,j \leq n$. Note that the definitions of $\dot{\xi}$ and $\dot{X}_{n,k}$ differs from the definitions in Section \ref{Sec:Reduction}. We will use the definition given in this appendix for the remainder of this proof. We can write
	\begin{equation*}
	\frac{1}{\sqrt{n}}\lnorm \mathring{X}_{n,k}-\tilde{X}_{n,k}\rnorm  = \frac{1}{\sqrt{n}}\lnorm \dot{X}_{n,k}\rnorm.\\
	\end{equation*}
	By \cite[Lemma 5.9]{BSbook}, for $L$ sufficiently large
	\begin{equation} 
	\frac{1}{\sqrt{n}}\lnorm \dot{X}_{n,k}\rnorm \leq \frac{c'}{8}\label{Equ:Lem:OmeganOverwhelming:2}
	\end{equation}
	with overwhelming probability. Thus, by choosing $L$ large enough to satisfy both conditions, by \eqref{Equ:Lem:OmeganOverwhelming:1} and \eqref{Equ:Lem:OmeganOverwhelming:2}, 
	\[\max_{1 \leq k \leq m}\frac{1}{\sqrt{n}}\lnorm \check{X}_{n,k}-\hat{X}_{n,k}\rnorm <\frac{c'}{4}\]
	with overwhelming probability. By recalling \eqref{Equ:Lem:Omega_nOverwhelming}, this implies that, for $L$ sufficiently large,
	\[\inf_{|z|>1+\delta/2}s_{mn}\left(\blmat{Y}{n}-zI\right)\geq c\]
	with overwhelming probability where $c=\frac{c'}{2}$.
\end{proof}

\begin{lemma}
	\label{Lem:E_nHatOverwhelming}
	Fix $\varepsilon>0$. For a fixed integer $m>0$, let $\xi_{1},\xi_{2},\dots \xi_{m}$ be real-valued random variables each mean zero, variance one, and finite $4+\tau$ moment for some $\tau>0$. Let ${X}_{n,1},{X}_{n,2},\dots,{X}_{n,m}$ be independent iid random matrices with atom variables $\xi_{1},\xi_{2},\dots,\xi_{m}$ respectively. Define $\hat{X}_{n,1},\hat{X}_{n,2},\dots\hat{X}_{n,m}$ as in \eqref{def:Xhat}, and define $\hat{P}_{n}$ as in \eqref{Def:TruncatedProducts}. For any $\delta>0$, there exists a constant $c>0$ depending only on $\delta$ such that 
	\[\inf_{|z|> 1+\delta/2}s_{mn}\left(\hat{P}_{n}-zI\right)\geq c\]
	with overwhelming probability.
\end{lemma}
\begin{proof}
	Fix $\delta>0$. By Lemma \ref{Lem:Omega_nOverwhelming}, we know that there exists some $c'>0$ such that $\inf_{|z|>1+\delta/2}s_{mn}\left(\blmat{Y}{n}-zI\right)\geq c'$ with overwhelming probability as well. Recall that $s_{mn}\left(\blmat{Y}{n}-zI\right)=s_{1}\left(\left(\blmat{Y}{n}-zI\right)^{-1}\right)$ provided $z$ is not an eigenvalue of $\blmat{Y}{n}$. A block inverse matrix calculation reveals that 
	\[\left(\left(\blmat{Y}{n}-zI\right)^{-1}\right)^{[1,1]}=z^{m-1}\left(\hat{P}_{n}-z^{m}I\right)^{-1}\]
	where the notation $A^{[1,1]}$ denotes the upper left $n\times n$ block of $A$. Therefore,
	\begin{equation*}
	\frac{1}{c'} \geq \sup_{|z|> 1+\delta/2}s_{1}\left(\left(\blmat{Y}{n}-zI\right)^{-1}\right)\geq \sup_{|z|> 1+\delta/2}|z|^{m-1}\lnorm\left(\hat{P}_{n}-z^{m}I\right)^{-1}\rnorm.
	\end{equation*}
	This implies that there exists a constant $c>0$ such that 
	\[\frac{1}{c}\geq \sup_{|z|> 1+\delta/2}s_{1}\left(\left(\hat{P}_{n}-zI\right)^{-1}\right)\] 
	with overwhelming probability. This gives $\inf_{|z| > 1+\delta/2}s_{n}\left(\hat{P}_{n}-zI\right)\geq c$ with overwhelming probability. 
\end{proof}

\begin{lemma}
	\label{Lem:E_nOverwhelming}
	For a fixed integer $m>0$, let $\xi_{1},\xi_{2},\dots \xi_{m}$ be real-valued random variables each satisfying Assumption \ref{assump:4PlusTau}. Fix $\delta >0$ and let $X_{n,1},X_{n,2},\dots X_{n,m}$ be independent iid random matrices with atom variables $\xi_{1},\xi_{2},\dots \xi_{m}$ respectively. Then there exists a constant $c>0$ depending only on $\delta$ such that 
	\[\inf_{|z|> 1+\delta/2}s_{n}\left(P_{n}/\sigma-zI\right)\geq c\]
	with probability $1-o(1)$ where $\sigma = \sigma_{1}\cdots \sigma_{m}$. 
\end{lemma}
\begin{proof}
	 By a simple rescaling, it is sufficient to assume that the variance of each random variable is 1 so that $\sigma=1$. Let $\delta>0$ and recall by Lemma \ref{Lem:E_nHatOverwhelming} there exists a $c'>0$ depending only on $\delta $ such that $\inf_{|z|> 1+\delta/2}s_{n}\left(\hat{P}_{n}-zI\right)\geq c'$
	with overwhelming probability. Then by Lemma \ref{Lem:ProductsCloseToTruncatedProducts},
	\begin{align*}
	&\P\left(\inf_{|z|> 1+\delta/2}s_{n}\left(P_{n}-zI\right)<\frac{c'}{2}\right)\\
	&\quad\quad = \P\left(\inf_{|z|> 1+\delta/2}s_{n}\left(P_{n}-zI\right)<\frac{c'}{2}\;\;\text{ and }\;\;\lnorm P_{n}-\hat{P}_{n}\rnorm\leq n^{-\varepsilon}\right)\\
	&\quad\quad\quad\quad + \P\left(\inf_{|z|> 1+\delta/2}s_{n}\left(P_{n}-zI\right)<\frac{c'}{2}\;\;\text{ and }\;\;\lnorm P_{n}-\hat{P}_{n}\rnorm> n^{-\varepsilon}\right)\\
	&\quad\quad \leq \P\left(\inf_{|z|> 1+\delta/2}s_{n}\left(P_{n}-zI\right)<\frac{c'}{2}\;\;\text{ and }\;\;\lnorm P_{n}-\hat{P}_{n}\rnorm\leq n^{-\varepsilon}\right)\\
	&\quad\quad\quad\quad + \P\left(\lnorm P_{n}-\hat{P}_{n}\rnorm> n^{-\varepsilon}\right)\\
	&\quad\quad \leq \P\left(\inf_{|z|> 1+\delta/2}s_{n}\left(P_{n}-zI\right)<\frac{c'}{2}\;\;\text{ and }\;\;\lnorm P_{n}-\hat{P}_{n}\rnorm\leq n^{-\varepsilon}\right)+o(1).
	\end{align*}
	Suppose that there exists a $z_{0}\in\C$ with $|z_{0}|\geq 1+\delta/2$ such that $s_{n}\left(P_{n}-z_{0}I\right)<\frac{c'}{2}$ and $\lnorm P_{n}-\hat{P}_{n}\rnorm <n^{-\varepsilon}<\frac{c'}{2}$. Then, by Weyl's inequality \eqref{Equ:weyl},\\ $\left|s_{n}(P_{n}-z_{0}I)-s_{n}(\hat{P}_{n}-z_{0}I)\right|<\frac{c'}{2}$ which implies $s_{n}(\hat{P}_{n}-z_{0}I)<c'$. Thus, for $n$ sufficiently large to ensure that $n^{-\varepsilon}<\frac{c'}{2}$, by Lemma \ref{Lem:ProductsCloseToTruncatedProducts}
	\begin{equation*}
	\P\left(\inf_{|z|> 1+\delta/2}s_{n}\left(P_{n}-zI\right)<\frac{c'}{2}\right)\leq\P\left(\inf_{|z|> 1+\delta/2}s_{n}\left(\hat{P}_{n}-zI\right)<c'\right)+o(1).
	\end{equation*}
Thus, selecting $c=\frac{c'}{2}$, we have $\inf_{|z|> 1+\delta/2}s_{n}\left(P_{n}-zI\right)\geq c$ with probability $1-o(1)$.
\end{proof}

\begin{lemma}\label{Lem:SingValColRemovedBoundedBelow}
	Let $A$ be an $n\times n$ matrix. Let $R$ be a subset of the integer set $\{1,2,\dots n\}$. Let $A^{(R)}$ denote the matrix $A$, but with the $r$th column replaced with zero for each $r\in R$. Then 
	\begin{equation*} 
	s_{n}\left(A^{(R)}-zI\right)\geq \min\{s_{n}(A-zI),|z|\}.
	\end{equation*}
\end{lemma}
\begin{proof}
	Let $A^{((R))}$ denote the matrix $A$ with column $r$ removed for all $r\in R$. Note that $A^{((R))}$ is an $n\times (n-|R|)$ matrix, which is distinct from the $n\times n$ matrix $A^{(R)}$. Also, let $I^{((R))}$ denote the $n\times n$ identity matrix with column $r$ removed for all $r\in R$. In order to bound the least singular value of $(A^{(R)}-zI)$, we will consider the eigenvalues of $\left(A-zI\right)^{*}\left(A-zI\right),$ $\left(A^{(R)}-zI\right)^{*}\left(A^{(R)}-zI\right),$ and $\left(A^{((R))}-zI^{((R))}\right)^{*}\left(A^{((R))}-zI^{((R))}\right).$
	
	Now, observe that $\left(A^{((R))}-zI^{((R))}\right)^{*}\left(A^{((R))}-zI^{((R))}\right)$ is an $(n-|R|)\times (n-|R|)$ matrix, and is a principle sub-matrix of the Hermitian matrix $(A-zI)^{*}(A-zI)$. Therefore, the eigenvalues of $\left(A^{((R))}-zI^{((R))}\right)^{*}\left(A^{((R))}-zI^{((R))}\right)$ must interlace with the eigenvalues of $\left(A-zI\right)^{*}\left(A-zI\right)$ by Cauchy's interlacing theorem \cite[Theorem 1]{HW}. This implies 
	\[s_{n}\left(A^{((R))}-zI^{((R))}\right)^{2}\geq s_{n}\left(A-zI\right)^{2}.\]
	Next, we compare the eigenvalues of $\left(A^{(R)}-zI\right)^{*}\left(A^{(R)}-zI\right)$ to the eigenvalues of $\left(A^{((R))}-zI^{((R))}\right)^{*}\left(A^{((R))}-zI^{((R))}\right)$. Note that, after a possible permutation of columns to move all zero columns of $A^{(R)}$ to be in the last $|R|$ columns, the product $\left(A^{(R)}-zI\right)^{*}\left(A^{(R)}-zI\right)$ becomes
	\[\left[\begin{array}{cc}
	\left(A^{((R))}-zI^{((R))}\right)^{*}\left(A^{((R))}-zI^{((R))}\right) & 0\cdot I_{|R|\times (n-|R|)}\\
	  &   \\
	0\cdot I_{(n-|R|)\times |R|} & |z|^{2}\cdot I_{|R|\times |R|}
	\end{array}\right].\]
	Due to the block structure of the matrix above, if $w$ is an eigenvalue of\\ $\left(A^{(R)}-zI\right)^{*}\left(A^{(R)}-zI\right)$, then either $w$ is an eigenvalue of\\ $\left(A^{((R))}-zI^{((R))}\right)^{*}\left(A^{((R))}-zI^{((R))}\right)$ or $w$ is $|z|^{2}$. Ergo,
	\begin{align*}
	s_{n}\left(A^{(R)}-zI\right)^{2} &= \min\left\{s_{n}\left(A^{((R))}-zI^{((R))}\right)^{2},\;|z|^{2}\right\}\\
	&\geq  \min\left\{s_{n}\left(A-zI\right)^{2},\;|z|^{2}\right\}
	\end{align*}
	which implies $s_{n}\left(A^{(R)}-zI\right)\geq  \min\left\{s_{n}\left(A-zI\right),\;|z|\right\}$ concluding the proof.
\end{proof}
This lemma gives way to the following two corollaries.

\begin{corollary}
	\label{Cor:Omega_nkOverwhelming}
	Fix $\varepsilon>0$. For a fixed integer $m>0$, let $\xi_{1},\xi_{2},\dots \xi_{m}$ be real-valued random variables each mean zero, variance one, and finite $4+\tau$ moment for some $\tau>0$.  Let ${X}_{n,1},{X}_{n,2},\dots,{X}_{n,m}$ be independent iid random matrices with atom variables $\xi_{1},\xi_{2},\dots,\xi_{m}$ respectively, and define $\hat{X}_{n,1},\hat{X}_{n,2},\dots\hat{X}_{n,m}$ as in \eqref{def:Xhat}. Define $\blmat{Y}{n}$ as in \eqref{Def:Y_n} and $\blmat{Y}{n}^{(k)}$ as $\blmat{Y}{n}$ with the columns $c_{k},c_{n+k},c_{2n+k},\dots,c_{(m-1)n+k}$ replaced with zeros. For any $\delta>0$, there exists a constant $c>0$ depending only on $\delta$ such that 
	\[\inf_{|z|> 1+\delta/2}s_{mn}\left(\blmat{Y}{n}^{(k)}-zI\right)\geq c\]
	with overwhelming probability. 
\end{corollary}
\begin{proof}
	Note that by Lemmas \ref{Lem:Omega_nOverwhelming} and \ref{Lem:SingValColRemovedBoundedBelow}, 
	\begin{align*}
	\inf_{|z|> 1+\delta/2}s_{mn}\left(\blmat{Y}{n}^{(k)}-zI\right)&\geq \inf_{|z|> 1+\delta/2} \min\left\{s_{mn}\left(\blmat{Y}{n}-zI\right),\;|z|\right\}\\
	&\geq \inf_{|z|> 1+\delta/2} \min\left\{s_{mn}\left(\blmat{Y}{n}-zI\right),\;1\right\}\\
	& \geq \min\left\{c',\;1\right\}
	\end{align*}
	with overwhelming probability for some constant $c'>0$ depending only on $\delta$. The result follows by setting $c=\min\left\{c',\;1\right\}$.
\end{proof}

\begin{corollary}
	\label{Cor:Omega_nksOverwhelming}
	Fix $\varepsilon>0$. For a fixed integer $m>0$, let $\xi_{1},\xi_{2},\dots \xi_{m}$ be real-valued random variables each mean zero, variance one, and finite $4+\tau$ moment for some $\tau>0$. Let $\hat{X}_{n,1},\hat{X}_{n,2},\dots,\hat{X}_{n,m}$ be independent iid random matrices with atom variables as defined in \eqref{def:Xhat}. Define $\blmat{Y}{n}$ as in \eqref{Def:Y_n} and $\blmat{Y}{n}^{(k,s)}$ as $\blmat{Y}{n}$ with the columns $c_{k},c_{n+k},c_{2n+k},\dots,c_{(m-1)n+k}$ and $c_{s}$ replaced with zeros. For any $\delta>0$, there exists a constant $c>0$ depending only on $\delta$ such that 
	\[\inf_{|z|> 1+\delta/2}s_{mn}\left(\blmat{Y}{n}^{(k,s)}-zI\right)\geq c\]
	with overwhelming probability.  
\end{corollary}
The proof of Corollary \ref{Cor:Omega_nksOverwhelming} follows in exactly the same way as the proof of Corollary \ref{Cor:Omega_nkOverwhelming}.

\section{Useful Lemmas}
\label{Sec:AppendexMisc}


\begin{lemma}[Lemma 2.7 from \cite{BScov}]
	For $X = (x_{1},x_{2},\ldots,x_{N})^{T}$ iid standardized complex entries, $B$ an $N\times N$  complex matrix, we have, for any $p\geq 2$,
	\begin{equation*}
	\E\left|X^{*}BX-\tr(B)\right|^{p}\leq K_{p}\left(\left(\E\left|x_{1}\right|^{4}\emph{tr} B^{*}B\right)^{p/2}+\E|x_{1}|^{2p}\emph{tr}(B^{*}B)^{p/2}\right)
	\end{equation*}
	where the constant $K_{p}>0$ depends only on $p$. 
	\label{Lem:BilinearFormsWithTrace}
\end{lemma}

\begin{lemma}
	Let $A$ be an $N\times N$ complex-valued matrix. Suppose that $\xi$ is a complex-valued random variable with mean zero and unit variance. Let $S\subseteq [N]$, and let $w=(w_{i})_{i=1}^{N}$ be a vector with the following properties:
	\begin{enumerate}[label=(\roman*)]
		\item $\{w_i : i \in S \}$ is a collection of iid copies of $\xi$, 
		\item $w_{i}=0$ for $i\not\in S$.
	\end{enumerate} 
	Additionally, $A_{S\times S}$ denote the $|S|\times|S|$ matrix which has entries $A_{(i,j)}$ for $i,j\in S$.
Then for any even $p\geq 2$,
\[\E\left|w^{*}Aw-\tr(A_{S\times S})\right|^{p}\ll_{p}\E\left|\xi \right|^{2p}\left(\tr(A^{*}A)\right)^{p/2}.\]
\label{Lem:BlockBilinearFormsWithTrace}
\end{lemma}
\begin{proof}
	Let $w_{S}$ denote the $|S|$-vector which contains entries $w_{i}$ for $i\in S$ and observe
	\begin{equation*}
	w^{*}Aw = \sum_{i,j}\bar{w}_{i}A_{(i,j)}w_{j}=  w_{S}^{*}A_{S\times S}w_{S}.
	\end{equation*}
	Therefore, by Lemma \ref{Lem:BilinearFormsWithTrace}, for any even $p\geq 2$,
	\begin{align*}
	\E\left|w^{*}Aw-\tr(A_{S\times S})\right|^{p}& = \E\left|w_{S}^{*}A_{S\times S}w_{S}-\tr(A_{S\times S})\right|^{p}\\
	&\ll_{p}\left(\E\left|\xi\right|^{4}\tr(A_{S\times S}^{*}A_{S\times S})\right)^{p/2}+\E\left|\xi\right|^{2p}\tr(A_{S\times S}^{*}A_{S\times S})^{p/2}\\
	&\ll_{p}\E\left|\xi\right|^{2p}\left(\tr(A_{S\times S}^{*}A_{S\times S})\right)^{p/2}.
	\end{align*}
	Now observe that 
	\[\tr(A_{S\times S}^{*}A_{S\times S})= \sum_{i,j\in S} A_{i,j}^{*}A_{j,i} \leq \sum_{i,j=1}^{N} A_{i,j}^{*}A_{j,i}= \tr(A^{*}A).\]
	Therefore 
	\begin{Details} 
		\begin{align*}
	\tr(A_{S\times S}^{*}A_{S\times S}) & = \sum_{i,j\in S} A_{i,j}^{*}A_{j,i}\\
	& = \sum_{i,j\in S} \overline{A}_{j,i}A_{j,i}\\
	& = \sum_{i,j\in S} \left|A_{j,i}\right|^{2}\\
	& \leq \sum_{i,j=1}^{N} \left|A_{j,i}\right|^{2}\\
	& \leq \sum_{i,j=1}^{N} A_{i,j}^{*}A_{j,i}\\
	& = \tr(A^{*}A).
	\end{align*}
	\end{Details}
\begin{equation*}
\E\left|w^{*}Aw-\tr(A_{S\times S})\right|^{p}\ll_{p}\E\left|\xi\right|^{2p}\left(\tr(A_{S\times S}^{*}A_{S\times S})\right)^{p/2}\leq \E\left|\xi\right|^{2p}\left(\tr(A^{*}A)\right)^{p/2}.
\end{equation*}
\end{proof}

\begin{lemma}[Lemma A.1 from \cite{BScov}]
	For $X = (x_{1},x_{2},\ldots,x_{N})^{T}$ iid standardized complex entries, $B$ an $N\times N$ complex-valued Hermitian nonnegative definite matrix, we have, for any $p\geq 1$,
	\begin{equation*}
	\E\left|X^{*}BX\right|^{p}\leq K_{p}\left(\left(\emph{tr} B\right)^{p}+\E|x_{1}|^{2p}\emph{tr}B^{p}\right).
	\end{equation*}
	where $K_{p}>0$ depends only on $p$. 
	\label{Lem:BilinearForms}
\end{lemma}



\begin{lemma} \label{Lem:ConjLessThanPartialTrace}
	Let $A$ be an $N\times N$ Hermitian positive semidefinite matrix. Suppose that $\xi$ is a complex-valued random variable with mean zero and unit variance. Let $S\subseteq [N]$, and let $w = (w_i)_{i=1}^N$ be a vector with the following properties:  
	\begin{enumerate}[label=(\roman*)]
		\item $\{w_i : i \in S \}$ is a collection of iid copies of $\xi$, 
		\item $w_{i}=0$ for $i\not\in S$.
	\end{enumerate} 
	Then for any $p\geq 2$,
	\begin{equation}
	\E\left|w^{*}Aw\right|^{p} \ll_{p}\E|\xi|^{2p}\left( \tr A\right)^{p}.
	\end{equation} 
\end{lemma}

\begin{proof}
	Let $w_{S}$ denote the $|S|$-vector which contains entries $w_{i}$ for $i\in S$, and let $A_{S\times S}$ denote the $|S|\times|S|$ matrix which has entries $A_{(i,j)}$ for $i,j\in S$. Then we have
	\begin{equation*}
	w^{*}Aw = \sum_{i,j}\bar{w}_{i}A_{(i,j)}w_{j}=  w_{S}^{*}A_{S\times S}w_{S}.
	\end{equation*}
	By Lemma \ref{Lem:BilinearForms}, we get
	\begin{equation*}
	\E\left|w^{*}Aw\right|^{p} \ll_{p}\left( \tr A_{S\times S}\right)^{p}+\E|\xi|^{2p}\tr A_{S\times S}^{p}.
	\end{equation*}
	Since $A$ is non-negative definite, the diagonal elements are non-negative so that $\tr(A_{S\times S}^{p})\leq (\tr(A_{A\times A}))^{p}$. By this and the fact that for a Hermitian positive semidefinite matrix, the partial trace is less than or equal to the full trace, we observe that 
	\begin{equation*}
	\left( \tr A_{S\times S}\right)^{p}+\E|\xi|^{2p}\tr A_{S\times S}^{p} \ll_{p} \E|\xi|^{2p}\left( \tr A_{S\times S}\right)^{p} \ll_{p} \E|\xi|^{2p}\left( \tr A\right)^{p}.
	\end{equation*}
\end{proof}

\begin{Details}
\begin{lemma} \label{lem:BilinearFormWithDifferentVectorsBig}
	Let $A$ be a deterministic complex $mn\times mn$ matrix for some fixed $m>0$. Suppose that $\xi$ is a complex-valued random variable with mean zero, unit variance, and which satisfies $|\xi| \leq n^{-\varepsilon}$ almost surely for some $\varepsilon > 0$. Let $S,R\subseteq [mn]$, and let $w = (w_i)_{i=1}^{mn}$ and $t = (t_i)_{i=1}^{mn}$ be independent vectors with the following properties:  
	\begin{enumerate}[label=(\roman*)]
		\item $\{w_i : i \in S \}$ and $\{t_j : j \in R \}$ are collections of iid copies of $\xi$, 
		\item $w_{i}=0$ for $i\not\in S$, and  $t_{j}=0$ for $j\not\in R$.
	\end{enumerate} 
	Then for any $p\geq 1$,
	\begin{equation}
	\E\left|w^{*}At\right|^{p}\ll_{m,p}n^{-p\varepsilon}(\tr(A^{*}A))^{p/2}.
	\end{equation}
\end{lemma}
\begin{proof}
	Let $w_{S}$ denote the $|S|$-vector which contains entries $w_{i}$ for $i\in S$, and let $t_{R}$ denote the $|R|$-vector which contains entries $t_{j}$ for $j\in R$.  For an $N \times N$ matrix $B$, we let $B_{S \times S}$ denote the $|S| \times |S|$ matrix with entries $B_{(i,j)}$ for $i,j \in S$.  Similarly, we let $B_{R \times R}$ denote the $|R| \times |R|$ matrix with entries $B_{(i,j)}$ for $i,j \in R$.  
	
	We first note that, by the Cauchy--Schwarz inequality, it suffices to assume $p$ is even.  In this case, since $w$ is independent of $t$, Lemma \ref{Lem:BilinearForms} implies that
	\begin{align*}
	\E|w^{*}At|^{p} &= \E|w^{*}Att^{*}A^{*}w|^{p/2}\\
	&= \E\left|w^{*}_{S}(Att^{*}A^{*})_{S\times S}w_{S}\right|^{p/2}\\
	&\ll_{p}\E\left[\left(\tr(Att^{*}A^{*})_{S\times S}\right)^{p/2}+(n^{-\varepsilon})^{p}\tr(Att^{*}A^{*})_{S\times S}^{p/2}\right]. 
	\end{align*}
	Recall that for any matrix $B$, $\tr(B^{*}B)^{p/2}\leq (\tr(B^{*}B))^{p/2}$. By this and by the fact that for a Hermitian positive semidefinite matrix, the partial trace is less than or equal to the full trace, we observe that 
	\begin{equation*} 
	\E\left[\left(\tr(Att^{*}A^{*})_{S\times S}\right)^{p/2}+(n^{-\varepsilon})^{p}\tr(Att^{*}A^{*})_{S\times S}^{p/2}\right]\ll_{p}n^{-p\varepsilon}\E\left[(\tr(Att^{*}A^{*}))^{p/2}\right].
	\end{equation*}
	By a cyclic permutation of the trace, we have 
	\begin{equation*}
	\E\left[(\tr(Att^{*}A^{*}))^{p/2}\right] = \E\left[(t^{*}A^{*}At)^{p/2}\right]\leq\E\left|t^{*}A^{*}At\right|^{p/2}.
	\end{equation*}
	By Lemma \ref{Lem:BilinearForms} and a similar argument as above, we have
	\begin{align*}
	\E\left|t^{*}A^{*}At\right|^{p/2}&=\E\left|t_{R}^{*}(A^{*}A)_{R\times R}t_{R}\right|^{p/2}\\
	&\ll_{p}(\tr(A^{*}A)_{R\times R})^{p/2}+(n^{-\varepsilon})^{p}\tr(A^{*}A)_{R\times R}^{p/2}\\
	&\ll_{p}n^{-p\varepsilon}(\tr(A^{*}A))^{p/2}, 
	\end{align*}
	completing the proof.
\end{proof}
\end{Details}

\begin{Details}
	\begin{lemma}
		Suppose that $X_{n}$ and $Y_{n}$ are two sequences of random variables and assume that $Y_{n}$ converges in distribution to a random variable $Y$ as $n\rightarrow\infty$. If any of the following hold:
		\begin{enumerate}
			\item For any $\varepsilon >0$, $\P(| X_{n}-Y_{n}| >\varepsilon)=o(1)$,
			\item $\P(X_{n}=Y_{n})=1-o(1)$,
			\item $\E| X_{n}-Y_{n}|^{p} =o(1)$ for $p\geq 1$,
		\end{enumerate}
		then $X_{n}$ converges in distribution to $Y$ as $n\rightarrow\infty$ as well.
		\label{Lem:ConvergenceInDistCharacterizations}
	\end{lemma}
	\begin{proof}
		These all imply $|X_{n}-Y_{n}|\rightarrow0$ in distribution as $n\rightarrow\infty$. Thus the claim follows by \cite[Theorem 25.4]{Bill}.
	\end{proof}

\begin{lemma}[Theorem 3.2 from \cite{Bmart}]
	Let $\{X_{i}\}_{i=1}^{N}$ be a complex martingale difference sequence with respect to the filtration $\{\mathcal{F}_{i}\}_{i=1}^{N}$. Then
	\begin{equation*}
	\E\left|\sum_{i=1}^{N}X_{i}\right|^{2}\ll \sum_{i=1}^{N}\E|X_{i}|^{2}.
	\end{equation*} 	
	\label{Lem:MDSSumOfSquares}
\end{lemma}
\end{Details}

\begin{lemma}
	Let $A$ and $B$ be $n\times n$ matrices. Then
	\begin{equation*}
	\left|\tr(AB)\right|\leq \sqrt{n}\lnorm AB\rnorm _{2}\leq \sqrt{n}\lnorm A\rnorm \cdot \lnorm B\rnorm_{2}. 
	\end{equation*}
	\label{Lem:BoundOnTraceByNorms}
\end{lemma}
\begin{proof}
	This follows by an application of the Cauchy--Schwarz inequality and an application of  \cite[Theorem A.10]{BSbook}.
	\begin{Details} 
		Observe
		\begin{align*}
		|\tr(AB)| &= \left|\sum_{i=1}^{n}(AB)_{i,i}\right|\\
		&=\sqrt{\left|\sum_{i=1}^{n}(AB)_{i,i}\right|^{2}}\\
		&\leq \sqrt{\sum_{i,j=1}^{n}\left|(AB)_{i,j}\right|^{2}}\\
		&=\lnorm AB\rnorm_{2}
		\end{align*}
		proving the first inequality. The second inequality follows since by \cite[Theorem A.10]{BSbook},
		\begin{align*}
		\lnorm AB\rnorm_{2} &=\sqrt{\sum_{i=1}^{n}s_{i}(AB)^{2}}\\
		&\leq \sqrt{s_{1}(A)^{2}\sum_{i=1}^{n}s_{i}(B)^{2}}\\
		&=\lnorm A\rnorm \lnorm B\rnorm_{2}.
		\end{align*}
	\end{Details}
\end{proof}

\begin{Details} 
\section{Proofs of Remark \ref{Remark:ImprovementOnVGU}}
\label{sec:ProofRemark}
\begin{lemma}
	Let $U_{k}$ be the $mn\times m$ matrix which contains as its columns $c_{k},c_{n+k},\dots,c_{(m-1)n+k}$, and define $V_{k}$ to be the $mn\times m$ matrix which contains as its columns $e_{k},e_{n+k},\dots, e_{(m-1)n+k}$ where $e_{1},\dots e_{mn}$ denote the standard basis elements of $\C^{mn}$. Let $\mathcal{G}_{n}^{(k)}(z)$ be defined as in \eqref{Def:G^{(k)}_{n}}. Then 
	\begin{equation*}
	\E\lnorm V_{k}^{T}\mathcal{G}_{n}^{(k)}(z)U_{k}\oindicator{\Omega_{n,k}}\rnorm^{2} \ll n^{-1}.
	\end{equation*}
	\label{Lem:BoundingVGU2}
\end{lemma}
\begin{proof}
	Begin by observing that 
	\begin{align*}
	&\E\lnorm V_{k}^{T}\mathcal{G}_{n}(z)U_{k}\oindicator{\Omega_{n,k}}\rnorm^{2}\\
	&\quad\quad \ll \max_{1\leq i,j\leq m} \E\left|  (V_{k}^{T}\mathcal{G}_{n}^{(k)}(z)U_{k})_{(i,j)}\oindicator{\Omega_{n,k}}\right|^{2}\\
	&\quad\quad= \max_{1\leq i,j\leq m}
	\E\left|e_{(i-1)n+k}\mathcal{G}_{n}^{(k)}(z)c_{(j-1)n+k}\oindicator{\Omega_{n,k}}\right|^{2}\\
	&\quad\quad = \max_{1\leq i,j\leq m} \E\left|c_{(j-1)n+k}^{*}(\mathcal{G}_{n}^{(k)}(z))^{*}e_{(i-1)n+k}e_{(i-1)n+k}^{T}\mathcal{G}_{n}(z)c_{(j-1)n+k}\oindicator{\Omega_{n,k}}\right|.
	\end{align*}
	By Lemma \ref{Lem:conjLessThanConstant}, and since the rank of $(\mathcal{G}_{n}^{(k)}(z))^{*}e_{(i-1)n+k}e_{(i-1)n+k}^{T}\mathcal{G}_{n}(z)$ is at most 1, for any $1\leq j\leq m$ we have 
	\begin{align*}
	&\E\left|c_{(j-1)n+k}^{*}(\mathcal{G}_{n}^{(k)}(z))^{*}e_{(i-1)n+k}e_{(i-1)n+k}^{T}\mathcal{G}_{n}(z)c_{(j-1)n+k}\oindicator{\Omega_{n,k}}\right|\\
	&\quad\quad\ll n^{-1}\lnorm(\mathcal{G}_{n}^{(k)}(z))^{*}e_{(i-1)n+k}e_{(i-1)n+k}^{T}\mathcal{G}_{n}(z)\oindicator{\Omega_{n,k}}\rnorm\\
	&\quad\quad \ll n^{-1}.
	\end{align*}
	Thus,
	\[\E\lnorm V_{k}^{T}\mathcal{G}_{n}(z)U_{k}\oindicator{\Omega_{n,k}}\rnorm^{2}\ll n^{-1}\]
	as advertised. 
\end{proof}

\begin{lemma}
	Let $U_{k}$ be the $mn\times m$ matrix which contains as its columns $c_{k},c_{n+k},\dots,c_{(m-1)n+k}$, and define $V_{k}$ to be the $mn\times m$ matrix which contains as its columns $e_{k},e_{n+k},\dots, e_{(m-1)n+k}$ where $e_{1},\dots e_{mn}$ denote the standard basis elements of $\C^{mn}$. Let $\mathcal{G}_{n}^{(k)}(z)$ be defined as in \eqref{Def:G^{(k)}_{n}}. Then 
	\begin{equation*}
	\E\lnorm V_{k}^{T}(\mathcal{G}_{n}^{(k)}(z))^{2}U_{k}\oindicator{\Omega_{n,k}}\rnorm ^{2}\ll n^{-1}.
	\end{equation*}
	\label{Lem:BoundingVG^2U2}
\end{lemma}
\begin{proof}
	By the same argument as in the proof of Lemma \ref{Lem:BoundingVGU2} above, we have  
	\begin{align*}
	&\E\lnorm V_{k}^{T}(\mathcal{G}_{n}(z))^{2}U_{k}\oindicator{\Omega_{n,k}}\rnorm^{2}\\
	&\quad\quad \ll \max_{1\leq i,j\leq m} \E\left|c_{(j-1)n+k}^{*}(\mathcal{G}_{n}^{(k)}(z))^{2*}e_{(i-1)n+k}e_{(i-1)n+k}^{T}(\mathcal{G}_{n}(z))^{2}c_{(j-1)n+k}\oindicator{\Omega_{n,k}}\right|.
	\end{align*}
	By Lemma \ref{Lem:conjLessThanConstant}, and since the rank of $(\mathcal{G}_{n}^{(k)}(z))^{2*}e_{(i-1)n+k}e_{(i-1)n+k}^{T}(\mathcal{G}_{n}(z))^{2}$ is at most 1, for any $1\leq j\leq m$ we have 
	\begin{align*}
	&\E\left|c_{(j-1)n+k}^{*}(\mathcal{G}_{n}^{(k)}(z))^{2*}e_{(i-1)n+k}e_{(i-1)n+k}^{T}(\mathcal{G}_{n}(z))^{2}c_{(j-1)n+k}\oindicator{\Omega_{n,k}}\right|\\
	&\quad\quad\ll n^{-1}\lnorm(\mathcal{G}_{n}^{(k)}(z))^{2*}e_{(i-1)n+k}e_{(i-1)n+k}^{T}(\mathcal{G}_{n}(z))^{2}\oindicator{\Omega_{n,k}}\rnorm\\
	&\quad\quad \ll n^{-1}.
	\end{align*}
	Thus,
	\[\E\lnorm V_{k}^{T}(\mathcal{G}_{n}(z))^{2}U_{k}\oindicator{\Omega_{n,k}}\rnorm^{2}\ll n^{-1}\]
	as advertised. 
\end{proof}
\end{Details}


\begin{thebibliography}{99}




\bibitem{ARRS} K. Adhikari, N. Kishore Reddy, T. Ram Reddy, K. Saha, \emph{Determinantal point processes in the plane from products of random matrices}, Ann. Inst. Henri Poincar\'{e} Probab. Stat., 52(1):16–46, 2016. 

\bibitem{AB} G. Akemann, Z. Burda, \emph{Universal microscopic correlation functions for products of independent Ginibre matrices}, J. Phys. A: Math. Theor. \textbf{45} (2012).

\bibitem{ABK} G. Akemann, Z. Burda, M. Kieburg, \emph{Universal distribution of Lyapunov exponents for products of Ginibre matrices}, J. Phys. A: Math. Theor. 47(39) (2014).

\bibitem{AIK} G. Akemann, J. R. Ipsen, M. Kieburg, \emph{Products of rectangular random matrices: singular values and progressive scattering}, Phys. Rev. E \textbf{88}, (2013).  

\bibitem{AIK2} G. Akemann, J. R. Ipsen, E. Strahov, \emph{Permanental processes from products of complex and quaternionic induced Ginibre ensembles}, Random Matrices: Theory and Applications Vol. 3, No. 4 (2014) 1450014.   

\bibitem{AKW} G. Akemann, M. Kieburg, L. Wei, \emph{Singular value correlation functions for products of Wishart random matrices}, J. Phys. A: Math. Theor. \textbf{46} (2013).

\bibitem{AS} G. Akemann, E. Strahov, \emph{Hole probabilities and overcrowding estimates for products of complex Gaussian matrices}, J. Stat. Phys. (2013), Volume 151, Issue 6, pp 987--1003.  


\bibitem{A}  G. Anderson, \emph{Convergence of the largest singular value of a polynomial in independent Wigner matrices}, Ann. Probab., Volume 41, Number 3B (2013), 2103--2181.

\bibitem{AZ} G. Anderson, O. Zeitouni, \emph{CLT for a band matrix model}, Probab. Theory and Related Fields, vol. 134 (2006), 283--338.



\bibitem{Bcirc} Z.~D.~Bai, {\it Circular law}, Ann. Probab. \textbf{25} (1997), 494--529. 

\bibitem{BS:CLT} Z. D. Bai, J.W. Silverstein, {\em CLT for linear spectral statistic of large-dimensional sample covariance matrix}, Ann. Probab. Volume 32 (2004), 553-605.

\bibitem{BScov} Z. D. Bai, J. Silverstein, {\em No eigenvalues outside the support of the limiting spectral distribution of large-dimensional sample covariance matrices}, Ann. Probab. Volume 26, Number 1 (1998), 316-345.

\bibitem{BSbook} Z. D. Bai, J. Silverstein, {\em Spectral analysis of large dimensional random matrices}, Mathematics Monograph Series \textbf{2}, Science Press, Beijing 2006.









\bibitem{Bhatia} R. Bhatia, \emph{Matrix Analysis}, Graduate Texts in Mathematics, Springer-Verlag New York, 1997.  

\bibitem{Bill} P. Billingsley, \emph{Probability and Measure}, 3rd Edition, Wiley Series in Probability and Mathematical Statistics. Wiley, New York, 1995.

\bibitem{BillOld} P. Billingsley, \emph{Convergence of Probability Measures}, 1st Edition, John Wiley and Sons, Inc. 1968.

\bibitem{B} C. Bordenave, {\em On the spectrum of sum and product of non-hermitian random matrices}, Elect. Comm. in Probab. \textbf{16} (2011), 104--113.



\bibitem{BC} C. Bordenave, D. Chafa\"i, {\em Around the circular law}, Probability Surveys \textbf{9} 1--89, (2012).

\bibitem{BJW} Z. Burda, R. A. Janik, B. Waclaw, \emph{Spectrum of the product of independent random Gaussian matrices}, Phys. Rev. E \textbf{81} (2010).

\bibitem{BJLNS} Z. Burda, A. Jarosz, G. Livan, M. A. Nowak, A. Swiech, \emph{Eigenvalues and singular values of products of rectangular Gaussian random matrices}, Phys. Rev. E \textbf{82} (2010).  

\bibitem{BNS} Z. Burda, M. A. Nowak, A. Swiech, \emph{Spectral relations between products and powers of isotropic random matrices}, Phys. Rev. E \textbf{86}, 061137 (2012)

\bibitem{Bsurv} Z. Burda, \emph{Free products of large random matrices - a short review of recent developments}, J. Phys.: Conf. Ser. \textbf{473} 012002 (2013).





\bibitem{COW} N. Coston, S. O'Rourke, P. Wood, \emph{Outliers in the spectrum for products of independent random matrices}, available at {\tt arXiv:1711.07420}.

\bibitem{D} C. Y. Deng, \emph{A generalization of the Sherman--Morrison--Woodbury formula}, Applied Mathematics Letters, Volume 24, Issue 9, Sept. 2011, 1561--1564.

\bibitem{DS} P. Diaconis, M. Shahshahani, \emph{On the eigenvalues of random matrices}, J. Appl. Probab. 31A (1994), 49-62.

\bibitem{DE} P. Diaconis, S.N. Evans, \emph{Linear functionals of eigenvalues of random matrices}, Trans. Amer. Math. Soc. vol. 353, no. 7 (2001), 2615--2633.



\bibitem{Ed-cir} A.~Edelman, {\it The probability that a random real Gaussian matrix has $k$ real eigenvalues, related distributions, and the circular law}, J. Multivariate Anal. \textbf{60}, 203--232 (1997). 


\bibitem{F} P. J. Forrester, \emph{Lyapunov exponents for products of complex Gaussian random matrices}, J. Stat. Phys. \textbf{151}, 796-808 (2013).

\bibitem{F2} P. J. Forrester, \emph{Probability of all eigenvalues real for products of standard Gaussian matrices}, J. Phys. A, vol. 47, 065202 (2014).

\bibitem{Gi} J. ~Ginibre, {\it Statistical ensembles of complex, quaternion, and real matrices}, J. Math. Phys. \textbf{6} (1965), 440--449.

\bibitem{G1} V.~L.~Girko, {\it Circular law}, Theory Probab. Appl. \textbf{29} (1984), 694--706.

\bibitem{G2} V.~L.~Girko, \emph{The circular law}, Teor. Veroyatnost. i Primenen. \textbf{29} (4): 669--679 (1984).  


\bibitem{G4} V.~L. Girko, A. Vladimirova, \emph{L.I.F.E.: and Halloween Law}, Random Operators and Stochastic Equations, 18(4), pp. 327-353 (2010).


\bibitem{GNT} F.~G\"otze, A. Naumov, T.~Tikhomirov, \emph{On local laws for non-Hermitian random matrices and their products}, Doklady Mathematics. 96. 10.1134/S1064562417060072. 

\bibitem{GTcirc} F.~G\"otze, T.~Tikhomirov, \textit{The circular law for random matrices}, Ann. Probab. Volume 38, Number 4 1444--1491 (2010).  

\bibitem{GTprod} F.~G\"otze, T.~Tikhomirov, \emph{On the asymptotic spectrum of products of independent random matrices}, available at {\tt arXiv:1012.2710}.

\bibitem{HJ} R.~A.~Horn, C.~R.~Johnson, \textit{Matrix Analysis}, Cambridge Univ. Press (1991).  

\bibitem{HW} S. Hwang, \emph{Cauchy's interlace theorem for eigenvalues of Hermitian matrices}, The American Mathematical Monthly, vol. 111, no. 2, 2004, pp. 157–-159.

\bibitem{I} J.R. Ipsen, \emph{Products of independent Gaussian random matrices}, Bielefeld: Bielefeld University; 2015.

\bibitem{IK} J.R. Ipsen, M. Kieburg, \emph{Weak commutation relations and eigenvalue statistics for products of rectangular random matrices}, Phys. Rev. E, 89:032106, Mar 2014.

\bibitem{J} K. Johansson, \emph{On fluctuations of eigenvalues of random Hermitian matrices}, Duke Math. J., vol. 91 (1998), 151--204.



\bibitem{Kphil} P. Kopel, \emph{Linear Statistics of Non-Hermitian Matrices Matching the Real or Complex Ginibre Ensemble to Four Moments}, available at {\tt arXiv:1510.02987 [math.PR]}.  

\bibitem{KOV} P. Kopel, S. O'Rourke, V. Vu \emph{Random matrix products: Universality and least singular values}, available at {\tt arXiv 1802.03004}. 

\bibitem{KZ} A. B. J. Kuijlaars, L. Zhang, \emph{Singular values of products of Ginibre random matrices, multiple orthogonal polynomials and hard edge scaling limits}, Communications in Mathematical Physics, Volume 332, Issue 2, pp 759-781 (2014).  

\bibitem{LP} A. Lytova, L. Pastur, \emph{Central limit theorem for linear eigenvalue statistics of random matrices with independent entries}, Annals of Probability, vol. 37 (2009), 1778--1840.

\bibitem{M} M. L.~Mehta,  {\em Random matrices and the statistical theory of energy levels}, Acad. Press (1967).

\bibitem{M:B} M.~L.~Mehta, {\it Random Matrices}, third edition.  Elsevier/Academic Press, Amsterdam (2004).

\bibitem{N} Y. Nemish, \emph{No outliers in the spectrum of the product of independent non-Hermitian random matrices with independent entries}, J. Theor. Probab. (2018) 31: 402.

\bibitem{N2} Y. Nemish, \emph{Local law for the product of independent non-Hermitian random matrices with independent entries}, Electron. J. Probab. 22 (2017), no. 22, 1--35.

\bibitem{NP} I. Nourdin, G. Peccati, \emph{Universal Gaussian fluctuations of non-Hermitian matrix ensembles: from weak convergence to almost sure CLTs}, Latin American journal of probability and mathematical statistics 7, 341--375 (2010).

\bibitem{OR:CLT} S. O'Rourke, D. Renfrew, \emph{Central limit theorem for linear eigenvalue statistics of elliptic random matrices}, Journal of Theoretical Probability, Volume 29, Issue 3 (2016), pp. 1121--1191.

\bibitem{OR} S. O'Rourke, D. Renfrew, \emph{Low rank perturbations of large elliptic random matrices}, 	Electronic Journal of Probability Vol. 19, No. 43, 1--65 (2014). 

\bibitem{ORSV} S. O'Rourke, D. Renfrew, A. Soshnikov, V. Vu, \emph{Products of independent elliptic random matrices}, J. Stat. Phys., Vol. 160, No. 1 (2015), 89--119.

\bibitem{OS} S. O'Rourke, A. Soshnikov, \emph{Products of independent non-Hermitian random matrices}, Electronic Journal of Probability, Vol. 16, Art. 81, 2219--2245, (2011). 


\bibitem{PZ} G.~Pan, W.~Zhou, {\it Circular law, extreme singular values and potential theory}, Journal of Multivariate Analysis \textbf{101}, 645--656 (2010).






\bibitem{RiS} B. Rider, J.W. Silverstein, \emph{Gaussian fluctuations for non-Hermitian random matrix ensembles}, Ann. Probab. 34 (2006), 2118–2143.

\bibitem{Sh} M. Shcherbina, \emph{Central limit theorem for linear eigenvalue statistics of the Wigner and sample covariance random matrices}, Zh. Mat. Fiz. Anal. Geom., 7:2 (2011), 176--192.

\bibitem{SS} Y. Sinai, A. Soshnikov, \emph{Central limit theorem for traces of large random symmetric matrices with independent matrix elements}, Bol. Soc. Brasil. Mat. (N.S.), vol. 29 (1998), 1--24.

\bibitem{So} A. Soshnikov, \emph{The central limit theorem for local linear statistics in classical compact groups and related combinatorial identities}, Ann. Probab., vol. 28 (2000), 1353--1370.

\bibitem{SW} P. Sosoe, P. Wong, \emph{Regularity conditions in the CLT for linear eigenvalue statistics of Wigner matrices}, Advances in Mathematics, Vol. 249, 20, Dec. 2013, pp. 37--87.

\bibitem{S} E. Strahov, \emph{Differential equations for singular values of products of Ginibre random matrices}, J. Phys. A: Math. Theor. \textbf{47} 325203, 2014.  

\bibitem{Tout} T.~Tao, \emph{Outliers in the spectrum of iid matrices with bounded rank perturbations}, Probab. Theory Related Fields \textbf{155} (2013), 231--263. 


\bibitem{TVcirc}  T. Tao, V. Vu, \emph{Random matrices: The circular law}, Communication in Contemporary Mathematics \textbf{10} (2008), 261--307.

\bibitem{TVbull} T. Tao, V. Vu, {\it  From the Littlewood-Offord problem to the circular law: universality of the spectral distribution of random matrices}, {Bull.~Amer.~Math.~Soc.}  (N.S.) \textbf{46} (2009), no. 3, 377--396.

\bibitem{TVesd} T.~Tao, V.~Vu, \emph{Random matrices: Universality of ESDs and the circular law}, Ann. Probab. Volume 38, Number \textbf{5} (2010), 2023--2065. 



\end{thebibliography}
\end{document}